\documentclass[12pt]{amsart}
\usepackage[english]{babel}
\usepackage[utf8x]{inputenc}
\usepackage[T1]{fontenc}
\usepackage[]{float}
\usepackage[margin=2cm]{geometry}
\usepackage{empheq}
\usepackage{amsmath}
\usepackage{amsthm}
\usepackage{amssymb}
\usepackage{graphicx}
\usepackage[colorinlistoftodos]{todonotes}
\usepackage[colorlinks=true, allcolors=blue]{hyperref}

\newtheorem{theorem}{Theorem}[section]
\newtheorem{proposition}[theorem]{Proposition}
\newtheorem{corollary}{Corollary}[theorem]
\newtheorem{lemma}[theorem]{Lemma}
\newtheorem{remark}{Remark} [section]

\DeclareMathSymbol{,}{\mathpunct}{operators}{"2C}
\DeclareMathSymbol{.}{\mathpunct}{operators}{"2E}
\numberwithin{equation}{section}

\newcommand{\be}{\begin{equation}}
\newcommand{\ee}{\end{equation}}
\newcommand{\ep}{\varepsilon}
\newcommand{\p}{\partial}

\newcommand{\R}{\mathbb{R}}
\newcommand{\Z}{\mathbb{Z}}
\newcommand{\C}{\mathbb{C}}

\newcommand{\CH}{\mathcal{H}}
\newcommand{\CB}{\mathcal{B}}

\newcommand{\CD}{\mathcal{D}}
\newcommand{\CP}{\mathcal{P}}

\newcommand{\N}{\mathbb{N}}

\newcommand{\CE}{\mathcal{E}}
\newcommand{\CU}{\mathcal{U}}

\newcommand{\CG}{\mathcal{G}}
\newcommand{\CS}{\mathcal{S}}

\newcommand{\CL}{\mathcal{L}}

\newcommand{\CF}{\mathcal{F}}

\newcommand{\BBT}{\mathbb{T}}
\newcommand{\BBF}{\mathbb{F}}
\newcommand{\BBB}{\mathbb{B}}
\newcommand{\CK}{\mathcal{K}}
\newcommand{\CT}{\mathcal{T}}
\newcommand{\wt}{\widetilde}

\newcommand{\BA}{\mathbf{A}}

\newcommand{\BBL}{\mathbb{L}}
\newcommand{\BL}{\mathbf{L}} 
\newcommand{\BH}{\mathbf{H}} 
\newcommand{\BQ}{\mathbf{Q}} 
\newcommand{\CA}{\mathcal{A}} 
\newcommand{\CO}{\mathcal{O}}
\newcommand{\CM}{\mathcal{M}}
\newcommand{\BD}{\mathbf{D}}

\newcommand{\Vc}{\vec{\mathbf{c}}}

\begin{document}

\allowdisplaybreaks

\title[Invariant manifolds]{Invariant Manifolds for Capillary Waves and a Class of  quasilinear PDE\MakeLowercase{s}}

\author[J. Shatah]{Jalal Shatah}
\address[Jalal Shatah]{Courant Institute of Mathematical Sciences, New York University, New York, NY 10012 USA.}
\email{shatah@cims.nyu.edu}
\thanks{JS was supported in part by the Simons Collaboration Grant on Wave Turbulence}

\author[C. Zeng]{Chongchun Zeng}
\address[Chongchun Zeng]{School of Mathematics, Georgia Institute of Technology, Atlanta, GA 30332}
\email{zengch@math.gatech.edu}
\thanks{CZ was supported in part by the National  Science Foundation grant  DMS-2350115.}

\begin{abstract}

This paper studies the local stable and unstable manifolds of equilibria for quasilinear and fully nonlinear PDEs. These manifolds are fundamental objects in the analysis of local dynamics. While their existence is well understood for ODEs, semilinear PDEs, and certain parabolic-type quasilinear PDEs, invariant manifold theorems are often unavailable for quasilinear PDEs whose nonlinearities involve a loss of regularity and whose linear parts do not provide sufficient smoothing.

Our main results establish the existence, uniqueness, and smoothness  of local stable and unstable manifolds for nonlinear PDEs that satisfy suitable energy estimates. With the main focus on irrotational water waves with surface tension, this framework applies to a broad class of  PDEs, including nonlinear Schrödinger equations, nonlinear wave equations, and the MMT model, as well as to certain gradient-type PDEs.
\end{abstract}

\maketitle


\section{Introduction} \label{S:Intro}


Motivated by the problem of invariant manifolds in water wave systems, we develop a general framework based on energy estimate to construct local stable and unstable manifolds of equilibria for a broad class of nonlinear PDEs. Within this framework, we study irrotational water waves with surface tension in detail.

Consider an evolution equation in $\R^n$ with an equilibrium 
\be \label{E:DE-0}
u_t = F(u), \quad \text{ where } \ F(0)=0.
\ee
To study the local dynamics near the equilibrium, one often starts with the linearization $A= \BD F(0)$. 
Fundamental structures in the linear dynamics $e^{tA}$ include invariant subspaces, which often come from the spectral decomposition of $A$. 
A natural question is whether the invariant subspaces of $A$ deform into locally invariant manifolds in the nonlinear dynamics \eqref{E:DE-0}.  
More specifically, suppose there exist subspaces $X_\pm \subset \R^n$ invariant under $A$ such that $\R^n = X_+\oplus X_-$. 
Do there exist manifolds $W^\pm$ locally invariant under \eqref{E:DE-0} such that $\{0\} = W^+ \cap W^-$ and the tangent spaces $T_0 W^\pm =X_\pm$? Such results have been proved under the exponential dichotomy condition
\be \label{E:ED-0} 
\exists \lambda _+ > \lambda_- \ \text{ such that }\;  \sup_{t\le 0} e^{-\lambda_+ t} \big|e^{tA}|_{X_+} \big|, \ \sup_{t\ge 0} e^{-\lambda_- t} \big|e^{tA}|_{X_-} \big| < \infty. 
\ee
The classical invariant manifold theory, which is essentially weakly nonlinear, has been successfully extended to semilinear PDEs, some functional differential equations, and quasilinear PDEs with smoothing properties such as parabolic equations. 

Invariant manifolds are both core structures that organize the local dynamics and crucial tools in a variety of problems. 
For example, 
\begin{itemize} 
\item The existence of unstable manifolds $W^+$ with $\lambda_+ >0$ directly leads to results stronger than nonlinear instability, a fundamental issue in many systems arising in physics, engineering, {\it etc.} In particular, the existence of a finite-dimensional local unstable manifold yields solutions that grow in all relevant norms. See, e.~g.~\cite{LZ13}.  
\item Codim-1 stable (or center-stable) manifolds $W^-$ of saddles 
are often the borderline between regions in the phase space with different asymptotic behaviors. 
\item Local invariant manifolds are the starting points in the study of special structures such as homoclinic/heteroclinic orbits and temporal chaos. See, e.~g.~\cite{LMSW96, Ze00, SZ03}.  
\item Local invariant manifolds and foliations are basic tools in many local bifurcation analyses. See, e.~g.~\cite{CH82}.
\item In some focusing dispersive PDEs, the ground states and their stable manifolds are the non-scattering structures with the lowest energy. See, e.~g.~\cite{DM09, DM10}. 
\item Local invariant manifolds are fundamental tools in the construction of some special solutions such as traveling water waves, breather-type solutions of nonlinear Klein--Gordon equations, {\it etc.} See, e.~g.~\cite{Ki88, IK92, GGSZ25}.
\end{itemize} 

The theory of invariant manifolds has a long and rich history dating back to Poincar\'e, and there is no way to give an exhaustive list of references even for invariant manifolds near equilibria. The early works include the graph transform method developed by Hadamard and the integral equation method often referred to as the Lyapunov--Perron method. The general theory for finite-dimensional systems can be found in, e.~g.~\cite{Pl64, Ke67, Ca81}. For infinite dimensions, including semilinear PDEs and functional differential equations, we refer the reader to, e.~g.~\cite{Hal61, He81, CL88, BJ89, CFNT89, CLL91, IV92}. When the linear part of a quasilinear PDE is smoothing, invariant manifolds have also been obtained with the help of the maximal regularity property; see, e.~g.~\cite{Mi88, Ki88, DL88, PS16, SX21}. For the incompressible Euler equation on fixed domains, local stable and unstable manifolds were constructed in \cite{LZ13} using the Lagrangian formulation. 

However, for quasilinear or more nonlinear PDEs without sufficient smoothing, such as most Hamiltonian PDEs (nonlinear Schr\"odinger equations, nonlinear waves, water waves, {\it etc.}), the invariant manifold theory has remained largely missing. In both classical approaches---the integral equation method and the graph transform method---the issue is primarily the loss of regularity in the seemingly small nonlinear terms. Consider the Taylor expansion near the equilibrium $u=0$ of a PDE written symbolically as \eqref{E:DE-0} in a certain function space $X$, decomposed according to the linear invariant splitting $X = X_+ \oplus X_-$, 
\[
u_t^+ = A_+ u^+ + f_+ (u), \quad u_t^- = A_-u_- + f_- (u). 
\]
In the Lyapunov--Perron method, the unstable manifold $W^+$ is found by solving for the fixed points of the following integral equation, which coincide with orbits on $W^+$, 
\[
u_+(t) = e^{tA_+} u_{+} (0) + \int_0^t e^{(t-\tau)A_+} f_+(u(\tau))\, d\tau, \quad u_-(t) = \int_{-\infty}^t e^{(t-\tau)A_-} f_-(u(\tau))\, d\tau, \quad t\le 0. 
\]
For a quasilinear PDE where $e^{tA_\pm}$ is not smoothing and $f_\pm(u)$ includes spatial derivatives, this regularity loss cannot be recovered, and the iterations cannot be repeated infinitely many times. In fact, smooth dependence of solutions of quasilinear PDEs on initial data is usually achieved only at some level weaker than that of a priori estimates (see Appendix \ref{SS:NLPDE-LWP}). This issue also appears in the other traditional approach of graph transforms.

One observes that such loss of regularity in quasilinear PDEs also arises in establishing local well-posedness and in proving smooth dependence of solutions on the initial data. One strategy to overcome this issue is to first convert \eqref{E:DE-0} into a quasilinear equation 
\[
v_t = \CA (u) v + f(v), 
\]
where $f(v)$ is more like a semilinear nonlinearity. Suppose $\CA(u)$ satisfies certain energy estimates; then for each function $u(t)$, the linear operator $\CA(u)$ generates a solution map $U (u(\cdot), t, t_0)$. This helps one to construct an iteration scheme with closed a priori estimates. The contraction estimate is often carried out in a less regular space due to the regularity loss from the term $(\CA(u_1) -\CA(u_2))v$.

In this paper, we first consider a broad class of nonlinear PDEs symbolically in the form of \eqref{E:DE-0} in a certain real Hilbert space $X$, which is often the natural basic energy space of the system. 
Aiming at (but not limited to) Hamiltonian PDEs, the following are roughly our main assumptions.
\begin{itemize} 
\item There exist closed subspaces $X_\pm \subset X$ invariant under $e^{tA}$ such that $X= X_+\oplus X_-$ with associated projections $\Pi_\pm$. 
\item Denote $X^n = Dom(A^n)$ and $X_\pm^n = X_\pm \cap X^n$, $n \ge 0$, where $A= \BD F(0)$. For some $n\ge 2$, $\CA(u) = \BD F(u) \in \BBL(X^r, X^{r-1})$ is $C^1$ in $u\in X^{n'}$ for $1\le r \le n'$ and $n'= n, n-1$.  
\item There exist equivalent metrics $\CL_\pm(u)$ on $X_\pm$, $C^1$ in $u\in X^{n-1}$, such that 
\be \label{E:dissipativity-0}
\langle \CL_- (u) \CA(u) v, v\rangle \le \lambda_- |v|^2, \; \forall v \in X_-^1, \quad \langle \CL_+ (u) \CA(u) v, v\rangle \ge \lambda_+ |v|^2, \; \forall v \in X_+^1. 
\ee
\end{itemize}
For some PDEs, $X^n$ are simply Sobolev spaces. 
Obviously, in order to construct local invariant manifolds, the above assumptions are required only for $|u|_{X^{n-1}} \ll 1$.

Our main general results establish the existence, uniqueness, and smoothness of the local stable and unstable manifolds, obtained by incorporating this energy-estimate-based technique from the analysis of quasilinear PDEs into the Lyapunov–Perron approach. 

\vspace {0.08in} \noindent {\bf Main general results.} 
{\it Assume $\lambda_+ > \max\{0, \lambda_-\}$. Then there exist $\delta>0$ and $q^+ : X_+^{n} (\delta) \to X_-^n$, where $X_+^n (\delta)$ is the $\delta$-ball in $X_+^n$, such that 
\begin{enumerate} 
\item $q^+(0)=0$, $q^+ \in Lip(X_+^n (\delta), X_-^{n-1})$. Moreover,  
$q^+ \in C^{m,1} (X_+^n (\delta), X_-^{n-m-1})$ if $F$ satisfies certain $C^{m,1}$ smoothness assumptions for $m \ge 1$. 
\item $W^+ \triangleq graph(q^+)$ satisfies $T_0 W^+ = X_+^n$ and is locally invariant under \eqref{E:DE-0}. For $\lambda_0= \big( \max\{0, \lambda_-\}, \lambda_+ \big)/2$, there exists $M^*>0$ such that any $u_0 \in W^+$ has a solution $u(t)$, $t \le 0$, unique in the category of $|u(t)|_{X^n} \le M^* \delta e^{\lambda_0 t}$, and it also satisfies    
\be \label{E:expD-1}
\sup_{t\le 0} e^{-\lambda t} |u(t)|_{X^n} < \infty, \; \forall \lambda \in (\lambda_-, \lambda_+). 
\ee
\item Suppose a solution $u(t) \in X^n$, $t\le 0$, to \eqref{E:DE-0} satisfies \eqref{E:expD-1}   
for some $\lambda \in (\max\{0, \lambda_-\}, \lambda_+)$, then $u(t) \in W^+$ for all $t \ll -1$. 
\end{enumerate}
If $\lambda_- < \min\{0, \lambda_+\}$, then there exists $q^- : X_-^{n} (\delta) \to X_+^n$ along with its graph $W^-$ satisfying parallel properties for $t \ge 0$. 
}

\vspace {0.08in}
The majority of the above assumptions and results are consistent with those in the standard local unstable and stable manifold theory for ODEs and semilinear PDEs, with the exceptions of \eqref{E:dissipativity-0} and the loss of regularity in $q^+\in C^{m,1} (X_+^n (\delta), X_-^{n-m-1})$. 
More specifically, the first assumption is the existence of an invariant splitting of $X$ under the linearized flow $e^{tA}$ at $u=0$. The second essentially requires that the right-hand side of the PDE \eqref{E:DE-0} be a smooth vector field in $X^n$ with a certain loss of regularity. The last assumption is the dissipativity of the linearization $\CA(u)= \BD F(u)$ at each $u \in X^{n-1}$ forward in $X_-$ and backward in $X_+$, with respect to a state-dependent energy form $\CL_\pm(u)$ which is $C^1$ in $u \in X^{n-1}$.
This is equivalent to the availability of a linearized energy estimate of \eqref{E:DE-0} in $X_\pm$. If $\lambda_+ > \lambda_-$, it yields an exponential dichotomy of $e^{t \Pi_\pm \CA(u)}$ on $X_\pm$ {\it at each $u \in X^n$}. The smoothness of $\CA(u)$ and $\CL_\pm(u)$ allows one to carry out these estimates in higher-order spaces $X^r$, $r \le n$, as often seen in the local well-posedness theory of quasilinear PDEs. 
In the main results, statements (2--3) mean that $W^\pm$ consists of solutions that decay exponentially as $t\to \mp \infty$. The loss of regularity in the smoothness of $W^\pm$ is comparable to the loss in the smooth dependence of solutions to quasilinear PDEs (see Appendix \ref{SSS:smoothness-NLPDE}).

%
%

The precise assumptions are given in (D.1--4) in Subsection \ref{SS:NLPDE-LInMa} and (B.5) in Appendix \ref{SS:NLPDE-LWP}, the latter of which is only for the smoothness. The precise results are stated in Theorems \ref{T:NLPDE-UM} and \ref{T:UM-smoothness} and Remarks \ref{R:SM} and \ref{R:smoothness-SM} in Subsection \ref{SS:NLPDE-LInMa}. The assumptions in Appendix \ref{S:pre-LWP} and Section \ref{S:LInMa} look technical and tedious, but they are designed to be as directly applicable to concrete nonlinear PDEs as possible. 

In Section \ref{S:examples}, we demonstrate how various nonlinear PDEs fit into this framework. Since parabolic PDEs have relatively strong smoothing effect, this paper mainly aims at Hamiltonian PDEs (Subsection \ref{SS:examples-Ham}) such as nonlinear Schr\"odinger equations, nonlinear waves, KdV type equations, the MMT model, {\it etc.}, which often have natural estimate structures. Theorem \ref{T:Ham-UM} along with Remark \ref{R:Ham-SM} are specifically formulated for the local stable and unstable manifolds of Hamiltonian PDEs. To illustrate how the main results also apply to other types of nonlinear PDEs, we discuss more examples in Subsection \ref{SS:examples-other}, including mean curvature flows, where the local stable and unstable manifolds are obtained without using the maximal regularity property. 


As the primary motivation for this paper, stable and unstable manifolds of irrotational water waves with surface tension will be established in Theorem \ref{T:CWW-LInMa-main} in Section \ref{S:CWW}. This result requires substantial technical preparation due to the presence of the nonlocal Dirichlet–Neumann operator. The free interface problem between two incompressible, inviscid, irrotational fluids with surface tension will be outlined in Subsection \ref{SS:2F}, as the analysis is roughly parallel to that of the water waves.

To prepare for the construction of the local stable and unstable manifolds, a non-autonomous linear estimate with detailed dependence on various parameters is given in Appendix \ref{SS:Linear}. As a byproduct, in Appendix \ref{SS:NLPDE-LWP} we also provide a proof of the local well-posedness of a class of nonlinear evolutionary PDEs and the smooth dependence on the initial values by a similar energy-estimate-based approach. The derivation of both the local well-posedness and the local invariant manifolds of nonlinear PDEs are first achieved for quasilinearized PDEs (by a local diffeomorphism) in Appendix \ref{SS:QLPDE-LWP} and Subsection \ref{SS:QLPDE-LInMa}, respectively.

\vspace {0.08in}\noindent {\bf Notations.} 
$X^r (u, R)$, $R>0$, denotes the open ball centered at $u \in X^r$ with radius $R$, while $X^r (R)$ is the ball centered at $0$. For an unbounded linear operator $A$, its domain is denoted by $Dom(A)$. 
We often use $\BD$ to denote the Fr\'echet differentiation of nonlinear mapping between function spaces and $\nabla$ or $D$ to denote the directional derivatives or gradient vector with respect to certain spatial variables, while $'$ for time derivatives.   

\section{Local stable and unstable manifolds} \label{S:LInMa}

We shall first obtain the stable/unstable manifolds of an equilibrium of a model quasilinear PDE 
\eqref{E:QLPDE-1} in Subsection \ref{SS:QLPDE-LInMa}, followed by a class of nonlinear PDEs in Subsection \ref{SS:NLPDE-LInMa}.  Some concrete nonlinear PDEs will be discussed in Section \ref{S:examples}.

\subsection{Stable and unstable manifolds of a model quasilinear PDE} \label{SS:QLPDE-LInMa}

Since the construction of local stable and unstable manifolds are similar, we shall mainly focus on that of unstable manifolds. See Remark \ref{R:SM-0} for local stable manifolds. 



Let $X_+$ and $X_-$ be real Hilbert spaces with scales of dense subspaces $X_{\pm}^r \subset X_{\pm}$, $r \ge 0$.
Let $X^r= X_+^r \oplus X_-^r$ 
and consider solutions $v=(v_+, v_-)  \in X^r$ for $t \le 0$ to the system 
\be \label{E:QLPDE-1} \begin{cases} 
\p_t v_+ = \BA_+(v) v_+ + f_+ (v), \\
\p_t v_- = \BA_-(v) v_- + f_-(v), 
\end{cases} \ee
with the asymptotic  condition 
\be \label{E:QLPDE-1-BC}
v_\pm (t) \to 0 \; \text{ as } \; t\to -\infty \; \text{ at certain exponential rate.}
\ee
The above  asymptotic exponential rate will be specified later. 
We assume that there exist
\[
k \ge 1, \;\; R_0, C_{f}>0, \;\; \lambda_\pm \in \R,  \;\; C_0\ge 1, \;\;  \BL_\pm \in C^1 \big(X^{k-1} (R_0)
, \BBL(X_{\pm}, (X_{\pm})^*)\big), 
\]
 where $X^r (R) = X_+^r (R) \oplus X_-^r (R)$ for $R>0$, such that the following are satisfied.
\begin{enumerate} 
\item [(C.1)] For any $v \in X^{k-1} (R_0)$ and any $w \in X_\pm$
\[
\BL_\pm (v) = \BL_\pm (v)^*, \;\;      
C_{0}^{-1} |w|_{X_\pm} \le  |w |_{\BL_\pm (v)} \le C_{0} |w |_{X_\pm}, \;\; |\BD\BL_\pm |_{C^0 (X^{k-1} (R_0), \BBL (X^{k-1} \otimes X_\pm, X_\pm^*))} \le C_0,
\]
where $|w|_{\BL_\pm (v)} = \sqrt{\langle \BL_\pm (v) w, w \rangle}$. 

\item [(C.2)] For any $v \in X^{k-1} (R_0)$, the domain $Dom(\BA_{\pm}(v)) = X_{\pm}^1\subset X_{\pm}^0$ and $\lambda \pm \BA_{\pm} (v): X_\pm^1 \to X_\pm$  is surjective for some $\lambda > \mp \lambda_\pm$. Moreover, $\BA_\pm \in C^{1} \big(X^{k-1} (R_0), \BBL(X_\pm^{r}, X_\pm^{r-1})\big)$, $1\le r\le k$, also satisfies    
\[
|\BA_\pm|_{ C^{1}(X^{k-1} (R_0), \BBL(X_\pm^{r}, X_\pm^{r-1}))}, \, |\bar \BA_\pm|_{ C^{1}(X^{k-1} (R_0), \BBL(X_\pm^{r}, X_\pm^{r-1}))}, \, |\bar \BA_\pm^{-1}|_{ C^{1}(X^{k-1} (R_0), \BBL(X_\pm^{r-1}, X_\pm^{r}))} \le C_0,  
\]
where $\bar \BA_\pm (v) = \BA_\pm (v) - (\lambda_\pm \mp 1)$, 
and for any $v \in X^{k-1} (R_0)$ and $w \in X_\pm^1$, 
\be \label{E:dissipativity-3}
\langle \BL_\pm (v) w , \mp \BA_{\pm} (v)w \rangle \le \mp \lambda_\pm \langle \BL_\pm (v) w, w \rangle. 
\ee

\item [(C.3)] Assume $f_\pm \in C^0(X^{k} (R_0), X_\pm^k)$ satisfies 
\[
|f_\pm (v) |_{X_\pm^k} \le C_{f} |v|_{X^k}, \quad |f_\pm(v_1)- f_\pm(v_2)|_{X_\pm^{k-1}} \le C_{f} |v_2 -v_1|_{X^{k-1}}, \quad \forall v, v_1, v_2 \in X^{k} (R_0). 
\]

\item [(C.4)] Assume 
\[
\Sigma_+ \triangleq \Big\{ 0 \le \lambda \in (\lambda_-, \lambda_+) \mid  L_1 (\lambda) \triangleq \frac {C_0^{2(k+1)} C_f}{\lambda- \lambda_-} + \frac {C_0^{2(k+1)} C_f}{\lambda_+- \lambda}  < 1 \Big\} \ne \emptyset. 
\]
\end{enumerate}

\begin{remark} \label{R:LInMa}
The readers are referred to Appendix \ref{SS:Linear} and \ref{SS:QLPDE-LWP} for some remarks on the above assumptions. In particular, by the Lumer-Phillips Theorem, the surjectivity of $\lambda \pm  \BA_\pm (v)$ and the dissipativity \eqref{E:dissipativity-3} in (C.2) imply that, for any $v \in X^{k-1} (R_0)$, $\lambda \pm \BA_\pm (v)$ is an isomorphism for any $\lambda > \mp \lambda_\pm$ and $\BA_\pm (v)$ is the generator of a $C^0$ semigroup $e^{s \BA_\pm (v)}$, $\mp s\ge 0$, on $X_{\pm}$ with the exponential growth rate bounded by $\lambda_\pm$. In the special case where $\BA_+$ (or $\BA_-$) is independent of $v$, namely in the semilinear case, 
it is sufficient to assume that the semigroup $e^{s\BA_+}$ is well-posed with exponential rate $\lambda_+$ for $s\le 0$ (or $e^{s\BA_-}$ with exponential rate $\lambda_-$ for $s\ge 0$). In this case $\BL_+$ (or $\BL_-$) is not needed.  

Assumption (C.4) includes the necessary conditions $\lambda_+> \max \{0, \lambda_-\}$ and that their difference is much greater than the Lipschitz constant $C_{f}$.
This indicates that the linear dynamics is truly unstable in $X_+$, which is possibly a strongly unstable subspace of the unstable subspace of the equilibrium. 

It is not assumed that $\BA_+ (v)$ generates a semigroup $e^{s\CA_+(v)}$, $s \ge 0$, so the argument here can potentially be applied to ill-posed problems such as elliptic PDEs on channels treated as ill-posed evolution systems.  
\end{remark}

The main result of this subsection is the existence and uniqueness of the following unstable manifold under the above conditions.

\begin{theorem} \label{T:QLPDE-UM} 
Assume (C-1)--(C.4). For any $\lambda \in \Sigma_+$, $M_0 > \frac {2 C_0^{2(k+1)}}{1-L_1(\lambda)}$, and $l \in \big(L_1(\lambda), 1\big)$, there exists $\ep\in (0, \frac {R_0}{M_0})$ determined by $k, R_0, \lambda_\pm, C_f$, and $C_0$ (see \eqref{E:temp-18} and Remark \ref{R:size-1}) such that  there exists $h_+ :X_+^k (\ep)\to X_-^k$ satisfying the following properties. 
\begin{enumerate} 
\item For any $v_{0+} \in X_+^k (\ep)$, there exists a solution to \eqref{E:QLPDE-1}
\[
v(t) \in C^0 \big((-\infty, 0], X^k(R_0)\big) \cap C^1 \big((-\infty, 0], X^{k-1}(R_0)\big),  
\]
unique in the category
\be \label{E:UM-0.5}
v_+ (0) = v_{0+}, \quad |v(t)|_{X^k} \le M_0 \ep e^{\lambda t}, \; \forall t \le 0.
\ee
Moreover $v(t)$ satisfies 
\be \label{E:UM-1}
v_-(0) = h_+ (v_{0+}), \quad |v(t)|_{X^k} \le \frac {2 C_0^{2(k+1)}}{1-L_1(\lambda)} |v_{0+}|_{X_+^k}e^{\lambda t}.
\ee
\item The above solution $v(t)$ defined by $v_{0+}$ also satisfies that, for any $t<0$ such that $v_+(t) \in X_+^k(\ep)$, it holds $h_+(v_+ (t)) = v_-(t)$.
\item Suppose the above solution $v(t)$ defined by $v_{0+}$ can be extended to $C^0 ((-\infty, T_0), X^k)$ for some $T_0>0$. Let $T = \sup\{ t >0 \mid v_+(\tau ) \in X_+^k (\ep), \, \forall \tau \in [0, t) \} \in (0, T_0]$, then $h_+(v_+ (t)) = v_-(t)$ for all $t \in [0, T)$. 
\item $h_+$ satisfies $h_+(0)=0$ and, for any $v_{0+}, \wt v_{0+} \in X_+^k(\ep)$, 
\be \label{E:Lip-1}
|h_+(v_{0+})|_{X_-^k} \le \frac {2l C_0^{2(k+1)}}{1-L_1(\lambda)} |v_{0+}|_{X_+^k}, \quad |h_+ (v_{0+}) - h_+ (\wt v_{0+})|_{X_-^{k-1}} \le \frac {l C_0^{2k}}{1-l} |v_{0+} - \wt v_{0+}|_{X_+^{k-1}}, 
\ee
and their corresponding solutions $v(t)$ and $\wt v(t)$ satisfy, for $t\le 0$, 
\[
|v(t) - \wt v(t)|_{X^{k-1}} \le \frac {C_0^{2k}}{1-l}  e^{\lambda t}|v_{0+} - \wt v_{0+}|_{X_+^{k-1}}. 
\]
\item Suppose $0< \wt \lambda \in \Sigma_+$ and $v(t)$, $t\le 0$, is a solution to \eqref{E:QLPDE-1} such that $\sup_{t\le 0} e^{- \wt \lambda t}|v(t)|_{X^k}  <\infty$, then there exists $t_0\le 0$ such that $v_-(t) = h_+(v_+(t))$ for all $t\le t_0$.  
\end{enumerate}  \end{theorem}

The graph $W^+ \triangleq graph(h_+)$ is often referred to as the local unstable manifold of the equilibrium $v=0$ (or a strongly unstable manifold if $\lambda_->0$ too). 
One notices that the estimate in \eqref{E:UM-0.5} is rougher than that in \eqref{E:UM-1}, so statement (1) gives the uniqueness in a larger category. The above statements (2--3) indicate the local invariance of $W^+$ in $t$. 
Statement (5) implies that any solution exponentially decaying as $t \to -\infty$ is in $W^+$ eventually. Along with statement (1), they characterize $W^+$ as the set of nearby solutions which decay to $0$ at exponential rates in $\Sigma_+ \subset (\lambda_-, \lambda_+)$ and thus it is locally unique. 
Hence even though $W^+$ seems to depend on $\lambda$, $M_0$, $l$, and $\ep$ as stated in the theorem, $W^+$ is essentially independent of these parameters (see also Lemma \ref{L:UM-1}). 
The whole set-up and the theorem can be also be put on $X^k (R)$, $R\in (0, R_0]$, and essentially the same $W^+$ is obtained, see Remark \ref{R:varyingR}.  
If the nonlinearity $f_\pm$ are superlinear near $0$, i.~e.~
$C_f \to 0$ as $R_0 \to 0$ in assumption (C.3), then $L_1(\lambda) \to 0$ as $R_0 \to 0$. Therefore, by taking $R_0 \to 0$ and $l(\lambda) = 2 L_1(\lambda)$, the 
upper bounds in \eqref{E:Lip-1} also converge to $0$ as $\ep \to0$ (see also Remark \ref{R:size-1}). It means that $W^+$ is tangent to $X_+^{k}$ at $v=0$. 

\begin{remark} \label{R:SM-0} 
Assume \eqref{E:QLPDE-1} satisfies (C.1)--(C.3) and 
\be  \tag {C.4'} 
\Sigma_- \triangleq \Big\{ 0 \ge \lambda \in (\lambda_-, \lambda_+) \mid  L_1 (\lambda) \triangleq \frac {C_0^{2(k+1)} C_f}{\lambda- \lambda_-} + \frac {C_0^{2(k+1)} C_f}{\lambda_+- \lambda}  < 1 \Big\} \ne \emptyset,
\ee 
then Theorem \ref{T:QLPDE-UM}, as well as Lemma \ref{L:UM-1}, holds with the sign of $t$ reversed.  This implies the existence of a mapping $h_-: X_-^k(\ep) \to X_+^k$ and the local stable manifold $W^-=graph(h_-)$. 
\end{remark}

The theorem will be proved by a fixed point argument on a subset of  
\[
Y=  \{ v (\cdot) \in 
W^{1, \infty} ((-\infty, 0], X^{k-1}) \mid v (t) = (v_+(t), v_-(t)) \in X^{k} (R_0), \ \forall t \le 0\} 
\] 
along with the exponentially weighted norms
\be \label{E:weighted-1}
|v_\pm|_{\pm, r, \lambda} = \sup_{t\le 0} e^{-\lambda t} |v_\pm(t)|_{X_\pm^r}, \; \; |v|_{r, \lambda} = |v_+|_{+, r, \lambda} + |v_-|_{-, r, \lambda}, \quad 0\le r \le k. 
\ee

For $v \in Y$, let  
\[
A_\pm (t) = \BA_\pm (v(t)), \quad L_\pm (t) = \BL_\pm (v(t)), \quad Q_\pm(t) = \bar \BA_\pm (v(t)),
\]
where $\bar \BA_\pm (v)$ is defined in (C.2).
Much as in Appendix \ref{SS:QLPDE-LWP}, Proposition \ref{P:linear} implies that $A_\pm(t)$ generate strongly $C^0$ evolution operators 
\[
U_\pm (t, t_0) \in \BBL(X^r), \quad t, t_0 \le 0, \;\ \pm (t-t_0) \le 0, \;\ 0\le r \le k. 
\]
Similar to \eqref{E:U-bdd-1} and \eqref{E:U-bdd-2}, it satisfies 
\be \label{E:U-bdd-4} 
 |U_\pm (t, t_0) |_{\BBL(X^r)} \le C_0^{2(r+1)} e^{\lambda_\pm (t-t_0)  +  (r+1) C_0^{2(r+1)} |\int_{t_0}^t |v'(t)|_{X^{k-1}} dt| }. 
\ee 

For $v_{0+} \in X_+^{k-1} (R_1) \cap X_+^k (R_0)$ 
and $v(\cdot) \in Y$, let $U_\pm(t, t_0)$ be the evolution operators generated by $v$ and define the Lyapunov-Perron integral operator
\be \label{E:CT-2} 
\CT(v_{0+}, v) = \wt v(t) = (\wt v_+ (t), \wt v_- (t)), \quad t\le 0,  
\ee
where 
\begin{align*}
& \wt v_+ (t)= \CT_+ (v_0, v) = U_+ (t, 0) v_{0+} + \int_{0}^t U_+ (t, \tau) f_+ (v(\tau)) d\tau, \\
& \wt v_- (t)= \CT_- (v_0, v) = \int_{-\infty}^t U_- (t, \tau) f_- (v(\tau)) d\tau.
\end{align*}
Clearly, if convergent, $\wt v_\pm(t)$ satisfy  
\be \label{E:QLPDE-iter-2}
\p_t \wt v_\pm = \BA(v(t)) \wt v_\pm + f_\pm (v(t)), \quad \wt v_+ (0) = v_{0+}. 
\ee
In the following lemma, recall that $L_1 (\lambda)\in (0, 1)$ was defined by $\lambda_\pm, k, C_0, C_f$ and $\lambda \in \Sigma_+$ in assumption (C.4). 

\begin{lemma} \label{L:LP-contraction}
For any  $\lambda \in \Sigma_+$ and $l \in \big(L_1(\lambda), 1\big)$, let 
\[
\Gamma_{\lambda, M_0, \ep}= \Big\{ v \in Y \mid |v|_{k, \lambda} \le M_0 \ep, \ 
|\p_t v_\pm |_{\pm, {k-1}, 0} \le M_1\ep \Big\},
\]
where $M_0$, $M_1$, and $\ep \in (0, \frac {R_0}{M_0})$ satisfy 
\be \label{E:temp-17}
M_0 \ge \frac {2C_0^{2(k+1)}}{1-L_1(\lambda)}, \quad M_1 = (C_0+ C_f) M_0, \quad 2(k+1)C_0^{2(k+1)} M_1 \ep < \min \{ \lambda_+-\lambda, \lambda-\lambda_-\}, 
\ee
\be \label{E:temp-18} \begin{split}
& \frac {C_0^{2(k+1)} C_f}{\lambda- \lambda_- - 2(k+1)C_0^{2(k+1)} M_1\ep } + \frac {C_0^{2(k+1)} C_f}{\lambda_+- \lambda - 2(k+1)C_0^{2(k+1)} M_1\ep} \\
&\qquad \qquad \qquad \qquad \qquad \qquad \qquad \qquad \qquad \qquad \qquad \le \min\Big\{l,  \frac {1 + L_1(\lambda)}2\Big\}, \\
&\frac {C_0^{2k} (C_0 M_0 \ep +C_f)}{\lambda- \lambda_- -2 kC_0^{2k} M_1\ep } + \frac {C_0^{2k} (C_0 M_0 \ep + C_f)}{\lambda_+- \lambda - 2kC_0^{2k} M_1\ep} \le  l, 
\end{split} \ee
then 
\[
\CT \big( X_+^k (\ep) \times \Gamma_{\lambda, M_0, \ep} \big) \subset \Gamma_{\lambda, M_0, \ep}, \quad |\CT_- (v_{0+}, v)|_{-, k, \lambda} \le l|v|_{k, \lambda},   
\]
and for any $v_1, v_2 \in \Gamma_{\lambda, M_0, \ep}$ and $v_{0+, 1}, v_{0+, 2} \in X_+^k (\ep)$, it holds 
\[
\big|\CT(v_{0+, 1}, v_1) - \CT(v_{0+, 2}, v_2)\big|_{k-1, \lambda} \le C_0^{2k} |v_{0+, 1} - v_{0+, 2}|_{X^{k-1}} + l |v_2 - v_1|_{k-1, \lambda}, 
\]
\[
\big|\CT_- (v_{0+, 1}, v_1) - \CT_- (v_{0+, 2}, v_2)\big|_{-, k-1, \lambda} \le l |v_2 - v_1|_{k-1, \lambda}.  
\]
\end{lemma} 

Clearly a solution $v(\cdot)$ to \eqref{E:QLPDE-1} belongs to $\Gamma_{\lambda, M_0, \ep}$ if and only if $\CT( v_+(0), v)=v$. The existence of $\ep>0$ satisfying \eqref{E:temp-18} is ensured by  the assumption $\lambda\in \Sigma_+$ which bounds the left sides of \eqref{E:temp-18}  by  $L_1(\lambda)$ when $\ep=0$. 

\begin{remark} \label{R:size-1}
Suppose we fix the linear operators $\BA_\pm(v)$ in \eqref{E:QLPDE-1} and consider nonlinearity $f_\pm (v)$ with small $C_f \ll 1$ on $X^k (R_0)$, which is the case when $f_\pm$ is superlinear near $v=0$ and $R_0\ll 1$. In this case we may take $l = 2 L_1(\lambda) \ll 1$, then there exists $C>0$ determined by $C_0$ and $|\lambda - \lambda_\pm|$ such that $\ep = L_1(\lambda)/(CM_0)$ satisfies  condition \eqref{E:temp-18} in the above lemma. 
\end{remark}

\begin{proof} 
For $v_{0+} \in  X_+^k (\ep)$ and $v(\cdot) \in Y$, using the weight norms defined in \eqref{E:weighted-1},  we obtain from \eqref{E:CT-2}, \eqref{E:U-bdd-4}, and assumption (C.3), for any $t\le 0$,  
\begin{align*} 
|\wt v_+ (t) & |_{X_+^{k}} \le  |U_+ (t, 0) v_{0+}|_{X_+^{k}} + \int_t^0 |U_+ (t, \tau)|_{\BBL(X_+^{k})} |f_+ (v(\tau))|_{X_+^{k}} d\tau \\
\le & C_0^{2(k+1)}  \Big( e^{(\lambda_+ - (k+1) C_0^{2(k+1)} |v'|_{k-1, 0})t  }  |v_{0+}|_{X_+^{k}} \\
&\qquad  \qquad + C_{f}  \int_t^0  e^{(\lambda_+ - (k+1) C_0^{2(k+1)} |v'|_{k-1, 0})  (t-\tau)} |v(\tau)|_{X^k} d\tau \Big) \\
\le & C_0^{2(k+1)} e^{(\lambda_+ - (k+1) C_0^{2(k+1)} |v'|_{{k-1}, 0})t }  |v_{0+}|_{X_+^{k}} + \frac {C_0^{2(k+1)} C_{f}}{\lambda_+ - \lambda - (k+1) C_0^{2(k+1)} |v'|_{{k-1}, 0}} e^{\lambda t} |v|_{k, \lambda}. 
\end{align*}
For $v \in \Gamma_{\lambda, M_0, \ep}$, it implies 
\[
|\wt v_+|_{+, k, \lambda} \le C_0^{2(k+1)}  |v_{0+}|_{X_+^{k}} + \frac {C_0^{2(k+1)} C_{f}}{\lambda_+ - \lambda - 2(k+1) C_0^{2(k+1)}|v'|_{{k-1}, 0}}  |v|_{k, \lambda}. 
\] 
Similarly, 
for $v \in \Gamma_{\lambda, M_0,  \ep}$,  
\[
|\wt v_-|_{-, k, \lambda} \le \frac {C_0^{2(k+1)} C_{f}}{\lambda - \lambda_- - 2(k+1) C_0^{2(k+1)} |v'|_{{k-1}, 0}}  |v|_{k, \lambda}. 
\] 
Therefore 
\be \label{E:temp-19} \begin{split}
|\CT(v_{0+}, v)|_{k, \lambda} \le C_0^{2(k+1)}  |v_{0+}|_{X_+^{k}} + \Big( & \frac {C_0^{2(k+1)} C_f}{\lambda- \lambda_- - 2(k+1)C_0^{2(k+1)} |v'|_{{k-1}, 0} } \\
&+ \frac {C_0^{2(k+1)} C_f}{\lambda_+- \lambda - 2(k+1)C_0^{2(k+1)} |v'|_{{k-1}, 0}} \Big) |v|_{k, \lambda}. 
\end{split} \ee
Moreover, according to \eqref{E:QLPDE-iter-2},  
\be \label{E:temp-20} \begin{split}
|\wt v_\pm' (t) |_{X_\pm^{k-1}} = & |\BA_\pm (v(t)) \wt v_\pm (t) + f_\pm (v(t))|_{X_\pm^{k-1}}
\le C_0 |\wt v_\pm (t)|_{X_\pm^{k}} + C_{f} |v (t)|_{X^{k-1}}.
\end{split} \ee
Summarizing the above estimates, according to the choice of the constants $M_0$, $M_1$, and $\ep$, we obtain that $\CT \big( X_+^k (\ep) \times \Gamma_{\lambda, M_0, \ep} \big) \subset \Gamma_{\lambda, M_0, \ep}$ and the estimate on $\CT_-$. 

To obtain the contraction estimates, for $v_1, v_2 \in Y$, let $U_{1\pm}(t, t_0)$ and $U_{2\pm}(t, t_0)$, $\mp (t-t_0) \ge 0$ be the evolution operators generated by $v_1, v_2$, respectively, 
and   
\[
\wt w =( \wt w_+(\cdot) , \wt w_-(\cdot)) = \wt v_2 - \wt v_1 \triangleq \CT(v_{0+}, v_2) - \CT(v_{0+}, v_1). 
\]
Using \eqref{E:QLPDE-iter-2}, much as in the proof of Lemma \ref{L:Contraction-QLPDE-1} in Appendix \ref{SS:QLPDE-LWP}, we have 
\[
\wt w_\pm (t) = \int_{T_\pm}^t U_{1\pm} (t, \tau) \Big( \big(\BA_\pm (v_2(\tau)) - \BA_\pm (v_1(\tau)) \big)  \wt v_{2\pm}(\tau)  + f_\pm (v_2(\tau) ) - f_\pm (v_1(\tau) ) \Big) d\tau,
\]
where $T_-=-\infty$ and $T_+=0$. Hence   
\begin{align*} 
|\wt w_\pm(t) &|_{X_\pm^{k-1}} \le  
\Big| \int_{T_\pm}^t |U_{1\pm} (t, \tau)|_{\BBL(X_\pm^{k-1})} \Big( |\BD \BA_\pm |_{C^0(X^{k-1} (R_0), \BBL(X^{k-1} \otimes X_\pm^{k}, X_\pm^{k-1})} |\wt v_{2\pm} (\tau)|_{X_\pm^{k}}  +  C_{f} \Big)\\
&\qquad \qquad \times  |v_2 (\tau) - v_1(\tau)|_{X^{k-1}}d\tau \Big| \\
\le & C_0^{2k} \Big| \int_{T_\pm}^t  e^{(\lambda_\pm \mp k C_0^{2k}  |v'|_{{k-1}, 0}) (t-\tau)}   \Big( C_0  |\wt v_{2\pm}|_{\pm, {k}, 0} + C_{f} \Big) |v_2 (\tau) - v_1(\tau)|_{X^{k-1}}  d\tau \Big|\\
\le & \frac {C_0^{2k} (C_0 |\wt v_{2\pm}|_{\pm, {k}, 0} + C_{f})}{|\lambda- \lambda_\pm| - k C_0^{2k}  |v'|_{{k-1}, 0}} e^{\lambda t  } |v_2  - v_1|_{{k-1}, \lambda}. 
\end{align*}
The choices $M_0, M_1$, and $\ep$ along with \eqref{E:U-bdd-4} imply the contraction estimates. 
\end{proof}

\begin{corollary} \label{C:LP-contraction} 
For any $v_{0+} \in X_+^{k}(\ep)$, there exists a unique $v= (v_+(\cdot), v_-(\cdot)) \in \Gamma_{\lambda, M_0,  \ep}$ such that $\CT (v_{0+}, v) =v$. Moreover $v\in C^0 ((-\infty, 0], X^k (M_0 \ep)) \cap C^1 \big((-\infty, 0], X^{k-1}(R_0)\big)$. 
\end{corollary}

The corollary follows from the same proof as in the proof Theorem \ref{T:QLPDE-LWP} and we omit it. 

Under the assumptions and notations of Lemma \ref{L:LP-contraction},  $\CT (v_{0+}, v(\cdot)) =v(\cdot)$ implies $v_+(0) = v_{0+}$. The $X_-$ component of $v(0)$ allows us to define  
\be \label{E:UM-2}
h_+: X_+^{k}(\ep) \to X_-^{k} (R_0) \; \text{ as } \; h_+ (v_{0+}) = v_-(0).
\ee 

We are ready to complete the proof of Theorem \ref{T:QLPDE-UM}. 

\begin{proof}[Proof of Theorem \ref{T:QLPDE-UM}.] 
The estimate on $|v|_{k, \lambda}$ in \eqref{E:UM-1} is a direct consequence of \eqref{E:temp-19} and the fact that $v$ is the unique fixed point of $\CT(v_{0+}, \cdot)$ in $\Gamma_{\lambda, M_0,  \ep}$. The rest in Theorem \ref{T:QLPDE-UM}(1) are included in Corollary \ref{C:LP-contraction} and the definition of $\Gamma_{\lambda, M_0,  \ep}$. 

Suppose $v_+(t) \in X_+^k(\ep)$ for some $t<0$ . Since $\lambda \ge 0$,  $v(\cdot +t) \in \Gamma_{\lambda, M_0,  \ep}$ is the fixed point of $\CT( v_+(t), \cdot)$, we have $h_+(v_+ (t)) = v_-(t)$ which proves Theorem \ref{T:QLPDE-UM}(2). 

To prove Theorem \ref{T:QLPDE-UM}(3), one observes that  $M_0> \frac {2 C_0^{2(k+1)}}{1-L_1(\lambda)}$ is assumed. By the continuity  of $v(t) \in X^k$ along with \eqref{E:UM-1} we obtain that \eqref{E:UM-0.5} is satisfied by $v(\cdot + \delta)$ for some $\delta \in (0, T)$. So the uniqueness in Theorem \ref{T:QLPDE-UM}(1) implies $v(t) \in W^+$ for all $t \in (0, \delta)$. Theorem \ref{T:QLPDE-UM}(3) follows from a continuation argument. 

In Theorem \ref{T:QLPDE-UM}(4), $h_+(0) =0$ is clear since $v=0 \in \Gamma_{\lambda, M_0,  \ep}$ is the unique fixed point of $\CT(0, \cdot)$. 
The  estimates on $h_+$ and $|v(t) - \wt v(t)|_{X^{k-1}}$ 
follow directly from \eqref{E:UM-1} and the estimates of $\CT$ given in Lemma \ref{L:LP-contraction}.  

Theorem \ref{T:QLPDE-UM}(5) will be a part of Lemma \ref{L:UM-1} below. 
\end{proof}

Even though the above procedure defining $h_+$  involves parameters $\lambda$, $l$, $M_0$, and $\ep$, the following lemma shows that $h_+$ is essentially independent of these parameters. . 

\begin{lemma} \label{L:UM-1} 
The mapping $h_+$ satisfies the following properties. 
\begin{enumerate} 
\item Suppose $\lambda$ along with $\wt l \in \big(L_1(\lambda), 1\big)$, $\wt M_0$, and $\wt \ep$ also satisfy the conditions in Lemma \ref{L:LP-contraction}. Let $\wt h_+$ be the mapping defined by \eqref{E:UM-2} and Corollary \ref{C:LP-contraction} accordingly and $\ep^* = \min \{\ep, \wt \ep\}$, then $h_+ = \wt h_+$ on $X_+^k(\ep^*)$. 
\item Suppose $\wt \lambda \in \Sigma_+$, $\wt l \in \big(L_1(\wt \lambda), 1\big)$, $\wt M_0$, and $\wt \ep$ also satisfy the conditions in Lemma \ref{L:LP-contraction}. Let $\wt h_+$ be the mapping defined by \eqref{E:UM-2} and Corollary \ref{C:LP-contraction} accordingly. There exists $\ep_0 >0$ such that $h_+ = \wt h_+$ on $X_+^k(\ep_0)$.
\item Suppose $0< \wt \lambda \in \Sigma_+$ and $v(t)$, $t\le 0$, is a solution to \eqref{E:QLPDE-1} such that $|v|_{k, \wt \lambda} <\infty$, then there exists $t_0\le 0$ such that $v_-(t) = h_+(v_+(t))$ for all $t\le t_0$.  
\end{enumerate}\end{lemma} 

\begin{proof} 
To prove statement (1), without loss of generality, suppose $M_0 \le \wt M_0$. Clearly $(\lambda, \min\{l, \wt l\}, M_0, \ep^*)$ satisfy \eqref{E:temp-17} and \eqref{E:temp-18} and $\Gamma_{\lambda, M_0,  \ep^*} \subset \Gamma_{\lambda, M_0,  \ep} \cap \Gamma_{\lambda, \wt M_0, \wt \ep}$. For any $v_{0+} \in X_+^k(\ep^*)$, Lemma \ref{L:LP-contraction} implies that $\CT (v_{0+}, \cdot)$ has a unique fixed point $v (\cdot) \in \Gamma_{\lambda, M_0,  \ep^*}$, which is also its unique fixed point in both $\Gamma_{\lambda, M_0,  \ep}$ and  $\Gamma_{\lambda, \wt M_0, \wt \ep}$. Therefore we obtain $h_+(v_{0+}) = \wt h_+ (v_{0+}) = v_-(0)$. 

To prove statement (2), without loss of generality, suppose $0 \le \wt \lambda \le \lambda$. 
Let 
\[
\ep_0 = \min \{ \ep, \ \wt M_0 (1-L_1(\lambda)) C_0^{-2(k+1)}\wt \ep/ 2 \}. 
\]
For any $v_{0+} \in X_+^k(\ep_0)$, Lemma \ref{L:LP-contraction} implies that $\CT (v_{0+}, \cdot)$ has fixed points $v (\cdot) \in \Gamma_{\lambda, M_0,  \ep}$ and $\wt v(\cdot) \in \Gamma_{\wt \lambda, \wt M_0, \wt \ep}$. From  \eqref{E:UM-1}, the choice of $\ep_0$, and $\wt  \lambda \le \lambda$, it holds $v(\cdot) \in \Gamma_{\wt \lambda, \wt M_0, \wt \ep}$. The uniqueness of the fixed point implies $v = \wt v$ and thus $h_+(v_{0+}) = \wt h_+ (v_{0+})$. 

Finally we prove statement (3). Let 
\[
\wt M_0 = 2 C_0^{2(k+1)}/(1-L_1(\wt \lambda)), \quad \wt l = (1+ L_1(\wt \lambda))/2, 
\]
and $\wt \ep>0$ be sufficiently small such that  \eqref{E:temp-18} is satisfied by $\wt \lambda$, $\wt l$, $\wt M_0$, and $\wt \ep$, which thus defines a mapping $\wt h_+: X_+^k (\wt \ep) \to X_-^k$. According to the above statement (2), there exists $\ep_0>0$ such that $\wt h_+ = h_+$ on $X_+^k(\ep_0)$. Since $\wt \lambda>0$, there exists $t_0 <0$ such that $v (t_0 +\cdot) \in \Gamma_{\wt \lambda, \wt M_0, \wt \ep}$ and $|v(t)|_{X^k} < \ep_0$ for all $t \le t_0$. Therefore $v (t_0 +\cdot)$ is the unique fixed point of $\CT(v_+(t_0), \cdot)$ in $\Gamma_{\wt \lambda, \wt M_0, \wt \ep}$ and thus $ v_-(t)= \wt h_+(v_+(t))=h_+(v_+(t))$ for all $t \le t_0$ by Theorem \ref{T:QLPDE-UM}(2). 
\end{proof}

\begin{remark} \label{R:varyingR}
For $R \in (0, R_0]$, let $C_0(R) \le C_0$ and $C_f(R) \le C_f$ be upper bounds  satisfying assumption (C.1)--(C.3) on $X_+^{k-1}(R)$ and $X_+^k(R)$, respectively. The set $\Sigma_+(R)$ in (C.4) can be defined accordingly with $L_1 (\lambda, R) \le L_1(\lambda)$ depending on both $\lambda$ and $R$. For any $\lambda \in \Sigma_+ \subset \Sigma_+(R)$, $\wt l \in (L_1(\lambda, R), 1)$, and $\wt M_0> \frac {2C_0(R)^{2(k+1)}}{1-L_1(\lambda, R)}$, Theorem \ref{T:QLPDE-UM} implies that there exist $\wt \ep>0$ mapping $\wt h_+ (R, \cdot): X_+^k(\wt \ep) \to X_-^k$ satisfying the same properties. The same proof as in Lemma \ref{L:UM-1}(1) yields that $\wt h_+$ and $h_+$ (obtained from $R_0$, the same $\lambda$, $l = \max\{\wt l, \frac 12(1+L_1(\lambda)) \}$, and $M_0= \max\{\wt M_0, \frac {3C_0^{2(k+1)}}{1-L_1(\lambda)} \}$) coincide on the intersection of their domains.  Hence the mapping $h_+$ is independent of which $R \in (0, R_0]$ we start with. 
\end{remark}

\begin{remark} \label{R:lambda=0}
A corollary of Lemma \ref{L:UM-1}(2) is that, even if $0\in \Sigma_+$ and $h_+$ is constructed using $\lambda=0$, since there exists $0< \wt \lambda \in \Sigma_+$, small initial values $v_{0+} + h_+(v_{0+})$ still have exponentially decaying backward solutions. 
\end{remark}


\subsection{Stable and unstable manifolds of a class of nonlinear PDEs}  \label{SS:NLPDE-LInMa}

Consider 
\be \label{E:NLPDE-1} 
u_t= F(u), \qquad F(0)=0.
\ee
We assume  that there exist 
\[
2\le n \in \N, 
\quad \omega_\pm \in \R,  \;  \text{ and an open neighborhood } \; \CO \subset X^{n-1}  \; \text{ of } \; 0, 
\]
such that the following hold for $\CA =\BD F$, where $\CO_n = \CO \cap X^n$ 
is equipped with the $|\cdot|_{X^n}$ topology.  

\begin{enumerate} 

\item [(D.1)] There exist subspaces $X_{j\pm} \subset X$, $j=1,2,3$, such that, for any $0\le r\le n$, $X_{j\pm}^r = X_{j\pm} \cap X^r$ are closed subspaces of $X^r$, and  
\be \label{E:decom-1}
X^r= X_+^r \oplus X_-^r, \; \text{ where } \; X_\pm = \oplus_{j=1}^3 X_{j+}^r. 
\ee
Let $\Pi_{j\alpha} \in \BBL(X^r, X_{j\alpha}^r)$ and $\Pi_\alpha \in \BBL(X^r, X_\alpha^r)$, $j=1,2, 3$, $\alpha =+, -$ and $0\le r\le n$, be the associated bounded projections and 
\[
\CA_{j\alpha, j'\alpha'} (u) = \Pi_{j\alpha} \CA(u)|_{X_{j'\alpha'}}, \;\; \CA_{j\alpha} = \CA_{j\alpha, j\alpha}, \;\; \CA_{\alpha, \alpha'} (u) = \Pi_{\alpha} \CA(u)|_{X_{\alpha'}}, \;\; \CA_\alpha (u) = \CA_{\alpha, \alpha} (u). 
\]
In the decomposition \eqref{E:decom-1}, we assume $\CA(0)$ takes the following  upper triangular form 
\be \label{E:LinearD-1}
\CA(0)= \begin{pmatrix} \CA_{1+} (0) & \CA_{1+, 2+} (0) & \CA_{1+, 3+}(0) & 0 & 0 &0 \\ 0 & \CA_{2+} (0) & \CA_{2+, 3+} (0) & 0 & 0 & 0 \\ 0 & 0 & \CA_{3+}(0) & 0 & 0 &0 \\ 0 & 0 & 0 & \CA_{1-} (0) & \CA_{1-, 2-}(0) & \CA_{1-, 3-} (0) \\ 0 & 0 & 0 & 0 & \CA_{2-}(0) & \CA_{2-, 3-} (0) \\ 0 & 0 & 0 & 0 & 0 & \CA_{3-} (0)  \end{pmatrix}. 
\ee


\item [(D.2)] For $j=1,2,3$ and $1\le r \le n$,  assume 
\[
(\omega_\pm \mp 1 - \CA_{j\pm}(0))^{-1} \in 
\BBL( X_{j\pm}^{r-1}, X_{j\pm}^{r}).
\]

\item [(D.3)] The components of $\CA(u)$ 
satisfy, for 
$0\le r\le n-1$,  
\begin{align*} 
\CA_{j\alpha, j'\alpha'}
\in C^1 (\CO , \BBL(X_{j'\alpha'}^r, X_{j\alpha}^r)) \cap C^1 (\CO_n, \BBL(X_{j'\alpha'}^{n}, X_{j\alpha}^{n})), \quad j\alpha \ne j'\alpha', \\
 \CA_{j\alpha} - \CA_{j\alpha} (0) \in C^1 (\CO , \BBL(X_{j\alpha}^r)) \cap C^1 (\CO_n, \BBL( X_{j\alpha}^{n})), \quad j\ne 2, \\
\CA_{j\alpha} (0) \in \BBL(X_{j\alpha}^r, X_{j\alpha}^{r-1}) \cap \BBL( X_{j\alpha}^{n}, X_{j\alpha}^{n-1}), \quad j \ne 2, \; r\ge 1
\end{align*}
\be \label{E:regularity-1}
\CA_{2\pm} \in C^1 (\CO , \BBL(X_{2\pm}^r, X_{2\pm}^{r-1})) \cap C^1 (\CO_n, \BBL(X_{2\pm}^{n}, X_{2\pm}^{n-1})), \quad 1\le r.
\ee

\item [(D.4)] There exist $\CL_{j\pm} \in \BBL(X_{j\pm}, X_{j\pm}^*)$, $j=1,3$, and $\CL_{2\pm} \in C^1 (\CO, \BBL(X_{2\pm}, (X_{2\pm})^*))$ such that $\CL_{i\pm} (u) = \CL_{i\pm} (u)^* >0$, $i=1,2,3$, and for any $u \in \CO$ 
\be \label{E:coercivity-4}
\sup_{u \in \CO, w\in X_{i\pm} \setminus \{0\}} \Big\{ \frac { \langle \CL_{i\pm} (u) w, w \rangle}{|w|_{X_\pm}^2}, \, \frac {|w|_{X_\pm}^2}{ \langle \CL_{i\pm} (u) w, w \rangle} \Big\} < \infty. 
\ee
Moreover, 
\be \label{E:dissipativity-4}
\langle \CL_{2\pm} (u) w, \mp \CA_{2\pm} (u)w \rangle \le \mp \omega_\pm \langle \CL_{2\pm} (u) w, w \rangle, \quad \forall w \in Dom(\CA_{2\pm} (u)) = X_{2\pm}^1,
\ee
\be \label{E:dissipativity-5}
\langle \CL_{j\pm} w, \mp \CA_{j\pm} (0)w \rangle \le \mp \omega_\pm \langle \CL_{j\pm} w, w \rangle, \quad \forall w \in Dom(\CA_{j\pm}) = X_{j\pm}^1, \; j=1,3.
\ee

  
\end{enumerate}

\begin{remark} \label{R:(D.2)} 
Assumption (D.3) implies  
\be \label{E:(D.3)}
F(u)- \sum_{\alpha =\pm} \Pi_{2\alpha} F ( \Pi_{2\alpha} u) - \sum_{\substack{j=1,3, \, \alpha =\pm}} \CA_{j\alpha}(0) \Pi_{j\alpha} u \in C^2 (\CO, X^{n-1}) \cap C^2 (\CO_n, X^{n}). 
\ee
\end{remark}

In the following we will use the notations  
\[
X^r (R) = X_+^r (R) \oplus X_-^r(R), \quad X_\pm^r(R)= X_{1\pm}^r (R) \oplus X_{2\pm}^r (R) \oplus X_{3\pm}^r (R).
\]
The main result of this section (also of the paper) is the following local unstable manifold theorem. See Remark \ref{R:SM} for comments on local stable manifolds. 

\begin{theorem} \label{T:NLPDE-UM}
Assume (D.1)--(D.4). In addition, assume $\omega_+ > \lambda_0 > \max\{\omega_-, 0\}$, then there exist $\delta, C, M^*>0$ and $q_+: X_+^n (\delta) \to X_-^n$ 
such that the following hold. 
\begin{enumerate} 
\item For any $u_{0+} \in X_+^n (\delta)$, there exists a solution to \eqref{E:NLPDE-1}
\[
u(t) \in C^0 \big((-\infty, 0], \CO_n \big) \cap C^1 \big((-\infty, 0], \CO\big),  
\]
unique in the category 
\be \label{E:temp-20.5}
u_+ (0) = u_{0+}, \quad |u(t)|_{X^n} \le 2 M^* \delta e^{\lambda_0 t}, \; \forall t \le 0. 
\ee
Moreover $u(t)$ satisfies 
\be \label{E:UM-3}
u_-(0) = q_+ (u_{0+}), \quad |u(t)|_{X^n} \le M^* |u_{0+}|_{X_+^n}e^{\lambda_0 t}, \; \forall t \le 0,
\ee
\be \label{E:UM-3.1} 
\sup_{t\le 0}  |u(t)|_{X^n} e^{-\wt \lambda t} < \infty, \quad \forall \wt \lambda \in \big( \max\{\omega_-, 0\}, \omega_+\big). 
\ee
\item The above $u(t)$ defined by $u_{0+}$ also satisfies that, for any $t<0$ satisfying $u_+(t) \in X_+^n(\delta)$, it holds $q_+(u_+ (t)) = u_-(t)$.
\item Suppose the above solution $u(t)$ defined by $u_{0+}$ can be extended to $C^0 ((-\infty, T_0), X^n)$ for some $T_0>0$. Let $T = \sup\{ t >0 \mid u_+(\tau ) \in X_+^n (\delta), \, \forall \tau \in [0, t) \} \in (0, T_0]$, then $q_+(u_+ (t)) = u_-(t)$ for all $t \in [0, T)$. 
\item $q_+ \in C^{1,1} (X_+^n (\delta), X_-^{n-2})$ satisfies, for any $\delta' \in (0, \delta]$ and $u_{0+}, \wt u_{0+} \in X_+^n(\delta')$, 
\be \label{E:Lip-3} 
|q_+(u_{0+})|_{X_-^n} \le C \delta' |u_{0+}|_{X_+^n}, \quad |q_+ (u_{0+}) - q_+ (\wt u_{0+})|_{X_-^{n-1}} \le C \delta' |u_{0+} - \wt u_{0+}|_{X_+^{n-1}},
\ee
and their corresponding solutions $u(t)$ and $\wt u(t)$, $t\le 0$, satisfy 
\[
|u(t) - \wt u(t)|_{X^{n-1}} \le C e^{\lambda_0 t} |u_{0+} - \wt u_{0+}|_{X_+^{n-1}}. 
\] 
\item Suppose $\lambda \in \big( \max\{\omega_-, 0\}, \omega_+\big)$ and $u(t)$, $t\le 0$, is a solution to \eqref{E:NLPDE-1} such that $\sup_{t\le 0} e^{- \lambda t}|u(t)|_{X^n}  <\infty$, then $\exists t_0\le 0$ such that $u_-(t) = q_+(u_+(t))$, $\forall  t\le t_0$.  
\end{enumerate} 
\end{theorem}

The graph $W^+ \triangleq graph(q_+)$ is often referred to as the local unstable manifold of the
equilibrium $u=0$ (or a strongly unstable manifold if $\omega_->0$, too). 
It is locally unique due to \eqref{E:UM-3.1} and statement (5), see also Lemma \ref{L:UM-1}.
The remarks given below Theorem \ref{T:QLPDE-UM} 
are still valid. In particular, the above statement (4) implies that $W^+$ is tangent to $X_+^n$ at $u=0$. 

\begin{remark} \label{R:SM}
If we assume $\omega_- < \min\{\omega_+, 0\}$ instead, then there exists $q_-: X_-^n (\delta) \to X_+^n$ such that the conclusions of Theorem \ref{T:NLPDE-UM} still hold, except for $t\ge 0$ and $
\lambda_0 \in (\omega_-,  \min\{\omega_+, 0\}\big)$.  
The graph $W^- \triangleq graph(q_-)$ gives the local stable manifold of $0$. In addition, more analysis on the smoothness of the local stable/unstable manifolds are given in Theorem \ref{T:UM-smoothness} in Subsection \ref{SSS:smoothness-InMa}. In particular, the second Lipschitz constant in \eqref{E:Lip-3} can be slightly improved from $O(|u_{0+}|_{X_+^n})$ to $O(|u_{0+}|_{X_+^{n-1}})$ by the estimate on $\BD_{u_{0+}} q^+$ in Theorem \ref{T:UM-smoothness}. 
\end{remark}

Let us first make some comments on the assumptions, which are essentially the exponential dichotomy and regularity of $\CA(u)$. 
In the decomposition $X= X_+\oplus X_-$, $\CA(0)$ can be reduced to a $2\times 2$ blockwise diagonal matrix. However, in a class of Hamiltonian systems, the Hessian of the energies, which are often the natural candidates of $\CL_{j\pm} (u)$, may 
have some negative and degenerate directions. The form \eqref{E:LinearD-1} is motivated by the structural decomposition theorem for  linearized Hamiltonian PDEs in \cite{LZ22}. See Subsection \ref{SS:examples-Ham}. 
Very often in the construction of local unstable manifolds, the phase space can be decomposed such that 
\be \label{E:decom-2}
\dim X_{1\pm}, \dim X_{3-}<\infty, \quad X_{2+}=X_{3+}=\{0\}, \quad \omega_+> \max\{0, \omega_-\},
\ee
while for local stable manifold, a different decomposition would satisfy similar properties.  
Under assumptions (D.1)--(D.4), in the subspaces $X_{2\pm}$, equation \eqref{E:NLPDE-1} has the worst interaction between the nonlinearity and the regularity issues, often in the form of highest order derivatives appearing nonlinearly. As seen in Section \ref{S:examples}, the positive quadratic forms $\CL_{2\pm} (u)$ in (D.4) are often the Hessian of the energies restricted to $X_{2\pm}$ which help handle the regularity issues. The systems are essentially semilinear in the directions of $X_{j\pm}$, $j=1,3$, due to the regularity assumptions in (D.3). The dissipativity assumption (D.4) allows us to apply Proposition \ref{P:linear} to obtain directly the linear flows in the proof. In some systems, there are no obvious natural choices of $\CL_{j\pm}$, $j=1,3$, as the Hessian of the energies may not have clear signs in $X_{j\pm}$. Instead, as in the case of \eqref{E:decom-2}, it is often easier to verify directly 
\begin{enumerate} 
\item [(D.5')] $\CA_{j\pm}(0)$, $j=1,3$, generate  strongly $C^0$ semigroups $e^{s\CA_{j\pm} (0)}$, $\mp s \ge 0$, on $X_{j\pm}$, 
such that 
\be \label{E:ED-1}
\sup_{\mp s\ge 0} e^{-\omega_\pm s}|e^{s\CA_{j\pm} (0)}|_{\BBL(X_{j\pm})} < \infty. 
\ee
\end{enumerate}
Due to (D.2), the above semigroup estimate can be extended to $X_{j\pm}^r$ by applying $\CA_{j\pm}(0)$ to $e^{s\CA_{j\pm} (0)}$ repeatedly. 
On the one hand, replacing \eqref{E:coercivity-4} for $j \ne 2$ and \eqref{E:dissipativity-5} by (D.5') is sufficient, indeed. In this case, instead of applying Proposition \ref{P:linear} in the whole $X_\pm$ as in the proof below, but only in $X_{2\pm}$, and then one may solve the complete linear evolutions by working in  $X_{j\pm}$, $j=1,3$, directly using (D5') and the upper triangular structure \eqref{E:LinearD-1}. On the other hand, 
with the slight additional assumption that $\CA_{j\pm}(0)$, $j=1,3$, generate groups $e^{s \CA_{j\pm}(0)}$, $ s\in \R$,
we give the following general lemma which  
provides positive definite quadratic forms $\CL_{j\pm}$ desired in \eqref{E:dissipativity-5}. 

\begin{lemma} \label{L:Lumer-P-1}
If $A: X^1= D(A) \to X$ generates a group $e^{tA}$ satisfying 
\[
|e^{sA}| \le M e^{\lambda s}, \quad  s\ge 0,
\]
then, for any $\omega > \lambda$, let $L: X \to X^*$ be defined as 
\be \label{E:quadraticF}
\langle Lu, v\rangle  = \int_0^\infty e^{-2\omega s} (e^{sA}u,  e^{sA}v) ds, 
\ee
then $|L|_{ \BBL(X, X^*)} \le \frac {M^2}{2(\omega -\lambda)}
$ and 
$L^*=L$.
Moreover $L^{-1} \in \BBL(X^*, X)$ and, for any $u \in X$, 
\[
\langle Lu, u\rangle = \int_0^\infty e^{-2\omega s} |e^{sA} u|^2 ds \ge \int_0^1 e^{-2\omega s} |e^{sA} u|^2 ds \ge  \int_0^1 e^{-2\omega s}  ds (\sup_{s\in [0, 1]} |e^{-sA}|)^{-2} |u|^2, 
\]
and for any $u \in D(A)$, 
\[
\langle LAu, u\rangle = \int_0^\infty e^{-2\omega s} ( A e^{sA} u,  e^{sA}u) ds = \frac 12 \int_0^\infty e^{-2\omega s} \frac d{ds} |e^{sA} u|^2 ds = - \frac 12 |u|^2 + \omega \langle Lu, u\rangle. 
\]
\end{lemma} 

As the off-diagonal blocks $\CA_{j\alpha, j'\alpha'} (u)$ of $\CA(u)$ are assume to be bounded and depend on $u$ smoothly in (D.3), in principle they would neither contribute much to the linear growth/decay nor pose challenges in the regularity analysis. Often (D.3) can be verified easily in concrete systems, e.~g.~when properties such as \eqref{E:decom-2} is satisfied. 

Again  we do not assume the well-posedness of the semigroup $e^{s \CA(0)}$ for $s\ge 0$, so this framework could be applied to ill-posed PDEs like elliptic equations on cylindrical domains. 

Finally one notices that, unlike in Appendix \ref{SS:NLPDE-LWP}, a condition like (B.4) is not assumed. This is due to $F(0)=0$. In fact, \eqref{E:tameE-1} would not be needed in Appendix \ref{SS:NLPDE-LWP} if $F(u_*)=0$ there. See Remark \ref{R:LWP-R-2}. 

To prove the above theorem, we shall transform the nonlinear PDE \eqref{E:NLPDE-1} into a quasilinear PDE in the form of \eqref{E:QLPDE-1} satisfying assumptions (C.1)--(C.4) given in Subsection \ref{SS:QLPDE-LInMa}. Fix $\sigma >0$ to be determined later. For $u\in \CO \subset X^{n-1}$ close to $0$, let $v= \CB(u)$ be defined by 
\be \label{E:CB-2}
v_{j\pm} = \Pi_{j\pm} v=\CB_{j\pm} (u) \triangleq  \sigma ^{2-j} \Pi_{j\pm} \big( F(u) - (\omega_\pm \mp 1) u \big).  
\ee
Clearly $\CB(0)=0$ and 
\[
\CB \in  C^{2} (\CO, X^{n-2}) \cap C^{2} (\CO_n, X^{n-1}), \quad 
\BD\CB_{j\pm}(u) = \sigma ^{2-j} \Pi_{j\pm} (\CA(u) - (\omega_\pm \mp 1)). 
\]
Hence $\BD\CB(0)$ takes a $6 \times 6$ blockwise upper triangular form much as in \eqref{E:LinearD-1} with isomorphic diagonal entries 
\[
\sigma ^{2-j} (\CA_{j\pm} (0) - (\omega_\pm \mp 1)) \in \BBL(X_{j\pm}^r, X_{j\pm}^{r-1}), \quad 1\le r\le n, 
\]
which implies that $\BD\CB(0)$ is also isomorphic. From the Implicit Function Theorem, $\CB$ 
is a local diffeomorphism near $0 \in X^{n'}$, $n'=n-1, n$. Namely, there exists $\delta_0 \in (0,1]$ such that 
$\CB \in C^2 (X^{n'} (\delta_0), X^{n'-1})$ and $\CB^{-1} \in C^2 (X^{n'-1} (\delta_0), X^{n'})$ are both diffeomorphisms  to their ranges. 

From \eqref{E:NLPDE-1}, the evolution of $v=\CB(u)$ satisfies 
\[
\p_t v_{j\pm} =  
\sigma ^{2-j} \Pi_{j\pm} \big( \CA(u) - (\omega_\pm \mp 1) \big) F(u).   
\]
According to the definition of $\CB$, 
\[
F(u) = \sum_\pm \sum_{j'=1}^3 \sigma^{j'-2} v_{j'\pm} + \big( (\omega_+-1) \Pi_++  (\omega_- +1) \Pi_-\big) u. 
\]
Hence we have 
\[ \begin{split} 
\p_t v_{j\pm} =& \sigma ^{2-j} \Pi_{j\pm} \big( \CA(u) - (\omega_\pm \mp 1) \big) \Big( \sum_{\alpha'=\pm} \sum_{j'=1}^3 \sigma^{j'-2} v_{j'\alpha'} + \big( (\omega_+-1) \Pi_++  (\omega_- +1) \Pi_-\big) u \Big) \\
\triangleq & \sum_{j'=1}^3 \sigma^{j'-j} \CA_{j\pm, j'\pm} (u) v_{j'\pm} +  \wt F_{j\pm}(u),
\end{split}\]
where 
\begin{align*}
\wt F_{j\pm} (u) = & \sigma ^{2-j} \Pi_{j\pm} \big( \CA(u)  \big( (\omega_+-1) \Pi_++  (\omega_- +1) \Pi_-\big) u - (\omega_\pm \mp 1) F(u) \big) \\
& + \sum_{j'=1}^3 \sigma^{j'-j} \CA_{j\pm, j'\mp} (u) \CB_{j'\mp} (u).
\end{align*}
It can be rewritten as 
\be \label{E:QLPDE-2}
\p_t v_\pm = \BA_\pm (v) v_\pm + f_\pm (v), 
\ee
where 
\be \label{E:QLPDE-2-A-f}
\BA_\pm (v) w = \sum_{j, j'=1}^3 \sigma^{j'-j} \CA_{j\pm, j'\pm} (\CB^{-1} (v)) \Pi_{j'\pm} w, \;\; w\in X_\pm;  \quad f_\pm (v) = \sum_{j=1}^3 \wt F_{j\pm} (\CB^{-1} (v)). 
\ee

For equation \eqref{E:QLPDE-2}, 
let 
\be \label{E:BL-1}
\BL_\pm (v) = diag ( \CL_{1\pm}, \CL_{2\pm}(\CB^{-1}(v)), \CL_{3\pm} ) \in \BBL(X_\pm, (X_\pm)^*), \quad v\in X^{n-2} (\delta_0), 
\ee
where $\CL_{j\pm}$ was given in assumption (D.4). 

\begin{lemma} \label{L:C1C2}
Let 
$\BA_\pm(v)$ and  
$\BL_\pm(v)$ be defined as in \eqref{E:QLPDE-2-A-f} and \eqref{E:BL-1}, then for any $\lambda_\pm$ satisfying $ \pm (\omega_\pm - \lambda_\pm) >0$, 
there exist  $R_0 \in (0, \delta_0), \sigma, C_0>0$ determined by 
$|\lambda_\pm -\omega_\pm|$ and the norms of $\CA_\pm (u)$ and $\CL_{2\pm}(u)$ involved in (D.1)--(D.4),  
such that $\BA_\pm (v)$ and $\BL_\pm(v)$ satisfy assumptions (C.1) and (C.2) given in Subsection \ref{SS:QLPDE-LInMa} for $k=n-1$. 
\end{lemma}

\begin{proof} 
Due to the upper triangular form \eqref{E:LinearD-1} of $\CA(0)$ and the definition \eqref{E:QLPDE-2-A-f} of $\BA_\pm(v)$, the exponents of $\sigma$ in the off-diagonal terms of $\BA_\pm (0)$ are all positive. They can be made arbitrarily small by taking $\sigma>0$ sufficiently small. Hence from the regularity assumptions (D.3), the dissipativity assumptions \eqref{E:dissipativity-4}--\eqref{E:dissipativity-5}, 
and the definitions of $\BA_\pm(v)$ and $\BL_\pm (v)$, there exist $R_0 \in (0, \delta_0)$ and $ \sigma >0$ such that $\BL_\pm(v)$ and $\BA_\pm (v)$ satisfy \eqref{E:dissipativity-3} for all $v \in X^{n-2} (R_0)$. 
Apparently the rest of the assumptions (C.1)--(C.2) are satisfied on $X^{n-2} (R_0)$ for some $C_0$ determined the norms of $\CA$ and $\CL_{j\pm}$ on $X^{n-1} (\delta_0)$. 
\end{proof}


Concerning assumptions (C.3)--(C.4) in Subsection \ref{SS:QLPDE-LInMa}, we analyze $f$ defined in  \eqref{E:QLPDE-2-A-f}. 

\begin{lemma} \label{L:Df-1}
It holds that $f_\pm \in C^1 (X^{n'} (R_0), X^{n'})$, $n'=n-1, n-2$, $f_\pm (0)=0$, and there exists $C>0$ depending on the norms of $\CA_\pm (u)$ and $\CL_{2\pm}(u)$ involved in (D.1)--(D.4) so that 
\[
|\BD f_\pm (v) |_{\BBL(X^{n'})} \le C (1+ \sigma^{-2}) |v|_{X^{n'}}, \quad \forall v \in X^{n'} (R_0). 
\]
\end{lemma} 

\begin{proof} 
The $C^1$ smoothness of $f_\pm$ and $f_\pm (0)=0$ follow directly from the regularity assumption (D.3) and the definition $\CB$. To show $\BD f_\pm (0)=0$, for any $v \in  X^{n'}$ and $w \in X^{n'}$, $n'=n-1, n-2$,  one may compute 
\begin{align*}
\Pi_{j\pm} \BD f_\pm (v) w = & \sigma ^{2-j} \Pi_{j\pm} \Big( \big(\CA(u)  \big( (\omega_+-1) \Pi_++  (\omega_- +1) \Pi_-\big) - (\omega_\pm \mp 1) \CA(u) \big) \BD\CB^{-1} (v) w \\
& + \BD\CA(u) (\BD\CB^{-1} (v)w) \big( (\omega_+-1) \Pi_++  (\omega_- +1) \Pi_-\big) u \Big) \\
& + \sum_{j'=1}^3 \sigma^{j'-j} \Big( \CA_{j\pm, j'\mp} (u) w_{j'\mp} + \BD \CA_{j\pm, j'\mp} (u) (\BD\CB^{-1} (v)w) v_{j'\mp} \Big), 
\end{align*}
where $u = \CB^{-1} (v)$. Consolidating the first part of the expression, we obtain 
\begin{align*}
\Pi_{j\pm} \BD f_\pm (v) w = & \sigma ^{2-j} \Pi_{j\pm} \Big( \mp(\omega_+- \omega_- -2) \CA(u) \Pi_\mp \BD\CB^{-1} (v) w \\
& + \BD \CA(u) (\BD\CB^{-1} (v)w) \big( (\omega_+-1) \Pi_++  (\omega_- +1) \Pi_-\big) u \Big) \\
& + \sum_{j'=1}^3 \sigma^{j'-j} \Big( \CA_{j\pm, j'\mp} (u) w_{j'\mp} + \BD \CA_{j\pm, j'\mp} (u) (\BD\CB^{-1} (v)w) v_{j'\mp} \Big). 
\end{align*}
Therefore 
\begin{align*}
\Pi_{j\pm} \BD f_\pm (0) w = & \mp (\omega_+- \omega_- -2) \sigma ^{2-j} \Pi_{j\pm} \CA(0) \Pi_\mp \BD\CB^{-1} (0) w  + \sum_{j'=1}^3 \sigma^{j'-j} \CA_{j\pm, j'\mp} (0) w_{j'\mp}. 
\end{align*}
Due to the invariance of $X_\pm$ under $\CA(0)$ assumed in \eqref{E:LinearD-1}, we obtain $\BD f_\pm (0)=0$. The estimate on $\BD f_\pm$ follows from the regularity assumptions in (D.3). 
\end{proof} 

Before proving Theorem \ref{T:NLPDE-UM} by applying Theorem \ref{T:QLPDE-UM}, we need the following technical lemma showing the image of a manifold as given in Theorem \ref{T:QLPDE-UM} under $\CB^{-1}$ is still a manifold. 

\begin{lemma} \label{L:LipMap-1} 
Suppose $\ep_0, a_0, a_1 >0$ and $h: \overline {X_+^{n-1} (\ep_0)} \to X_-^{n-1}$ such that and 
\be \label{E:Lip-2}
|h(v_+)|_{X_-^{n-1}} \le a_0 |v_+|_{X_+^{n-1}}, \quad |h(v_{1+}) - h(v_{2+})|_{X_-^{n-2}} \le a_1 |v_{1+}-v_{2+}|_{X_+^{n-2}}, 
\ee
for all $v_+, v_{1+}, v_{2+} \in \overline {X_+^{n-1} (\ep_0)}$. Let 
\[
\phi: \overline{X_+^{n-1} (\ep_0)} \to \overline{X_+^{n}} \; \text{ as } \; \phi(v_+) = \Pi_+ \CB^{-1} \big( v_+ + h(v_+)\big). 
\]
Then there exist $\delta_1 \in (0, \delta_0)$ (see \eqref{E:temp-25}), $\ep_1 \in (0, \ep_0)$, and $q : \overline{X_+^{n} (\delta_1)} \to X_-^{n}$, such that $\phi^{-1}: \overline{X_+^{n} (\delta_1)} \to \overline{X_+^{n-1} (\ep_0)}$ is well defined and  
\[
\overline{X_+^{n-1} (\ep_1)} \subset \phi^{-1} \big(\overline{X_+^{n} (\delta_1)}\big); \quad  \CB \big( u_+ + q (u_+) \big) = \phi^{-1}(u_+) + h(\phi^{-1}(u_+)), \; \forall u_+ \in \overline{X_+^{n} (\delta_1)}. 
\]
Moreover, 
\begin{align*}
& |\phi^{-1}(u_+)|_{X_+^{n-1}} \le 2 |\BD \CB(0)|_{\BBL(X_+^{n}, X_+^{n-1})} |u_+|_{X_+^n}, 
& |\phi(v_+)|_{X_+^{n}} \le 2 |\BD \CB(0)^{-1}|_{\BBL(X_+^{n-1}, X_+^{n})} |v_+|_{X_+^{n-1}}, 
\end{align*}
and there exists $C>0$ 
determined by $a_0, a_1$, and the norms of $\CA_\pm (u)$  involved in (D.1)--(D.4) such that 
for any $u_+, u_{1+}, u_{2+} \in \overline{X_+^n (\delta)}$, we have 
\[
\begin{split}
& |q(u_+)|_{X_-^{n}} \le \big( a_0 |\BD \CB^{-1} (0)|_{\BBL(X_-^{n-1}, X_-^{n})} |\BD \CB(0)|_{\BBL(X_+^{n}, X_+^{n-1})} + C\delta\big) |u_+|_{X_+^n},  \\
&|q(u_{1+}) -  q(u_{2+})|_{X_-^{n-1}} \le \big(a_1 |\BD \CB^{-1} (0)|_{\BBL(X_-^{n-2}, X_-^{n-1})} |\BD \CB(0)|_{\BBL(X_+^{n-1}, X_+^{n-2})} + C\delta\big) |u_{1+} -  u_{2+}|_{X_-^{n-1}}. 
\end{split} \]
\end{lemma}

The definition of the local diffeomorphism $\CB$ and the blockwise structure \eqref{E:LinearD-1} imply 
\be \label{E:temp-21}
\BD \CB(0) X_\pm^{r+1} = X_\pm^{r}, \quad \Pi_\pm \BD \CB(0) \Pi_\pm \BD \CB^{-1} (0)|_{X_\pm^r} = I_{X_\pm^r}, \quad \forall \, 0\le r \le n-1. 
\ee
Hence the operator norms $|\BD \CB(0)|_{\BBL(X_\pm^{n}, X_\pm^{n-1})}$ and $|\BD \CB^{-1} (0)|_{\BBL(X_\pm^{n-1}, X_\pm^{n})}$ are meaningful. 
If the spaces $X_\pm^r$ were independent of $r$, the graph of $h$ is a Lipschitz manifold.
Hence the lemma follows directly from a standard argument based on the Implicit Function Theorem. However, the Lipschitz assumption here is under a weaker norm. We have to carefully go through the proof based on the Contraction Mapping Theorem. 

\begin{proof} 
Let $M=\max \{1, a_0\}$ and, without loss of generality, we may assume that $\ep_0$ is so small that $\overline {X^{n-1} (M \ep_0)} \supset (id + h) \big(\overline {X_+^{n-1} (\ep_0)}\big)$ is contained in the domain of $\CB^{-1}$ and thus 
\[
\CS(u_+, v_+) = \Pi_+ \BD \CB(0) \big (u_+ - \Pi_+ \CB^{-1} \big( v_+ + h(v_+) \big) \big)+ v_+ , \quad v_+ \in \overline {X_+^{n-1} (\ep_0)}, \; u_+ \in X_+^n, 
\]
is well defined
and 
\be \label{E:temp-22}
\CS(u_+, v_+) = v_+ \; \text{ iff } \; u_+ = \phi(v_+)= \Pi_+ \CB^{-1} \big( v_+ + h(v_+) \big). 
\ee

Let 
\begin{align*}
C_*  = \max &  \big\{  ((1+a_0)^2/2) |\BD \CB(0)|_{\BBL(X_+^{n}, X_+^{n-1})} |\Pi_+ \BD^2 \CB^{-1} |_{C^0 (\overline {X^{n-1} (M \ep_0)}, \BBL(X^{n-1} \otimes X^{n-1}, X_+^{n}))}, \\ 
& (1+a_0) (1+a_1) |\BD \CB(0)|_{\BBL(X_+^{n-1}, X_+^{n-2})}  |\Pi_+ \BD^2 \CB^{-1} |_{C^0 (\overline {X^{n-1} (M \ep_0)}, \BBL(X^{n-1} \otimes X^{n-2}, X_+^{n-1}))} \big\}. 
\end{align*}
From \eqref{E:temp-21}, 
we have  
\begin{align*}
\CS(u_+, v_+) 
=&  \Pi_+ \BD \CB(0) \Big (u_+ - \Pi_+ \int_0^1 \BD \CB^{-1} \big( \tau(v_+ + h(v_+)) \big) \big( v_+ + h(v_+) \big)- \BD \CB^{-1} (0) v_+ d\tau \Big)\\
=& \Pi_+ \BD \CB(0) \Big (u_+ - \Pi_+ \int_0^1 \Big( \BD \CB^{-1} \big( \tau(v_+ + h(v_+)) \big) - \BD \CB^{-1} (0)\Big) \big( v_+ + h(v_+) \big) d\tau \Big). 
\end{align*}
Hence  from \eqref{E:Lip-2} we obtain
\be \label{E:temp-23} \begin{split}
\big|\CS(u_+, v_+) - \Pi_+ \BD&  \CB(0) u_+ \big|_{X_+^{n-1}} \le 
((1+a_0)^2/2) |\BD \CB(0)|_{\BBL(X_+^{n}, X_+^{n-1})} 
\\
& \times | \Pi_+ \BD^2 \CB^{-1} |_{C^0 (\overline {X^{n-1} (M \ep_0)}, \BBL(X^{n-1} \otimes X^{n-1}, X_+^n))} |v_+|_{X_+^{n-1}}^2 \le C_* |v_+|_{X_+^{n-1}}^2. 
\end{split} \ee
Similarly, for any $\ep \in (0, \ep_0]$ and $v_{1+}, v_{2+} \in \overline {X_+^{n-1} (\ep)}$, we have 
\begin{align*}
\CS(&u_+,  v_{1+}) - \CS(u_+, v_{2+}) =   \Pi_+ \BD \CB(0) \Pi_+ \int_0^1 \BD \CB^{-1} (0) \big(v_{1+} + h(v_{1+}) - v_{2+} - h(v_{2+}) \big) \\
&\; \; - \BD \CB^{-1} \Big( \tau(v_{1+} + h(v_{1+})) + (1-\tau) (v_{2+} - h(v_{2+})  \Big) \big(v_{1+} + h(v_{1+}) - v_{2+} - h(v_{2+}) \big) d\tau, 
\end{align*}
which implies 
\be \label{E:temp-24}
|\CS(u_+, v_{1+}) - \CS(u_+, v_{2+})|_{X_+^{n-2}} \le 
C_* \ep |v_{1+} - v_{2+}|_{X_+^{n-2}}. 
\ee
Let 
\be \label{E:temp-25}
\ep_2 = \min \{\ep_0, 1/(4C_*),  \delta_0 |\BD \CB(0)|_{\BBL(X_+^{n}, X_+^{n-1})} \}, \quad \delta_1= \ep_2/(2  |\BD \CB(0)|_{\BBL(X_+^{n}, X_+^{n-1})}) < \delta_0, 
\ee
and   
\[
\ep \le \ep_2, \quad  \delta = \ep/(2  |\BD \CB(0)|_{\BBL(X_+^{n}, X_+^{n-1})}) \le \delta_1, 
\]
then for any $u_+ \in \overline{X_+^{n} (\delta)}$, $\CS(u_+, \cdot)$ is a contraction mapping on $\overline {X_+^{n-1} (\ep)}$ with Lipschitz constant $1/4$ in the $|\cdot|_{X_\pm^{n-2}}$ norms. 
Therefore, much as in the proof of Theorem \ref{T:QLPDE-LWP}, each iteration sequence $\CS(u_+, \cdot)^{(k)} v_{0+}$ converges in $X_+^{n-2}$ to a limit $v_+ \in X_+^{n-2}$ which is independent of $v_{0+}$. Moreover each such iteration sequence also has a subsequence converging weakly in the $|\cdot|_{X_+^{n-1}}$ topology to a limit in $\overline {X_+^{n-1} (\ep)}$ and thus $\CS(u_+, \cdot)$ has the unique fixed point in $ \overline {X_+^{n-1} (\ep)}$, which is clearly $v_+= \phi^{-1} (u_+)$ according to \eqref{E:temp-22}. 
Let 
\[
\ep_1= \delta_1 / \big( (1+ a_0) |\Pi^+ \CB^{-1}|_{C^0(X^{n-1} (M_0 \ep_0), \BBL(X^{n-1}, X_+^n))} \big) \implies \overline{X_+^{n-1} (\ep_1)} \subset \phi^{-1} \big(\overline{X_+^{n} (\delta_1)}\big).
\]
Define  
\[
q (u_+) = \Pi_- \CB^{-1} \big( \phi^{-1}(u_+) + h(\phi^{-1}(u_+)) \big) \implies \CB \big( u_+ + q (u_+) \big) = \phi^{-1}(u_+) + h(\phi^{-1}(u_+)),
\]
where \eqref{E:temp-22} was used. 
From \eqref{E:temp-23} and \eqref{E:temp-24}, we obtain, for $u_+, u_{1+}, u_{2+} \in \overline{X_+^n (\delta)}$, 
\[
\begin{split}
& |\phi^{-1}(u_+)|_{X_+^{n-1}} \le (1- C_*\ep)^{-1} |\BD \CB(0)|_{\BBL(X_+^{n}, X_+^{n-1})} |u_+|_{X_+^n}, \\ 
&|\phi^{-1}(u_{1+}) - \phi^{-1}(u_{2+}) |_{X_+^{n-2}} \le 
(1- C_*\ep)^{-1} |\BD \CB(0)|_{\BBL(X_+^{n-1}, X_+^{n-2})} |u_{1+} - u_{2+}|_{X_+^{n-1}}. 
\end{split} \]
The above first inequality proves the desired estimate on $\phi^{-1}$. The estimate on $\phi$ follows from a similar argument using its defintion. 
Much as in the derivation of \eqref{E:temp-23} and \eqref{E:temp-24}, along with \eqref{E:temp-21}, they further imply
\begin{align*} 
|q(u_+)|_{X_-^n} \le & \big| \Pi_- \BD \CB^{-1} (0) h(\phi^{-1}(u_+)) |_{X_-^n} \\
& + |\Pi_- \BD^2 \CB^{-1}|_{C^0 (\overline {X^{n-1} (M \ep_0)}, \BBL(X^{n-1} \otimes X^{n-1}, X_-^n))} \big|  \phi^{-1}(u_+) + h(\phi^{-1}(u_+)) \big|_{X^{n-1}}^2 /2,
\end{align*}
\begin{align*} 
|q(u_{1+}) -  q(u_{2+})|_{X_-^{n-1}} \le & \big| \Pi_- \BD \CB^{-1} (0) \big(h(\phi^{-1}(u_{1+})) - h(\phi^{-1}(u_{2+})) \big) \big|_{X_-^{n-1}} \\
& + |\Pi_- \BD^2 \CB^{-1}|_{C^0 (\overline {X^{n-1} (M \ep_0)}, \BBL(X^{n-1} \otimes X^{n-2}, X_-^{n-1}))} \\
& \quad \times  \big|\phi^{-1}(u_{1+}) + h(\phi^{-1}(u_{1+})) + \phi^{-1}(u_{2+}) + h(\phi^{-1}(u_{2+}))\big|_{X^{n-1}} \\
& \quad \times  \big|  \phi^{-1}(u_{1+}) + h(\phi^{-1}(u_{1+})) - \phi^{-1}(u_{2+}) - h(\phi^{-1}(u_{2+})) \big|_{X^{n-2}}.  
\end{align*}
The desired estimates follow immediately. 
\end{proof}

We are ready to complete the proof of Theorem \ref{T:NLPDE-UM}.

\begin{proof}[Proof of Theorem \ref{T:NLPDE-UM}.]
Let 
\be \label{E:temp-26}
\lambda_\pm \in (\omega_-, \omega_+), \; \text{ s.~t. } \; \lambda_0 
\in (\lambda_-, \lambda_+). 
\ee
In the rest of the proof, we shall not repeat if a constant is determined by the norms of $\CA_\pm (u)$ and $\CL_{2\pm}(u)$ involved in (D.1)--(D.4), $|\lambda_0-\lambda_\pm|$, and $|\omega_\pm - \lambda_\pm|$. We will use `$C$' to denote a generic constant with such dependence. 

According to Lemma \ref{L:C1C2}, there exist $\sigma\in (0, 1)$, $R_1>0$, and $C_0\ge 1$ 
such that $v= \CB(u)$ transforms \eqref{E:NLPDE-1} into \eqref{E:QLPDE-1} and $\BA_\pm (v)$ and $\BL_\pm(v)$ defined in \eqref{E:QLPDE-2-A-f}--\eqref{E:BL-1} satisfy assumptions (C.1) and (C.2) given in Subsection \ref{SS:QLPDE-LInMa} for $k=n-1$ and $R_0=R_1$. 
Concerning assumption (C.3) for any $R\in (0, R_1]$, let 
\be \label{E:temp-27}
C_f (R) = \max \{ |\BD f_\pm|_{C^0( X^{n-1} (R), \BBL(X^{n-1}))}, \,  |\BD f_\pm|_{C^0( X^{n-2} (R), \BBL(X^{n-2}))} \} \le C R,
\ee
where Lemma \ref{L:Df-1} was used.
Since $f_\pm (0)=0$ 
there exists $R_*\in (0, R_1]$ such that, for any $R_0 \in (0, R_*]$, (C.1)--(C.4) are satisfied by $\BA_\pm (v)$, $\BL_\pm(v)$, and $f_\pm(v)$ with the above constants $k=n-1, C_0, R_0, C_f(R_0)$ and $\lambda_\pm$ with $\lambda_0  \in \Sigma_+$ with $L_1(\lambda_0, R_0) \le L_1(\lambda_0, R_*)<\frac 12$. Here  $L_1$ was defined in assumption (C.4) and we also added the parameter $R_0$ to emphasize its (non-decreasing) dependence on $R_0$.

Take $\lambda= \lambda_0$ and 
\[
R_0\in (0, R_*], \quad  \bar M = (2C_0^{2n})/(1-L_1(\lambda_0, R_*)), \quad M_0 = 2\bar M, \quad l_0 (R_0) = 2L_1(\lambda_0, R_0) \le CR_0.
\]
From Theorem \ref{T:QLPDE-UM} along with Remark \ref{R:size-1}, for there exist $\ep_0 (R_0) = \frac {R_0}{CM_0}$ and $h_+: \overline{X_+^{n-1} (\ep_0 (R_0))} \to X_-^{n-1}$ satisfying the properties given in Theorem \ref{T:QLPDE-UM} for \eqref{E:QLPDE-1}. When $R_0$ varies, 
a.) according to Remark \ref{R:varyingR}, $h_+$ is independent of $R_0$ and we simply view it as defined on $\overline{X_+^{n-1} (\ep_0 (R_*))}$; and b.) \eqref{E:temp-27} and Theorem \ref{T:QLPDE-UM}(4) implies the estimates 
\be \label{E:temp-28} 
a_0(\ep_0(R_0)), \, a_1 (\ep_0(R_0)) \le C R_0 \le C \ep(R_0),
\ee
where 
\[
a_0(\ep) = \sup_{v_+\in X_+^{n-1} (\ep)} \frac {|h_+(v_+)|_{X_-^{n-1}}}{|v_+|_{X_+^{n-1}}}, \quad a_1(\ep)= \sup_{v_+, \wt v_+\in X_+^{n-1} (\ep)} \frac {|h_+(v_+) - h_+(\wt v_+)|_{X_-^{n-2}}}{|v_+ - \wt v_+|_{X_-^{n-2}}}. 
\]

Let $\delta_1 \in (0, \delta_0]$, $\ep_1 \in (0, \ep_0(R_*)]$, $\phi_+: \overline{X_+^{n-1} (\ep_0 (R_*))} \to X_-^n$,  and $q_+: \overline{X_+^n (\delta_1)} \to X_-^n$ be determined by $h_+$ according to Lemma \ref{L:LipMap-1}. Let 
\[
M^*= 2 \bar M |\BD \CB(0)|_{\BBL(X_+^{n}, X_+^{n-1})}  |\BD \CB^{-1}|_{C^0(\overline{X^{n-1} (R_*)}, \BBL(X^{n-1}, X^n))}, 
\]
\[ 
\delta = \min\Big\{\delta_1, \, 
\frac {\ep_0(R_*)}{ 2 M^* |\Pi_+|_{\BBL(X^{n-1})} |\BD \CB|_{C^0(\overline{X^{n} (\delta_0)}, \BBL(X^{n}, X^{n-1}))}} \Big\}. 
\]
We shall complete the proof of Theorem \ref{T:NLPDE-UM} with the above $M^*$ and $q_+$ restricted on $X_+^n(\delta)$.   

For any $u_{0+} \in X_+^n (\delta_1)$, let $v(t)$, $t \le 0$, be the solution to \eqref{E:QLPDE-1} with initial value 
\[
v(0)= \CB (u_{0+} + q_+(u_{0+})) = \phi_+^{-1} (u_{0+}) + h_+(\phi_+^{-1} (u_{0+})) 
\]
given by Theorem \ref{T:QLPDE-UM}(1) and $u(t) = \CB^{-1} (v(t))$ be the corresponding solution to \eqref{E:NLPDE-1}. 
From \eqref{E:UM-1} and  Lemma \ref{L:LipMap-1}, it holds 
\begin{align*}
|u(t)|_{X^n} \le & |\BD \CB^{-1}|_{C^0(\overline{X^{n-1} (R_*)}, \BBL(X^{n-1}, X^n))} |v(t)|_{X^{n-1}} \\
\le & \bar M |\BD \CB^{-1}|_{C^0(\overline{X^{n-1} (R_*)}, \BBL(X^{n-1}, X^n))} |\phi_+^{-1} (u_{0+})|_{X_+^{n-1}} e^{\lambda_0t}\\
\le & 2 \bar M  |\BD \CB(0)|_{\BBL(X_+^{n}, X_+^{n-1})}  |\BD \CB^{-1}|_{C^0(\overline{X^{n-1} (R_*)}, \BBL(X^{n-1}, X^n))} |u_{0+}|_{X_+^{n}}  e^{\lambda_0t} 
= M^* |u_{0+}|_{X_+^{n}} e^{\lambda_0t},
\end{align*}
which proves \eqref{E:UM-3} for all $u_{0+} \in X_+^n (\delta_1)$. 

Suppose $|u_{+0}|_{X_+^n} < \delta$ and  $\wt u(t)$, $t\le 0$, is a solution to \eqref{E:NLPDE-1} satisfying \eqref{E:temp-20.5}. Let $\wt v(t) = \CB(\wt u(t))$ which solves \eqref{E:QLPDE-1}. One may estimate using \eqref{E:temp-20.5} 
\begin{align*}
|\wt v(t)|_{X^{n-1}} \le & |\BD \CB|_{C^0(\overline{X^{n} (\delta_0)}, \BBL(X^{n}, X^{n-1}))} |\wt u(t)|_{X^n}  \\
\le & 2 M^* \delta  |\BD \CB|_{C^0(\overline{X^{n} (\delta_0)}, \BBL(X^{n}, X^{n-1}))} e^{\lambda_0t} \le \ep_0(R_*) e^{\lambda_0t} \le M_0 \ep_0(R_*) e^{\lambda_0t},
\end{align*}
where $|\Pi_+|_{\BBL(X^{n-1})} \ge 1$, true for any non-trivial projection, is also used. 
From the uniqueness statement in Theorem \ref{T:QLPDE-UM}(1), the definition of $q_+$, and again \eqref{E:temp-20.5}, we have 
\[
\CB(\wt u(0)) = \wt v(0) = \wt v_+(0) + h_+ (\wt v_+(0)) = \CB \big(\phi_+ (\wt v_+(0)) + q_+\big( \phi_+ (\wt v_+(0)) \big) \big),
\] 
which implies $\wt u_-(0) = q_+(\wt u_+(0)) = q_+(u_{0+})$ and thus $\wt v(0)=v(0)$. Since both $\wt v(t)$ and $v(t)$ satisfy \eqref{E:UM-0.5}, we obtain  $\wt v(t) = v(t)$ and thus $\wt u(t) = u(t)$ for all $t \le 0$.

To show \eqref{E:UM-3.1}, let $\wt q^+: X_+^n(\wt \delta) \to X_-^n$ be defined by $\wt \lambda$ by the same procedure. Let us tentatively assume Theorem \ref{T:NLPDE-UM}(5), which will be proved below. Since $u(t)$ satisfies \eqref{E:UM-3}, Theorem \ref{T:NLPDE-UM}(5) implies that $u_-(t_0) = \wt q^+(u_+(t_0))$ for $t_0 \ll -1$. Hence $u(\cdot +t_0)$ also satisfies \eqref{E:UM-3} with $\lambda_0$ replaced by $\wt \lambda$, which implies \eqref{E:UM-3.1} and completes 
the proof of Theorem \ref{T:NLPDE-UM}(1). 

Suppose $|u_{0+}|_{X_+^n} \le \delta$ and  $u(t)$, $t\le 0$, is the unique solution to \eqref{E:NLPDE-1} satisfying \eqref{E:temp-20.5}. 
For any $t_0\le 0$, from the exponential growth bound \eqref{E:UM-3}, $v(t)= \CB(u(t))$ satisfies 
\begin{align*}
|v_+(t_0)|_{X_+^{n-1}} = & |\Pi_+ \CB (u(t_0))|_{X_+^{n-1}} \le |\Pi_+|_{\BBL(X^{n-1})} |\BD \CB|_{C^0(\overline{X^{n} (\delta_0)}, \BBL(X^{n}, X^{n-1}))} |u(t_0)|_{X^n} \\
\le & |\Pi_+|_{\BBL(X^{n-1})} |\BD \CB|_{C^0(\overline{X^{n} (\delta_0)}, \BBL(X^{n}, X^{n-1}))} M^* |u_{0+}|_{X_+^n} e^{\lambda_0t_0} \le \ep_0(R_*)/2. 
\end{align*}
Theorem \ref{T:QLPDE-UM}(2) implies $v_-(t_0)= h_+(v_+(t_0))$ and thus $u_-(t_0) = q_+(u_+(t_0))$ if $|u_+(t_0)|_{X_+^n} < \delta_1$. This immediately yields Theorem \ref{T:NLPDE-UM}(2). Along with a continuation argument, we also obtain Theorem \ref{T:NLPDE-UM}(3) from Theorem \ref{T:QLPDE-UM}(3). 

Except for the $C^{1,1}$ smoothness of $q^+$, Theorem \ref{T:NLPDE-UM}(4) is a direct consequence of \eqref{E:temp-28}, Lemma \ref{L:LipMap-1}, and Theorem \ref{T:QLPDE-UM}(4). The smoothness of $q^+$ is given in Theorem \ref{T:UM-smoothness} in Subsection \ref{SSS:smoothness-InMa} (see also Remark \ref{R:smoothness-1}).  

Suppose both $\sigma, \sigma'>0$ satisfy the properties given by Lemma \ref{L:C1C2}, the resulted equation \eqref{E:QLPDE-1} from the transformation $\CB$ defined by $\sigma, \sigma'$ differ by a bounded linear transformation. By the uniqueness of $q^+(u_{0+})$ given in Theorem \ref{T:NLPDE-UM}(1) for each of $\sigma$ and $\sigma'$, their local unstable manifolds correspond to the same local unstable manifold of the original equation \eqref{E:NLPDE-1}.  
Observe that any fixed $\lambda_\pm$ satisfying \eqref{E:temp-26} are sufficient for the above arguments. Varying $\lambda_\pm$ satisfying \eqref{E:temp-26} does not change $q_+$ essentially (see Lemma \ref{L:UM-1}). To prove Theorem \ref{T:NLPDE-UM}(5), take any $\lambda \in (\max\{0, \omega_-\}, \omega_+)$, we may choose $\lambda_\pm$ and $R_0\ll 1$ such that \eqref{E:temp-26} is satisfied and $\lambda \in \Sigma_+$ defined by $\lambda_\pm$ and $R_0$. Since the exponential decay of $|u(t)|_{X^n}$ as $t \to -\infty$ implies the decay of $|\CB(u(t))|_{X^{n-1}}$ at the same exponential rate, Theorem \ref{T:NLPDE-UM}(5) follows from Theorem \ref{T:QLPDE-UM}(5) and Lemma \ref{L:LipMap-1} immediately. 
\end{proof}

\subsubsection{Smoothness of stable and unstable manifolds} \label{SSS:smoothness-InMa} 

To end this section, we study the smoothness of local invariant manifolds, $W^\pm = graph (q_\pm)$ over $X_\pm^n (\delta)$, obtained in Theorem \ref{T:NLPDE-UM} and Remark \ref{R:SM}. where again we shall focus on the unstable manifold $W^+$. In addition to (D.1--4), we also assume (B.5) as in Appendix \ref{SSS:smoothness-NLPDE}. In particular, according to Remark \ref{R:(D.2)}, (B.5) is essentially applied to $\Pi_{2\pm} (F \circ \Pi_{2\pm})$. The proof of the smoothness of $W^+$ is in a fashion similar to that of the solution map obtained in Theorem \ref{T:LWP-smoothness}. 

\begin{theorem} \label{T:UM-smoothness} 
Assume (D.1--4), (B.5), and $\omega_+ > \max\{0, \omega_-\}$. For any $\lambda_0 \in (\max \{0, \omega_-\}, \omega_+)$, there exist $C, \delta>0$ such that $q^+ \in C^{m, 1} \big(X_+^n (\delta), X_-^{n-m-1}\big)$ for any $1\le m \le m_0$. Moreover, for any $1\le m \le m_0$, $r_0 \le r \le n-1$, and $u_{0+} \in X_+^n (\delta)$, it holds 
\[
|\BD_{u_{0+}} q^+ (u_{0+}) |_{\BBL (X_+^r, X_-^{r})} \le C|u_{0+}|_{X_+^{n-1}}, \quad |\BD_{u_{0+}}^m q^+ (u_{0+}) |_{\BBL (\otimes_{j=1}^m X_+^r, X_-^{r-m+1})} \le C,
\] 
and for  $u_{01+}, u_{02+} \in X_+^n(\delta)$, 
\[
|\BD_{u_{0+}}^m q^+ (u_{02+}) - \BD_{u_{0+}}^m q^+ (u_{01+}) |_{\BBL (\otimes_{j=1}^m X_+^r, X_-^{r-m})} \le C  |u_{02+} - u_{01+}|_{X_+^{n-1}}.
\]
\end{theorem} 

\begin{remark} \label{R:smoothness-SM}
The same results hold for the local stable manifold $W^- = graph (q^-)$ if $\omega_- <  \min\{0, \omega_+\}$ is assumed instead. 
\end{remark}


Our strategy to prove the theorem is similar to that in Appendix \ref{SSS:smoothness-NLPDE}. Namely, we first obtain the candidate symmetric multilinear operators for $\BD^m q^+$ and then prove that they are indeed the derivatives of $q^+$. 

For any $u_{0+} \in X_+^n (\delta)$ and $t \le 0$, let 
\be \label{E:temp-28.5}
u(t) = \phi(t, u_{0+}) = \phi_+ (t, u_{0+}) + \phi_- (t, u_{0+}), 
\quad  \phi_- (t, u_{0+}) = q^+ \big( \phi_+(t, u_{0+})\big), \quad t\le 0, 
\ee
denote the solution on $W^+$ to \eqref{E:NLPDE-1} with initial value $u_0 = u_{0+} + u_{0-}$, where $u_{0-} = q^+ (u_{0+})$, given by Theorem \ref{T:NLPDE-UM}. 
For any $m \in \N$, formally linearizing $\phi(t, u_{0+})$ with respect to $u_{0+}$ yields that the symmetric $m$-linear operators 
$U^m (t, u_{0+}) = \BD_{u_{0+}}^m \phi (t, u_{0+})$ satisfies the same equation as  \eqref{E:smoothness-1} 
\[
\p_t U^m = \CA(u) U^m   + \CF_{m} (t, u_{0+}), 
\]
where 
the symmetric $m$-linear operator $\CF_{m} (t, u_{0+}) (w, \ldots, w)$, $w\in X_+$, takes the same form \eqref{E:temp-10.1} and satisfies \eqref{E:recursion-1} with $u(t, u_0)$ and $u_0$ replaced by $\phi(t, u_{0+})$ and $u_{0+}$. So it involves $U^1, \ldots, U^{m-1}$.
Projecting $U^m$ into $X_\pm$, the above equation 
is equivalent to 
\be \label{E:smoothness-2} 
\p_t U_\pm^m = \CA_{\pm}(u) U_\pm^m  + \CA_{\pm, \mp} (u) U_\mp^m + \CF_{m\pm} (t, u_{0+}).
\ee
At $t=0$ and $t=-\infty$, we have  
\be \label{E:UpmBC} \begin{cases} 
U_+^1 (0, u_{0+}) = I, & \\
U_+^m (0, u_{0+}) =0, &  m>1,
\end{cases}
\quad \text{ and } \; U_-^m (-\infty, u_{0+})=0.
\ee
We single out the upper triangular and  leading order diagonal parts of $\CA_\pm(u)$ 
\be \label{E:CA-upperT}
\CA_\pm^U (u) = \sum_{1 \le j \le j' \le 3} \CA_{j\pm, j'\pm} (u) \Pi_{j'\pm}, \quad \CA_\pm^d(u) = diag \big( \CA_{1\pm} (0), \CA_{2\pm} (u), \CA_{3\pm} (0) \big),
\ee
and then \eqref{E:smoothness-2} can be rewritten as 
\be \label{E:smoothness-3} 
\p_t U_\pm^m = \CA_{\pm}^U (u) U_\pm^m  + \wt \CA_{\pm} (u) U^m + \CF_{m\pm} (t, u_{0+}), 
\ee
where 
\be \label{E:wtCA}
\wt \CA_{\pm} (u)  = \big(\CA_\pm (u) - \CA_\pm^U (u)\big) \Pi_\pm  + \CA_{\pm, \mp} (u) \Pi_\mp. 
\ee

\begin{lemma} \label{L:smoothness-UM-1} 
For any $\lambda_\pm$ satisfying $ \pm (\omega_\pm - \lambda_\pm) >0$, there exist $\delta, C>0$, such that for any $u_{0+} \in X_+^{n} (\delta)$, $\CA_\pm^U (u(t))$ determines unique evolution operators $U_\pm (t, t_0, u_{0+}) \in \BBL(X^r)$, for $t,t_0 \le 0$ and $\pm (t-t_0) \le 0$, which, for any $0 \le r \le n-1$, are strongly continuous in $t$ and $t_0$ on $X^r$ and satisfy 
\[ 
|U_\pm (t, t_0, u_{0+})|_{\BBL(X^r)} \le C e^{\lambda_\pm (t-t_0)}. 
\]
\end{lemma}

\begin {proof} 
In the same spirit as the proof of Theorem \ref{T:NLPDE-UM}, we introduce equivalent norms on $X_\pm^r$  
\[
\| w\|_{X_\pm^r, \sigma}^2 = | \sigma w|_{X_{1\pm}^r}^2 + |w|_{X_{2\pm}^r}^2 + | \sigma^{-1} w|_{X_{3\pm}^r}^2, \quad w\in X_\pm^r, \quad \sigma \in (0, 1], 
\] 
and then one may estimate, for any $u \in \CO$, 
\begin{align*}
\big \| \big(\CA_\pm^U (u) - \CA_\pm^d(u) \big) w \big\|_{X_\pm^r, \sigma} = &  \Big\| \sum_{j=1,3} \big( \CA_{j\pm} (u) - \CA_{j\pm} (0) \big) \Pi_{j\pm} w + \sum_{1\le j < j' \le 3}  \CA_{j\pm, j'\pm} (u) \Pi_{j'\pm} w 
\Big\|_{X_\pm^r, \sigma}\\
\le & \sum_{j=1,3} | \CA_{j\pm} (u) - \CA_{j\pm} (0)|_{\BBL(X_{j\pm}^r)} \|\Pi_{j\pm} w\|_{X_{j\pm}^r, \sigma} \\
&+ \sum_{1\le j < j' \le 3} \sigma^{j'-j} |\CA_{j\pm, j'\pm} (u)|_{\BBL(X_{j\pm}^r, X_{j'\pm}^r)} \| \Pi_{j'\pm} w \|_{X_{j'\pm}^r, \sigma}.
\end{align*}
From assumption (D.3) we obtain 
\be \label{E:temp-29} 
\big \| \big(\CA_\pm^U (u) - \CA_\pm^d(u) \big) \big\|_{\BBL((X_\pm^r, \| \cdot \|_{X_\pm^r, \sigma}))} \le C (|u|_{X^{n-1}} + \sigma),
\ee
where $C$ is determined by the norms of $\CA$ involved in (D.3). 

From Theorem \ref{T:NLPDE-UM}, 
for any $u_{0+} \in X_+^n (\delta)$, \eqref{E:UM-3} implies that the solution $u(t)$ satisfies 
\[
|u_t (t)|_{X^{n-1}} = |F(u (t))|_{X^{n-1}} \le |\CA|_{C^0 (X^n (M^* \delta), \BBL(X^n, X^{n-1}))} |u(t)|_{X^n} \le C |u_{0+}|_{X_+^n} e^{\lambda_0 t}, \quad t \le 0.
\]
As in \eqref{E:U-bdd-4}, $\CA_\pm^d (u(t))$ generate evolution operators $U_\pm (t, t_0, u_{0+}) \in \BBL(X^r)$, for $t, t_0 \le 0$ and $\pm (t-t_0) \le 0$, which, for any $0 \le r \le n-1$, are strongly continuous in $t$ in $X^r$ and satisfy 
\[
|U_\pm^d (t, t_0)|_{\BBL(X_\pm^r)}  \le C e^{\omega_\pm (t-t_0) + C |\int_{t_0}^t |u_t (\tau)|_{X^{n-1}} d\tau |} \le C e^{\omega_\pm (t-t_0)}. 
\]
Due to \eqref{E:temp-29}, for sufficiently small $\delta, \sigma >0$, $\CA_\pm^U (u(t)) - \CA_\pm^d(u(t)) \in \BBL(X_\pm^r)$ is a small bounded operator in an equivalent norm, hence $\CA_\pm^U (u(t))$ also generate evolution operators  strongly $C^0$ in $t$ and $t_0$ on $X_\pm^r$ satisfying the desired estimates. 
\end{proof}

Using the above lemma, we shall solve \eqref{E:smoothness-3} and \eqref{E:UpmBC} in the space 
\begin{align*}
\Gamma_{r, \lambda}^m = \big\{ & U \in C^0 \big( (-\infty, 0], \BBL (\otimes_{j=1}^m X^r, X^{r-m+1}) \big) \mid \\
& \qquad \qquad |U|_{m, r, \lambda} \triangleq \sup_{t\le 0} e^{-\lambda t} |U(t)|_{ \BBL (\otimes_{j=1}^m X_+^r, X^{r-m+1})} < \infty \big\}. 
\end{align*} 

\begin{lemma} \label{L:smoothness-UM-2}
For any $\lambda_0 \in (\max \{0, \omega_-\}, \omega_+)$, there exist $C, \delta>0$ such that \eqref{E:smoothness-3} and \eqref{E:UpmBC} has a unique solution $U^m (\cdot, u_{0+}) \in \Gamma_{r, \lambda}^m$ for any $1\le m \le m_0$, $r_0 \le r \le n-1$, and $u_{0+} \in X_+^n (\delta)$. Moreover, it satisfies, for $t\le 0$,  
\[
|U^1 (t, u_{0+}) - U_+(t, 0) |_{\BBL (X_+^r, X^{r})} \le C e^{\lambda_0 t} |u_{0+}|_{X^{n-1}}, \quad |U^m (t, u_{0+}) |_{\BBL (\otimes_{j=1}^m X_+^r, X^{r-m+1})} \le C e^{\lambda_0 t},
\] 
and for  $u_{01+}, u_{02+} \in X_+^n(\delta)$, 
\[
|U^m (t, u_{02+}) - U^m (t, u_{01+}) |_{\BBL (\otimes_{j=1}^m X_+^r, X^{r-m})} \le C e^{\lambda_0 t} |u_{02+} - u_{01+}|_{X_+^{n-1}}.
\]
\end{lemma}

\begin{proof} 
Like the Lyapunov-Perron equation, \eqref{E:smoothness-3} and \eqref{E:UpmBC} are equivalent to 
\be \label{E:smoothness-LP-1} \begin{cases} 
U_+^m (t) = U_+(t, 0) U_+^m(0) + \int_0^t U_+(t, \tau) \big(  \wt \CA_{+} (u(\tau)) U^m (\tau) + \CF_{m+} (\tau) \big) d\tau,\\
U_-^m (t) = \int_{-\infty}^t U_-(t, \tau) \big(  \wt \CA_{-} (u(\tau)) U^m (\tau) + \CF_{m-} (\tau) \big) d\tau, 
\end{cases} \ee
where we skipped the dependence on $u_{0+}$ when there is no confusion. This system is in the form of a fixed point equation
\be \label{E:smoothness-LP-2}
U^m  = \CS (u_{0+}) U^m  + g_m (u_{0+}),
\ee
where, for $t \le 0$, 
\[
\big(\CS (u_{0+}) U^m\big) (t) =  \int_0^t U_+(t, \tau)  \wt \CA_{+} (u(\tau)) U^m (\tau) d\tau + \int_{-\infty}^t U_-(t, \tau) \wt \CA_{-} (u(\tau)) U^m (\tau) d\tau,  
\]
\[
g_m(u_{0+}) (t)
= U_+(t, 0) U_+^m(0) + \int_0^t U_+(t, \tau) \CF_{m+} (\tau, u_{0+})  d\tau + \int_{-\infty}^t U_-(t, \tau) \CF_{m-} (\tau, u_{0+})  d\tau,
\]
like $\CF_m$, the latter of which is computed using $U^1, \ldots, U^{m-1}$.
From assumptions (D.1) and (D.3) and the definition \eqref{E:wtCA},  
\[
|\wt \CA_\pm (u) |_{\BBL(X^r, X_\pm^r)} \le C |u|_{X^{n-1}}, \quad u \in \CO, \;\; 0 \le r \le n-1. 
\]

Let $\lambda_\pm = (\omega_\pm + \lambda_0)/2$ and $\delta>0$ be determined by Lemma \ref{L:smoothness-UM-1} and Theorem \ref{T:NLPDE-UM}. Using Lemma \ref{L:smoothness-UM-1}, \eqref{E:UM-3}, and the above bound on $\wt \CA_\pm$, 
we obtain, for $0 \le r \le n-1$,  
\begin{align*}
|\CS (u_{0+}) U^m|_{m, r, \lambda_0} \le & \sup_{t\le 0} C e^{-\lambda_0 t} \Big(\int_t^0 e^{\lambda_+(t-\tau)}  |u(\tau)|_{X^{n-1}} |U^m (\tau)|_{\BBL (\otimes_{j=1}^m X_+^r, X^{r-m+1})} d\tau \\
& \qquad \qquad \quad + \int_{-\infty}^t e^{\lambda_- (t-\tau)}  |u(\tau)|_{X^{n-1}} |U^m (\tau)|_{\BBL (\otimes_{j=1}^m X_+^r, X^{r-m+1})} d\tau\Big) \\
\le & C \Big( \frac 1{\lambda_+-\lambda_0} + \frac 1{\lambda_0 -\lambda_-}\Big)  |u_{0+}|_{X^{n-1}} |U^m|_{m, r, \lambda_0}. 
\end{align*} 
For sufficiently small $\delta>0$, $I-\CS(u_{0+}) \in \BBL(\Gamma_{r, \lambda_0}^m)$ has a bounded inverse and 
\[
|(I - \CS(u_{0+}))^{-1} - I |_{\BBL(\Gamma_{r, \lambda_0}^m)} \le C |u_{0+}|_{X_+^{n-1}}, \quad 0\le r\le n-1. 
\]
Therefore the solution to  \eqref{E:smoothness-LP-1} must be given by 
\[
U^m (\cdot, u_{0+}) = (I - \CS(u_{0+}))^{-1} g_m (u_{0+}) (U_+^m(0), U^1, \ldots, U^{m-1}).  
\]

For $m=1$, we have $U_+^1(0) =I$ and $g_1(u_{0+})= U_+(t, 0)$, hence the desired estimate on $U^1 (t, u_{0+})$ follows directly from Lemma \ref{L:smoothness-UM-1}. 

For $1 < m \le m_0$, $\CF_{m\pm}$  in the non-homogeneous term $g_m(u_{0+})$ satisfy the same inequality \eqref{E:temp-10.22} in the proof of Theorem \ref{T:LWP-smoothness}. Inductively we have 
\begin{align*}
| g_m(u_{0+})|_{m, r, \lambda_0} \le & \sup_{t\le 0} C e^{-\lambda_0 t} \Big(\int_t^0 e^{\lambda_+(t-\tau)} e^{2\lambda_0 \tau} d\tau + \int_{-\infty}^t e^{\lambda_- (t-\tau)}  e^{2\lambda_0 \tau} d\tau\Big) \le C,
\end{align*}
which yields the desired estimates on $U^m(\cdot, u_{0+})$. 

To obtain the Lipschitz estimates of $U^m (\cdot, u_{0+})$ with respect to $u_{0+}$, let $u_{0j+}\in X_+^n (\delta)$ and denote $u_j(t) = \phi(t, u_{0j})$ and $U_{j}^m (t) = U^m (t, u_{0j+})$, $j=1,2$. Equation \eqref{E:smoothness-3} implies 
\be \label{E:temp-30} \begin{split}
(U_{2\pm}^m - U_{1\pm}^m)_t = & \CA_\pm^U (u_1(t)) (U_{2\pm}^m - U_{1\pm}^m) + \wt \CA_\pm (u_1(t)) (U_{2}^m - U_{1}^m) \\
& + \big(\CA_\pm^U (u_2(t)) - \CA_\pm^U (u_1(t))\big) U_{2\pm}^m + \big(\wt \CA_\pm (u_2(t)) - \wt \CA_\pm (u_1(t))\big) U_{2}^m \\
& + \big( \CF_{m\pm} (t, u_{02+}) -\CF_{m\pm} (t, u_{01+}) \big),
\end{split} \ee
with boundary conditions 
\[
U_{2+}^m (0) - U_{1+}^m (0) =0, \quad U_{2-}^m (-\infty) - U_{1-}^m (-\infty) =0.
\] 
Like \eqref{E:smoothness-3} it can be converted into an integral equation in the form of \eqref{E:smoothness-LP-1}, equivalent to a non-homogeneous linear equation in the form of \eqref{E:smoothness-LP-2} with the same homogeneous linear part $I - \CS(u_{01+})$. In the non-homogeneous terms, for any $1\le r \le n-1$ and $t\le 0$, 
\begin{align*}
& \big|\CA_\pm^U (u_2(t)) - \CA_\pm^U (u_1(t))\big|_{\BBL(X_\pm^r, X_\pm^{r-1})} +  \big|\wt \CA_\pm (u_2(t)) - \wt \CA_\pm (u_1(t))\big|_{\BBL(X_\pm^r)}  \\
\le & C |u_2(t) - u_1(t)|_{X^{n-1}} \le C e^{\lambda_0 t} |u_{02+} - u_{01+} |_{X_+^{n-1}}. 
\end{align*}
The $\CF_m$ in the non-homogeneous part satisfies the same estimate \eqref{E:temp-10.24} as in Appendix \ref{SSS:smoothness-NLPDE} which implies, for $r_0 \le r \le n-1$, 
\begin{align*}
& |\CF_{m\pm}(t, u_{02+}) - \CF_{m\pm} (t, u_{01+}) |_{\BBL(\otimes_{j=1}^m X_+^r, X^{r-m})}\\
\le & C \Big( e^{\lambda_0 t} |u_{02+} - u_{01+} |_{X_+^{n-1}} + \sum_{j=1}^{m-1} |U_2^{j} (t) - U_1^{j} (t)|_{\BBL(\otimes_{j=1}^{j} X_+^r, X^{r-j})} \Big). 
\end{align*}
The desired Lipschitz estimates of $U^m(\cdot, u_{0+})$ follows inductively using Lemma \ref{L:smoothness-UM-1}. 
\end{proof}

Finally we are ready to prove the smoothness of the local unstable manifolds. 

\begin{proof}[Proof of Theorem \ref{T:UM-smoothness}.]
To complete the proof, we mainly need to show that $U^m(t, u_{0+}) = \BD_{u_{0+}}^m \phi (t, u_{0+})$ where we still adopt the notations \eqref{E:temp-28.5}.  
For any initial value 
\[
u_{0+} + \tilde u_{0+} \in X_+^{n}(\delta),  \quad 0\le |\tilde u_{0+}|_{X^{n}} \ll 1. 
\]
Let $\tilde u(t) = \phi (t, u_{0+} + \tilde u_{0+})$ denote the solution to \eqref{E:NLPDE-0} on $W^+$ and 
we often skip the initial $u_{0+}$ in $u$ (or $\phi)$ and $U^m$ if $\tilde u_0=0$. 
Consider 
\[
w(t, \tilde u_0) = \phi(t, u_{0+} + \tilde u_{0+}) - \phi(t, u_{0+}) - U^1 (t, u_{0+}) \tilde u_{0+} = \tilde u(t) - u(t) - U^1 (t) \tilde u_{0+}. 
\]
One may compute much as in the proof of Theorem \ref{T:LWP-smoothness}
\[
w_t  =  F(\tilde u(t)) - F(u(t)) - \CA(u(t)) U^1 (t) \tilde u_{0+} =  \CA(u(t)) w + F_1 (t),
\]
where 
\begin{align*}
F_1 (t) = & F(\tilde u(t)) - F(u(t)) - \CA(u(t)) \big(\tilde u(t) - u(t) \big)  \\
= & \int_0^1 \Big( \CA \big( u(t) + \tau (  \tilde u(t) - u(t))\big) - \CA(u(t))\Big) \big(\tilde u(t) - u(t) \big) d\tau. 
\end{align*}
From  assumption (D.3) and the Lipschitz dependence of $\phi(t, u_{0+}) \in X^{n-1}$ in $u_{0+} \in X_+^{n-1}$ (Theorem \ref{T:NLPDE-UM}(4)), we have
\[
|F_1 (t)|_{X^{n-2}} \le C e^{2\lambda_0 t} |\tilde u_{0+}|_{X^{n-1}}^2.  
\]
By rewriting the equation of $w_t$ into the form of the Lyapunov-Perron integral equations \eqref{E:smoothness-LP-1}/\eqref{E:smoothness-LP-2} with the same $\CS(u_{0+})$ and $g_1$ replaced by the integrals of $F_1(t)$, using $w_+ (0, \tilde u_0) =0$, the same arguments in the exponentially weight spaces as in the proof of Lemma \ref{L:smoothness-UM-2} imply  
\[
|w (t)|_{X^{n-2}} \le C e^{\lambda_0 t} |\tilde u_{0+}|_{X^{n-1}}^2.  
\]
Therefore $\phi(t, \cdot) : X^{n}(\delta) \to X^{n-2}$  is Fr\'echet differentiable and $\BD_{u_0} \phi(t, u_{0+}) = U^1 (t)$ for any $u_{0+}$. Along with the above estimates on $U^1$ it completes the proof of the theorem for the case of $m_0=1$. Since $\CF_m$ in \eqref{E:smoothness-3} satisfies the same properties \eqref{E:temp-10.1} and \eqref{E:recursion-1} as in Appendix \ref{SSS:smoothness-NLPDE}, the inductive proof of the higher order derivatives $\BD_{u_{0+}}^m \phi(t, u_{0+}) = U^m(t)$ follows from a similar procedure as in the proof of Theorem \ref{T:LWP-smoothness} using Lemma \ref{L:smoothness-UM-1} and the same arguments  in the exponentially weight spaces as in the proof of Lemma \ref{L:smoothness-UM-2}. 

Finally the theorem is obtained from the fact $q^+ ( u_{0+}) = \phi_- (0, u_{0+})$. 
\end{proof}

\section{Some concrete nonlinear PDEs} \label{S:examples}

In this section, we illustrate how some nonlinear evolutionary PDEs fit conveniently (not necessarily optimally) into the general frameworks laid out in Appendix \ref{S:pre-LWP} and Section \ref{S:LInMa}, while the  water waves with surface tension will be discussed in Section \ref{S:CWW}. 
These include many PDEs with a natural {\it energy}. 
After a general outline, 
Hamiltonian PDEs will be discussed in Subsection \ref{SS:examples-Ham} and some other examples  including the mean curvature flow in Subsection \ref{SS:examples-other}. 
In both subsections, mostly the local stable and unstable manifolds are obtained after some brief discussions on the local well-posedness and linear analysis. 

\vspace{0.08in} \noindent {\bf Energy generated flows.} The principally positive symmetric quadratic form $\BL(v)$ or $\CL(u)$ in Appendix \ref{S:pre-LWP} and Section \ref{S:LInMa} often originate from some intrinsic energy structures of the PDEs. We first outline roughly the key ingredients. Let 
\begin{itemize} 
\item $X$ be a real Hilbert space and $\CE(u)$ be a nonlinear functional smoothly defined on an open set of a dense subspace of $X$; and 
\item $\Omega: X^* \supset Dom(\Omega) \to X$ be a (possibly unbounded) densely defined linear operator
such that 
\be \label{E:Omega-1}
\Omega + \Omega^* \ge 0, \; \text{ namely } \; \langle \gamma, \Omega \gamma \rangle \ge 0, \; \forall \gamma \in Dom(\Omega). 
\ee
\end{itemize} 
Here if $\CE$ is differentiable at $u$, naturally $\BD \CE(u)$ is a linear functional.  
Consider 
\be \label{E:energyF-1} 
u_t = F(u) \triangleq - \Omega \BD \CE(u). 
\ee
Roughly the main  assumptions are 
\begin{enumerate} 
\item [(E.1)] there exist $\omega_1 \in \R$, $u_* \in Dom(\CE)$, and a dense subspace $X^1 \subset X$ such that, for $\CA_0 =\BD F(u_*) = - \Omega \BD^2 \CE(u_*)$, $\omega_1 - \CA_0: X \supset X^1 \to X$ is closed and bijective;
\item [(E.2)] with $X^r \triangleq (\omega_1 - \CA_0)^{-r} X$ equipped with graph norms, 
there exist $n_0 \ge 2$ and $\ep>0$ 
such that $u_*\in X^{n_0+1}$ and (B.1) in Appendix \ref{SS:NLPDE-LWP} is satisfied on $\CO= X^{n_0} (u_*, \ep)$;
\item [(E.3)] $\CL \triangleq \BD^2 \CE \in C^1 \big(X^{n_0} (u_*, \ep),  \BBL(X, X^*)\big)$ and it satisfies \eqref{E:coercivity-2} at $u=u_*$. 
\end{enumerate}


At first $\BD F(u_*)$ in (E.1) may be a formal differentiation which helps to identify $\omega_1$ and the spaces $X$ and $X_1$ to establish the subsequent framework. The rigorous justification is then required in (E.2) and (E.3). 
The key dissipativity condition \eqref{E:dissipativity-2.5} is satisfied automatically  
\be \label{E:dissipativity-6}
\big \langle \CL (u) w , \CA (u)w \big \rangle = \big \langle \BD^2 \CE (u) w , - \Omega \BD^2 \CE (u) w \rangle \le 0,
\quad \forall u \in X^{n_0} (u_*, \ep), \ w\in X^1. 
\ee
In practices, assumptions in Appendix \ref{S:pre-LWP} and Section \ref{S:LInMa} are often verified in a neighborhood of $u_*$ as shown for the PDEs below. 

Two typical categories of such energy flows are {\it Hamiltonian} flows where $\Omega^*= - \Omega$ and {\it gradient flows}, where $\Omega^*=\Omega>0$. We shall mainly focus on the Hamiltonian cases while also discuss some examples of gradient flows. 

\begin{remark} \label{R:energyF}
a.) In many problems, the forms of $\CE(u)$ and $\Omega$ are given by the underlying physical  background. Often one first computes the symmetric quadratic form $\BD^2 \CE(u)$ and decide the Hilbert space $X$ accordingly such that the coercivity \eqref{E:coercivity-2} is satisfied. The spaces $X^r$ are determined by $\CA$ subsequently. 
\\
b.) More lower order terms could also be included in \eqref{E:energyF-1}, then some $\omega_* \in \R$ would appear in \eqref{E:dissipativity-6} as in assumption (B.3) in Appendix \ref{SS:NLPDE-LWP}.\\
c.) It is also possible to consider $u$-dependent $\Omega = \Omega(u)$ 
under appropriate conditions on $\Omega(u)$ which are essentially 
the control of $\BD \Omega \BD \CE$ by $\Omega \BD^2 \CE$. See Subsection \ref{SS:examples-other}.  
\end{remark} 

{\it For example}, consider an energy functional whose energy density depends  pointwisely on the state function $u(x)$ and its gradient 
defined on $\R^d$ or the $d$-dim torus $\BBT^d$: 
\be \label{E:energy-1} 
\CE (u) = \int E\big(x, u(x), \nabla u(x) \big) dx,
\ee
where $E= E(u, p)$ is smooth in $u$ and $p$. 
The usual ellipticity is also assumed on $E$ 
\be \label{E:ellipticity-1} 
\exists \, M, a_0>0, \quad \sum_{j,k=1}^d E_{p_j p_k} (x, u, p) \xi_j \xi_k \ge a_0 |\xi|^2, \;\; \forall \xi \in \R^d, \; \forall |u|, |p| \le M. 
\ee
In this case, 
\be \label{E:energy-2} \begin{split} 
\langle \CL(u)\phi_1, \phi_2 \rangle = & \BD^2 \CE (u) (\phi_1, \phi_2) \\
= & \int \sum_{j, k=1}^d E_{p_j p_k} \phi_{1x_j} \phi_{2x_k} + \sum_{j=1}^d E_{u p_j} (\phi_1 \phi_{2x_j} + \phi_{1x_j} \phi_{2}) +     E_{uu} \phi_1\phi_2 dx\\
=& \int \Big( - \sum_{j, k=1}^d \p_{x_k}( E_{p_j p_k} \phi_{1x_j})  - \sum_{j=1}^d \p_{x_j} (E_{u p_j})  \phi_{1} +     E_{uu} \phi_1 \Big) \phi_2 dx, 
\end{split} \ee
where the derivatives of $E$ are evaluated at $(u, \nabla u)$ and the last line gives the form of $\CL(u)\phi$ through the $L^2$ duality.  The ellipticity \eqref{E:ellipticity-1} suggests taking $X= H^1$ or $\dot H^1$. In the above calculations, $u(x)$ may also be a vector valued function. 


\subsection{Hamiltonian PDEs where $\Omega^*=-\Omega$} \label{SS:examples-Ham} 

In this case, as the linearization $\CA=-\Omega \BD^2 \CE$ is not self-ajoint, the linear analysis is already a delicate issue. We start the discussion with some general results on the linearization of \eqref{E:energyF-1} proved in \cite{LZ22} based on the Pontryagin invariant subspace theorem. 

Let $Y$ be a real Hilbert space and suppose the linear operators $L \in \BBL(Y, Y^*)$ and $J: Y^* \supset Dom(J) \to Y$ satisfy
\be \label{E:LHam-1}
L^* = L,  \quad J^*=-J.
\ee
We further assume that $L$ is uniformly positive except in finite many directions. Namely, there exist closed subspaces $\wt Y_\pm \subset Y$ such that 
\be \label{E:LHam-2} \begin{split}
& Y = \wt Y_-\oplus \ker L \oplus \wt Y_+, \quad \dim \ker L <\infty, \quad m^-(L) \triangleq \dim \wt Y_-<\infty, \\ 
& \langle Lu, u\rangle <0, \; \forall u \in \wt Y_-\setminus \{0\}, \inf_{u\in \wt Y_+\setminus \{0\}} \langle Lu, u\rangle / |u|^2 >0. 
\end{split} \ee
Here $m^-(L)$ is referred to as the Morse index of $L$, which is independent of the choice of $\wt Y_-$. Another index $m_0^{\le 0} (L)$ is also useful which is defined as the number of non-positive directions of the induced quadratic form $\langle L \cdot, \cdot \rangle$ on the quotient space $\wt E_0 = E_0 /\ker L$ where $E_0$ is the generalized kernel of $L$ defined as $E_0 = \{ u \in Y \mid \exists l \in \N, \, (JL)^l u=0\}$.

\begin{proposition} \label{P:LHam}
Assume \eqref{E:LHam-1} and \eqref{E:LHam-2},  then the group $e^{tJL} \in \BBL(Y)$ is defined for all $t\in \R$  and there exist closed subspaces $Y_+, Y_1, Y_2, Y_3, Y_-\subset Y$,
such that the following hold. 
\begin{enumerate} 
\item $Y= Y_+ \oplus_{j=1}^3 Y_j \oplus Y_-$, $Y_\alpha \subset Dom((J L)^r)$ and $d_\alpha = \dim Y_\alpha <\infty$ for any $r \in \N$ and $\alpha \in \{+, 1, 3, -\}$.
\item In this decomposition $J L$ and $L$ take the following form
\[
J L \longleftrightarrow \begin{pmatrix} A_{+} & 0& 0& 0 &0 \\ 0 & A_{1} & A_{12} & A_{13} & 0 \\ 0 & 0 & A_{2} & A_{23} & 0 \\ 0 & 0 &0 & A_{3} & 0 \\ 0 & 0 &0 & 0 & A_- \end{pmatrix}, \quad L  \longleftrightarrow \begin{pmatrix} 0 & 0 & 0& 0 & L_{+-} \\ 0 & L_{1} & 0 & L_{13} & 0  \\ 0 & 0 & L_{2} & 0 & 0 \\ 0 & L_{31} & 0 & 0 & 0 \\ L_{-+} & 0 &0 &0 &0 \end{pmatrix}. 
\]
\item $\exists \delta >0$ such that $ \BBL(Y_2, Y_2^*) \ni L_2 \ge \delta$ and thus defines an equivalent norm on $Y_2$. 
\item All blocks of $J L$ are bounded operators except $A_{2}: Y_2 \supset Dom (A_{2}) \to Y_2$ is anti-self-adjoint with respect to the equivalent inner product on $Y_2$ defined by $L_{2}$, i.~e. $L_{2} A_{2} = - A_{2}^* L_{2}$. 
\item The spectra of the diagonal blocks of $A_\alpha$ satisfy
\[
\sigma(A_-) = - \sigma (A_+), \quad \Re \lambda >0, \; \forall \lambda \in \sigma(A_{+}); \quad \Re \lambda = 0, \; \forall \lambda \in \sigma(A_{j}), \; j=1, 2, 3.
\]
\item $m_0^{\le 0} (L) < m^-(L)$ and, if $m^- (L) - m_0^{\le 0} (L)$ is odd, then $m^- (L) - m_0^{\le 0} (L) \ge d_+=d_- >0$, i.~e.~$JL$ is unstable. 
\item Let $Z_1 = Y_+ \oplus Y_1 \oplus Y_3 \oplus Y_-$ and $Z_2 = Y_2$ and  $P_j \in \BBL(Y, Z_j)$, $j=1,2$,  be the projections  associated with the decomposition $Y=Y_1 \oplus Y_2$, then $P_1 J: Y^* \to Z_1$ is bounded. 
\end{enumerate}
\end{proposition}   

Statements (1--6) in the proposition follow directly from Remark 2.3, Theorem 2.1, and 2.3 in \cite{LZ22}. Specifically, $Y_+$ corresponds to $X_5$, $Y_-$ to $X_6$, $Y_1$ to $\ker L \oplus X_1 \oplus X_2$, $Y_2$ to $X_3$, and $Y_3$ to $X_4$ in Theorem 2.1 in \cite{LZ22}. 
Statements (7) is somewhat hidden in the proofs in \cite{LZ22}. Firstly, if $\ker L = \{0\}$, then \eqref{E:LHam-2} implies that $L \in \BBL(Y, Y^*)$ is isomorphic (Lemma A.2(2) \cite{LZ22}). From the $L$-orthogonality between the complementary closed subspaces $Z_{1,2}$, we have  that $L_{Z_{j}} \triangleq i_{Z_{j}}^* L i_{Z_{j}} \in \BBL(Z_{j}, Z_{j}^*)$, $j=1,2$, are also isomorphic (Lemma A.2(2) and A.3(2) \cite{LZ22}), where $i_{Z_j}^* \in \BBL(Y^*, Z_j^*)$ is the dual operator of the embedding $i_{Z_j} \in \BBL(Z_j, Y)$. Using this non-degeneracy, one can prove $\ker i_{Z_j}^* = L Z_{3-j}$. Since $Z_1 \subset Dom(JL)$, it follows $\ker i_{Z_2}^* \subset Dom(J)$ and thus Lemma A.3(1) \cite{LZ22} implies $P_1 J : Y^* \to Z_1$ is bounded, which is equal to $J_{11} + J_{12}$ there. 
If $\ker L \ne \{0\}$, Lemma A.4 \cite{LZ22} implies that there exists a closed subspace $\wt Y \subset Y$ such that  $\ker i_{\wt Y}^* \subset Dom(J)$ and $Y= \ker L \oplus \wt Y$ associated to the projection $P_0: \BBL(X, \ker L)$. Again from Lemma A.3(1) \cite{LZ22}, $P_0 J: X \to \ker L$ is bounded. Moreover, Lemma A.3(3) implies that the problem can be reduced to $\wt Y$ where $L$ is non-degenerate. Together with the above argument for the non-degenerate case statement (7) follows  (also see the proof of Theorem 2.1 in \cite{LZ22}).  

The finite Morse index assumption \eqref{E:LHam-2} allows a convenient way to verify some properties of the linearizations of the Hamiltonian PDEs via Proposition \ref{P:LHam} and obtain the local well-posedness and invariant manifolds through Propositions \ref{P:Ham-LWP} and Theorem \ref{T:Ham-UM} below. However, \eqref{E:LHam-2} is not absolutely necessary and it may be easier to work with Theorems \ref{T:NLPDE-LWP} and \ref{T:NLPDE-UM} directly in some cases. 

\begin{remark} \label{R:SeparableHam}
In addition to \eqref{E:LHam-1}, if $J$ and $L$ also possess the so-called separable structure, then, instead of using the results in \cite{LZ22}, Theorem 2.3 in \cite{LZ22CPAM} could be applied, too. It provides stronger results of the linear dynamics. 
\end{remark}

\noindent $\bullet$ {\it Local well-posedness of Hamiltonian PDEs.} 
Formally differentiating $\CE$ to obtain $\CL(u) =\BD^2 \CE(u)$, we may apply Proposition \ref{P:LHam} to establish the framework and verify assumptions (B.1)--(B.3) in Appendix \ref{SS:NLPDE-LWP}. 

\begin{proposition} \label{P:Ham-LWP}
Assume $\Omega^*=-\Omega$, then the following hold for \eqref{E:energyF-1}.
\begin{enumerate}
\item Assume $u_* \in Dom(\CE) \subset X$ and $L \triangleq \BD^2 \CE (u_*) \in \BBL(X)$ satisfies \eqref{E:LHam-1} and \eqref{E:LHam-2} for $Y=X$, then there exists $\lambda_0 \ge 0$ such that $(\lambda - \CA_0)^{-1} \in \BBL(X, X^1)$ for all $\lambda \ge \lambda_0$ where $\CA_0= - \Omega L$ and  $X^1 = Dom(\Omega L) \subset X$. 
\item Let $X^r = Dom (\lambda_0 -\CA_0)^{r} \subset X$ equipped with the graph norm, $\omega^*> \lambda_0$, and in addition, assume there exist $n\ge 2$ and $\ep>0$ such that $u_* \in X^n$, $\CL\triangleq \BD^2 \CE \in C^1 (\wt \CO, \BBL(X, X^*))$, and (B.1) in Appendix \ref{SS:NLPDE-LWP} is satisfied on $\wt \CO= X^{n-1} (u_*, \ep)$, then there exist $\ep_0 \in (0, \ep]$ and $a, C_\CL, C_*>0$ such that (B.2) and (B.3) are also satisfied on $\CO= X^{n-1} (u_*, \ep_0)$. 
\end{enumerate} 
\end{proposition}

Statement (1) of the proposition follows from the upper triangular form of $\CA_0$, codim-$Y_2 <\infty$, and in particular 
Proposition \ref{P:LHam}(4) applied to $J =-\Omega$ and $L = \BD^2 \CE(u_*)$. 
In statement (2) the dissipativity \eqref{E:dissipativity-2.5} is already justified in \eqref{E:dissipativity-6} with $\omega_*=0$. 
The coercivity \eqref{E:coercivity-2} at $u=u_*$ is satisfied due to assumption \eqref{E:LHam-2},  $Y_\alpha \subset X^1$ for $\alpha \ne 2$ (Proposition \ref{P:LHam}(1)), and  the invertibility of $\lambda -\CA_0$. The rest of  (B.2) and (B.3) hold in a small neighborhood of $u_*$ by the continuity of $\BD \CA$ and $\BD \CL$ assumed in (B.1). 

To obtain the local well-posedness from Theorem \ref{T:NLPDE-LWP}, as commented before, the assumption (B.4) in Appendix \ref{SS:NLPDE-LWP} is often verified directly or via Lemma \ref{L:CB-1} for concrete PDEs.

\vspace{0.08in} \noindent $\bullet$ {\it Local stable and unstable manifolds of Hamiltonian PDEs.} 
Suppose $\BD \CE(0)=0$ in equation  \eqref{E:energyF-1} and $u=0$ is spectrally unstable, i.~e.~$\sigma (\CA(0)) \not\subset i\R$, we construct its local unstable and stable manifolds. 

Assume $J=-\Omega$ and $L =\BD^2 \CE(0)$ satisfy \eqref{E:LHam-1} and \eqref{E:LHam-2} and thus Proposition \ref{P:LHam} applies. In particular Proposition \ref{P:LHam}(6) provides a sufficient condition on the spectral instability in the same spirit of \cite{GSS87, GSS90}. In the spectrally unstable case, the subspaces $Y_\pm$ with $\dim Y_\pm \in \N$ given in Proposition \ref{P:LHam} are clearly the unstable/stable subspaces of $\CA(0) = -\Omega \BD^2 \CE(0)$, while the center subspace is decomposed into $\oplus_{j=1}^3 Y_j$.
%
%
For any $r \in \N \cup \{0\}$, let 
\be \label{E:splitting-1} \begin{split} 
& X_+^r=X_{1+}^r = Y_+, \quad X_{1-}^r = Y_-\oplus Y_1, \quad X_{3-}^r = Y_3, \quad X_{2+}^r =X_{3+}^r =\{0\}, \\
&X_{2-}^r = (1-A_2)^{-r}Y_{2} \subset Y_2, \quad X_-^r = X_{1-} \oplus X_{2-}^r \oplus X_{3-}, \quad X^r = X_{+}^r \oplus X_-^r,
\end{split} \ee 
where $A_2$ 
is the block of $\CA(0)$ given in Proposition \ref{P:LHam}. 

\begin{theorem} \label{T:Ham-UM} 
Suppose $\BD \CE(0)=0$ in \eqref{E:energyF-1}. In addition, we assume   
\begin{enumerate} 
\item $J=\Omega$ and $L=\BD^2 \CE(0)$ satisfy \eqref{E:LHam-1} and \eqref{E:LHam-2} for $Y=X$;
\item $\CA(0)=-\Omega \BD^2 \CE(0)$ is spectrally unstable on $X$;
\item There exist $n\ge 2$   and a neighborhood of $\CO \subset X^{n-1}$ of $0$ such that, for any  $1\le r \le n-1$, 
\[
\CL \triangleq \BD^2 \CE \in C^1 (\CO,  \BBL(X, X^*)), \quad \BD F  \in C^1 (\CO, \BBL(X^r, X^{r-1})) \cap C^1 (\CO\cap X^n, \BBL(X^n, X^{n-1})),
\]
\end{enumerate} 
then for any $\omega_\pm$ given in \eqref{E:spectraPM}, there exist $\delta>0$ and $q_+: X_+ (\delta) \to X_-^{n}$ satisfying the properties in Theorems \ref{T:NLPDE-UM} and \ref{T:UM-smoothness}
\end{theorem}

The graph $W^+$ of $q^+$ gives the unique local unstable manifold of $u=0$. 

\begin{remark} \label{R:Ham-SM}
If (B.5) is also satisfied, then the solution map of \eqref{E:energyF-1} is $C^{m_0, 1}$ in the initial values in the sense of Theorem \ref{T:UM-smoothness}. Moreover, since $X_+$ is a closed subspace of $X^r$ for any $r$, $W^+$ is a $C^{m_0, 1}$ manifold in $X^{n-1-m_0}$ for any $m_0\le n-1$. 

By rearranging the definition of subspaces in \eqref{E:splitting-1} as 
\begin{align*}
& X_-^r=X_{1-}^r = Y_-, \quad X_{1+}^r = Y_+ \oplus Y_1, \quad X_{3+}^r = Y_3, \quad X_{2-}^r =X_{3-}^r =\{0\}, \\
&X_{2+}^r = (1-A_2)^{-r}Y_{2} \subset, \quad X_+^r = X_{1+} \oplus X_{2+}^r \oplus X_{3+}, \quad X^r = X_{+}^r \oplus X_-^r,
\end{align*}
then the same theorems yield the unique local stable manifold $W^-$ of $u=0$ according to 
Remarks \ref{R:SM} and \ref{R:smoothness-SM}. 
\end{remark}

\begin{proof}[Proof of Theorem \ref{T:Ham-UM}]
The above theorem is proved by verifying assumptions (D.1)--(D.4) in Subsection \ref{SS:NLPDE-LInMa} and then applying Theorems \ref{T:NLPDE-UM} and \ref{T:UM-smoothness}. Based on the spectral instability assumption, we let 
\be \label{E:spectraPM}
\wt \omega_+ = \min \sigma (A_+) = - \max \sigma(A_-)>0, \quad 0< \omega_- < \omega_+ < \wt \omega_+. 
\ee
Clearly (D.1) and (D.2) follow from Proposition \ref{P:LHam} and the definition of the subspaces. 

Assumption (D.3) is basically a consequence of the above assumption (3) and $\dim Y_\alpha < \infty$ and 
Proposition \ref{P:LHam}. More precisely, let $\pi_{j\alpha}: X \to X_{j\alpha}$ be the projections associated with the decomposition, then assumption (3) yields $\CA_{j\alpha, j'\alpha'}  = \pi_{j\alpha} \Omega \BD^2 \CE i_{ X_{j'\alpha'}^r} \in C^m(\CO, \BBL(X_{j'\alpha'}^{r+1}, X_{j\alpha}^r))$ where $i_{ X_{j'\alpha'}^r}: X_{j'\alpha'}^r \to X^r$ is the embedding. Clearly $\CA_{2-}$ satisfies (D3). For $j'\alpha' \ne 2-$, $X_{j'\alpha'} = X_{j'\alpha'}^r$ for any $r \in \N$. Hence $\BBL(X_{j'\alpha'}^{r+1}, X_{j\alpha}^r)$ is isomorphic to $\BBL(X_{j'\alpha'}^{r}, X_{j\alpha}^r)$ and (D3) holds for such $\CA_{j\alpha, j'\alpha'}$. For $j\alpha \ne 2-$, for the same reason $\BBL(X_{j'\alpha'}^{r}, X_{j\alpha}^{r-1})$ is isomorphic to $\BBL(X_{j'\alpha'}^{r}, X_{j\alpha}^r)$ for $r \ge 1$ and thus $\CA_{j\alpha, j'\alpha'} \in C^m(\CO, \BBL(X_{j'\alpha'}^{r}, X_{j\alpha}^{r}))$ and thus satisfies (D3). For $r=0$ and $j\alpha \ne 2-$, $\pi_{j \alpha} \Omega \in \BBL(X^*, X_{j\alpha})$ in Proposition \ref{P:LHam}(7) implies $\CA_{j\alpha, j'\alpha'} = \pi_{j\alpha} \Omega \BD^2 \CE i_{ X_{j'\alpha'}^r} \in C^m(\CO, \BBL(X_{j'\alpha'}, X_{j\alpha}))$.  

From the boundedness of the blocks of $\Omega \BD^2 \CE(u_*)$ except $A_2$ in Proposition \ref{P:LHam}, their spectral properties and Lemma \ref{L:Lumer-P-1} imply that there exist 
\[
\CL_{j\alpha} \in \BBL(X_{j\alpha}, X_{j\alpha}^*),  \quad j\alpha \in \{1+, 1-, 3-\},
\]
such that \eqref{E:coercivity-4} and \eqref{E:dissipativity-5} are satisfied (or verify (D5') easily). Finally let 
\[
\langle \CL_{2-} (u)w_1, w_2\rangle = \langle \BD^2 \CE(u) w_1, w_2 \rangle, \quad \forall u \in \CO \triangleq X^{n-1}(\ep), \; w_1, w_2 \in X_{2-}.
\]
From Proposition \ref{P:LHam}(3) and and the continuous dependence of $\BD^2 \CE(u)$ on $u \in \CO \triangleq X^{n-1} (\ep)$, $\CL_{2-}$ is uniformly positive on $X_{2-}$ if $\ep\ll1$. Moreover, from the skew-symmetry of $\Omega$, 
\[
\langle \CL_{2-} (u)w, \CA_{2-} (u) w\rangle = - \langle \BD^2 \CE(u) w, \pi_{2-} \Omega  \BD^2 \CE(u) w \rangle = \langle \BD^2 \CE(u) w, (I-\pi_{2-}) \Omega  \BD^2 \CE(u) w \rangle,
\]
which vanishes if $u=u_*$ due to Proposition \ref{P:LHam}(2). Thus \eqref{E:dissipativity-4} on $X^{n-1} (\ep)$ follows follows from the continuity of $\BD^2 \CE(u)$ and the boundedness of $(I-\pi_{2-}) \Omega$ by Proposition \ref{P:LHam}(7). 
\end{proof}

In the following, we consider several nonlinear Hamiltonian PDEs. Their local well-posedness and stable/unstable manifolds follow from Theorems \ref{T:NLPDE-LWP}, \ref{T:LWP-smoothness}, \ref{T:Ham-UM}, \ref{T:NLPDE-UM}, \ref{T:UM-smoothness}, and Remarks \ref{R:SM} and \ref{R:smoothness-SM}. The statements in Theorems \ref{T:CWW-LWP-1} and \ref{T:CWW-LInMa-main} for  water waves with surface tension also apply to these equations. 


\vspace{.08in} \noindent $\bullet$ {\bf Nonlinear Schr\"odinger (NLS) equations}. Consider \eqref{E:energyF-1} with $\CE(u)$ given in \eqref{E:energy-1} and 
\be \label{E:temp-45}
u: \BBT^d \to \R^2 \sim \C, \quad E = \frac 12 \wt E\big(|u|^2, |\nabla u|^2\big), \quad 
\Omega = \begin{pmatrix} 0 & -1 \\ 1 & 0  \end{pmatrix} \sim   i,
\ee
where the smooth energy density $\wt E(\rho, s)$ takes a gauge invariant form. 
In this vector valued case, most of the calculations are similar to the scalar case
\[
u_t = F(u) = - \Omega \BD \CE(u) = i \Big( \nabla \cdot  \big(\wt E_{s} (|u|^2, |\nabla u|^2) \nabla u \big) - \wt E_\rho  (|u|^2, |\nabla u|^2) u\Big), 
\]
\be \label{E:energy-3} \begin{split} 
\langle \CL(u)\phi_1, \phi_2 \rangle =  \BD^2 \CE (u) (\phi_1, \phi_2) = & \int_{\BBT^d} \wt E_s (\nabla \phi_1 \cdot \nabla \phi_2) + 2 \wt E_{ss} (\nabla u \cdot \nabla \phi_1)  (\nabla u \cdot \nabla \phi_2) \\
&+ 2 \wt E_{\rho s} \big( (u\cdot \phi_1)  (\nabla u \cdot \nabla \phi_2) + (u\cdot \phi_2)  (\nabla u \cdot \nabla \phi_1)\big) \\
&+ \wt E_\rho (\phi_1 \cdot \phi_2) + 2 \wt E_{\rho\rho} (u\cdot \phi_1)  (u\cdot \phi_2)  dx,
\end{split} \ee
\begin{align*}
\CA(u) \phi = \BD F(u) \phi = & i  \nabla \cdot \big(\wt E_s \nabla \phi + 2 \wt E_{ss} (\nabla u \cdot \nabla \phi) \nabla u + 2 \wt E_{s\rho} (u\cdot \phi)\nabla u \big) \\
& - i \big( 2\wt E_{s\rho} (\nabla u \cdot \nabla \phi) u + 2\wt E_{\rho\rho} (u\cdot \phi) u + \wt E_\rho \phi\big).
\end{align*}
The forms of  $\CL$ and $\CA$ suggest taking $X^n = H^{1+2n} (\BBT^d)$.
A sufficient  ellipticity condition  is 
\[
\exists \, M, M_1, a_0>0, \;\; \wt E_s (\rho, s)   + \min\{ 0, \, 2 \wt E_{ss} (\rho, s) s\}  \ge a_0, \quad \forall  \rho \le M^2, \ s \le M_1^2.
\]
which ensures   \eqref{E:LHam-2}. 
For any $M_2>0$, 
\be \label{E:temp-44}
n_0> d/4 +1/2, \quad \CO = \{ u \in H^{1+2n_0} \mid  |u|_{L^\infty} <M, \, |\nabla u|_{L^\infty} < M_1, \, |D^2 u|_{L^\infty} < M_2\}, 
\ee
it is straight forward to verify the regularity of $\CA$ and $\CL$ and thus assumptions (B.1)--(B.3) and (B.5) hold for $n=n_0+1$ due to Proposition \ref{P:Ham-LWP} and Remark \ref{R:smoothness-2}. 
The inequality \eqref{E:tameE-1} on $\BD^2 F$ can be derived using \eqref{E:tameE-2} and thus Lemma \ref{L:CB-1} applies. Hence the local well-posedness of the NLS 
follows from Theorems \ref{T:NLPDE-LWP} and \ref{T:LWP-smoothness}. 

If $x \in \R^d$ instead, one may consider $v(t, x) = e^{- i \omega t} u(t, x)$. It satisfies an NLS whose energy has an additional $\frac 12 \omega |v|_{L^2}^2$ with a uniformly positive Hessian near any given $v_*$ for reasonably large $\omega>0$. Hence the same argument applies to yield the local well-posedness. 

$\bullet$ {\it Invariant manifolds of standing waves.} While the argument also works for $u : \BBT^d \to \C$, we consider $u: \R^d \to \C$ in the following. For a simple example, consider the energy 
\be \label{E:energy-4}
\CE(u) = \int_{\R^d} \frac 12 (1 + a |u|^2 ) |\nabla u|^2 - \frac 1p |u|^p dx, \quad a>0, \; 2 < p < \frac {2d}{d-2} \; \text{ if } \; d\ge 3, 
\ee
where the ellipticity condition is satisfied. The term $(1 + a |u|^2 )$ may be viewed as a nonlinear metric on the codomain $\C \sim \R^2$ of $u(t, x)$. A standing wave is a relative equilibrium in the form of $u (t, x) = e^{i\omega t} \phi(\omega, x)$ where $\phi$ is an equilibrium of the NLS in the rotating frame 
\be \label{E:NLS}
iv _t + \nabla \cdot \big( (1+ a |v|^2) \nabla v\big) - a |\nabla v|^2 v - \omega v + |v|^{p-2} v =0,
\ee
with the energy given by $\CE(v) + \frac 12 \omega |v|_{L^2}^2$. By a variational approach \cite{CJ04}, standing waves exist for $\omega >0$, which are positive, radially symmetric, smooth, and exponentially localized. Along with the phase invariance, they form a 2-dim cylinder in the radial function space $H_{rad}^{1+ 2n} (\R^d)$ parametrized by $\omega$ and the phase $e^{i\theta}$. By the criterion given in \cite{GSS87, GSS90} {\it etc.}, such standing waves are spectrally unstable if $\p_\omega \big(|\phi(\omega, \cdot)|_{L^2}^2\big) <0$ with 1-dim unstable subspace, and stable if the sign is opposite.  In the semilinear NLS, the standing waves and their stable/unstable are the non-scattering solutions with the lowest energy. 
For $0< a\ll 1$ the standing waves $\phi(\omega)$ are perturbations to those of the seminlinear focusing NLS with power nonlinearity, where the standing waves are known to be spectrally unstable in the mass supercritical case of $p > 2+ 4/d$. So spectrally instability persists if $0< a\ll1$. 
Hence Theorem \ref{T:Ham-UM} and Remark \ref{R:Ham-SM} imply the existence of 1-dim smooth  local unstable and stable manifolds. The solutions on the unstable manifold escape small $L^2$ neighborhoods of the standing wave cylinder in logarithmic time while the solutions on the stable manifold converge exponentially to the standing wave in any $H^s$, $s\ge 0$, norm as $t \to +\infty$.

\vspace{.08in} \noindent $\bullet$ {\bf KdV type equations.} For $d=1$, $\Omega =  - \p_x$, and $\CE$ given in \eqref{E:energy-1}, the flow \eqref{E:energyF-1} becomes 
\be \label{E:KdV} 
u_t = F(u) =- \p_x \big( \p_x ( E_{p} (x, u, u_x) ) - E_u (x, u, u_x) \big), 
\ee
with 
\[
\CA(u) w= - \p_x \big( \p_x ( E_{pp} (x, u, u_x) w_x ) + \p_x ( E_{up} (x, u, u_x)) w - E_{uu} (x, u, u_x) w \big). 
\]
The ellipticity \eqref{E:ellipticity-1} turns out to be $E_{pp} >0$. If $x \in \BBT$, we take $X^n = H^{1+3n} (\BBT)$. One may verify (B.1)--(B.4) using Proposition \ref{P:Ham-LWP} and Lemma \ref{L:CB-1} and obtain the local well-posedness. 

If $x \in \R$, we may put $u(t, x)$ in a moving frame. Namely consider $v(t, x) = u(t, x-ct)$, $c>0$, which adds a momentum $\frac 12 c|v|_{L^2}^2$ in the energy. The Hessian of the energy is positive definite at any $v_*$ for  reasonably large $c>0$. Therefore the same argument applies.

$\bullet$ {\it Invariant manifolds of traveling waves.} 
For simplicity, consider 
\[
u_t = - \p_x \big( \p_x \big( (1+ au^2) u_x\big) - auu_x^2 + |u|^{p-1} u \big),   \text{ with } \; \CE(u) = \int_\R \frac 12 (1+au^2) u_x^2 - \frac 1{p+1} |u|^{p+1} dx. 
\]
A traveling wave with wave speed $c$ is a relative equilibrium in the form of $u (t, x) = \phi(c, x -ct)$ where $\phi \in H^1 (\R)$ is a solution to the ODE 
\be \label{E:KdV-TW}
\p_x \big( (1+ a\phi^2) \phi_x\big) - a\phi\phi_x^2 - c\phi + |\phi|^p=0
\ee
homoclinic to $0$. It is also an equilibrium of the KdV equation in the moving frame 
\be \label{E:KdV-1}
v _t =- \p_x \big( \p_x \big( (1+ av^2) v_x\big) - a v v_x^2 - c v + |v|^p \big) =0,  \; \text{ with } \; \wt \CE(v) = \CE(v) + c |v|_{L^2}^2/2.
\ee
Equation \eqref{E:KdV-TW} has a conserved quantity
\[
H(\phi, \phi_x) = \frac 12 (1+ a\phi^2)\phi_x^2 - \frac c2 \phi^2 + \frac 1{p+1} |\phi|^{p+1},
\]
and thus $\phi (c, x)$ exists for $c>0$ and is given by the zero level curve of $H$, where usually the positive branch limiting to $0$ is chosen. The spectral instability criterion of the traveling wave is again given by $\p_c \big(|\phi(c, \cdot)|_{L^2}^2\big) <0$ (see, e.~g.~\cite{LZ22}). Particularly it is unstable if $p>5$ in the more classical case when $a=0$. The instability persists if $0< a \ll 1$. Hence Theorem \ref{T:Ham-UM} and Remark \ref{R:Ham-SM} imply the existence of 1-dim smooth  local unstable and stable manifolds.

\vspace{.08in} \noindent $\bullet$ {\bf Nonlinear wave type equations} in $X^n = H^{1+n} \times H^n$ with the same $\Omega$ as in \eqref{E:temp-45} and the corresponding energy functional in the form of \eqref{E:energy-1}  
\[
\CE(u_1, u_2) = \int_{\R^d} E (x, u_1, u_2, \nabla u_1) dx, 
\]
where $E= E(u, v, p)$ is smooth on $\R^{2+d}$ and $u_2$ is the Legendre transform of $u_{1t}$. One may compute the Hamiltonian PDE \eqref{E:energyF-1} and its linearization as 
\[
\vec{u}_t = {F} (\vec{u}) =  \begin{pmatrix} E_v (x, u_1, u_2, \nabla u_1) \\ \nabla \cdot \big(E_p(x, u_1, u_2, \nabla u_1) \big) - E_u (x, u_1, u_2, \nabla u_1) \end{pmatrix},
\]
\[
\CA(u_1, u_2) \begin{pmatrix} \phi_1 \\ \phi_2 \end{pmatrix} = \begin{pmatrix} E_{pv} \nabla \phi_1 + E_{uv} \phi_1 + E_{vv} \phi_2 \\ \nabla \cdot \big(E_{pp} \nabla \phi_1 + E_{pv} \phi_2 \big) + (\nabla \cdot E_{pu}) \phi_1  - E_{uv} \phi_2 -E_{uu} \phi_1 \end{pmatrix},  
\]
\begin{align*}
\CL (u_1, u_2) \big((\phi_1, \phi_2), (\psi_1, \psi_2) \big) = \int & D_{(v, p)}^2 E \big( (\phi_2, \nabla \phi_1), (\psi_2, \nabla \psi_1) \big)  + E_{uu} \phi_1\psi_1 \\
& + E_{up} (\phi_1 \nabla \psi_1 + \psi_1 \nabla \phi_1) +  E_{uv} (\phi_1 \psi_2 + \phi_2\psi_1)  dx.     
\end{align*}
The coercivity of the energy density $E(u, v, p)$ is given by the uniform positive definiteness of the $(d+1) \times (d+1)$ Hessian $D_{(v, p)}^2 E$ at each $(u, v, p)$ in a ball in $\R^{2+d}$. If $x \in \BBT^d$, this condition implies \eqref{E:LHam-2} and it is standard to verify (B.1)--(B.5) using Proposition \ref{P:Ham-LWP}, Remark \ref{R:smoothness-2}, and Lemma \ref{L:CB-1} to obtain the local well-posedness. 

If $x \in \R^d$, instead of going through \eqref{E:LHam-2}, it is easier to verify assumption (B.3) directly. In fact, $(\omega - \CA(\vec{u}))^{-1}$ provide the $L^2 \times H^{-1}$ control. Along with the $\dot H^1 \times L^2$ bound from the leading part of $\CL(\vec{u})$, they dominate the lower order part of $\CL(\vec{u})$ and thus (B.3) is verified and the local well-posedness follows from Theorem \ref{T:NLPDE-LWP} (see also Remark \ref{R:energyF}b). 

\begin{remark} 
For NLS or nonlinear wave equations {\it etc.} with $x \in \R^d$, one could also work in $X^n= \dot H^{1+2n}$ or $X^n = (\dot H^{1+n} \cap \dot H^1) \times H^n$. 
\end{remark}

$\bullet$ {\it Invariant manifolds of stationary waves.} 
Again for simplicity, consider 
\[
\CE(u) = \int_{\R^d} \frac 12 (1 + a u_1^2 ) \big(|\nabla u_1|^2 +  u_2^2 \big) + \frac 12 u_1^2 - \frac 1p |u_1|^p dx, \quad a>0, \; 2 < p < \frac {2d}{d-2} \; \text{ if } \; d\ge 3. 
\]
For a stationary wave $u_2=0$ and $u_1$ has to be critical point of $\CE(u_1, 0)$. By the same variation argument \cite{CJ04}, a positive, radially symmetric, smooth, and exponentially localized stationary wave $U_1$ exists. The Hessian of $\BD^2 \CE(U_1, 0)$ has one negative direction and a 1-dim kernel. Therefore the linearized wave equation has a 1-dim unstable and stable subspace and satisfies the exponential trichotomy (see, e.~g.~\cite{LZ22CPAM}). Theorem \ref{T:Ham-UM} and Remark \ref{R:Ham-SM} enures the existence of 1-dim smooth  local unstable and stable manifolds.

\vspace{.08in} \noindent $\bullet$ {\bf MMT equations}. This is a Hamiltonian model whose energy density depends on $u: \BBT^d \to \R^2 \sim \C$ nonlocally 
\[
\CE(u) = \int_{\BBT^d} \frac 12 \big| |\nabla|^\alpha u \big|^2 + \frac \sigma4 \big| |\nabla|^\beta u \big|^4, \quad \lambda =\pm 1, \; \alpha>0, \; \beta \le \alpha.  
\]
With the symplectic form $\Omega$ given in \eqref{E:temp-45}, the evolution is governed by 
\be \label{E:MMT-1}
u_t = F(u) = - \Omega \BD \CE(u) = - i \big(|\nabla|^{2\alpha} u +  \sigma |\nabla|^\beta \big( \big| |\nabla|^\beta u \big|^2 |\nabla|^\beta u \big) \big). 
\ee
Following the same procedure, one may compute 
\begin{align*}
\langle \CL(u)\phi_1, \phi_2 \rangle =  \BD^2 \CE (u) (\phi_1, \phi_2) = \int_{\BBT^d}  & |\nabla|^\alpha \phi_1 \cdot |\nabla|^\alpha \phi_2 +  \sigma \big| |\nabla|^\beta u \big|^2 |\nabla|^\beta \phi_1 \cdot |\nabla|^\beta \phi_2  \\
&+ 2  \sigma\big(|\nabla|^\beta u \cdot |\nabla|^\beta \phi_1\big) \big(|\nabla|^\beta u \cdot |\nabla|^\beta \phi_2\big)  dx,
\end{align*}
\[
\CA(u) \phi = - i \BD^2 \CE(u) \phi = - i \big(|\nabla|^{2\alpha} \phi +  \sigma |\nabla|^\beta \big( \big| |\nabla|^\beta u \big|^2 |\nabla|^\beta \phi + 2 \big(|\nabla|^\beta u \cdot |\nabla|^\beta \phi\big) |\nabla|^\beta u  \big) \big). 
\]
We take  
\[
X^n = H^{(1+2n) \alpha} (\BBT^d), \quad n_0 > d/(4\alpha) + \beta/\alpha - 1/2.
\]
To ensure the ellipticity, let   
\[
\CO = 
\begin{cases} 
X^{n_0} (R)  & \text{ if } \; \alpha > \beta \ \text{ or } \ \sigma =1,  \\
\big\{ u \in X^{n_0} \mid \big| |\nabla|^\beta u \big|_{L^\infty}^2 <\delta \big\}  & \text{ if } \; \alpha = \beta \ \text{ and } \  \sigma =-1, 
\end{cases}\]
where $R>0$ and $\delta <1/3$. The verification of (B.1)--(B.5) through Proposition \ref{P:Ham-LWP} and Remark \ref{R:smoothness-2} with $n= n_0+1$ is similar due to the dominance of $\big| |\nabla|^\alpha u \big|_{L^2}^2$ 
and the local well-posedness follows.

$\bullet$ {\it Invariant manifolds of plane wave} 
$u_0 (t, x) = a e^{i(\omega t + \xi_0 \cdot x})$ where 
\be \label{E:MMT-PW}
\omega + |\xi_0|^{2\alpha} + \sigma |a|^2 |\xi_0|^{4\beta} =0, \quad \xi_0 \in \Z^d. 
\ee
In the rotating frame $u(t, x) = e^{i \omega t} v(t, x)$, it corresponds to an equilibrium $v_0 (x) = a e^{i \xi_0 \cdot x}$ of 
\[
v_t = - i \big(|\nabla|^{2\alpha} v + \omega v  +  \sigma |\nabla|^\beta \big( \big| |\nabla|^\beta v \big|^2 |\nabla|^\beta v \big) \big).
\]
If $v_0(x)$ is spectrally unstable, Theorem \ref{T:Ham-UM} and Remark \ref{R:Ham-SM} yield its finite dimensional smooth local stable and unstable manifolds. 

To see that some plane wave are indeed spectrally unstable, we first observe that in the linearized equation $v_t = \CA(v_0) v -i \omega v$, the $\xi$-th mode interacts only with the $(2\xi_0 -\xi)$-th mode. Due to the phase invariance we may assume $a>0$. One may compute that $b_+(t) e^{i\xi \cdot x} + b_-(t) e^{i(2\xi_0 - \xi) \cdot x}$ is a linearized solution iff 
\[
\begin{pmatrix} b_+ \\ b_- \end{pmatrix}_t = -i \begin{pmatrix} c_+ b_+ + c \bar b_- \\ c \bar b_+ + c_- b_- \end{pmatrix}, \; \text{ where } \ c= \sigma a^2 |\xi_0|^{2\beta} |\xi|^{\beta}|2\xi_0 -\xi|^{\beta},
\]
\[
c_+ = \omega + |\xi|^{2\alpha} + 2 \sigma a^2 |\xi_0|^{2\beta} |\xi|^{2\beta}, \quad c_- =  \omega + |2\xi_0 - \xi|^{2\alpha} + 2 \sigma a^2 |\xi_0|^{2\beta} |2\xi_0 -\xi|^{2\beta}. 
\]
Eigenvalues of this (real) 4-dim linear system are the roots of 
\[
\lambda^4 + (c_+^2 + c_-^2 - 2c^2) \lambda^2 + (c_+c_- - c^2)^2=0. 
\]
A sufficient condition for the instability is $c_+^2 + c_-^2 - 2c^2 <0$, which according to \eqref{E:MMT-PW} is equivalent to 
\begin{align*}
2 a^4 |\xi_0|^{4\beta} |\xi|^{2\beta}|2\xi_0 -\xi|^{2\beta} > & \big( |\xi|^{2\alpha} - |\xi_0|^{2\alpha} + 2 \sigma a^2 |\xi_0|^{2\beta} (|\xi|^{2\beta} -  |\xi_0|^{2\beta}) \big)^2 \\
& + \big( |2\xi_0 -\xi|^{2\alpha} - |\xi_0|^{2\alpha} + 2 \sigma a^2 |\xi_0|^{2\beta} (|2\xi_0 -\xi|^{2\beta} -  |\xi_0|^{2\beta}) \big)^2. 
\end{align*}
For example, this can be ensured if  $a \gg 1$ and 
\[
|\xi|^{2\beta}|2\xi_0 -\xi|^{2\beta} > 2 (|\xi|^{2\beta} -  |\xi_0|^{2\beta})^2 + 2(|2\xi_0 -\xi|^{2\beta} -  |\xi_0|^{2\beta})^2,
\]
which can be verified easily if $|\xi - \xi_0| \ll |\xi_0|$.

\subsection{Other PDE systems} \label{SS:examples-other} 

In this subsection, we consider some non-Hamiltonian nonlinear PDEs which easily fit into the framework of Appendix \ref{S:pre-LWP} and Section \ref{S:LInMa}.

\vspace{.08in} 
\noindent $\bullet$ {\it A nonlinear dissipative PDE} with $\CE$ as in \eqref{E:energy-1} and \eqref{E:ellipticity-1} for $u \in \R^m$ defined on\footnote{The domain can also be a Riemannian manifold, see the next example on the mean curvature flow.} $\R^d$ or $\BBT^d$
\be \label{E:heat-1} \begin{split} 
u_t = & F(u) 
= - \Omega (x, u, D u) \BD \CE(u) + f(x, u, D u), 
\end{split} \ee
where $\Omega(x, u, p)$ and $f(x, u, p)$ are $m\times m$ matrixe and $m$-dim vector valued, respectively,  smoothly depending on $x$, $u$, and the matrix $p$ and 
\be \label{E:temp-47}
\Omega (x, u, p) = \Omega_S(x, u, p)+ \Omega_A (x, u, p),
\quad \Omega_S^T = \Omega_S\ge 1, \;\; \Omega_A^T = - \Omega_A.
\ee  
T simply the notation, for any function $u(x)$, in the rest of this example we shall write 
\[
\Omega[u] (x)= \Omega(x, u(x), Du(x)).  
\]
Equation \eqref{E:heat-1} is 
a generalization of \eqref{E:energyF-1} where $\Omega$ depends on $u$ and $D u$. 
Denoting the $j$-th column of $p$ by $p_j$ (corresponding to $\p_{x_j} u$), the linearization is given by   
\be \label{E:NHeatE-2} \begin{split}
\CA(u) w = & - \Omega \BD^2\CE w
+ \Big(\Omega_u  w + \sum_{j=1}^d \Omega_{p_j}   w_{x_l}\Big) \Big( \sum_{j=1}^d \p_{x_j} E_{p_j}  - E_u \Big)  + f_u  w + \sum_{j=1}^d f_{p_j} w_{x_j}, 
\end{split} \ee
where $E$ and $\Omega$ are evaluated at $(x, u, D u)$. Let 
\[
\CL(u)= \BD^2\CE (u), \quad \CL_0(u)  = a_0 - \sum_{j, k=1}^d  \p_{x_j} \big( E_{p_jp_k} (x, u, D u ) \p_{x_k} \big), 
\]
where $a_0>0$ is given in \eqref{E:ellipticity-1}. 
Take $X^n = H^{1+2n}$ and for $M_1>0$, 
\be \label{E:temp-48}
n_0> (d+2)/4, \quad \CO = \{ u \in H^{1+2n_0} \mid  |u|_{L^\infty}, \, |\nabla u|_{L^\infty} < M, \ |D^2 u|_{L^\infty} < M_1\}.
\ee
Due to \eqref{E:ellipticity-1}, the symmetric $\CL_0(u) \in \BBL(X, X^*)$ is uniformly positive on $X= H^1$ and defines an equivalent metric.
So there exist $a_1, a_2>0$ such that for any $u \in \CO$ and $w \in H^1$,  
\be \label{E:temp-49}
\langle\CL_0(u) w, w \rangle \ge a_0 |w|_{H^1}^2, \;\;  a_1|w|_{H^1}^2 \ge \langle\CL(u) w, w \rangle \ge (a_0/2) |w|_{H^1}^2 - a_2 |w|_{H^{-1}}^2.
\ee
From \eqref{E:temp-47} and the definition of $\CL(u)$ and $\CL_0(u)$, there exists $C>0$ such that for any $\lambda >0$, $u \in \CO$, and $w\in X$ 
\begin{align*} 
& C |w|_{H^1} |(\lambda - \CA(u))w|_{H^1} \ge \langle\CL_0(u) w, (\lambda - \CA(u)) w \rangle \\
= & \lambda\langle\CL_0(u) w, w \rangle + \big\langle\CL_0(u) w, \Omega[u] \CL_0(u) w + \big(\Omega[u]  (\CL(u) - \CL_0(u)) -  ((\Omega[u]\CL(u) + \CA(u) \big) w\big\rangle \\
\ge & \lambda a_0  |w|_{H^1}^2 + |\CL_0(u)w|_{L^2}^2 -C |\CL_0(u)w|_{L^2} |w|_{H^1}. 
\end{align*}
Therefore $(\lambda - \CA(u))^{-1} \in \BBL(H^1, H^3)$ for large $\lambda >0$. By considering $\CL_0(u)^r(\lambda-\CA(u)) $ whose commutators are lower order terms, we verify assumption (B.2) in Appendix \ref{SS:NLPDE-LWP}. Along with \eqref{E:temp-49} it also implies \eqref{E:coercivity-2}. 
Similarly 
\be \label{E:temp-49.3} \begin{split}
\langle\CL(u) w, \CA(u) w \rangle = &  \big\langle\CL(u) w, - \Omega[u] \CL(u) w + \big(  \Omega [u] \CL(u) +  \CA(u) \big) w\big\rangle \\
\le & -|\CL(u)w|_{L^2}^2 + C |\CL(u)w|_{L^2} |w|_{H^1} \le -|\CL_0(u)w|_{L^2}^2 + C |w|_{H^2} |w|_{H^1}
\end{split} \ee
and thus \eqref{E:dissipativity-2.5} follows. 
Subsequently, the verification of assumptions (B.1)--(B.5) for $n=n_0+1$ is straight forward 
using Lemma \ref{L:CB-1} and Remark \ref{R:smoothness-2}. Hence the local well-posedness of \eqref{E:heat-1} follows. 


To simply the spectral analysis in obtaining 
local invariant manifolds, we further assume 
\be \label{E:temp-49.5}
u: \BBT^d \to \R^m, \quad \Omega_A\equiv 0, \quad f\equiv 0. 
\ee
Let $u_0 \in \CO_n \triangleq \CO \cap X^n$ be an equilibrium of \eqref{E:heat-1}. 
The above  assumptions imply that $\CA(u_0) = - \Omega[u_0] \CL(u_0)$ is a self-adjoint operator on $L^2(\BBT^d, \Omega[u_0]^{-1} dx)$, bounded from above, and have compact resolvents. So $\sigma(\CA(u_0))$ consists of only real semi-simple eigenvalues  arranged into a sequence  $\lambda_1 > \lambda_2 > \ldots \to -\infty$ whose eigenspaces $\ker (\lambda_j -\CA(u_0))$
satisfy 
\[
\dim \ker (\lambda_j -\CA(u_0)) < \infty, 
\quad  \ker (\lambda_j -\CA(u_0)) \subset X^n= H^{1+2n}, 
\]
\be \label{E:temp-50} 
-\lambda_{j} \langle \Omega[u_0]^{-1} w, \wt w \rangle = \langle \CL(u_0) w, \wt w \rangle, \; \forall w\in  \ker (\lambda_j -\CA(u_0)), \; \wt w\in  \ker (\lambda_{j'} -\CA(u_0)),
\ee
and are complete on $L^2$. In $X=H^1$, $\sigma(\CA(u_0))$ and the eigenspaces remain the same. Let 
\[
Y_+ (\omega) =  \oplus_{\lambda \in \sigma(\CA(u_0)) \cap [\omega, +\infty)} \ker (\lambda -\CA(u_0)), \;\; Y_- (\omega) =  \overline {\oplus_{\lambda \in \sigma(\CA(u_0)) \cap (-\infty, \omega)} \ker (\lambda -\CA(u_0))} \cap H^1, 
\] 
which are closed subspaces of $H^1$, invariant under $\CA(u_0)$, and satisfy $\dim Y_+ (\omega) < \infty$ and $H^1 = Y_+(\omega) \oplus Y_-(\omega)$. 
In particular, 
\eqref{E:temp-50} includes the $\CL(u_0)$-orthogonality and $\CL(u_0)\CA(u_0)$-orthogonality between the different eigenspaces and $\CL(u_0)|_{Y_-(0)} >0$. 
In fact, a straight forward argument based on the compactness and \eqref{E:temp-49} yields 
\be \label{E:temp-51} 
\exists\delta>0, \; \text{ s. t. } \; \langle \CL(u_0)w, w \rangle >\delta|w|_{H^1}, \; \forall w\in  Y_-(0). 
\ee
Hence $\CL(u_0)$ induces a metric on $Y_-(0)$ equivalent $|\cdot|_{H^1}$. 
For any $\omega <0$ 
\be \label{E:temp-52} 
\langle \CL(u_0)w, \CA(u_0) w \rangle \le \lambda_*\langle \CL(u_0)w,  w \rangle, \; \forall w\in  Y_-(\omega), \quad \lambda_* = \max \sigma(\CA(u_0)) \cap (-\infty, \omega).
\ee

If $\lambda_- < \lambda_+$ and $\sigma(\CA(u_0)) \cap [\lambda_-, \lambda_+) = \emptyset$, clearly $ H^1 = Y_-(\lambda_-) \oplus Y_+(\lambda_+)$.  Let 
\[
X_{1+} = Y_+(\lambda_+), \quad X_{2+}=X_{3+} = X_{3-} = \{0\}, \quad X_{1-} = Y_-(\lambda_-) \cap Y_+(0),  
\]
\[
X_{2-}^r= Y_-(\lambda_-) \cap Y_-(0) \cap H^{1+2r}, \quad \langle \CL_{2-} (u) w, w\rangle = \langle \CL(u) w, w\rangle, \quad \forall w \in H^1. 
\]
It is easy to see 
\[
X^r=H^{1+2r} = X_{1+} \oplus X_{1-} \oplus X_{2-}^r,   \quad \dim X_{1+} , \dim X_{1-} < \infty,
\]
where the associated projections $\Pi_{1\pm}$ and $\Pi_{2-}$ can be written using the $\CL(u_0)$-orthogonality due to \eqref{E:temp-50}.  
For any $\omega_\pm \in (\lambda_-, \lambda_+)$ with $\omega_+> \omega_-$, (D.1--D.2) in Subsection \ref{SS:NLPDE-LInMa} are clearly satisfied. Assumption (D.3) holds due to \eqref{E:temp-50} and the smoothness of $E(x, u, p)$ and $\Omega[u_0]^{-1} w$ for any $w \in X_{1+} \oplus X_{2-}$. For $\ep \ll 1$, $\CL_{2-}(u)\ge \delta/2$ on $X_{2-}$ for any $u \in X^{n-1}(u_0, \ep)$ by \eqref{E:temp-51} and the continuity of $\CL(u) \in \BBL(X, X^*)$ in $u \in X^{n-1}$. Finally from \eqref{E:NHeatE-2},  \eqref{E:temp-49.5}, \eqref{E:temp-49.3}, \eqref{E:ellipticity-1}, and \eqref{E:temp-52} we have, for any $u \in X^{n-1} (u_0, \ep)$ and $w \in X_{2-}^1$ 
\begin{align*}
\big|\langle \CL_{2-}(u)w, \CA_{2-}(u) w \rangle - \langle \CL(u_0)w, \CA (u_0) w \rangle \big|  \le& C |u-u_0|_{X^{n-1}} |w|_{H^2}^2 \\
\le & - C |u-u_0|_{X^{n-1}} \langle \CL(u_0)w, \CA (u_0) w \rangle,
\end{align*}
which along with \eqref{E:temp-52} implies, for $\lambda_1 = \max \sigma (\CA(u_0)) \cap (-\infty, 0)<0$, 
\begin{align*}
\langle \CL_{2-}(u)w, \CA_{2-}(u) w \rangle \le & ( 1- C |u-u_0|_{X^{n-1}}) \langle \CL(u_0)w, \CA (u_0) w \rangle \\
\le &( 1- C |u-u_0|_{X^{n-1}})  \min\{\lambda_1, \lambda_-\} \langle \CL(u_0)w, w \rangle \\
\le & ( 1- C |u-u_0|_{X^{n-1}})  \min\{\lambda_1, \lambda_-\} \langle \CL_{2-}(u)w, w \rangle. 
\end{align*}
So (D.4) is verified for $\omega_-$, $\CL_{2-}(u)$, $\CA_{2-}(u)$ and $u \in X^{n-1}(u_0, \ep)$. The rest of (D.4) follows directly from the finite dimensionality and spectral properties of $\CA(u_0)$ on $X_{1\pm}$. 
Therefore the existence of the unstable  manifold $W^+$ (and stable manifold $W^-$, respectively) of $u_0$ follows from Theorems \ref{T:NLPDE-UM} and \ref{T:UM-smoothness} (and Remarks \ref{R:SM} and \ref{R:smoothness-SM}) by taking $\lambda_+ > 0$ (or $\lambda_-< 0$, respectively). Here $W^+$ (or $W^-$) may be a strong unstable submanifold (or strong stable submanifold) if $\omega_->0$ (or $\lambda_+<0$, respectively).  

\begin{remark} 
It is also possible to consider dissipative flows with other $\Omega$ satisfying \eqref{E:Omega-1}, for example, $\Omega = -\Delta$ 
which leads to a quasilinear Cahn-Hilliard type equation. 
\end{remark}

\noindent $\bullet$ {\bf Mean curvature flow (MCF).} 
This geometric flow is one of the classical quasilinear gradient flow type PDEs. Heuristically, inside a  $(d+1)$-dim ambient Riemannian manifold $(\CM^{d+1}, g)$, for $n < d+1$, the set $\Gamma$ of all $n$-dim submanifolds $M \subset \CM$ can be viewed as an infinite dimensional manifold. The tangent space $T_M \Gamma$ is equivalent to the space of normal vector fields along $M$, which we equip with the $L^2$ metric (induced by the metric $g$ of $\CM$). Thus $\Gamma$ becomes an infinite dimensional Riemannian manifold itself. The MCF is the gradient flow of the area function from $ \Gamma$ to $\R$ under the $L^2$ metric on $T\Gamma$. 

We first review briefly how to formulate the MCF as a PDE in the form of \eqref{E:energyF-1} (see also, e.~g.~\cite{Eck04, Ger06, Mant11, CM19}). 
For simplicity, we consider the MCF of compact hypersurfaces 
near an orientable reference hypersurface $\CM_0^d \subset \CM$. Let $N(x)$, $x \in \CM_0$, be a unit normal vector field along $\CM_0$, then near any $(x, 0)$, $\Phi(x, s)=\exp_x(s N(x))$ is a local diffeomorphism from $\CM_0 \times \R$ to $\CM$. The Gauss Lemma implies that $\p_s \Phi(x, s) \perp D \Phi(x, s) \tau$ for any $\tau \in T_x \CM_0$. Hence without loss of generality, locally we may make the following assumptions on $\CM$ and its metric $g \in \BBL( T\CM, T^*\CM)$, 
\be \label{E:metric-1}
\CM= \CM_0 \times \R, \quad g(x, s) = \bar g (x, s) + ds^2, \;\; \bar g(x, s) \in \BBL(T\CM_0, T^* \CM_0), 
\ee
where naturally  $\bar g$ is symmetric and positive. Let $\CU \subset \CM_0$ be open with compact $\overline{\CU}$ and smooth $\p \CU$ (or $\p \CU=\emptyset$). Any hypersurface $S \subset \CM$ near $\CU$ satisfying $\p S = \p \CU$ can be represented as the graph $s= \phi(x)$ of a small function $\phi (x)$ in a Sobolev space depending on the regularity of $S$ 
\[
X^n = H_0^{2n+1} (\CU) = \big \{ \phi:   \CU \to \R \, \big| \, \phi|_{\p \CU}=0, \ |\phi|_{X^{n}}^2 \triangleq \Sigma_{j=0}^{2n+1} |\CD_0^j \phi |_{L^2 ( \CU)}^2 < \infty\big\},
\]
where $\CD_0$ is the covariant differentiation on $T \CM_0$ induced by its Riemannian metric $g_0 \triangleq \bar g(\cdot, 0)$. Hence $\phi \in X^n$ can be used as a local coordinate of $\Gamma$ near $\CU$. 

For any $\phi \in X^n$, denote  $P_\phi (x) = (x, \phi(x))$ and $\CS_\phi = graph (\phi) = P_\phi(\CU)$. In the coordinates $x \to P_\phi(x)$ of $S_\phi$, through standard calculations, its metric, the measure, and the  upward unit normal vector 
at $P_\phi (x) $ are given by 
\[
g_\phi (P_\phi (x)) ) = \bar g ( P_\phi(x)) + d \phi (x) \otimes d \phi(x) \in \BBL(T\CM_0, T^* \CM_0), 
\]
\[
dS_\phi = \sqrt { \det(g_0^{-1} (g_\phi \circ P_\phi))} dS_0, 
\quad N_\phi (x) = \frac {- \bar g ( P_\phi(x))^{-1} d\phi + \p_s}{\sqrt {1+ \langle d\phi,  \bar g ( P_\phi(x))^{-1} d\phi \rangle }}, 
\]
where $dS_0$ is the Riemannian measure on $\CM_0$. In particular $g_0^{-1} (g_\phi \circ P_\phi) \in \BBL(T\CM_0)$ and thus its determinant is well-defined. Using an orthonormal basis of $T_x \CM_0$ under the metric $\bar g (P_\phi(x))$, where one of them is parallel to $\bar g(P_\phi(x))^{-1} d\phi$, one may calculate
\[
\det(\bar g^{-1} g_\phi)|_{P_\phi(x)} = 1+ \langle d\phi,  \bar g ( P_\phi(x))^{-1} d\phi \rangle. 
\] 
Therefore we can single out  the dependence on $d \phi$ and rewrite 
\[
dS_\phi = (a \circ P_\phi) \sqrt { 1+ \langle d\phi,  (\bar g^{-1} \circ P_\phi) d\phi \rangle} dS_0, \quad  a(x, s) =  \sqrt { \det(g_0^{-1} \bar g)} \big|_{(x, s)}.
\] 

At any $\phi \in X^n$, an infinitesimal variation $\wt \phi$ to $\phi$ is associated to the normal variational vector field along $\CS_\phi$  
\[
\psi(x) N_\phi(x), \; x\in \CM_0, \ \text{ where } \ \psi = (\wt \phi \p_s, N_\phi)_g = \wt \phi /  \sqrt {1+ \langle d\phi,  (\bar g^{-1} \circ P_\phi) d\phi \rangle }.   
\]
Hence, given variations $\wt \phi_1$ and $\wt \phi_2$ , the $L^2$ inner product of their corresponding normal vector fields defines a symmetric $\Omega(\phi)^{-1} \in \BBL(X, X^*)$ as 
\[
\langle \Omega(\phi)^{-1} \wt \phi_1, \wt \phi_2\rangle = \int_{S_\phi} \psi_1 \psi_2 dS_\phi = \int_{\CU} \frac {(a \circ P_\phi) \wt \phi_1 \wt \phi_2}{\sqrt{1+ \langle d\phi,  (\bar g^{-1}  \circ P_\phi) d\phi \rangle}} dS_0. 
\]
A solution $\phi(t)$ of the MCF in this graph formulation satisfies  
\be \label{E:MCF-1} 
\phi_t = - \Omega(\phi) \BD \CE(\phi)
\ee
where $\CE(\phi)$ is the surface area of $\CS_\phi$ 
\[
\CE (\phi) = \int_{\CM_0} (a \circ P_\phi) \sqrt { 1+ \langle d\phi,  (\bar g^{-1} \circ P_\phi) d\phi \rangle} dS_0. 
\]
The above equation \eqref{E:MCF-1} is equivalent to 
\begin{align*}
-\langle \Omega(\phi)^{-1} \phi_t,& \wt \phi \rangle = \langle \BD \CE(\phi), \wt \phi \rangle 
=  \int_{\CM_0} \Big\langle d\wt \phi, \frac {(a \bar g^{-1}) \circ P_\phi } {\sqrt { 1+ \langle d\phi,  (\bar g \circ P_\phi)^{-1} d\phi \rangle}} d\phi\Big \rangle \\
& + \frac {  \langle d\phi,  ((a\p_s (\bar g^{-1})) \circ P_\phi) d\phi \rangle }{2 \sqrt { 1+ \langle d\phi,  (\bar g \circ P_\phi)^{-1} d\phi \rangle}}  \wt \phi + ((\p_s a) \circ P_\phi) \sqrt { 1+ \langle d\phi,  (\bar g^{-1}  \circ P_\phi)d\phi \rangle} \wt \phi dS_0, 
\end{align*}
for any $\wt \phi$. Therefore  
\be \label{E:MCF-2} \begin{split} 
\phi_t = F(\phi) \triangleq  & - \frac { \sqrt { 1+ \langle d\phi,  (\bar g^{-1} \circ P_\phi) d\phi \rangle}} {a \circ P_\phi}  d^* \Big( \frac {(a \bar g^{-1}) \circ P_\phi } {\sqrt { 1+ \langle d\phi,  (\bar g^{-1} \circ P_\phi) d\phi \rangle}} d\phi  \Big) \\
& - \frac 12  \langle d\phi,  ((\p_s (\bar g^{-1})) \circ P_\phi) d\phi \rangle  -  \big( 1+ \langle d\phi,  (\bar g^{-1}  \circ P_\phi)d\phi \rangle \big) \Big(\frac {\p_s a}a \Big)\circ P_\phi,  
\end{split} \ee
where the covariant divergence $d^*$ on $(\CM_0, g_0)$ is the dual operator of the differentiation "$-d$" under the zero boundary condition on $\p \CU$, 
mapping tangent vector fields to scalar functions on $\CM_0$. For any vector field $W: \CM_0 \to T\CM_0$, in a local coordinates $x= (x_1, \ldots, x_d)$, 
\[
d^* W = \big( \nabla \cdot \big( \sqrt {\det g_0} W \big) \big) / \sqrt{\det g_0}. 
\]
When $\CM=\R^{d+1}$ is flat, $\bar g = I_{d\times d}$ and $a \equiv 1$. Thus the last two terms in \eqref{E:MCF-2} drop out and it takes the well-known standard form (see, e.~g.~\cite{GT01}). 

The local well-posedness of MCF has been well-established, see, e.~g.~\cite{Eck04, Ger06, Mant11, CM19}. Some local invariant manifolds of MCF were also constructed in, e.~g.~\cite{EW87, SX21}. The maximal regularity property was used. 
Like \eqref{E:heat-1}, the MCF  \eqref{E:MCF-1}/\eqref{E:MCF-2} is also a generalization of \eqref{E:energyF-1} with $\Omega$ depending on $\phi$ and $d\phi$, which only create lower order terms. Even though the   unknown functions of \eqref{E:MCF-1} are defined on a Riemannian manifold $\CU \subset \CM_0$, an argument similar to that of \eqref{E:heat-1} also applies and we obtain the local-wellposedness and the local stable and unstable manifolds near critical surfaces through the energy estimate based on Theorems \ref{T:NLPDE-LWP}, \ref{T:LWP-smoothness}, \ref{T:NLPDE-UM}, 
and \ref{T:UM-smoothness} and there remarks.

\section{Irrotational waver waves with surface tension} \label{S:CWW}

Consider an irrotational incompressible inviscid fluid in a $d$-dim moving 
domain $\Omega_t$ with a free surface $\CS_t$. Let 
\be \label{E:fluidD-0}
\CU = \Pi_{j=1}^{d_1} \big( \R/(2 \pi l_j) \big) \times \R^{d_2} , \quad d_1, d_2 \in \{0\} \cup \N, \; d_1 + d_2 = d-1 \ge 1, 
\ee
denote the horizontal directions of the fluid domain which can  be $2\pi l_j$-periodic, in the $j$-th direction, $j=1, \ldots, l_{d_1}$,  and extend to infinity in other directions. We focus on the case where the fluid domain is in the form of  
\be \label{E:fluidD-2}
\Omega_h = \{ -h_0 < x_d < h(x') \mid x' \in \CU\}, \quad h_0 \in (0, \infty],
\ee
with the free surface given by a graph 
\be \label{E:surface-1}
\CS_h = \{ x_d = h(x'), \, x'= (x_1, \ldots x_{d-1}) \in \CU\}, \quad \inf_\CU h> - h_0,
\ee
where $h_0>0$, if finite, 
is the typical depth of the fluid. 
The surface separates the fluid and vacuum. 
As times evolves, $h= h(t,\cdot)$ and thus $\Omega_t \triangleq \Omega_{h(t, \cdot)}$ are parts of the unknowns\footnote{The notations $\Omega_h$ and $\Omega_t = \Omega_{h(t, \cdot)}$ are mostly interchangeable throughout this section, while the former generally emphasizes domains below the graph of $h$. }. 
We shall obtain the local stable and unstable manifolds of spectrally unstable steady irrotational water waves with surface tension based on Theorem \ref{T:Ham-UM} and Remark \ref{R:Ham-SM}. 
The problem of free interfaces separating two fluids will be discussed briefly in Subsection \ref{SS:2F}.  

The velocity field $u(t, x)$, $x \in \Omega_h$, of the fluid and the free surface are governed by the free boundary problem of the incompressible Euler equation 
\be \label{Euler}
u_t + (u \cdot \nabla) u + \nabla p =0, \quad \nabla \cdot u=0, \quad \text{ in } \ \Omega_h, 
\ee
along with the kinematic and dynamic boundary conditions. 
In the {\it irrotational} case\footnote{In fact, if $d_1 > 0$, in each periodic horizontal direction, the conserved horizontal momentum has to be assumed to be zero, i.~e.~the vanishing of the average horizontal velocity, in order to write $u = \nabla \phi$ and subsequently \eqref{E:Zakharov-1}. If the average horizontal velocity is not zero, 
it can be separated as a background velocity and we obtain \eqref{E:Zakharov-tf}. See also Remark \ref{R:background-V}.}, the systems can be reduced to the surface through the harmonic velocity potential $\phi(t, x)$ where 
\[
u(t, x) = \nabla_x \phi(t, x), \quad x\in \Omega_h, 
\]
satisfying  
\be \label{E:Harmonic}
\Delta_x \phi =0 \; \text{ in } \; \Omega_h, \quad \nabla_x \phi \in L^2 (\Omega_h), \; \text{ and, if } \ h_0 < \infty, \quad \p_{x_d} \phi|_{ x_d = -h _0} =0.
\ee 
In terms of the surface profile $h$ and the trace of the harmonic potential function $\phi$
along $\CS_h$ 
\be \label{E:HarTrace-1}
\Phi(t, x') = \phi(t,  x', h (t, x')), \quad x' \in \CU,  
\ee
the water wave problem takes the form of the well-known Zakharov system for $x' \in \CU$ 
\begin{subequations}  \label{E:Zakharov-1} 
\begin{align} 
& h_t  =\CG( h ) \Phi, \label{E:Z-1} \\
&\Phi_t = -\frac 12 |\nabla_{x'}\Phi|^2 + \frac {(\CG(h)\Phi+ \nabla_{x'} h \cdot \nabla_{x'} \Phi)^2}{2 (1+ |\nabla_{x'} h|^2)}  - g h + \sigma \nabla_{x'} \cdot \Big( \frac {\nabla_{x'}  h }{\sqrt{1+ |\nabla_{x'}  h |^2}} \Big), \label{E:Z-2} 
\end{align}
\end{subequations}
where 
\be \label{E:gravity}
g > 0 \; \text{ if } \; d_2 >0 \; \text{ or } \; g\ge 0 \; \text{ if } \; d_2=0, 
\ee
is the gravitational acceleration, $\sigma >0$ the coefficient of the surface tension, and $\CG(h)$ the Dirichlet-to-Neuman operator on $\CS_h$ (with the slip boundary condition in \eqref{E:Harmonic} if $h_0<\infty$) weighted by the surface area 
\be \label{E:CG}
\CG(h) \Phi = \p_{x_d} \phi - \nabla_{x'} h \cdot \nabla_{x'} \phi. 
\ee

There have been plenty of local well-posedness results of water waves, irrotational or rotational, with or without surface tension or gravity, {\it etc.}
in the literature. See, for example, \cite{Wu97, Wu99, Na74, Yo83, Cr85, BG98, La05, Sc05, CS07, AM05, ZZ08, AM09, SZ08a, MZ09, SZ11, ABZ11}. The readers are referred to \cite{La13} for more comments on the references. 

It is advantageous sometimes to study water waves in a traveling frame with a background velocity vector $\Vc = (c_1, \ldots, c_{d-1}) \in \R^{d-1}$, 
and the Zakharov system becomes 
\begin{subequations}  \label{E:Zakharov-tf} 
\begin{align} 
& h_t  =\Vc \cdot \nabla_{x'} h + \CG( h ) \Phi, \label{E:Z-tf-1} \\
&\Phi_t=\Vc\cdot \nabla_{x'}  \Phi  -\frac {|\nabla_{x'}\Phi|^2}2 + \frac {(\CG(h)\Phi+ \nabla_{x'} h \cdot \nabla_{x'} \Phi)^2}{2 (1+ |\nabla_{x'} h|^2)}  - g h + \sigma \nabla_{x'} \cdot \big( \frac {\nabla_{x'}  h }{\sqrt{1+ |\nabla_{x'}  h |^2}} \big) \label{E:Z-tf-2} 
\end{align}
\end{subequations}
Due to \eqref{E:DCG-1} and Remark \ref{R:Z-tf-2} below, \eqref{E:Zakharov-tf} has a Hamiltonian formulation 
\be \label{E:HamF}
\p_t \begin{pmatrix} h \\ \Phi \end{pmatrix} = J \BD \BH(h, \Phi) \triangleq F(h, \Phi), \quad J = \begin{pmatrix} 0 & 1 \\ -1 & 0  \end{pmatrix},
\ee
with 
\be \label{E:Ham-CG} 
\BH(h, \Phi) = \int_{\CU } \frac 12 \Phi \CG(h) \Phi + \frac 12 g h^2 + \sigma (\sqrt{1+ |\nabla_{x'}  h |^2} -1) +  \Phi \Vc \cdot \nabla_{x'} h \, dx'.
\ee

Equilibria $(h_*, \Phi_*)$ of \eqref{E:Zakharov-tf} satisfying $h_*> -h_0$ are irrotational steady traveling water waves, which can be proved to be $C^\infty$. There is a huge literature on their existence and properties. A good survey can be found in \cite{HHSTWW22} and see also, e.~g.~\cite{CN00}. 

In Subsection \ref{SS:CWW-set-up}, we will set up the function spaces for the capillary gravity water waves problem and state the main results in Theorem \ref{T:CWW-LInMa-main}  on the local stable and unstable manifolds followed by some comments on the nonlinear instability. Some analysis on the Dirichlet-Neumann operator $\CG(h)$ to be used in the rest of the paper are recalled or presented in Subsection \ref{SS:CWW-pre}.
In Subsection \ref{SS:CWW-proof} we finish the proof of Theorem \ref{T:CWW-LInMa-main} based on Theorems \ref{T:Ham-UM} and \ref{T:UM-smoothness} and Remark \ref{R:Ham-SM}, where the local well-posedness and the smooth dependence on the initial data also obtained as a byproduct based on Theorems \ref{T:NLPDE-LWP} and \ref{T:LWP-smoothness}. 
Finally  in Subsection \ref{SS:2F}  the fluid interface problem is considered. The analysis will only be outlined as  it is largely parallel to the fluid-vacuum case after an appropriate framework is set up.

\subsection{Set-up and main results} \label{SS:CWW-set-up} 

The choice of the function spaces depends on $d_2$ and $h_0$ in the set-up \eqref{E:fluidD-0} and \eqref{E:fluidD-2} of the water wave problem. Let 
\be \label{E:space-1}
Z_1^r = H^{r+1} \; \text{ if } \; d_2 >0, 
\; \text{ and } \; 
Z_1^r = \dot H^{r+1} \sim \Big\{h \in H^{1+r} \mid \int_{\CU} h dx' =0 \Big\} \; \text{ if  } \; d_2=0,
\ee
and for $R, \ep>0$,  
\be \label{E:space-1.5}
\CO_r(R, \ep) = \{ h \in Z_1^r \mid |h|_{Z_1^r} < R, \, h_0 + \inf h > \ep\}. 
\ee
It is standard to compute the operator $\CG(0)$ of the flat domain 
\be \label{E:CG0}
\CG(0) = \begin{cases} 
|\nabla_{x'}| \tanh (h_0 |\nabla_{x'}|), & \text{ if } \; h_0 \in (0, +\infty), \\
|\nabla_{x'}|, & \text{ if } \; h_0 =+\infty, 
\end{cases} 
\ee
where $|\nabla_{x'}|$ 
can be expressed in terms of Fourier multipliers.
Let 
\[
|f|_{Z_2^r} = |(1-\Delta_{x'})^{\frac r2}\CG(0)^{\frac 12} f|_{L^2}, 
\]
which is positive (unless $d_2=0$ in \eqref{E:fluidD-0} and $f =const$) and induced by the symmetric bilinear form $\langle (1-\Delta_{x'})^{r}\CG(0)f_1,  f_2 \rangle$. 
Let the real Hilbert space $Z_2^r$ be the completion of $C_0^\infty (\CU, \R)$ under $|\cdot|_{Z_2^r}$  
and 
\be \label{E:space-2} 
X^r = Z_1^{\frac 32r} \times Z_2^{\frac 32r}, \quad 
X \triangleq X^0= Z_1 \times Z_2, \quad Z_1= Z_2^0, \quad Z_2 = Z_2^0.   
\ee

\begin{remark} \label{R:spaces}
As usual the completion is defined as the space of the equivalence classes of Cauchy sequences under the metrics $|\cdot|_{Z_2^r}$. In the case of $d_2 \le 2$ in \eqref{E:fluidD-0}, there exist such sequences in $C_0^\infty$ whose pointwise limit is $1$ everywhere but their $|\cdot|_{Z_2^r}$ norms converge to $0$. Hence $f= const$ is equivalent to $f=0$ in $Z_2^r$ if $d_2 \le 2$. 
Since the right sides of \eqref{E:Zakharov-tf} do not change when adjusting $\Phi$ by a constant, with sufficient regularity the right sides of \eqref{E:Zakharov-tf} defines a mapping from $X^r$ to $X^{r-1}$.
It is also easy to see 
\be \label{E:space-3}
\text{ if } \;   h_0<\infty:  \Phi \in Z_2^r \; \text{ iff } \; \nabla_{x'} \Phi \in H^{r-\frac 12}; \quad \text{if } \;   h_0=\infty:  \Phi \in Z_2^r \; \text{ iff } \; |\nabla_{x'}|^{\frac 12} \Phi \in H^{r}. 
\ee
\end{remark} 

For $h \in \CO_{r_0} (R, \ep)$ and $\Phi \in Z_2^{r_0}$, let  
\be \label{E:CA-CL}
\CL(h, \Phi) = \BD^2 \BH (h, \Phi), \quad \CA(h, \Phi) = J \CL(h, \Phi). 
\ee
Let $m^- (\CL(h, \Phi))$ be the Morse index of $\CL(h, \Phi)$, i.~e., 
\be \label{E:Morse-I} \begin{split}
m^- (\CL(h, \Phi)) = \max \{ \dim Y \mid & \text{subspace }  Y \subset X,  \text{ s.~t. }  \\
& \langle \CL(h, \Phi) (\eta, \Psi), (\eta, \Psi) \rangle < 0, \ \forall (\eta, \Psi) \in Y \setminus \{0\} \}.   
\end{split} \ee
The following proposition ensures the exponential trichotomy of the linear flow $e^{t\CA(h, \Phi)}$. 

\begin{proposition} \label{P:CWW-ET-1} 
Suppose $\frac 32 n > \frac {d+1}2$, $(h, \Phi) \in X^n$, and $\inf h + h_0>0$. 
In addition, assume either $d_2=0$ or \eqref{E:coercivity-5} is satisfied at $(h, \Vc)$, then the strongly $C^0$ group 
$e^{t \CA(h, \Phi)} \in \BBL(X^r)$ is well-defined for all $t \in \R$ and $r \in [0, n]$. Moreover, there exist 
unique closed subspaces $X_\pm, X_0 \subset X$ invariant under $e^{t \CA(h, \Phi)}$ such that 
\[
X = X_+ \oplus X_0\oplus X_-, \quad X_\pm \subset X^n, 
\quad 0< \dim X_+ = \dim X_- \le m^- (\CL(h, \Phi)) < \infty,
\]
\[
X_\pm = span \{ (\eta, \Psi) \in X \mid \exists j \in \N, \, \lambda \in \C \text{ such that } \pm\Re \lambda >0, \, (\lambda - \CA(h, \Phi))^j  (\eta, \Psi) =0\},
\]
\[
\sigma (\CA(h, \Phi)|_{X_0}) \subset i\R, \;\ \sup_{t\in \R} (1 + |t|^{k_0})^{-1} |e^{t\CA(h, \Phi)}|_{\BBL(X^r)} < \infty,  \;\ k_0= 1+ 2 \big(m^-(\CL(h, \Phi)) - \dim X_+\big).
\]
\end{proposition} 

The proof of the proposition given in Subsection \ref{SS:CWW-proof} is based on the analysis of the linearizations $\CA$ and $\CL$ 
and Proposition \ref{P:LHam}. The latter is due to Theorems 2.1 and 2.2 in \cite{LZ22} which
require  $m^-(\CL) < \infty$. 
In the case of $d_2 =0$ where $\CU$ is compact, $\CL$ has only discrete eigenvalues and it is easy to prove $m^-(\CL) < \infty$ by a compactness argument. If $d_2 >0$, then $\CL$  has continuous spectra. Condition \eqref{E:coercivity-5}, which implies that the principle part of $\CL$ is uniformly positive on $X$, ensures $m^-(\CL) < \infty$. 

Let 
$(h_*, \Phi_*)$ be an  
equilibrium of \eqref{E:Zakharov-tf} 
with the background velocity $\Vc_*$.  Denote  
\be \label{E:CA-CL-*}
\CA_*\triangleq  J \CL_*, \quad \CL_*=\BD^2 \BH (h_*, \Phi_*),
\ee
Apparently their corresponding $X_\pm$ and $X_0$ are the unstable, stable, and center subspaces of the linearized water wave problem at $(h_*, \Phi_*)$. 
Clearly $X_\pm = \{0\}$ iff $(h_*, \Phi_*)$ is spectrally stable. In the unstable case, let $\Pi_{\pm, 0}$ denote the projections associated to the above decomposition and we often write 
\be \label{E:proj}
u= u_+ + u_0 + u_-, \quad u_{\pm, 0} = \Pi_{\pm, 0} u. 
\ee 
The  following is the main theorem of this section based on Theorems \ref{T:Ham-UM}, \ref{T:NLPDE-UM}, \ref{T:UM-smoothness} and Remarks \ref{R:SM}, \ref{R:smoothness-SM}, and \ref{R:Ham-SM}. 

\begin{theorem} \label{T:CWW-LInMa-main}
Suppose $(h_*, \Phi_*)$ is spectrally unstable on $X$, namely, $\sigma (\CA_*) \not \subset i\R$, and either $d_2=0$ or \eqref{E:coercivity-5} is satisfied at $(h_*, \Phi_*)$. For any  $n_0 \in \N$ and $\lambda_0 \in \R$ satisfying
\be \label{E:para-1}
3(n_0-1)/2 > (d+6)/2; \quad  
\lambda_0 \in \big(0, \lambda_+), \; \lambda_+=\min \{\Re \lambda \mid \lambda \in \sigma (\CA_*|_{X_+})\} >0, 
\ee
there exist  $\delta, M^* >0$ and  $q^\pm: X_\pm (\delta) \to (X_\mp \oplus X_0) \cap X^{n_0}$ such that  the following hold with $W^+ \triangleq graph(q^+)$.
\begin{enumerate}
\item $q^\pm \in C^\infty (X_\pm (\delta), X^{n})$ for any $n \ge 0$. Moreover, $q^\pm(0)=0$ and $\BD q^\pm (0)=0$. 
\item For any  $u_{0+} \in X_+ (\delta)$, there exists a solution $(h_*, \Phi_*) + u(t)$ to \eqref{E:Zakharov-tf} with   
\[
u(t) \in C^0 \big((-\infty, 0], X^{n_0} \big) \cap C^1 \big((-\infty, 0], X^{n_0-1}\big),  
\]
unique in the category  
\be \label{E:temp-91}
u_+ (0) = u_{0+}, \quad |u(t)|_{X^{n_0}} \le 2 M^* \delta e^{\lambda_0 t}, \; \forall t \le 0.
\ee
Moreover $u(t)$ satisfies 
\be \label{E:UM-3-1}
u(0) \in W^+, 
\quad |u(t)|_{X^{n_0}} \le M^* |u_{0+}|_{X_+}e^{\lambda_0 t}, \; \forall t \le 0,
\ee
\be \label{E:UM-3.1-1} 
\sup_{t\le 0}  |u(t)|_{X^n} e^{- \lambda t} < \infty, \; \text{ if } \ n,  \lambda \ \text{ also satisfy \eqref{E:para-1}}. 
\ee
\item The above $u(t)$ defined by $u_{0+}$ satisfies that, there exists $T>0$ such that, for any $t\in (-\infty, T)$ satisfying $u_+(t) \in X_+(\delta)$, it holds $u(t) \in W^+$. 
\item Suppose $n$ and $\lambda>0$ satisfy \eqref{E:para-1} and $(h_*, \Phi_*) + u(t)$, $t\le 0$, is a solution to \eqref{E:Zakharov-tf} such that $\sup_{t\le 0} |u(t)|_{X^{n}} e^{- \lambda t} <\infty$, then there exists $t_0\le 0$ such that $u(t) \in W^+$
for all $t\le t_0$.  
\end{enumerate}
Similar properties are satisfied by $q^-$ for $t \in [0, \infty)$. 
\end{theorem}

The graphs $W^{u, s} = (h_*, \Phi_*) +  W^\pm$ are referred to as the local unstable (u/+) and stable (s/-) manifolds of $(h_*, \Phi_*)$. Statement (3) indicates the local invariance of $W^{u, s}$ under \eqref{E:Zakharov-tf}. Inequalities \eqref{E:UM-3-1} and \eqref{E:UM-3.1-1} give the exponential decay of solutions in $W^{u,s}$. Qualitatively, $W^{u,s}$ are the sets of small solutions which decay to $0$ at some exponential rate $\pm \lambda \in (0, \lambda_+)$. Here the parameters $n_0$ and $\lambda_0$ help to fix $\delta$ and $M^*$, but they are not essential since a.) statement (1) yields $W^{u, s} \subset X^n$ for any $n \ge 0$, as well as their tangency to $X_\pm$ at $(h_*, \Phi_*)$; and b.) along with the local well-posedness (Theorem \ref{T:CWW-LWP-1}) and the continuation of higher regularity (Proposition \ref{P:continuation}), Theorem \ref{T:CWW-LInMa-main} implies $W^{u,s}$ are independent of $n_0$ and $\lambda_0$. 

One notices that some statements ($q^\pm \in C^\infty (X_\pm (\delta), X^n)$, $\forall n$, {\it etc.}) in Theorem \ref{T:CWW-LInMa-main} are stronger than those in Theorem \ref{T:NLPDE-UM}. They are due to $\CH \in C^\infty(X^n, \R)$ for any $n$. Hence the same statements in Theorem \ref{T:CWW-LInMa-main} also hold for the Hamiltonian PDEs discussed in Subsection \ref{SS:examples-Ham} provided those energy density $E$ are $C^\infty$ in $u$ and $\nabla u$. 

A corollary of Theorem \ref{T:CWW-LInMa-main} is that the spectral instability of traveling water waves with surface tension implies the nonlinear instability, while it has much stronger implications. 

$\bullet$ {\it Unstable manifolds of transversally unstable solitary capillary gravity water waves.}
When $\CU=R$ and $g, h_0 \in (0, \infty)$, small amplitude 1-dim solitary traveling waves were obtained in \cite{AK89}. 
It is proved to be conditionally orbitally stable based an energy method \cite{Mi02}. When the solitary wave is viewed as a traveling wave over $\CU=\R \times (\R / l \Z)$ where $\Vc = (c, 0)$ with $c\ne 0$, it could be spectrally and nonlinearly unstable subject to perturbations of certain wave length $l$ in the transversal direction \cite{RT11}. See also \cite{Br01, GHS01, PS04} {\it etc.} for other results on transversal instability of traveling solitary water waves. 
A direct corollary of Theorem \ref{T:CWW-LInMa-main} is the existence of the (possibly multi-dimensional) unstable manifolds $W^+$ of those transversally unstable traveling solitary capillary gravity  water waves such as those in \cite{RT11} ($\dim W^+=1$ in that case). See more discussions at the end of Subsection \ref{SS:CWW-proof}. 

$\bullet$ {\it Unstable manifolds of unstable Stokes waves with surface tension.} 
In the case of $d_2=0$ in \eqref{E:fluidD-0}, by Theorem \ref{T:CWW-LInMa-main}, the spectral instability of spatially periodic traveling waves (Stokes waves)  yields the existence of unstable manifolds and thus the nonlinear instability. In the literature, the spectral instability of periodic traveling water waves has mostly been studied detailedly and obtained for gravity waves under long wave perturbations often referred to as Benjamin-Feir or modulational instability, \cite{BM95, BMV22, NS23, HY23} {\it etc.} See also \cite{AN13, DT14} for numerical studies in the case of periodic capillary gravity water waves.

Finally, similar results on the local stable and unstable manifolds of  the fluid interface problem will be given in Subsection \ref{SS:2F}.

\subsection{Preliminary analysis} \label{SS:CWW-pre} 

In this section we 
give some technical results of the Dirichlet-Neumann operator $\CG(h)$ to be used in the rest of the paper. Firstly 
\be \label{E:CG-sym}
\int_{\CU } \Phi_1 \CG(h) \Phi_2 dx' = \int_{\Omega_h} \nabla \phi_1 \cdot \nabla \phi_2 dx, 
\ee
where 
$\phi_1$ and $\phi_2$ 
are defined as in \eqref{E:Harmonic} and \eqref{E:HarTrace-1}. 
Clearly $\CG$ is translation invariant, namely 
\be \label{E:TransInv}
\CG(h(\cdot + x_0')) \Phi(\cdot + x_0') = (\CG(h) \Phi) (\cdot + x_0'), \quad \forall x_0' \in \R^{d-1}. 
\ee
Differentiating in $x_0'$ yields the commutator formula,
for any $j=1, \ldots, d-1$, 
\be \label{E:CG-commutator}
[\p_{x_j}, \CG(h)] \Phi= \p_{x_j} (\CG(h) \Phi) - \CG(h) \Phi_{x_j} = (\BD \CG(h) h_{x_j}) \Phi. 
\ee
There are various ways to express the shape derivative $\BD \CG(h)$ with respect to $h$, see, e.~g.~\cite{La13, SZ08a}. From those formulas, one may verify 
\be \label{E:DCG-1} \begin{split}
&\langle  \BD  \CG(h)  (\eta) \Phi_1,  \Phi_2 \rangle \\
= & \int_\CU
\eta  \Big( \nabla_{x'} \Phi_1 \cdot \nabla_{x'} \Phi_2 - \frac {(\CG(h) \Phi_1 + \nabla_{x'} h \cdot \nabla_{x'} \Phi_1) (\CG(h) \Phi_2 + \nabla_{x'} h \cdot \nabla_{x'} \Phi_2)} {1+ |\nabla_{x'} h|^2} \Big) dx'. 
\end{split} \ee
One notices that this formula tolerates the possible singularity of $\Phi_{1,2} \in Z_2^r$ in lower wave numbers when $d_2>0$. 
 
\begin{remark} \label{R:Z-tf-2}
From \eqref{E:DCG-1}, we can rewrite \eqref{E:Z-tf-2} as\footnote{Formally $\langle \Phi, \BD\CG(h) (\cdot) \Phi \rangle \in Z_1^*$,
i.~e.~a linear functional on $Z_1$,
so $\langle \Phi, \BD\CG(h) (\cdot) \Phi \rangle^* \in Z_1$ denotes the function obtained  via the $L^2$ dualilty.}  
\be \label{E:Z-tf-2a}
\p_t \Phi = \Vc \cdot \nabla \Phi -\frac 12\langle \Phi, \BD\CG(h) (\cdot) \Phi \rangle^* - g h + \sigma \nabla_{x'} \cdot \Big( \frac {\nabla_{x'}  h }{\sqrt{1+ |\nabla_{x'}  h |^2}} \Big),
\ee
which leads to the Hamiltonian formulation \eqref{E:HamF}. 
\end{remark} 


\begin{proposition} \label{P:DN-0} 
Assume $r_0 > \frac {d-1}2$, then the follow properties  $\CG$ hold for any $R, \ep>0$. 
\begin{enumerate} 
\item For any 
$r \in [0, r_0+\frac 12]$, $\CG: \CO_{r_0} (R, \ep) \to \BBL(Z_2^r, Z_1^{r-\frac 32} )$ is analytic in $h$. Moreover, its $m$-th order derivative also satisfies, for any $r' \in [0, r_0]$,  
\[
\BD^m \CG \in C^0 \big(\CO_{r_0} (R, \ep),  \BBL\big((\otimes_{j=1}^m Z_1^{r_0}) \otimes Z_2^{r}, Z_1^{r-\frac 32} \big) \cap\BBL\big( (Z_1^{r'-\frac 12} \otimes (\otimes_{j=2}^m Z_1^{r_0})) \otimes Z_2^{r_0+\frac 12}, Z_1^{r'-\frac 32} \big)\big).
\]
\item For any 
$h \in \CO_{r_0} (R, \ep)$, $\CG(h) \in \BBL(Z_2, Z_2^*)$ is a uniformly positive isomorphism. 
\end{enumerate} 
Moreover, there exists $C>0$ depending only on $r_0$, $r$, $r'$, $m$, $R$, and $\ep$, such that the norms of the above operators are bounded by $C$ for all $h \in \CO_{r_0} (R, \ep)$.  
\end{proposition}


\begin{remark} 
Since $\BD^m \CG(h) (h_1, \ldots, h_m) \Phi$ is symmetric in $h_1, \ldots, h_m$ in statement (2), any $h_j$, $1\le j \le m$, can be the $Z_1^{r'-\frac 12}$ component.
By duality, for $h \in \CO_{r_0} (R, \ep)$, $\CG(h)$ can be extended to spaces with negative indicies. 
\end{remark}

In the case of $\CU =\R^{d-1}$, this proposition is proved in Chapter 3 and Appendix A of \cite{La13}. In particular, 
statement (1) follows from Theorem 3.21, Proposition 3.28, and Proposition 3.51 in \cite{La13}. 
The boundedness of $\CG(h)$ and $\CG(h)^{-1}$ in statement (2) are proved in Proposition 3.12, as well as the argument in Subsection 3.7.3 in \cite{La13}. 
For the general case of  the domain $\CU$ given in \eqref{E:fluidD-0} and \eqref{E:fluidD-2}, the same arguments (through a coordinate change of the maximal regularity to flatten the domain $\Omega_h$) can be easily adapted. See also \cite{CN00, SZ08a}. 

In oder to apply Theorems \ref{T:NLPDE-LWP} and \ref{T:Ham-UM}, where $\BD^2 \CH(h, \Phi) \in C^1 (X^{n-1}, \BBL(X, X^*))$ is assumed and $Z_2^*$ is more restrictive than $H^{-\frac 12}$, to obtain the well-posedness and local invariant manifolds of  \eqref{E:Zakharov-tf}, we have to obtain the smoothness of $\CG(h) \in \BBL(Z_2, Z_2^*)$ in $h \in \CO_{r_0} (R, \ep)$ and some related estimates. These will be done in the rest of this subsection.  

\begin{lemma} \label{L:DN-1}
For any $r_0 > \frac {d-1}2$, $r \in [0, r_0+\frac 12]$, and $R, \ep>0$, $(1-\Delta_{x'})^{\frac r2} \CG: \CO_{r_0} (R, \ep) \to  \BBL(Z_2^r, Z_2^*)$ is analytic and for any $m\in \N$,
\begin{align*}
& (1-\Delta_{x'})^{\frac r2} \BD^m \CG \in C^0 \big(\CO_{r_0} (R, \ep), \BBL\big((\otimes_{j=1}^m Z_1^{r_0}) \otimes Z_2^r, Z_2^* \big)\big), \\
& (1-\Delta_{x'})^{\frac {r'}2} \BD^m \CG \in C^0 \big(\CO_{r_0} (R, \ep), \BBL\big( (Z_1^{r'-\frac 12} \otimes (\otimes_{j=2}^m Z_1^{r_0})) \otimes Z_2^{r_0+\frac 12}, Z_2^* \big) \big), \quad \forall r' \in [0, r_0], 
\end{align*}
and their norms are bounded by some $C>0$ determined by $m, r_0, r, r', R$, and $\ep$. 
\end{lemma}

Compared to Proposition \ref{P:DN-0}, it is a slight improvement in near zero Fourier modes. 

\begin{proof}
Denote the cut-off operators in Fourier modes 
\be \label{E:temp-92} 
\CP_{l} \Phi = \chi_{[0, 1]} (|\nabla_{x'}|) \Phi, 
\quad \CP_{h} = I - \CP_{l}, 
\ee
where $\chi_{[0, 1]} (\tau)$ is the characteristic function of the interval $[0, 1]$. Clearly they satisfy 
\be \label{E:temp-93}
\CP_{l} \in \BBL(Z_2^{s'}, Z_2^s), \; \forall s', s \in \R,  \quad \CP_{h}  \in \BBL(Z_2, H^{\frac 12}) \subset \BBL(Z_2, Z_2), \quad \nabla_{x'} \CP_{l} = \CP_{l} \nabla_{x'}. 
\ee
From Proposition \ref{P:DN-0}(1), $\CP_h^* (1-\Delta_{x'})^{\frac r2}  \CG:  \CO_{r_0} (R, \ep) \to \BBL(Z_2^r, Z_2^*)$ is analytic and satisfies the desired estimates. 
It remains to handle $\CP_l^* (1-\Delta_{x'})^{\frac r2} \CG$.  

Given $h \in \CO_{r_0} (R, \ep)$, $\eta \in Z_1^{r_0}$, and $\Phi_1\in Z_2^r$, $\Phi_2 \in Z_2$, let $\Phi_{2, l} = \CP_l \Phi_2 \in Z_2^s$ for any $s \in \R$. By \eqref{E:DCG-1}, we have 
\begin{align*}
\langle (1- & \Delta_{x'})^{\frac r2}  \BD  \CG(h)  (\eta) \Phi_1,  \Phi_{2, l} \rangle = \int_\CU
\eta  \Big( \nabla_{x'} \Phi_1 \cdot \nabla_{x'} (1-\Delta_{x'})^{\frac r2} \Phi_{2, l} \\
& - \frac {(\CG(h) \Phi_1 + \nabla_{x'} h \cdot \nabla_{x'} \Phi_1) } {1+ |\nabla_{x'} h|^2}  \big( \CG(h)(1-\Delta_{x'})^{\frac r2} \Phi_{2, l} + \nabla_{x'} h \cdot  (1-\Delta_{x'})^{\frac r2} \nabla_{x'}\Phi_{2,l}\big) \Big) dx'. 
\end{align*}
According to Proposition \ref{P:DN-0}(1), \eqref{E:space-3}, and \eqref{E:temp-93}, it holds that $\nabla_{x'} \Phi_1, \CG(h) \Phi_1 \in H^{r-\frac 12}$, $\nabla_{x'} \Phi_{2,l}, \CG(h) (1-\Delta_{x'})^{\frac r2} \Phi_{2, l} \in H^{s}$ for any $s \in \R$, where $\CG(h) \Phi_1$ and $\CG(h)(1-\Delta_{x'})^{\frac r2} \Phi_{2, l}$ are analytic in $h$. Therefore, we obtain the analyticity of
\[
\CP_l^* (1- \Delta_{x'})^{\frac r2}  \BD  \CG(\cdot)  : \CO_{r_0} (R, \ep) \to \BBL\big( Z_1^{r_0} \otimes Z_2^r, Z_2^* \big)
\]
which along with the above analyticity of $\CP_h^* (1-\Delta_{x'})^{\frac r2}  \CG$ yields the analyticity of $(1-\Delta_{x'})^{\frac r2}  \CG$. 

The  estimates on 
\[
(1-\Delta_{x'})^{\frac {r'}2} \BD^m \CG \in \BBL\big( (Z_1^{r'-\frac 12} \otimes (\otimes_{j=2}^m Z_1^{r_0})) \otimes Z_2^{r_0+\frac 12}, Z_2^* \big) \big) 
\]
can be obtained similarly. 
\end{proof}

Instead of the commutator estimate in Proposition 3.32 in \cite{La13}, which emphasizes the dependence on the shallow water scaling, we shall use the following estimate which may not be optimal, but is a direct corollary of the above lemma and \eqref{E:CG-commutator}. 

\begin{corollary} \label{C:DN-1}
For any $r_0 > \frac {d+1}2$, and $R, \ep>0$, the following hold. 
\begin{enumerate} 
\item For any $r \in [0, r_0-\frac 12]$, $h \in \CO_{r_0} (R, \ep)$, $n \in [0, r+1] \cap \N$, and $1\le \alpha_1, \ldots, \alpha_n \le d-1$,  
\[
|(1-\Delta)^{\frac {r-n+1}2} [\p_{x_{\alpha_1}} \ldots \p_{x_{\alpha_n}}, \CG(h)]|_{\BBL(Z_2^r, Z_2^*)} \le C.
\]
\item For any $h \in \CO_{r_0}(R, \ep)$, $n \in [0, r_0+1] \cap \N$, and $1\le \alpha_1, \ldots, \alpha_n \le d-1$, 
\[
|(1-\Delta)^{\frac {r_0-n}2 + \frac 14} [\p_{x_{\alpha_1}} \ldots \p_{x_{\alpha_n}}, \CG(h)]|_{\BBL(Z_2^{r_0}, Z_2^*)} \le C.
\]
\end{enumerate}
Here 
$C>0$ is determined by $r_0, r, n, R$, and $\ep$.
\end{corollary} 

\begin{proof} 
For $n=1$, the corollary follows directly from \eqref{E:CG-commutator} and the two types of estimates on $\BD \CG$ in Lemma \ref{L:DN-1}. For $n>1$, using the estimate of the case $n=1$ and 
\[
[\p^n, \CG] = \p^{n-1} [\p, \CG] + \p^{n-2} [\p, \CG] \p + \ldots +  [\p, \CG] \p^{n-1}.
\]
we obtain the desired bounds. 
\end {proof}

The following lemma is a (very rough) refinement of Lemma \ref{L:DN-1} in high regularity cases. 

\begin{lemma} \label{L:DN-3} 
Let $m \in \N \cup \{0\}$, $r_0 > \frac {d-1}2$, $k \in \N$, and $r, R, \ep>0$ satisfy $k+r_0 \ge r\ge k$,  
then  there exists $C>0$ such that for any $h \in \CO_{r_0}(R, \ep) \cap Z_1^{r-\frac 12}$, $\Phi \in Z_2^r$, and $\eta_j \in Z_1^{r-\frac 12}$, $j=1, \ldots, m$, it holds 
\begin{align*}
|(1-\Delta)^{\frac r2} &\BD^m \CG(h) (\eta_1, \ldots, \eta_m) \Phi|_{Z_2^*} \le C \big(1+ |h|_{Z_1^{r_0 + \frac {k}2}}^{k-1}\big)  \Big(|\Phi|_{Z_2^{r}}  \prod_{l=1, \ldots, m} |\eta_l|_{Z_1^{r_0+ \frac k2}} \\
& + |h|_{Z_1^{r-\frac 12}}  |\Phi|_{Z_2^{r_0 + \frac {k+1}2}}  \prod_{l=1, \ldots, m} |\eta_l|_{Z_1^{r_0+ \frac k2}} + \sum_{j=1}^m |\Phi|_{Z_2^{r_0 + \frac {k+1}2}} |\eta_j|_{Z_1^{r-\frac 12}} \prod_{\substack{l=1, \ldots, m \\  l\ne j}} |\eta_l|_{Z_1^{r_0+ \frac k2}} \Big).
\end{align*}
\end{lemma} 

This inequality is useful when $r \ge r_0 + \frac {k+1}2$ and thus the highest order norms appear linearly on the right side. 

\begin{proof} 
The translation invariance \eqref{E:TransInv} also holds for $\BD^m \CG$ 
\[
\BD^m \CG\big(h(\cdot + x_0')\big) \big(\eta_1(\cdot + x_0'), \ldots, \eta_m(\cdot + x_0')\big) \Phi(\cdot + x_0') = \big(\CG(h) (\eta_1, \ldots, \eta_m) \Phi\big) (\cdot + x_0'). 
\]
Differentiating it repeatedly we obtain
\begin{align*}
&\p^k  \BD^m \CG(h)  (\eta_1, \ldots, \eta_m) \Phi =
\sum_{K \in \Lambda} a_K \BD^{m+l} \CG(h) \big(\p^{k_1} \eta_1, \ldots, \p^{k_m} \eta_m, \p^{k_{m+1}} h, \ldots, \p^{k_{m +l}} h \big) \p^{\tilde k} \Phi,
\end{align*}
where $a_K$'s are some positive integers and 
\[
\Lambda = \{ K= (k_1, \ldots, k_{m+l}, \tilde k) \mid k_1, \ldots, k_m, \tilde k\ge 0, \, k_{m+1}, \ldots, k_{m+l} \ge 1, \, k_1 + \ldots +k_{m+l} + \tilde k=k\}. 
\]
In each $K$, let $k_*$ denotes the maximal index and thus all others are at most $\frac k2$. 
Hence, according to Lemma \ref{L:DN-1}, we can obtain the desired estimate by bounding the term 
\[
|(1-\Delta)^{\frac {r-k}2} \BD^{m+l} \CG(h) \big(\p^{k_1} \eta_1, \ldots, \p^{k_m} \eta_m, \p^{k_{m+1}} h, \ldots, \p^{k_{m +l}} h \big) \p^{\tilde k} \Phi|_{Z_2^*}
\] 
by one of the following, depending on whether $\p^{k_*}$ happens to be applied to $\Phi$, $h$, or $\eta_j$,  
\[
C\big(1+ |h|_{Z_1^{r_0 + \frac {k}2}}^{k-1}\big)|\Phi|_{Z_2^{r}}  \prod_{l=1, \ldots, m} |\eta_l|_{Z_1^{r_0+ \frac k2}}, \quad C\big(1+ |h|_{Z_1^{r_0 + \frac {k}2}}^{k-1}\big) |h|_{Z_1^{r-\frac 12}} |\Phi|_{Z_2^{r_0 + \frac {k+1}2}}  \prod_{l=1, \ldots, m} |\eta_l|_{Z_1^{r_0+ \frac k2}}, 
\]
\[ 
C\big(1+ |h|_{Z_1^{r_0 + \frac {k}2}}^{k-1}\big)|\Phi|_{Z_2^{r_0 + \frac {k+1}2}} |\eta_j|_{Z_1^{r-\frac 12}} \prod_{\substack{l=1, \ldots, m \\  l\ne j}} |\eta_l|_{Z_1^{r_0+ \frac k2}}.
\]
Here the assumption $r-k \le r_0$ allows us to treat the term with $\p^{k_*}$ as the least regular term 
when applying 
Lemma \ref{L:DN-1}. 
\end{proof} 

The following proposition indicates that $\CG(0)$ can be replaced by $\CG(h)$ in defining $Z_2^r$. 

\begin{proposition} \label{P:DN-2} 
For any $r_0 > \frac {d+1}2$, $r \in [0, r_0+\frac 12]$, $R, \ep>0$, and $h \in \CO_{r_0}(R, \ep)$,  $(1-\Delta_{x'})^{\frac r2} \CG(h) \in  \BBL(Z_2^r, Z_2^*)$ is isomorphic and 
\be \label{E:temp-94}
\big|\CG(h)^{-1} (1-\Delta_{x'})^{-\frac r2} \big|_{\BBL(Z_2^*, Z_2^r)} \le C,
\ee
for $C>0$ determined by $r_0, r, R$, and $\ep$.
\end{proposition}

\begin{proof}
According to Proposition \ref{P:DN-0}(2),  $\CG(h) \in \BBL(Z_2, Z_2^*)$ is isomorphic which along with Lemma \ref{L:DN-1} implies $(1-\Delta_{x'})^{\frac r2} \CG(h) \in  \BBL(Z_2^r, Z_2^*)$ is injective for any $r \in [0, r_0+\frac 12]$. Moreover, for any $f \in Z_2^*$, there exists a unique $\Phi \in Z_2$ such that 
\be \label{E:temp-95} 
\CG(h) \Phi = (1-\Delta_{x'})^{-\frac r2} f \in (1-\Delta_{x'})^{-\frac r2} Z_2^* \subset Z_2^*. 
\ee
First let $r=1$. From Corollary \ref{C:DN-1} we have, for any $j =1, \ldots, d-1$, 
\[
|\CG(h) \Phi_{x_j}|_{Z_2^*} \le |\p_{x_j} (1-\Delta_{x'})^{-\frac 12} f |_{Z_2^*} + |[\CG(h), \p_{x_j}] \Phi|_{Z_2^*} \le |f|_{Z_2^*} + C |\Phi|_{Z_2} \le C |f|_{Z_2^*}, 
\]
which implies $|\Phi_{x_j}|_{Z_2} \le C |f|_{Z_2^*}$. Hence $\CG(h)^{-1} (1-\Delta_{x'})^{-\frac 12} \in \BBL(Z_2^*, Z_2^1)$ and we can write
\[
\CG(h)^{-1} \in \BBL\big( (1-\Delta_{x'})^{-\frac 12} Z_2^*, Z_2^1\big) \cap \BBL\big( Z_2^*, Z_2^0\big) . 
\]
Inequality \eqref{E:temp-94} for $r \in [0, 1]$ follows from interpolation. For $r\in [1,2]$, 
Corollary \ref{C:DN-1} and 
\[
(1-\Delta_{x'})^{\frac {r-1}2} \CG(h) \Phi_{x_j} = \p_{x_j} (1-\Delta_{x'})^{-\frac 12} f + (1-\Delta_{x'})^{\frac {r-1}2} [\CG(h), \p_{x_j}] \Phi, 
\]
along with $r-1\in [0, 1]$, imply 
\[
|\Phi_{x_j}|_{Z_2^{r-1}} \le C |(1-\Delta_{x'})^{\frac {r-1}2} \CG(h) \Phi_{x_j} |_{Z_2^*} \le C |f|_{Z_2^*} + C |\Phi|_{Z_2^{r-1}} \le C |f|_{Z_2^*} + C |(1-\Delta_{x'})^{-\frac 12} f|_{Z_2^*}  
\]
and we obtain \eqref{E:temp-94} for $r \in [0, 2]$. Finally \eqref{E:temp-94} for $r\in [0, r_0+\frac 12]$ follows inductively. 
\end{proof}

\subsection{Proof of Theorem \ref{T:CWW-LInMa-main}} \label{SS:CWW-proof} 

The proof is based on Theorem \ref{T:Ham-UM}. As a byproduct, 
we also demonstrate how Theorems \ref{T:NLPDE-LWP} and \ref{T:LWP-smoothness} apply to yield the local well-posedness of \eqref{E:Zakharov-tf} and the smoothness in the initial data. The main steps are the analysis of $\BD^2 \BH$ and $J \BD^2\BH$. 

Recall the space $X^r$ defined in \eqref{E:space-2} and the operators $\CL =\BD^2\BH$ and $\CA = J\CL$ in \eqref{E:CA-CL} and \eqref{E:HamF}. 
Dropping the terms involving $\BD \CG$ and $\BD^2 \CG$, the principle part of $\CL(h, \Phi)$ is  
\be \label{E:CL0-1} \begin{split}
\big\langle \CL_0 (h) & (\eta_1, \Psi_1), (\eta_2, \Psi_2) \big \rangle = \int_\CU g\eta_1\eta_2 + \frac \sigma{\sqrt{1+ |\nabla_{x'}  h |^2}} \Big( \nabla_{x'} \eta_1 \cdot \nabla_{x'} \eta_2 \\
& - \frac {(\nabla_{x'} h \cdot \nabla_{x'} \eta_1)(\nabla_{x'} h \cdot \nabla_{x'} \eta_2) }{1+ |\nabla_{x'} h|^2}  \Big)  + \Psi_1 \CG(h) \Psi_2 +  \Psi_1\Vc \cdot \nabla\eta_2  + \Psi_2\Vc \cdot \nabla\eta_1  dx'. 
\end{split} \ee
By completing the square in the last three terms as in \cite{Mi02}, it can be rewritten as 
\be \label{E:CL0-1.4}
\big\langle \CL_0 (h) (\eta_1, \Psi_1), (\eta_2, \Psi_2) \big \rangle = \langle \CL_{01} (h, h) \eta_1, \eta_2 \rangle + \big\langle \CL_{02} (h) (\eta_1, \Psi_1), (\eta_2, \Psi_2) \big \rangle,
\ee
where 
\be \label{E:CL0-1.5} \begin{split}
\langle \CL_{01} (h_1, h_2) & \eta_1, \eta_2 \rangle = \int_\CU - (\Vc \cdot \nabla_{x'} \eta_1) \CG(h_2)^{-1}   (\Vc \cdot \nabla_{x'} \eta_2) + g\eta_1\eta_2 \\
&+ \frac \sigma{\sqrt{1+ |\nabla_{x'}  h_1 |^2}} \Big( \nabla_{x'} \eta_1 \cdot \nabla_{x'} \eta_2  - \frac {(\nabla_{x'} h_1 \cdot \nabla_{x'} \eta_1)(\nabla_{x'} h_1 \cdot \nabla_{x'} \eta_2) }{1+ |\nabla_{x'} h_1|^2}  \Big)  dx',
\end{split} \ee
\be \label{E:CL0-1.6}
\big\langle \CL_{02} (h) (\eta_1, \Psi_1), (\eta_2, \Psi_2) \big \rangle = \int \big( \Psi_1 + \CG(h)^{-1} (\Vc \cdot \nabla_{x'} \eta_1)\big) \CG(h)  \big( \Psi_2 + \CG(h)^{-1} (\Vc \cdot \nabla_{x'} \eta_2)\big) dx'.
\ee
As $\CL_{02}$ is already positive on $Z_2$, 
to use $\CL_0(h)$ as a control quantity we shall assume 
\be \label{E:coercivity-5} \begin{split}
&\exists \delta>0 \  \text{ and a closed subspace }\  \wt Z_+ \subset Z_1 \; \text{ s.~t.~} \ codim \wt Z_+< \infty \ \text{ and } \\ 
&\langle \CL_{01}(h, 0) \eta,  \eta \rangle \ge \delta |\eta|_{Z_1}^2, \; \forall \eta \in \wt Z_+. 
\end{split} \ee

\begin{lemma} \label{L:CWW-CL} 
For any $r_0 > \frac {d}2$, $R, \ep>0$, the following hold for $J$ and $\CL$. 
\begin{enumerate} 
\item $J: X^* \supset Dom(J) \to X $ as defined in \eqref{E:HamF} satisfies $J^*=-J$ and $\CL: \CO_{r_0} (R, \ep) \times Z_2^{r_0} \to \BBL (X, X^*)$ is analytic and the norms depend only on $r_0, \Vc, \sigma, g, R$, and $\ep$.
\item If either $d_2=0$ in \eqref{E:fluidD-0} or $(h, \Vc) \in \CO_{r_0} (R, \ep) \times \R^{d-1} $ satisfies \eqref{E:coercivity-5},  
then for any $\Phi \in Z_2^{r_0}$, there exist $\delta>0$ and a closed subspace $\wt X_+ \subset X$ such that 
\[
codim\wt X_+ < \infty \ \text{ and } \ \langle \CL(h, \Phi) (\eta, \Psi), (\eta, \Psi) \rangle \ge \delta |(\eta, \Psi)|_X^2, \quad \forall (\eta, \Psi) \in \wt X_+. 
\]
\end{enumerate}
\end{lemma}

\begin{proof} 
From the definitions, it is straight forward to verify the density of $Dom(J) = (Z_1^* \cap Z_2) \times (Z_2^* \cap Z_1) \subset X^*= Z_1^* \times Z_2^*$ and $J^*=-J$. 
Lemma \ref{L:DN-1} implies that $\CL_0 (h) \in \BBL(Z_2, Z_2^*)$ is analytic in $h$. 
From \eqref{E:Ham-CG}, one may compute,  
for any $(\eta_1, \Psi_1), (\eta_2, \Psi_2) \in X =Z_1 \times Z_2$, 
\be \label{E:D2H-1} \begin{split}
& \big \langle \CL_1(h, \Phi) (\eta_1,  \Psi_1), (\eta_2, \Psi_2) \big \rangle  \triangleq  \big \langle \big( \CL(h, \Phi) - \CL_0(h)\big) (\eta_1, \Psi_1), (\eta_2, \Psi_2) \big \rangle  \\
= & \int  \frac 12 \Phi \BD^2 \CG(h) (\eta_1, \eta_2) \Phi + \Psi_2 \BD \CG(h) (\eta_1) \Phi  + \Psi_1 \BD \CG(h) (\eta_2) \Phi   dx'. 
\end{split} \ee
The term involving $\BD^2 \CG$ is the main complication and can be calculated using \eqref{E:DCG-1} as 
\be \label{E:D2CG} \begin{split} 
\int \Phi \BD^2 \CG(h) (\eta_1, \eta_2) & \Phi dx' =  2 \int   \frac {\CG(h) \Phi + \nabla_{x'} h \cdot \nabla_{x'} \Phi } {1+ |\nabla_{x'} h|^2} \Big( - \BD\CG(h) (\eta_1) \Phi \\
&- \nabla_{x'} \eta_1 \cdot \nabla_{x'} \Phi 
+ \frac {(\CG(h) \Phi + \nabla_{x'} h \cdot \nabla_{x'} \Phi) \nabla_{x'} h \cdot \nabla_{x'} \eta_1} {1+ |\nabla_{x'} h|^2}  \Big) \eta_2 dx'. 
\end{split} \ee
Using Lemma \ref{L:DN-1} to handle all the terms in $\CL_1$ involving $\CG(h)$ and $\BD \CG(h)$ we obtain the analyticity of $\CL: \CO_{r_0} (R, \ep) \times Z_2^{r_0} \to \BBL(X, X^*)$ and the desired bounds on its derivatives. 

To prove statement (2), for any $(\eta, \Psi) \in X$, from \eqref{E:D2H-1}, \eqref{E:DCG-1}, and \eqref{E:D2CG}, we first rewrite  
\begin{align*}
\big \langle \CL_1(h, \Phi) (\eta,  \Psi), (\eta, \Psi) \big \rangle = \int \eta \Big( \nabla_{x'} \Psi \cdot f_1 + f_2 \CG(h) \Psi  + \nabla_{x'} \eta \cdot f_3 + f_4  (\BD \CG(h) \eta ) \Phi \Big) dx',    
\end{align*}
where $f_j$, $j=1, 2, 3, 4$, are linear or quadratic combinations of $\CG(h) \Phi$ and $\nabla_{x'} \Phi$ with coefficients in the forms of rational functions of $\nabla_{x'} h$. Due to \eqref{E:space-3} and Proposition \ref{P:DN-0}(1), we have i.) $\CG(h) \Phi, \nabla_{x'} \Phi \in H^{r_0-\frac 12}$ and thus $f_j \in H^{r_0-\frac 12}$ which decay as $|x'| \to \infty$; ii.) $\nabla_{x'} \Psi, \CG(h) \Psi \in H^{-\frac 12}$, and $(\BD \CG(h) \eta) \Phi \in L^2$, so they behave well near zero Fourier frequencies; and iii.) aperently $\CL_1$ requires less regularity on $(\eta, \Psi)$ then $X$. Hence from a standard argument, $\CL_1(h, \Phi)  \in\BBL(X, X^*)$ is compact. 

The Cauchy-Schwarz inequality applied to the curvature term yields 
\be \label{E:temp-96} 
\big\langle \CL_0 (h) (\eta, \Psi), (\eta, \Psi) \big \rangle \ge  \int  g\eta^2 +  \frac { \sigma|\nabla_{x'} \eta|^2}{(1+ |\nabla_{x'} h|^2)^{\frac 32}} + \Psi \CG(h) \Psi + 2 \Psi \Vc \cdot \nabla\eta   dx'.
\ee
According to Lemma \ref{L:DN-1} and the definition of $X$, $\CL_0(h, \Phi) \in\BBL(X, X^*)$ is uniformly positive subject to the lower order bounded perturbation of the momentum $2 \langle \eta,  \Vc \cdot \nabla \Psi \rangle$. The latter is a compact perturbation if $d_2=0$. In this case, the whole $\CL(h, \Phi) \in \BBL(X, X^*)$ is a compact perturbation to a uniformly positive symmetry operator and thus statement (2) holds. 

In the case of $d_2>0$, again we split $\CL_{01}$ by considering 
\[
\wt \CL_{01} (h) \triangleq \CL_{01}(h, h) - \CL_{01}(h, 0) = (\Vc \cdot \nabla_{x'}) \big(\CG(h)^{-1} - \CG(0)^{-1}\big) (\Vc \cdot \nabla_{x'}). 
\]
For any $\wt \eta \in Z_1^1$, \eqref{E:DCG-1} implies 
\begin{align*}
& \langle \wt \CL_{01} (h) \wt \eta, \wt \eta\rangle = \int_0^1 \big( \CG(\tau h)^{-1} (\Vc \cdot \nabla_{x'})\wt \eta\big) (\BD\CG(\tau h) h) \big(\CG(\tau h)^{-1} (\Vc \cdot \nabla_{x'})\wt \eta\big) d\tau \\
= & \int_0^1\int h \Big( \big| \nabla_{x'} \CG(\tau h)^{-1} (\Vc \cdot \nabla_{x'})\wt \eta\big|^2  - \frac {\big| (\CG(\tau h) + \tau \nabla_{x'}h \cdot \nabla_{x'})\CG(\tau h)^{-1} (\Vc \cdot \nabla_{x'})\wt \eta\big|^2}{1+ \tau^2|\nabla_{x'} h|^2}\Big) dx' d\tau.   
\end{align*}
From Propositions \ref{P:DN-0} and \ref{P:DN-2}, 
\[
 \nabla_{x'}\CG(\tau h)^{-1} (\Vc \cdot \nabla_{x'}), \  (\CG(\tau h) + \tau \nabla_{x'}h \cdot \nabla_{x'})\CG(\tau h)^{-1} (\Vc \cdot \nabla_{x'}) \in \BBL(Z_1^1, H^1).
\] 
Moreover $h \in Z_1^{r_0} \subset H^{r_0+1}$ and thus it decays as $x' \to \infty$. Therefore $\wt \CL_{01} \in \BBL(Z_1^1, (Z_1^1)^*)$ is also compact. 
For any $s \ge 1$, let 
\[
P_l = \chi_{[0, s]} (|\nabla|), \quad P_h = I- P_l,
\]
be the cut-off operators in the Fourier variables, then we obtain the compactness of $P_l^*  \wt \CL_{01} (h) P_l \in \BBL(Z_1, Z_1^*)$. 
From Proposition \ref{P:DN-2} we have for any $\eta_1, \eta_2 \in Z_1$, 
\begin{align*}
|\langle \wt \CL_{01} (h) \eta_1, P_h \eta_2\rangle| \le & |(1-\Delta)^{-\frac 14} (\Vc \cdot \nabla_{x'}) P_h \eta_2|_{Z_2^*} | (1-\Delta)^{\frac 14} (\CG(h)^{-1} - \CG(0)^{-1}) (\Vc \cdot \nabla_{x'})\eta_1|_{Z_2} \\
\le &C |\Vc| \langle s\rangle^{-1} |\eta_2|_{Z_1} |(1-\Delta)^{\frac 14} (\Vc \cdot \nabla_{x'})\eta_1|_{Z_2^*} \le C |\Vc|^2 \langle s\rangle^{-1} |\eta_1|_{Z_1} |\eta_2|_{Z_1},
\end{align*}
which implies 
\[
|P_h^* \wt \CL_{01} (h) P_h|_{\BBL(Z_1, Z_1^*)} + |P_l^* \wt \CL_{01} (h)P_h|_{\BBL(Z_1, Z_1^*)} + | P_h^*\wt \CL_{01} (h) P_l|_{\BBL(Z_1, Z_1^*)} \le C |\Vc|^2 \langle s\rangle^{-1}. 
\]
By taking $s \gg 1$, $I -P_l^* \wt \CL_{01} (h) P_l$ can be arbitrarily small. Along with the compactness of $P_l^* \wt \CL_{01} (h) P_l$, we obtain the compactness of  $\wt \CL_{01}(h) \in \BBL(Z_1, Z_1^*)$. 
The positivity of $\CL_{02}(h, \Phi)$ follows from the positivity of $\CG(h)$  on $Z_2$ and $\CG(h)^{-1} \nabla_{x'} \in \BBL(Z_1, Z_2^{\frac 12})$ due to Proposition \ref{P:DN-2}. Summarizing the above analysis, $\CL(h, \Phi) \in \BBL(X, X^*)$ is the sum of compact operators and an  operator uniformly positive  except in  finitely many directions, so statement (2) holds.   
\end{proof} 

In the above proof we separated the low and high Fourier modes mainly due to the form of the formula \eqref{E:DCG-1} where $\Phi_{1,2} \in \dot H^1$, instead of $Z_2$, is required. 


\begin{remark} \label{R:Vc-1}
From \eqref{E:temp-96}, $\CL_0 (h, \Phi)$ is uniformly positive on $X$ if $\Vc=0$. Even though $\Vc$ contributes only lower order derivative terms, it may create infinitely many negative directions if 
the domain $\CU$ is unbounded, even at $h=0$. This is why \eqref{E:coercivity-5} is required in the lemma. This assumption does not affect the local well-posedness as one can always let $\Vc=0$ by removing the moving frame. However, it does have an impact on the spectra of the linearizations at at traveling waves  and even the existence of traveling waves (see e.~g.~\cite{AK89, Mi02, RT11}). 
\end{remark}

\begin{remark} \label{R:Vc-2}
Since $\CL_{01}(h_1, h_2) \in \BBL(X, X^*)$ is analytic in $h_1, h_2$, \eqref{E:coercivity-5} is satisfied by $|h|_{Z_1^{r_0}} \ll1 $ if $\CL_{01}(0)$ satisfies \eqref{E:coercivity-5}. As the Fourier modes are decoupled in $\CL_{01}(0)$, its Fourier multipliers can be computed explicitly. A classical sufficient condition for $\CL_{01}(0)$ to satisfy \eqref{E:coercivity-5} 
\be \label{E:CL0-4}
\BBF<1 (\Longleftrightarrow gh_0 > |\Vc|^2) \; \text{ and } \; \BBB \le 3 \BBF^{-2} (\Longleftrightarrow 3\sigma \ge |\Vc|^2 h_0)
\ee
is given in terms of two commonly used dimensionless physical constants (see e.~g.~\cite{AK89, Mi02, RT11}), the Froud and Bond number  
\[
\BBF= |\Vc|/\sqrt{gh_0}, \quad \BBB= gh_0^2/\sigma. 
\]
\end{remark} 

In the following we consider the right side $F(h, \Phi)= J\BD \BH(h, \Phi)$ of system \eqref{E:Zakharov-tf} and 
the domains of the powers of $\CA= \BD F =J \BD^2 \BH$.  

\begin{lemma} \label{L:DomCA-1}
Assume $n \in \N$ and $(h, \Phi) \in X^{n}$ satisfy $\frac 32 n > \frac {d+1}2$ and $h_0 + \inf h >0$, there exists 
then $Dom(\CA(h, \Phi)^r) = X^r$ for $r =1,2, \ldots, n$. Moreover, for any $R, \ep>0$, $m \in \N$, there exists $C>0$ determined by $d, g, \sigma, h_0, \Vc, n, m, R, \ep$ such that  
\be \label{E:CA-1}
|\CA|_{C^m (\CO_{\frac 32n} (R, \ep) \times Z_2^{\frac 32n} (R), \BBL(X^r, X^{r-1}))} \le C, \quad 1\le r \le n. 
\ee
\end{lemma} 

\begin{proof} 
From \eqref{E:Zakharov-tf}, one may compute 
\[
\CA(h, \Phi) (\eta, \Psi) = \big(\CA_1 (h, \Phi) (\eta, \Psi), \CA_2 (h, \Phi) (\eta, \Psi) \big), 
\]
where 
\be \label{E:CA1}
\CA_1 (h, \Phi) (\eta, \Psi) =  \Vc \cdot \nabla_{x'} \eta + \CG(h) \Psi +  ( \BD\CG(h) \eta) \Phi,
\ee
\be \label{E:CA2} \begin{split}
\CA_2 (h, \Phi) (\eta, \Psi) = &  \Vc \cdot \nabla_{x'} \Psi-g \eta -  \sigma \BD \kappa(h) \eta - \nabla_{x'} \Phi \cdot \nabla_{x'} \Psi +  \frac {\CG(h)\Phi+ \nabla_{x'} h \cdot \nabla_{x'} \Phi}{1+ |\nabla_{x'} h|^2} \\
& \times \Big( (\BD\CG(h)\eta) \Phi + \CG(h) \Psi  + \nabla_{x'} h \cdot \nabla_{x'} \Psi + \nabla_{x'} \eta \cdot \nabla_{x'} \Phi \\
&- \frac {\CG(h)\Phi+ \nabla_{x'} h \cdot \nabla_{x'} \Phi}{1+ |\nabla_{x'} h|^2}{\nabla_{x'} h \cdot \nabla_{x'} \eta} \Big),
\end{split} \ee
\[
\BD \kappa (h) \eta =  - \nabla_{x'} \cdot \big( \frac { (1+ |\nabla_{x'} h|^2) \nabla_{x'} \eta - (\nabla_{x'} h \cdot \nabla_{x'} \eta) \nabla_{x'} h}{(1+ |\nabla_{x'} h|^2)^{\frac 32}} \big).
\]

Suppose $n \ge r \ge 1$ and 
\be \label{E:temp-97}
(\eta, \Psi),  \, (\zeta, \Theta) = \CA(h, \Phi) (\eta, \Psi)  \in X^{r-1} = Z_1^{\frac 32 (r - 1)} \times Z_2^{\frac 32 (r -1)},
\ee
From Proposition \ref{P:DN-0}, and \eqref{E:space-1}, \eqref{E:space-2}, \eqref{E:space-3}, and the above expressions, 
\[
\sigma \BD \kappa (h) \eta = -\Theta + \big( \CA_2 (h, \Phi) (\eta, \Psi) + \sigma \BD \kappa (h) \eta \big) \in H^{\frac 32r -2}. 
\]
As shown in the proof of \eqref{E:temp-96}, $\BD \kappa(h)$ is uniformly elliptic. So $\eta \in Z_1^{\frac 32r -1}$ and it implies 
\[
\CG(h) \Psi = \zeta - \Vc \cdot \nabla_{x'} \eta -  ( \BD\CG(h) \eta) \Phi \in Z_2^* \cap H^{\frac 32 r-1} = (1-\Delta_{x'})^{\frac 12 (\frac 12-\frac 32 r) }Z_2^*.
\]
Hence $\Psi \in Z_2^{\frac 32r -\frac 12}$ thanks to Proposition \ref{P:DN-2}.
Repeating this argument and using the improved regularity of $(\eta, \Psi)$,  we obtain
$\eta \in Z_1^{\frac 32r}$ and then $\Psi \in Z_2^{\frac 32r}$.  

The above analysis for $r=1$ implies that $Dom(\CA(h, \Phi)) = X^1$. Inductively we obtain $Dom(\CA(h, \Phi)^r) = X^r$ for $r =1,2, \ldots, n$.

Finally inequality \eqref{E:CA-1} is a direct corollary of Proposition \ref{P:DN-0}. 
\end{proof}

With the above preparations, we are ready to prove Proposition \ref{P:CWW-ET-1} of the {\bf exponential trichotomy} of the linearized water wave system \eqref{E:Zakharov-tf}. 

\begin{proof}[Proof of Proposition \ref{P:CWW-ET-1}]
Lemma \ref{L:CWW-CL} 
implies that \eqref{E:LHam-1} and \eqref{E:LHam-2} are satisfied by $J$ and $\CL(h, \Phi)$ on $X$.  Hence Proposition \ref{P:LHam} yields a decomposition $X$ into the direct sum of the closed subspaces $Y_1, Y_2, Y_3$, and $Y_\pm$. Let $X_\pm = Y_\pm$ and $X_0 = Y_1 \oplus Y_2 \oplus Y_3$. All the desired properties in Proposition \ref{P:CWW-ET-1} are easily verified using Proposition \ref{P:LHam} and the above Lemma \ref{L:DomCA-1}. Finally, even though the choices of subspaces $Y_\pm, Y_{1,2,3}$ may not be unique,  $X_\pm$ and $X_0$ are unique since they are the spectral subspaces corresponding to the subsets of $\sigma(\CA(h, \Phi))$ with $\Re \lambda >0$, $<0$, and $=0$, respectively. 
\end{proof} 

The next lemma will be used in the proof of the continuation of higher regularity. 

\begin{lemma} \label{L:D2F}
Suppose $\frac 32 n \ge \frac {d+9}2$ and $R, \ep>0$, then there exists $C>0$ such that $u=(h, \Phi) \in \CO_{\frac 32 (n-1)} (R, \ep) \times Z_2^{\frac 32(n-1)} (R) \subset X^{n-1}$, $v_j=(\eta_j, \Psi_j)$, $j=1,2$, 
\[
|\BD^2 F(u) (v_1, v_2)|_{X^{n-1}} \le C \big(|v_1|_{X^n} |v_2|_{X^{n-1}} + |v_1|_{X^{n-1}} |v_2|_{X^n} + |u|_{X^n} |v_1|_{X^{n-1}} |v_2|_{X^{n-1}}  \big). 
\]
\end{lemma}

\begin{proof} 
This type of estimates (even some stronger ones) would be rather standard if only algebraic or rational combinations are involved, while the nonlocal $\CG$ causes the main complication here. We use \eqref{E:CA1} and \eqref{E:CA2} to compute $\BD^2 F(u) = \BD \CA(u)$. Separating the terms   in $\BD \CA (u)(v_1, v_2)$ involving $\BD^m \CG$ applied to $v_j$, we can write 
\[
\BD \CA_1 (u) (v_1, v_2) = \BD \CG(h) (\eta_2) \Psi_1 + \BD \CG(h) (\eta_1) \Psi_2 + \BD^2 \CG(h) (\eta_1, \eta_2) \Phi,
\] 
\begin{align*} 
\BD \CA_2 (u) (v_1, v_2)  = & - \sigma \BD^2 \kappa(h) (\eta_1, \eta_2) + f_0(u, v_1, v_2) + f_1 (u, v_2) \CG(h) \Psi_1 + f_2 (u, v_1) \CG(h) \Psi_2 \\
& + f_3 (u, v_2) \BD \CG (h) (\eta_1) \Phi +  f_4 (u, v_1) \BD \CG (h) (\eta_2) \Phi + f_5 (u) \BD \CG (h) (\eta_1) \Psi_2  \\
& +  f_6 (u) \BD \CG (h) (\eta_2) \Psi_1 +  f_7 (u) \BD^2 \CG (h) (\eta_1, \eta_2) \Phi + f_8(u) (\CG(h) \Psi_1)  (\CG(h) \Psi_2) \\
&+ f_9(u) (\CG(h) \Psi_1)  (\BD \CG(h) (\eta_2) \Phi)  + f_{10} (u) (\CG(h) \Psi_2)  (\BD \CG(h) (\eta_1) \Phi)  \\
&+ f_{11} (u) (\BD \CG(h) (\eta_1) \Phi) (\BD \CG(h) (\eta_2) \Phi).
\end{align*}
Here $f_j$, $0\le j\le 4$, are linear or bilinear in $v_1$ and $v_2$ involving $\nabla_{x'} \eta_j$ and $\nabla_{x'} \Psi_j$ with coefficients given by rational combinations of $\nabla_{x'} h$, $\nabla_{x'} \Phi$, and $\CG(h)\Phi$, while $f_j(u)$, $5\le j \le 11$, are rational combinations of $\nabla_{x'} h$, $\nabla_{x'} \Phi$, and $\CG(h)\Phi$. 

To estimate the terms in $\BD \CA_1$, take 
\be\label{E:temp-72.6}
r=3n/2, \quad r_0 = d/2, \quad k = \N \cap \{(3n-d)/2, \, (3n-d+1)/2\}. 
\ee
Since our assumptions yields $r_0+ \frac {k+1}2 \le \frac 32 (n-1)$, applying Lemma \ref{L:DN-3} we obtain 
\begin{align*}
|\BD \CA_1 (u) (v_1, v_2)|_{Z_1^{\frac 32(n-1)}} \le & C | (1-\Delta)^{\frac r2} \BD \CA_1 (u) (v_1, v_2)|_{Z_2^*} \\
\le & C \big(|v_1|_{X^n} |v_2|_{X^{n-1}} + |v_1|_{X^{n-1}} |v_2|_{X^n} + |u|_{X^n} |v_1|_{X^{n-1}} |v_2|_{X^{n-1}}  \big).
\end{align*}


As $\BD^2 \kappa(h) (\eta_1, \eta_2)$ and $f_0(u, v_1, v_2)$ are rational combinations of pointwise derivatives of $u$, $v_1$, and $v_2$, it is straight forward to verify that their $H^{\frac 32n - 1} $ norms, which control the $Z_2^{\frac 32(n-1)}$ norms,  
satisfy the desired estimate by standard inequalities on the products of Sobolev functions when $\frac 32 (n-1) > \frac d2$. 

From the same argument, 
we also have for $n'=n, n-1$,
\be \label{E:temp-73} \begin{split} 
& |f_j (u, v)|_{H^{\frac 32 n'-1}} \le C\big(|v|_{X^{n'-\frac 13}} + |u|_{X^{n'-\frac 13}} |v|_{X^{n'-\frac 43}}\big), \quad 1\le j\le 4; \\
& |f_l (u)|_{H^{\frac 32 n'-1}} \le C (1+ |u|_{X^{n'-\frac 13}}), \quad 5\le l\le 11. 
\end{split} \ee
The estimates of the terms in $\BD \CA_2 (u)$ based on \eqref{E:temp-73}, Proposition \ref{P:DN-0},  and Lemma  \ref{L:DN-3} are similar and we shall illustrate only a few of them. Firstly,  
\begin{align*}
& |f_1 (u, v_2)  \CG(h) \Psi_1|_{Z_2^{\frac 32 (n-1)}} \le  C |f_1 (u, v_2) \CG(h) \Psi_1|_{H^{\frac 32 n-1}}\\ 
\le & C\big( |f_1 (u, v_2) |_{H^{\frac 32 n-1}} |\CG(h) \Psi_1|_{H^{\frac 32 n-\frac 52}} + |f_1 (u, v_2) |_{H^{\frac 32 n-\frac 52}} |\CG(h) \Psi_1|_{H^{\frac 32 n-1}} \big). 
\end{align*}
Applying \eqref{E:temp-73} to the $f_1$ terms, Proposition \ref{P:DN-0} to the above first $\CG(h) \Psi_1$, and then Lemma \ref{L:DN-3} to the second with parameters in \eqref{E:temp-72.6}, we obtain 
\begin{align*} 
& |f_1 (u, v_2)  \CG(h) \Psi_1|_{Z_2^{\frac 32 (n-1)}} \\
\le &  C\big( ( |v_2 |_{X^{n}} + |u|_{X^n} |v_2|_{X^{n-1}})  | \Psi_1|_{Z_2^{\frac 32 n-2}} + |v_2 |_{X^{ n- 1}} |(1-\Delta)^{\frac 34n}\CG(h) \Psi_1|_{Z_2^*} \big) \\
\le & C\big( ( |v_2 |_{X^{n}} + |u|_{X^n} |v_2|_{X^{n-1}}) | \Psi_1|_{Z_2^{\frac 32 n-2}} + |v_2 |_{X^{ n- 1}} (| \Psi_1|_{Z_2^{\frac 32 n}} +|h|_{Z_1^{\frac 32n}}  | \Psi_1|_{Z_2^{\frac 32 (n-1)}} ) \big).
\end{align*}
This yields the desired estimates. The terms involving $f_j$, $j=2,3,4$, can be treated similarly. 
Much as the above, \eqref{E:temp-73} implies 
\begin{align*}
& |f_7 (u) \BD^2 \CG(h)(\eta_1, \eta_2) \Phi|_{Z_2^{\frac 32 (n-1)}} \\
\le & C\big( |f_7 (u) |_{H^{\frac 32 n-1}} |\BD^2 \CG(h) (\eta_1, \eta_2) \Phi|_{H^{\frac 32 n-\frac 52}}  +|f_7 (u) |_{H^{\frac 32 n-\frac 52}} |\BD^2 \CG(h) (\eta_1, \eta_2) \Phi|_{H^{\frac 32 n-1}} \big) \\
\le & C\big( (1 + |u|_{X^n}) |\BD^2 \CG(h) (\eta_1, \eta_2) \Phi|_{H^{\frac 32 n-\frac 52}}  + |(1-\Delta)^{\frac 34n}\BD^2 \CG(h) (\eta_1, \eta_2) \Phi|_{Z_2^*}\big). 
\end{align*}
Controlling the first $\BD^2 \CG$ term by Proposition \ref{P:DN-0} and in the second term by Lemma \ref{L:DN-3}  with parameters in \eqref{E:temp-72.6}, we obtain the desired estimates. The terms involving $f_j$, $j=5,6,7$, can be handled  in a similar fashion. Finally we have from \eqref{E:temp-73} and  Proposition \ref{P:DN-0}
\begin{align*}
& | f_9(u) (\CG(h) \Psi_1)  (\BD \CG(h) (\eta_2) \Phi)|_{Z_2^{\frac 32 (n-1)}} \\ 
\le & C\big(  (1 + |u|_{X^n}) |\CG(h) \Psi_1 |_{H^{\frac 32 n-\frac 52}} |\BD \CG(h) (\eta_2) \Phi|_{H^{\frac 32 n-\frac 52}} + |\CG(h) \Psi_1 |_{H^{\frac 32 n-1}} |\BD \CG(h) (\eta_2) \Phi|_{H^{\frac 32 n-\frac 52}} \\
&+ |\CG(h) \Psi_1 |_{H^{\frac 32 n-\frac 52}} |\BD \CG(h) (\eta_2) \Phi|_{H^{\frac 32 n- 1}} \big) \\
\le & C\big( (1 + |u|_{X^n}) |v_1|_{X^{n-1}} |v_2|_{X^{n-1}}  + |(1-\Delta)^{\frac 34n}\CG(h) \Psi_1 |_{Z_2^*} |v_2|_{X^{n-1}}  \\
&+|v_1|_{X^{n-1}}|(1-\Delta)^{\frac 34n}\BD \CG(h) (\eta_2) \Phi|_{Z_2^*} \big), 
\end{align*}
which along with Lemma \ref{L:DN-3} yields the desired estimates. The terms involving $f_j$, $j=8, 10, 11$, can also be estimated similar. 
\end{proof}

$\bullet$ The above analysis allows us to obtain the {\bf local well-posedness} of the capillary water wave system \eqref{E:Zakharov-tf} including the smooth dependence of solutions on the initial data. 

\begin{theorem} \label{T:CWW-LWP-1} 
Assume \eqref{E:fluidD-0}, \eqref{E:fluidD-2}, and \eqref{E:gravity}. For any integers $n \ge 2$ and $m \ge 1$ satisfying $\frac 32 (n-m) > \frac {d+1}2$, $u_*= (h_*, \Phi_*) \in X^n$ such that $\ep_0=  (\inf h + h_*)>0$, and $R_1 > |u_*|_{X^n}$, there exist $\ep, T, M_0, C>0$ determined by $d, n, m, g, \sigma, h_0, \ep_0, R_1$ such that for any initial value $u_0=(h_0, \Phi_0) \in X^{n-1} (u_*, \ep) \cap X^n (u_*, R_1)$, \eqref{E:Zakharov-tf} has a unique solution $u(t) = (h(t), \Phi(t)) \in X^{n-1} (u_*, 2\ep) \cap X^n (M_0)$, $t \in [-T, T]$, which satisfies $u \in C_t^0 X^n \cap C_t^1 X^{n-1}$. 
Moreover, for any $t \in [-T, T]$, the solution map $u(t, \cdot) \in C^{m,1} (X^{n-1}(u_*, \ep) \cap X^n (R_1), X^{n-1-m})$. For any $u_0 \in X^{n-1}(u_*, \ep) \cap X^n (R_1)$, the $m$-linear operator $\BD_{u_0}^m u(t, u_0) \in \BBL( \otimes_{j=1}^m X^{n-1}, X^{n-m})$ is strongly $C^0$ in $t$, satisfies 
\be \label{E:temp-98}
|\BD_{u_0}^m u(t, u_0)|_{\BBL( \otimes_{j=1}^m X^{n-1}, X^{n-m})} \le 
\begin{cases} C, & \text{ if } \ m=1,\\
Ct,  & \text{ if } \ m>1, 
\end{cases} \ee
and for any $u_{01}, u_{02} \in X^{n-1}(u_*, \ep) \cap X^n (R_1)$, 
\[
|\BD_{u_0}^m u(t, u_{01}) - \BD_{u_0}^m u(t, u_{02})|_{\BBL( \otimes_{j=1}^m X^{n-1}, X^{n-1-m})} \le Ct |u_{01} - u_{02}|_{X^{n-1}}. 
\]
\end{theorem} 


\begin{proof} 
As \eqref{E:Zakharov-tf} is exactly \eqref{E:Zakharov-1} in a moving frame, while the spatial differentiation causes less regularity then the temporal differentiation in these equations with surface tention, it suffices to prove the theorem on the latter. The theorem for \eqref{E:Zakharov-1} would follow readily from Theorems \ref{T:NLPDE-LWP} and \ref{T:LWP-smoothness}, which require assumptions (B.1--B.5). 

Among these assumptions, (B.1) and (B.5) are concerned with the boundness and smoothness of $\CA$, which are satisfied (with $r_0=n-1$ in (B.5)) on any bounded open subset of 
\be \label{E:temp-52.1}
\CO_{\frac 32 (n-m)} \big(\ep_0/2, +\infty \big) \times Z_2^{\frac 32(n-m)} \subset X^{n-m}
\ee
due to Remark \ref{R:smoothness-2} and the above Lemma \ref{L:DomCA-1}. 

In assumption (B.3), the boundedness and smoothness and $\CL$ is ensured by Lemma \ref{L:CWW-CL} on any bounded subset of \eqref{E:temp-52.1}; the dominance \eqref{E:coercivity-2} of $\CL(u)$ on $X$ is satisfied for all $u \in X^{n-1} (u_*, \ep_1)$ for some small $\ep_1 >0$ due to Lemma \ref{L:CWW-CL}(2); and the dissipativity \eqref{E:dissipativity-2.5} of $\CA$ with respect to $\CL$ is due to the Hamiltonian structure. 

Regarding the non-degeneracy assumption (B.2), firstly Proposition \ref{P:CWW-ET-1} implies that for some $\omega_*>0$, $(\omega_* - \CA(u_*))^{-1} \in \BBL( X^{r-1}, X^{r})$, $1\le r\le n$. Due to the smoothness of $\CA$ given in Lemma \ref{L:DomCA-1}, $(\omega_* - \CA)^{-1} \in C^1 (X^{n-m} (u_*, \ep_2), \BBL( X^{r-1}, X^{r}))$, $1\le r\le n-m$, for some $\ep_2 \in (0, \ep_1]$. Again due to Lemma \ref{L:DomCA-1}, $X^r = Dom(\CA(u)^r)$ for all $u \in X^{n-m} (u_*, \ep_2) \cap X^l$, $n-m \le l \le n$, and $1\le r\le l$. Therefore it is straight forward to show, for any $R>0$, 
\be \label{E:temp-52.2}
|(\omega_* - \CA)^{-1}|_{C^1 (X^{n-m} (u_*, \ep_2) \cap X^l (R), \BBL( X^{r-1}, X^{r}))} < \infty, \quad n-m \le l \le n, \;  1\le r\le l. 
\ee
It implies (B.2) holds for $\CO = X^{n-m} (u_*, \ep_2) \cap X^{n-1} (R_1)$ which is open in $X^{n-1}$. 

Assumption (B.4) follows readily from Lemma \ref{L:CB-1} and the above Lemma \ref{L:D2F}. 
\end{proof} 

The following proposition gives a continuation of  higher regularity of solutions which will be used in the proof of the theorem on the stable/unstable manifolds. 

\begin{proposition} \label{P:continuation} 
Let $l, n \in \N$ and a solution $u \in C_t^0 X^n \cap C_t^1 X^{n-1}$ to \eqref{E:Zakharov-tf} on $[-T, T]$ satisfy $\frac 32 n \ge \frac {d+9}2$ and $u(0) \in X^l$, $l > n$, then $u \in C_t^0 X^{l} \cap C_t^1 X^{l-1}$ on $[-T, T]$. 
\end{proposition}

\begin{proof}
From Theorem \ref{T:CWW-LWP-1}, there exists $T_1 >0$ such that $u \in C_t^0 X^{l} \cap C_t^1 X^{l-1}$ on $[-T_1, T_1]$. So we only need to obtain the estimates of $|u(t)|_{X^l}$ and $|u_t (t)|_{X^{l-1}}$. 

Due to the continuity of $u(t)\in X^n$ in $t$, there exist $R, \ep>0$ such that $u(t) \in \CO_{\frac 32n} (R, \ep) \times Z_2^{\frac 32n} (R)$ for all $ t\in [-T, T]$. Moreover Proposition \ref{P:CWW-ET-1} and the continuity of $\CA(u)$ in $u$ (Lemma \ref{L:DomCA-1}) imply there exists $\omega >0$ such that 
\[
\bar A(t)^{-1} \triangleq (\omega - \CA(u(t)))^{-1} \in \BBL(X^{r-1}, X^r), \quad 1\le r \le n, \; t \in [-T, T].
\] 
Let $v=u_{tt}$ and it satisfies 
\be \label{E:temp-71}
v= \CA(u) F(u), \quad v_t = \CA(u) v + \BD^2 F(u) (u_t, u_t), \quad u_t = \bar A(t)^{-1} ( \omega u_t - v). 
\ee
One the one hand, using Lemmas \ref{L:CWW-CL} and \ref{L:DomCA-1}, as in the proof Theorem \ref{T:NLPDE-LWP}, one may prove that $w_t = \CA(u(t))w$ generates a solution operator $U(t, t_0) \in \BBL(X^r)$, $0\le r\le n-1$, strongly continuous in $(t, t_0) \in [-T, T]^2$, where $t_0$ is the initial time. On the other hand, Lemma \ref{L:D2F} implies that 
\[
\big|\BD^2 F(u(t))\big(u_t(t), u_t(t) \big)\big|_{X^{n-1}} \le C |u_t (t)|_{X^{n-1}} |u_t(t)|_{X^n} \le C |u_t (t)|_{X^{n-1}} | \omega u_t (t)- v(t)|_{X^{n-1}}. 
\]
Therefore from the variation of parameter formula 
\[
v(t) = U(t, 0) v(0) + \int_0^t U(t, \tau) \BD^2 F(u(\tau)) \big(u_t(\tau), u_t(\tau) \big) d\tau, \;\; v (0) = \CA(u(0)) F(u(0)) \in X^{l-2}, 
\]
we immediately obtain 
\[
u_{tt} = v \in C_t^0 X^{n-1} \implies F(u)= u_t = \bar A^{-1} (\omega u_t - v) \in C_t^0 X^n, \quad  t\in [-T, T].
\] 
The translation invariance of $F$, i.~e.~$F(u(\cdot + x_0'))=(F(u))(\cdot + x_0')$, implies 
\[
\CA(u) \p u = \p F(u) \in C_t^0 X^{n-\frac 23} \subset C_t^0 X^{n-1}. 
\]
Since $u \in C_t^0X^n$, along with Lemma \ref{L:DomCA-1} it yields $\p u \in C_t^0 X^n$ and thus $u \in C_t^0 X^{n+\frac 23}$. Repeating this argument, we obtain $u \in C_t^0X^{n+1}$. 

The proposition follows from the above argument inductively. 
\end{proof}

$\bullet$ {\bf Invariant manifolds.} Finally, we are ready to prove  Theorem \ref{T:CWW-LInMa-main}.  

\begin{proof}[Proof of  Theorem \ref{T:CWW-LInMa-main}]
We shall only focus on the local unstable manifold of the spectrally unstable equilibrium $(h_*, \Phi_*)$ as the derivation and properties of the stable manifold are similar (see Remarks \ref{R:smoothness-SM}, \ref{R:SM}, and \ref{R:Ham-SM}). 
Let  
\be \label{E:para-2} \begin{split} 
& \N \ni n > 1+ (d+1)/3, \quad 
\ep_0 = \inf (h_* + h_0) >0, \quad R_0 = 2 |h_*|_{H^{\frac 32n+1}}+1, \\
&  \CO= \CO_{\frac 32(n-1)} (R_0, \ep_0/2) \times Z_2^{\frac 32(n-1)} (R_0) \subset X^{n-1}, \quad 0 < \omega_- < \lambda < \omega_+ < \lambda_+. 
\end{split} \ee
The assumptions on $u_*\triangleq (h_*, \Phi_*)$ in Theorem \ref{T:CWW-LInMa-main} along with Lemmas \ref{L:CWW-CL} and \ref{L:DomCA-1} 
verify all the hypothesis of Theorem \ref{T:Ham-UM}. Therefore, for 
any $n \in \N$, $\lambda$, and $\omega_\pm$ satisfying \eqref{E:para-2}, there exist $\delta, M^*, C>0$ and $q^+: X_+ (\delta) \to X_- \oplus (X_0 \cap X^n)$ with the properties in Theorem \ref{T:NLPDE-UM}.  

Fix $n=n_0$ and $\lambda_0$ as given in Theorem \ref{T:CWW-LInMa-main} and let $q^+$ be determined by $n_0$, $\omega_- = \frac 12 \lambda_0$, and $\omega_+= \frac 12 (\lambda_0 + \lambda_+)$. Except for \eqref{E:UM-3.1-1}, Theorem \ref{T:CWW-LInMa-main}(2) is identical to Theorem \ref{T:NLPDE-UM}(1). Due to the local well-posedness of \eqref{E:Zakharov-tf}, Theorem \ref{T:CWW-LInMa-main}(3) follow from Theorem \ref{T:NLPDE-UM}(2--3). We shall prove that $q^+$ satisfies Theorem \ref{T:CWW-LInMa-main}(4), \eqref{E:UM-3.1-1},  and then Theorem \ref{T:CWW-LInMa-main}(1). 

The proof of Theorem \ref{T:CWW-LInMa-main}(4) is similar to that of Theorem \ref{T:NLPDE-UM}(5) except for that $n$ may also be different from $n_0$. Suppose $n$ and $\lambda$ also satisfy \eqref{E:para-1}. Let $\wt n= \min \{n, n_0\}$ and $\wt \lambda= \min\{\lambda, \lambda_0\}$. Choose $\wt \omega_\pm \in (0, \lambda_+)$ such that $\lambda_0, \lambda \in (\wt \omega_-, \wt \omega_+)$. 
From the same above analysis, there also exist $\wt \delta \in (0, \delta]$, $\wt M^*>0$, and $\wt q^+: X_+ \to X_- \oplus (X_0 \cap X^{\wt n})$ given by Theorem \ref{T:NLPDE-UM} for $(\wt n, \wt \lambda)$, which also satisfies Theorem \ref{T:CWW-LInMa-main}(1--2) except for \eqref{E:UM-3.1-1}. Let 
$(h_*, \Phi_*) +u(t)$, $t\le 0$, be a solution to \eqref{E:Zakharov-tf} such that $\sup_{t\le 0} |u(t)|_{X^{n}} e^{- \lambda t} <\infty$. From the choices of $\wt n$ and $\wt \lambda$, we have $\sup_{t\le 0} |u(t)|_{X^{\wt n}} e^{- \wt \lambda t} <\infty$ and thus Theorem \ref{T:NLPDE-UM}(5) implies $(I-\Pi_+) u(t) = \wt q_+(u_+(t))$ for all $t\ll -1$. Let $t_1 \ll -1$ and $(h_*, \Phi_*) + v(t)$ be the solution to \eqref{E:Zakharov-tf} with initial value $v(0) = u_+ (t_1) + q^+ (u_+ (t_1))$. Due to the definition of $q^+$ which satisfies \eqref{E:UM-3-1}, by taking $t_1 \ll -1$, we have 
\[
|v(t)|_{X^{\wt n}} \le |v(t)|_{X^{n_0}} \le M^* |u_+(t_1)|_{X_+} e^{\lambda_0 t} \le M^* |u_+(t_1)|_{X_+} e^{\wt \lambda t} \le 2 \wt M^* \wt \delta e^{\wt \lambda t}, \quad \forall t\le 0.
\]
Therefore the uniqueness property in Theorem \ref{T:CWW-LInMa-main}(1) satisfied by $\wt q^+$ implies $q^+ (u_+ (t_1)) = \wt q^+ (u_+ (t_1))$ and $u(t+ t_1) = v(t)$ for all $t \le 0$. Since $v(0) \in W^+$, the invariance property Theorem \ref{T:CWW-LInMa-main}(2) of $q^+$ implies $u(t) \in W^+$ 
for all $ t\ll -1$ and Theorem \ref{T:CWW-LInMa-main}(4) is proved. 

To prove \eqref{E:UM-3.1-1}, let $\wt q^+: X_+(\wt \delta) \to X_- \oplus (X_0 \cap X^n)$ be the mapping given by Theorem \ref{T:NLPDE-UM} for $(n, \lambda)$. For any $u_{0+} \in X_+(\delta)$, let $(h_*, \Phi_*) +u(t)$ be the solution with $u(0) = u_{0+} +  q^+( u_+(0))$. Since $u(t)$ satisfies \eqref{E:UM-3-1} with parameter$(n_0, \lambda_0)$, Theorem \ref{T:CWW-LInMa-main}(4) satisfied by $\wt q^+$ and parameter $(n_0, \lambda_0)$ implies $(I-\Pi_+) u(t) = \wt q^+( u_+(t))$ for $t \ll -1$. 
Therefore \eqref{E:UM-3.1-1} follows from \eqref{E:UM-3-1} satisfied by $\wt q^+$. 

To prove Theorem \ref{T:CWW-LInMa-main}(1), for any $n > n_0$, let $\wt q^+: X_+(\wt \delta) \to X_- \oplus (X_0 \cap X^n)$ be the mapping given by Theorem \ref{T:NLPDE-UM} for $(n, \lambda=\frac {\lambda_++\lambda_0}2)$ and $\wt W^+ = graph (\wt q^+)$. 
We first show the regularity $q^+(X_+(\delta)) \subset X^n$. For any $u_{0+} \in X_+(\delta)$, let $(h_*, \Phi_*) +u(t)$ be the solution with $u(0)= u_{0+} +  q^+( u_+(0)) \in W^+$, then \eqref{E:UM-3-1}  satisfied by $q^+$ and Theorem \ref{T:CWW-LInMa-main}(4) satisfied by $\wt q^+$ imply that for all $t_1 \ll -1$ such that $u(t_1) \in \wt W^+  \subset X^n$. 
Due to the assumption on $n_0$, applying Proposition \ref{P:continuation} to initial value $(h_*, \Phi_*) +u(t_1)$, we obtain $u(0) \in X^n$. 

We continue to show $q^+ \in C^\infty(X_+(\delta))$. From Remark \ref{R:smoothness-2} and Lemma \ref{L:DomCA-1}, assumption (B.5) is satisfied for all $n \ge n_0$, $r_0= n-1$, and $2\le m \le n- n_0 +2$.  Theorem \ref{T:UM-smoothness} implies  
\[
 |q^+|_{ C^{2, 1} (X_+(\delta),  X^{n_0-3})} \le C, 
 \quad q^+(0)=0, \quad \BD q^+(0) =0,
 \]
and thus $W^+ \subset  X^{n_0-3}$ is a $C^{2,1}$ manifold diffeomorphic and close to $X_+(\delta)$. 
For any $\N \ni m \le  \frac {n-n_0}2-2$, 
similarly Theorem \ref{T:UM-smoothness} also implies the above defined 
\[
\wt q^+ \in C^{m, 1} (X_+(\wt \delta), X^{n-m-1}), \quad \wt q^+(0)=0, \quad \BD \wt q^+(0) =0.
\] 
 Let $(h_*, \Phi_*) +\wt u(t, v_{+})$, $v_{+} \in X_+(\wt \delta)$,  denote the solution with initial value $(h_*, \Phi_*) +v_{+} + \wt q^+(v_{+})$. 
 From Theorem \ref{T:CWW-LWP-1}, the solution map of \eqref{E:Zakharov-tf} is $C^{m, 1}$ from $X^{n-m-1}$ to $X^{n-2(m+1)} \subset X^{n_0+2}$. Moreover, since the solution map is invertible (by the solution map in negative time), their linearizations are always injective. 
Therefore  
 \[
\wt u (-t_1, \cdot)  \in C^{m, 1} (X_+(\wt \delta), X^{n- 2(m+1)}), \quad  \wt u(-t_1, u_+(t_1)) = u(0) = u_{0+} + q^+(u_+(0)),  
\] 
and $\BD \wt u (-t_1, v_{+}) \in \BBL(X_+, X^{n-2(m+1)})$ is  injective. 
This implies $\Pi_+ \wt u(-t_1, \cdot)$ is $C^{m, 1}$ mapping from a neighborhood of $u_+ (t_1) \in X_+$ to a neighborhood of $u_{0+} \in X_+$.
Since $u(t)$ satisfies \eqref{E:UM-3-1}, by the continuity of $\wt u(-t_1, \cdot)$ and the solution map of \eqref{E:Zakharov-tf} and the choice of $\lambda > \lambda_0$, solutions $(h_*, \Phi_*) +\wt u(t, v_{+})$ to \eqref{E:Zakharov-tf} with 
$v_+$ close to $u_{+}(t_1)$ satisfy \eqref{E:temp-91}. From  Theorem \ref{T:CWW-LInMa-main}(2), we obtain $\wt u(-t_1, v_+) \in W^+ = graph(q^+)$. 
Hence the injectivity of $\BD \wt u (-t_1, v_{+})$ implies  the injectivity of $\BD \Pi_+ \wt u (-t_1, v_{+}) \in \BBL(X_+)$. As $\dim X_+< \infty$, $\Pi_+ \wt u (-t_1, \cdot)$ is a $C^{m, 1}$ local diffeomorphism near $u_{+} (t_1)$. Therefore 
\[
q^+ = (I- \Pi_+) \wt u (-t_1, \cdot) \circ \big( \Pi_+ \wt u (-t_1, \cdot) \big)^{-1}  : X_+ \to X^{n-2(m+1)} 
\]
is $C^{m, 1}$ in a neighborhood of $u_{0+}$. 
Since $n >n_0$, $\N \ni m \le \frac {n-n_0}2-2$,  and $u_{0+}\in X_+(\delta)$ can be chosen arbitrarily, we obtain $q^+ \in C^\infty (X_+(\delta), X^n)$ for any $n \in \N$. The proof of Theorem \ref{T:CWW-LInMa-main} is complete. 
\end{proof} 

$\bullet$ {\it  Unstable manifolds of transversally unstable solitary capillary gravity water waves.} Consider \eqref{E:Zakharov-tf} on $\CU = \R \times (\R/l\Z)$ with $\Vc = (c, 0)$. 
Small line solitons $\big(h^\ep (x) = \ep^2 \wt h(\ep, \ep x), \Phi^\ep (x) = \ep \wt \Phi(\ep, \ep x)\big)$  were found in \cite{AK89} under assumptions 
\be \label{E:AK89}
gh_0/c^2 -1 = \ep^2 \ll 1, \quad  0 < \sigma/(g h_0^2) - 1/3 = O(1), 
\ee
where $\wt h(\ep, x)$ and $\wt \Phi(\ep, x)$ are smooth in both $\ep$ and $x \in \R$ limiting to the KdV soliton as $\ep \to 0+$. Under perturbations depending on both $x \in \R$ and the transversal variable $y \in \R/l\Z$, the spectral and nonlinear instability was obtained  for certain $l>0$ in \cite{RT11}, where essential spectra of $\CL (h^\ep, \Phi^\ep)= \BD^2 \CH(h^\ep, \Phi^\ep)$ was analyzed (ee also, e.~g.~\cite{Mi02}). In fact, using \eqref{E:AK89}, it is straight forward to obtain $\CL_{01}(0) \ge \frac {\ep^2}C$ over the domain $\CU$. In the difference $\CL_{01}(h, 0) - \CL_{01}(0)$, everything is multipled by $\p_{x} h^\ep (x') = \ep^3 \p_{x} \wt h (\ep x)$ whose $Z_1^{r_0}$ norm of the order $\ep^{\frac 52}$. Hence \eqref{E:coercivity-5} is verified and Theorem \ref{T:CWW-LInMa-main} applies to yield (possibly multi-dimensional) $C^\infty$  unstable manifolds of the spectrally unstable line solitons as steady states of \eqref{E:Zakharov-tf} over the domain $\CU$.

\subsection{Fluid interface problem} \label{SS:2F} 

Consider two irrotational incompressible inviscid fluids  occupying $d$-dim domains 
\be \label{E:fluidD-3}
\Omega_h^- = \{ -h_- < x_d < h(x') \mid x' \in \CU\}, \; \Omega_h^+ = \{ h_+ > x_d > h(x') \mid x' \in \CU\}, \; h_\pm \in (0, \infty],
\ee
separated by an interface $\CS_h$ given by the graph of $h(x')$ satisfying 
\[
\inf_\CU h + h_-, \; h_+ - \sup_\CU h >0,  
\]
over a  horizontal domain $\CU$ as in \eqref{E:fluidD-0}. Again $h_\pm$, if finite, are the typical depth of the upper and lower fluid. Assume the fluids in $\Omega_h^\pm$ have densities $\rho_\pm$, pressures $p_\pm$, and the velocity field $v_\pm : \Omega_h^\pm \to \R^d$ with constant horizontal background velocities $\nu_\pm \in \R^{d-1}$ which count for the conserved horizontal momenta. 

In the irrotational case, there exist potentials $\phi_\pm: \Omega_h^\pm \to \R$ satisfying the Laplace equation \eqref{E:Harmonic} in $\Omega_t^\pm$ with the slip boundary condition along $x_d = \pm h_\pm$ if $h_\pm \in (0, \infty)$, such that  
\[
v_\pm = \nu_\pm + \nabla \phi^\pm. 
\]
The kinematic condition, which includes the consistency of the normal components of $v_\pm$,
\be \label{E:BC-2F-K}
h_t = (-\nabla_{x'} h, 1) \cdot v_\pm  \triangleq \mp \CG_\pm(h) \Phi^\pm - \nu_\pm \cdot \nabla_{x'} h,  
\ee
and the dynamic boundary conditions 
\be \label{E:BC-2F-D}
(p_- - p_+)|_{\CS_t} = \sigma \kappa = - \sigma \nabla_{x'} \cdot \Big( \frac {\nabla_{x'}  h }{\sqrt{1+ |\nabla_{x'}  h |^2}} \Big), 
\ee
are assumed along $\CS_h$, where $\Phi_\pm$ are the traces of $\phi_\pm$ defined in \eqref{E:HarTrace-1}. 
Here ``$\mp$'' sign in the definition of the  weighted Dirichlet-Neumann operators  ensures $\CG_\pm(h)$ are outward normal derivatives of $\phi^\pm$ along $\p \Omega_h^\pm= \CS_h$. 
As in the fluid-vacuum problem, $\int_\CU \Phi_\pm dx'=0$ is adopted if $d_2=0$ in the definition \eqref{E:fluidD-0} of $\CU$. It is included in \eqref{E:BC-2F-K} 
\be \label{E:temp-53}
\CG_+(h) \Phi^+ + \CG_-(h) \Phi^- + (\nu_+ -\nu_-) \cdot \nabla_{x'} h=0. 
\ee
Hence $\phi^\pm$, and $\Phi^\pm$ as well, are not independent. 
The Euler equation \eqref{Euler} implies 
\[
\nabla \big( \rho_\pm (\phi_t^\pm + \nu_\pm \cdot \nabla \phi^\pm + \tfrac 12 |\nabla \phi^\pm|^2 + g x_d) + p_\pm\big) =0, 
\]
which in turn yield the Bernoulli equation along $\CS_t$, 
\be \label{E:2F-Ber-1} 
\rho_- \big( \phi_t^- + \nu_- \cdot \nabla \phi^- + \tfrac 12 |\nabla \phi^-|^2 \big) - \rho_+ \big( \phi_t^+ +  \nu_+ \cdot \nabla \phi^+ + \tfrac 12 |\nabla \phi^+|^2 \big)   + g (\rho_- -\rho_+) h + \sigma \kappa =const.
\ee
The irrotational interface problem is equivalent to \eqref{E:BC-2F-K} and \eqref{E:2F-Ber-1}. 

\begin{remark} \label{R:background-V}
In the fluid-vacuum interface problem, by considering $h(t, x-t \nu_1)$ and $v(t, x- t \nu_1) + \nu_2$, $\nu_1, \nu_2 \in \R^{d-1}$, from the Euler equation \eqref{Euler} and boundary conditions one 
obtains \eqref{E:Zakharov-tf}  with a moving velocity $\nu_1 - \nu_2$. 
With the background velocities in the fluid interface problem, there is no need to add the moving frame into the problem. 
\end{remark}

\subsubsection{Hamiltonian formulation} \label{SSS:2F-Ham}
We use the canonical variables given in \cite{BB97} (see also \cite{CG00}) 
\[
(h, \Phi), \; \text{ where } \; \Phi= \rho_- \Phi^- - \rho_+ \Phi^+. 
\] 
One observes that for any given $h$, $\Phi$ and $(\Phi^+, \Phi^-)$ has a one-to-one correspondence due to \eqref{E:temp-53} and the defintion of $\Phi$. 
The above Bernoulli equation \eqref{E:2F-Ber-1} is equivalent to  
\be \label{E:2F-Ber-2} \begin{split}
\Phi_t - \big( \rho_- \phi_{x_d}^- -  \rho_+ & \phi_{x_d}^+ \big) h_t  + \rho_- \big( \nu_- \cdot \nabla \phi^- + \tfrac 12 |\nabla \phi^-|^2 \big) \\
& - \rho_+ \big(  \nu_+ \cdot \nabla \phi^+ + \tfrac 12 |\nabla \phi^+|^2 \big)   + g (\rho_- -\rho_+) h + \sigma \kappa =const. 
\end{split} \ee
To verify briefly the Hamiltonian structure, we start with  the kinetic energy 
\[
\CK= \sum_\pm \int_{\Omega_h^\pm} \frac 12 \rho_\pm ( |v_\pm|^2 - |\nu_\pm|^2) dx = \sum_\pm \int_{\Omega_h^\pm} \frac 12 \rho_\pm ( |\nabla \phi^\pm|^2 + 2 \nu_\pm \cdot \nabla \phi^\pm) dx. 
\]
To derive the variation of the kinetic energy, suppose $h$ and $\Phi$, and thus $\Phi^\pm$ and $\phi^\pm$ as well, depend on an external parameter $\alpha$, while \eqref{E:temp-53} holds for all $\alpha$. As in \cite{BB97}, one may calculate 
\begin{align*} 
\CK_\alpha = & \sum_\pm \Big( \int_{\Omega_h^\pm} \rho_\pm ( \nabla \phi^\pm +  \nu_\pm ) \cdot \nabla \phi_\alpha^\pm dx \mp \int_{\CU} \frac 12 \rho_\pm \big( |\nabla \phi^\pm|^2 +  2\nu_\pm \cdot \nabla \phi^\pm\big) h_\alpha dx' \Big)\\
=& \sum_\pm \int_{\CU} \rho_\pm \big( \CG_\pm(h) \Phi^\pm \pm  \nu_\pm \cdot \nabla_{x'} h \big) \phi_\alpha^\pm \mp \frac 12 \rho_\pm \big( |\nabla \phi^\pm|^2 +  2\nu_\pm \cdot \nabla \phi^\pm\big) h_\alpha dx' \\
=& \sum_\pm \int_{\CU} \rho_\pm \big( \CG_\pm(h) \Phi^\pm \pm \nu_\pm \cdot \nabla_{x'} h \big) (\Phi_\alpha^\pm - \phi_{x_d}^\pm h_\alpha ) \mp \frac 12 \rho_\pm \big( |\nabla \phi^\pm|^2 +  2\nu_\pm \cdot \nabla \phi^\pm\big) h_\alpha dx'.
\end{align*}
Using \eqref{E:temp-53}, we obtain 
\begin{align*} 
\CK_\alpha = \int_{\CU} \big( \CG_-(h) \Phi^- -  \nu_- \cdot \nabla_{x'} h \big) \Phi_\alpha - \sum_\pm  \rho_\pm \Big( &\big( \CG_\pm(h) \Phi^\pm \pm  \nu_\pm \cdot \nabla_{x'} h \big)  \phi_{x_d}^\pm \\
&\pm  \big( \frac 12 |\nabla \phi^\pm|^2 +  \nu_\pm \cdot \nabla \phi^\pm\big) \Big)h_\alpha dx'. 
\end{align*}
This equality yields the variation of $\CK$ with respect to $\Phi$ and $h$. 
Up to a constant, the gravitational potential energy on $\Omega_h^\pm$ is given by 
\[
\mp \int_{\CU}  \int_{0}^{h(x')} g \rho_\pm x_d dx_d dx' = \mp \frac 12 \int_{\CU} g \rho_\pm h^2 dx'.
\]
From these calculations along with the variation of the potential energy of the surface area, it is straight forward to show that the Hamiltonian flow generated by $\BH$ is indeed the irrotational fluid interface problem given in \eqref{E:BC-2F-K} and \eqref{E:2F-Ber-2}.

To write the Hamiltonian more explicitly as a nonlinear functional of $\Phi$ and $h$, much as above we convert $\CK$ into an integral on the surface via the divergence theorem
\begin{align*}
\CK
= & \sum_\pm \int_{\CU} \frac 12 \rho_\pm  \big( \CG_\pm (h) \Phi^\pm \pm 2 \nu_\pm \cdot \nabla_{x'} h \big) \Phi^\pm  dx' = 
\sum_\pm \int_{\CU} \frac 12 \rho_\pm ( \mp h_t \pm \nu_\pm \cdot  \nabla_{x'} h ) \Phi^\pm dx' \\
=& \sum_\pm \int_{\CU} \frac 12 \rho_\pm ( h_t - \nu_\pm\cdot \nabla_{x'} h ) \CG_\pm (h)^{-1} (h_t + \nu_\pm \cdot \nabla_{x'} h ) dx' \\
= & \frac 12   \int_{\CU}h_t \wt \CG(h)^{-1} h_t - \sum_\pm \rho_\pm (\nu_\pm \cdot \nabla_{x'} h )  \CG_\pm (h)^{-1} (\nu_\pm \cdot  \nabla_{x'} h)  dx',
\end{align*} 
where \eqref{E:BC-2F-K} was used and the modified Dirichlet-Neumann operator $\wt \CG(h)$ is given by 
\be \label{E:wtCG}
\wt \CG^{-1} =\rho_+ \CG_+^{-1} +  \rho_- \CG_-^{-1}, 
\quad \wt \CG = \CG_+ (\rho_+ \CG_- + \rho_- \CG_+)^{-1} \CG_-. 
\ee
One may also compute 
\be \label{E:temp-54}
\wt \CG(h)^{-1} h_t 
= \Phi - \Sigma_\pm  \big(\rho_\pm  \nu_\pm \cdot \CG_\pm (h)^{-1}\nabla_{x'} h\big), 
\ee
and thus the Hamiltonian of the fluid interface problem takes the form 
\be \label{E:Ham-2F} \begin{split} 
\BH(h, \Phi) = & \int_{\CU} \frac 12  \Big(\Phi - \sum_\pm \rho_\pm  \nu_\pm \cdot \CG_\pm (h)^{-1}\nabla_{x'} h \Big)
\wt \CG(h) \Big(\Phi - \sum_\pm \rho_\pm  \nu_\pm \cdot \CG_\pm (h)^{-1}\nabla_{x'} h \Big)\\
& 
- \frac 12 \sum_\pm \rho_\pm (\nu_\pm \cdot \nabla_{x'} h )  \CG_\pm (h)^{-1} (\nu_\pm \cdot  \nabla_{x'} h) + \frac 12 g (\rho_--\rho_+) h^2 \\
&+ \sigma (\sqrt{1+|\nabla_{x'} h|^2} -1) dx'.
\end{split} \ee


\subsubsection{Function spaces and preliminary analysis} \label{SSS:wtCG} 

Let $Z_1^r$ be the same function space  defined in \eqref{E:space-1} of $h(x')$. 
For the momentum component, let 
\[
|f|_{Z_2^r} = |(1-\Delta_{x'})^{\frac r2} \wt \CG(0)^{\frac 12} f|_{L^2}. 
\]
It is easy to compute 
\[
\wt \CG(0) = \Big( \frac {\rho_+}{|\nabla_{x'}| \tanh (h_+ |\nabla_{x'}|)} + \frac {\rho_-}{|\nabla_{x'}| \tanh (h_- |\nabla_{x'}|)}\Big)^{-1},
\]
where it is understood $\tanh(h_\pm |\nabla_{x'}|) = 1$ if $h_\pm =+\infty$. 
So $|\cdot|_{Z_2^r}$ is positive (unless $d_2=0$ in \eqref{E:fluidD-0} and $f =const$) and induced by the symmetric bilinear form $\langle (1-\Delta_{x'})^{r}\wt \CG(0)f_1,  f_2 \rangle$. 
Define the real Hilbert space $Z_2^r$ to be the completion of $C_0^\infty (\CU, \R)$ under $|\cdot|_{Z_2^r}$  
and 
\be \label{E:space-5} 
X^r = Z_1^{\frac 32r} \times Z_2^{\frac 32r}, \quad 
X \triangleq X^0= Z_1 \times Z_2, \quad Z_1= Z_2^0, \quad Z_2 = Z_2^0,   
\ee
which could be quotient spaces as discussed in Remark \ref{R:spaces}. As in the fluid-vacuum case, 
\be \label{E:space-6} \begin{split} 
& \text{ if } \;   h_+ \ \text{ or } \ h_-<\infty:  \Phi \in Z_2^r \; \text{ iff } \; \nabla_{x'} \Phi \in H^{r-\frac 12}, \\
& \text{ if } \;   h_+ =  h_- =\infty:  \Phi \in Z_2^r \; \text{ iff } \; |\nabla_{x'}|^{\frac 12} \Phi \in H^{r}. 
\end{split} \ee

We first discuss those properties of  given in Section \ref{SS:CWW-pre}. We shall need the following identities satisfied by operators $W, W_1, W_2, \ldots $ depending on $h$
\begin{align}  
& [W_1, W_2^{-1}] = W_2^{-1} [W_2, W_1] W_2^{-1}, \quad \BD (W (h)^{-1}) \eta = - W (h)^{-1} (\BD W(h) \eta ) W (h)^{-1}, \label{E:temp-75}\\
& \BD^2 (W(h)^{-1}) (\eta, \eta) = - W (h)^{-1} \BD^2 W(h) (\eta, \eta ) W (h)^{-1} \notag \\
& \qquad \qquad  \qquad \qquad \quad + 2W (h)^{-1} (\BD W(h) \eta ) W (h)^{-1}  (\BD W(h) \eta ) W (h)^{-1}, \ldots  \label{E:temp-76}.
\end{align}

Without loss of generality, we may assume $h_+ \le h_- \le \infty$. While $\CG_+(0)$ defines the same spaces $Z_2^r$ and $\CG_+(h)$ satisfies all the properties in Section \ref{SS:CWW-pre}, $\CG_-(0)$ might define another scale of spaces $Z_{2-}^r \subset Z_2^r$ and $\CG_-(h)$ satisfies the same properties on $Z_{2-}^r$. In particular, $Z_{2-}^r = Z_2^r$ iff $h_+= h_- =\infty$ or $h_+, h_- \in (0, \infty)$. Clearly $Z_{2}^*  \subset Z_{2-}^*$ and Proposition \ref{P:DN-2} yields the analyticity of 
\be \label{E:temp-74}
\CG_-(h)^{-1}(1-\Delta)^{-\frac r2}  \in \BBL(Z_{2-}^*, Z_{2-}^r) \subset \BBL(Z_2^*, Z_2^r), \quad  
\CG_\pm (h)^{-1} \nabla_{x'} \in \BBL(Z_1^r, Z_2^{r+\frac 12}),
\ee
in $h$. It implies $\rho_+ \CG_+^{-1}(h) +  \rho_- \CG_-^{-1}(h) \in \BBL(Z_2^*, Z_2)$, symmetric and bounded below by $\rho_+ \CG_+^{-1}(h)$. Therefore $\wt \CG(h) \in \BBL(Z_2, Z_2^*)$ is well-defined, isomorphic, positive, and analytic in $h$. Moreover $\wt \CG(h)^{-1}(1-\Delta)^{-\frac r2} \in \BBL(Z_2^*, Z_2^r)$ is also analytic in $h$. 

From Corollary \ref{C:DN-1} and \eqref{E:temp-75}, one proves that  $\CG_\pm(h)^{-1}$, and thus $\wt \CG(h)^{-1}$ as well, satisfy the corresponding commutator estimates, 
\[
[\p_{x_j}, \CG_-(h)^{-1} ]\in \BBL(Z_{2-}^*, Z_{2-}), \quad [\p_{x_j}, \CG_+(h)^{-1} ], \, [\p_{x_j}, \wt \CG(h)^{-1}] \in \BBL(Z_{2}^*, Z_{2}), 
\]
which are also analytic in $h$. Consequently, 
\[
\p_{x_j}\wt \CG(h)  = \wt \CG(h) \p_{x_j} + \wt \CG(h) [\wt \CG(h)^{-1}, \p_{x_j}] \wt \CG(h) \in \BBL(Z_2^1, Z_2^*). 
\]
Along with interpolation, inductively we obtain that $\wt \CG$ satisfies Lemma \ref{L:DN-1}. Together with the first relation in \eqref{E:temp-74}, it implies that Proposition \ref{P:DN-2} is also satisfied by $\wt \CG(h)$. 

Finally, since $\CG_\pm(h)$ satisfy Lemma \ref{L:DN-3}, from \eqref{E:temp-75} and \eqref{E:temp-76}, $\CG_\pm(h)^{-1}$, and thus $\wt \CG(h)^{-1}$ as well, satisfy corresponds estimates on their variations. Again applying \eqref{E:temp-75} and \eqref{E:temp-76} we obtain Lemma \ref{L:DN-3} satisfied by $\wt \CG(h)$.

In summary, $\wt \CG(h)$ satisfies all the properties in Section \ref{SS:CWW-pre}.

\subsubsection{Local well-posedness and stable and invariant manifolds} \label{SSS:2F-LWP-LInMa}

The local well-posedness of the fluid interface problems with surface tension has been established in, e.~g.~\cite{BHL93, AM03, CSS08, SZ08b, SZ11}. In the following we lay out the framework for Theorems \ref{T:Ham-UM}, \ref{T:NLPDE-LWP}, \ref{T:LWP-smoothness}, \ref{T:NLPDE-UM}, \ref{T:UM-smoothness} to apply to yield both the well-posedness and invariant manifolds. 
To analyze 
\[
F  = J \BD \BH, \quad \CL= \BD^2 \BH, \quad \CA= J \BD^2 \BH,
\]
as in \eqref{E:CL0-1} and \eqref{E:D2H-1} we still split $\CL(h, \Phi)$ into 
\begin{align*}
\langle \CL(h, \Phi) (\eta_1, \Psi_1),  (\eta_2, \Psi_2) \rangle = &\langle \CL_{01}(h) \eta_1,  \eta_2 \rangle + \langle \CL_{02}(h) (\eta_1, \Psi_1),  (\eta_2, \Psi_2) \rangle \\
&+ \langle \CL_1(h, \Phi) (\eta_1, \Psi_1),  (\eta_2, \Psi_2) \rangle,
\end{align*}
\begin{align*}
\langle \CL_{02}(h) (\eta_1, \Psi_1),  (\eta_2, \Psi_2) =\int_{\CU} \frac 12  \Big(\Psi_1 & - \sum_\pm \rho_\pm  \nu_\pm \cdot \CG_\pm (h)^{-1}\nabla_{x'} \eta_1 \Big) \\
& \times  \wt \CG(h) \Big(\Psi_1 - \sum_\pm \rho_\pm  \nu_\pm \cdot \CG_\pm (h)^{-1}\nabla_{x'} \eta_2 \Big) dx'
\end{align*}
\begin{align*}
\langle \CL_{01}(h) \eta_1,  \eta_2 \rangle & = \int_{\CU} - \sum_\pm \rho_\pm (\nu_\pm \cdot \nabla_{x'} \eta_1 )  \CG_\pm (0)^{-1} (\nu_\pm \cdot  \nabla_{x'} \eta_2) +  g (\rho_--\rho_+) \eta_1\eta_2 \\
&+\frac \sigma{\sqrt{1+ |\nabla_{x'}  h_1 |^2}} \Big( \nabla_{x'} \eta_1 \cdot \nabla_{x'} \eta_2  - \frac {(\nabla_{x'} h_1 \cdot \nabla_{x'} \eta_1)(\nabla_{x'} h_1 \cdot \nabla_{x'} \eta_2) }{1+ |\nabla_{x'} h_1|^2}  \Big)    dx'.
\end{align*}
Much as in the proof of Lemma \ref{L:CWW-CL}, the remainder $\CL_1(h, \Phi)$ turns out to be compact perturbation and $\CL_{02}(h)$ is positive. The leading part of $\CL_{01}(h)$ from the surface tension and possibly gravity is positive on $Z_1$. In the case of $d_2>0$, we shall assume the finite Morse index assumption 
\be \label{E:coercivity-2F} \begin{split}
&\exists \delta>0 \  \text{ and a closed subspace }\  \wt Z_+ \subset Z_1 \; \text{ s.~t.~} \ codim \wt Z_+< \infty \ \text{ and } \\ 
&\langle \CL_{01}(h) \eta,  \eta \rangle \ge \delta |\eta|_{Z_1}^2, \; \forall \eta \in \wt Z_+. 
\end{split} \ee

\begin{lemma} \label{L:2F-a}
For any $(h, \Phi) \in X^{n}$ satisfying $n \in \N$, $\frac 32 n > \frac {d+1}2$, and $h_- + \inf h,  h_+ - \sup h>0$, the following hold. 
\begin{enumerate} 
\item There exists $\lambda_0 >0$ such that $ (\lambda - \CA(h, \Phi) )^{-1} \in \BBL(X^{r-1}, X^{r})$ for any $1\le r \le n$ and $\lambda> \lambda_0$. 
\item If $d_2 =0$ in \eqref{E:fluidD-0} or \eqref{E:coercivity-2F} is satisfied, then the above $\CL(h, \Phi)$ satisfies all properties in Lemma \ref{L:CWW-CL}.
\end{enumerate}  
\end{lemma} 

\begin{remark} \label{R:2F}
a.) The above statement (1) does not require \eqref{E:coercivity-2F}. In fact, the same proof as in Section \ref{SS:CWW-proof} yields that statement (1) holds for $\wt \CA = J \BD^2 \wt \BH = J \BD^2 ( \BH + a |h|_{L^2}^2)$. Hence it also holds for $\CA$ which is a bounded perturbation of $\wt \CA$. \\
b.) The two terms in $\CL_{01}(h)$ with possible negative contributions are i.) $g(\rho_- - \rho_+)$ which represents the Rayleigh-Taylor instability if the heavier fluid on top of the lighter fluid, and ii.) the first term involving $\nu_\pm$. The latter is removed if $\nu_+=\nu_-$ and  the system is put in a moving frame, otherwise it corresponds to the Kelvin-Helmholtz instability. 
As in Remark \ref{R:Vc-2}, \eqref{E:coercivity-2F} is satisfied by $|h|_{Z_1^{r_0}} \ll1 $ if it holds for $\CL_{01}(0)$. One can compute   
\begin{align*}
\langle \CL_{01}(0) (\eta, \eta) \rangle = \int_\CU \Big( \sigma |\xi|^2 + g (\rho_- - \rho_+) - \sum_\pm \frac {\rho_\pm (\nu_\pm \cdot \xi )^2}{ |\xi| \tanh (h_\pm |\xi|)} \Big) |\hat \eta(\xi)|^2 dx'. 
\end{align*}
Therefore we obtain 
\be \label{E:KH-1}
\CL_{01} (0) \ge \min_{\tau \in [0, \infty)} \Big(  \sigma \tau^2 + g (\rho_- - \rho_+) - \sum_\pm \frac {\rho_\pm |\nu_\pm|^2 \tau}{ \tanh (h_\pm \tau)}\Big). 
\ee
In particular, if $h_\pm =\infty$, we have  
\be \label{E:KH-2}
\CL_{01} (0) \ge g (\rho_- - \rho_+) -  \big( \rho_+ |\nu_+|^2+ \rho_- |\nu_-|^2 \big)^2 /(4 \sigma). 
\ee
This calculation is consistent with the analysis of the classical Kelvin-Helmholtz instability, see, for example, \cite{DR04, BSWZ16}. 
\end{remark} 

With the above preparation we can obtain the following  theorems through exactly the same proofs. 

\vspace{0.08in} \noindent $\bullet$ {\bf Local well-posedness.} 
The same results as in Theorem \ref{T:CWW-LWP-1} and Proposition \ref{P:continuation} hold for solutions $(h, \Phi)$ of the fluid interface problem. Note that \eqref{E:coercivity-2F} is not required for the these results even if $d_2>0$. In fact, instead of going through Proposition \ref{P:LHam} and Proposition \ref{P:CWW-ET-1}, it is easier to verify assumption (B.3) (and thus (B.1)--(B.5)) directly using Lemma \ref{L:2F-a}(2). See also the remark on the well-posedness in the example of nonlinear wave type equations. 

\vspace{0.08in} \noindent $\bullet$ {\bf Local invariant manifolds.} 
The same results as in Proposition \ref{P:CWW-ET-1} and Theorem \ref{T:CWW-LInMa-main}  hold for  the fluid interface problem.  

Among various spectrally instability of equilibria, the most thoroughly understood is the trivial equilibrium, unlike the capillary gravity waves, which could be unstable. 

$\bullet$ {\it Unstable manifolds due to the Kelvin-Helmholtz or Rayleigh-Taylor  instability.} 
Clearly $(h_*=0, \Phi_*=0)$ is an equilibrium where the fluid flows under the background velocities separated by the flat interface. With surface tension,  it is subject to a.) the Rayleigh-Taylor  instability for long waves (small $|\xi|$)  if $\rho_- < \rho_+$ and b.) the Kelvin-Helmholtz instability for intermediate wave length if $|\nu_+ - \nu_-|$ are large relative to $g(\rho_- -\rho_+)$ and $\sigma$. See \cite{DR04} and the calculations \ref{E:KH-1} and \eqref{E:KH-2} in Remark \ref{R:2F}. If $d_2=0$ in \eqref{E:fluidD-0} (the fluids are periodic horizontally), then the above theorem yields the existence of smooth local unstable manifolds. 

$\bullet$ {\it Unstable manifolds of periodic or solitary waves of fluid interfaces.} 
A good survey on the existence of such steady interfacial waves can be found in \cite{HHSTWW22}, where local and global bifurcation method plays an important role. Some discussions on their stability can be found, e.~g.~, in \cite{CW22}. Again spectral instability in the periodic cases or in the case of $d_2>0$ along with  leads to the existence of smooth local unstable manifolds, and thus nonlinear instability. Very often \eqref{E:coercivity-2F} is satisfied automatically if the steady wave is constructed from a bifurcation approach.

\appendix

\section{A preliminary linear problem and the local well-posedness} \label{S:pre-LWP}

In Appendix \ref{SS:Linear}, following the framework as in \cite{Ka73, Ka75} and Chapter 5 of \cite{Pazy83}, we consider a basic non-autonomous linear evolution system and obtain detailed estimates convenient for the paper.  As a byproduct 
of the analysis, we give the local well-posedness of a model quasilinear PDEs in Appendix \ref{SS:QLPDE-LWP} and then that of a class of more nonlinear PDEs in Appendix \ref{SS:NLPDE-LWP}, where certain smooth dependence on the initial data (with some expected loss of regularity) is also given in Appendix \ref{SSS:smoothness-NLPDE}. The method of the proofs is not entirely new, essentially an abstract formulation of proofs performed to many concrete nonlinear PDEs, but it is tailored into a form which paves the road for the construction of invariant manifolds in Section \ref{S:LInMa}. Some concrete PDE systems including some nonlinear evolution PDEs based on certain energies, such as the quasilinear heat equation, quasilinear Schr\"odinger equation, quasilinear wave equations, the MMT equation, {\it etc.}, are discussed in Section \ref{S:examples}.

\subsection{A preliminary linear system} \label{SS:Linear} 

For $T_0< T_1$, consider a linear equation
\be \label{E:linear-1} 
v_t = A(t) v, \quad t\in [T_0, T_1].
\ee
We assume that there exist 
\[
n_0 \in \N, \; L: [T_0, T_1] \to \BBL(X, X^*), \; Q \in W^{1, 1} \big([T_0, T_1], \BBL(X^{1}, X)\big), \; C_L, C_Q, \ge 1,  \; \omega, C_1 \in \R, 
\]
such that the following hold.
\begin{enumerate} 
\item [(L.1)] For any $t \in [T_0, T_1]$, $L(t)^* = L(t)$, 
\be \label{E:coercivity-1}
C_L^{-1} |v|_{X} \le |v|_{L(t)}  \le C_L |v|_{X}, \; \text{ where } \; |v|_{L(t)} \triangleq  \langle L(t) v, v \rangle^{\frac 12}, 
\ee
and $\sup_{T_0\le t_0 \le t_1 \le T_1} l(t_0, t_1) < \infty$, where  
\begin{align*}
l(t_0, t_1) \triangleq \sup \Big \{ \prod_{j=1}^{j_0} \frac {|v_j|_{L(s_j)}}{|v_{j}|_{L(s_{j-1})}} :  j_0 \in \N, \  t_0 = s_0  \le s_1 & \le \ldots \le s_{j_0} = t_1, \\
& v_1, \ldots, v_{j_0} \in X \setminus\{0\} \Big\}.
\end{align*} 
\item [(L.2)] For any $t \in [T_0, T_1]$, the domain $Dom(A(t))$ of $A(t)$ contains $X^1$, $A(t): Dom(A(t))\to X$ is the closure of $A(t)|_{X^1}$, $\lambda - A(t): Dom (A(t)) \to X$ is surjective for some $\lambda > \omega$, and $A$ also satisfies 
\be  \label{E:dissipativity-1}
A \in C^0 \big([T_0, T_1], \BBL(X^{1}, X^0)\big), 
\quad \langle L(t) v, A (t)v \rangle \le \omega \langle L(t) v, v \rangle, \quad \forall v \in X_1. 
\ee
\item [(L.3)] For any $1\le n \le n_0$, $Q \in W^{1,1} \big([T_0, T_1], \BBL(X^{n}, X^{n- 1})\big)$, $Q(t) \in \BBL(X^{n}, X^{n- 1})$ is an isomorphism for any $t \in [T_0, T_1]$, and 
\be \label{E:S-1}
|Q|_{C_t^0 \BBL(X^{n}, X^{n-1})}, \ |Q^{-1}|_{C_t^0\BBL(X^{n-1}, X^{n})} \le C_Q. 
\ee
Moreover, the commutator $[Q, A]\in C^0 \big( [T_0, T_1], \BBL(X^{n}, X^{n-1})\big)$. 
\end{enumerate} 

Let us start with some general comments on the assumptions. Equation \eqref{E:linear-1} is often the linearization of a nonlinear PDE \eqref{E:NLPDE-0} along a solution $u(t)$. When the principal part of the nonlinear PDE 
is a gradient or Hamiltonian flow, the space $X$ is often taken as the energy space. 
The symmetric bounded operator $L(t)\in \BBL(X, X^*)$ is usually based on the principal part of the Hessian of the energy and defines an equivalent inner product on $X$. The finiteness of the auxiliary function $l(t_0, t_1)$ measures the total variation of the norm defined by $L(t)$ and it is easy to prove that, for $T_0 \le t_0 \le t_1 \le t_2 \le T_1$, 
\be \label{E:temp-1}
l(t_0, t_1) l(t_1, t_2)=l(t_0, t_2) \ge C_L^{-2},
\ee
and, if $L \in W^{1,1} \big([T_0, T_1], \BBL(X, X^*)\big)$, then 
\be \label{E:l-1}
l(t_0, t_1) \le e^{\frac 12 \int_{t_0}^{t_1} \sup_{v\ne 0} \frac {\langle L'(\tau)v, v\rangle}{\langle L(\tau)v, v\rangle} d\tau} \le e^{\frac 12 C_L^2 \int_{t_0}^{t_1} |L'(\tau)| d\tau} = e^{\frac 12 C_L^2 |L'|_{L^1 ([t_0, t_1], \BBL(X, X^*))}}.  
\ee

Assumptions (L.2) include the dissipatitivity of $A(t)$ with respect to the equivalent inner product given by $L(t)$, which basically means that \eqref{E:linear-1} satisfies the energy estimate with the energy given by $L(t)$. Along with the surjectivity of $\lambda - A(t)$, the Lumer-Phillips Theorem (see e.~g.~Theorem 4.3 in Chapter 1 of \cite{Pazy83}) implies that, for any $t \in [T_0, T_1]$, $e^{\tau A(t)}$ is a well-posed semigroup of bounded linear operators satisfying the estimate 
\be \label{E:semiG-1} 
|e^{\tau A(t)} v|_{L(t)} \le e^{\omega \tau} |v|_{L(t)}, \quad \forall \tau \ge 0, \; v \in X. 
\ee

The operator $Q(t)$ in (L.3) is a differential operator for which the commutator $[Q, A]$ does not cost additional regularity. In many cases,  
$Q(t)$ can be taken closely related to the principle part of $\lambda - A(t)$. However, the assumptions do not exclude the possibility that $Q(t)$ is a higher order differential operator than $A(t)$. 
We shall apply $Q(t)$ to \eqref{E:linear-1}  to obtain higher order estimates of the solutions. In fact, from 
\be \label{E:comm-1}
Q^n A =A Q^n + \sum_{j=1}^{n} Q^{j-1} [Q, A] Q^{n-j}, 
\ee
we obtain inductively, for any integer $1\le n \le n_0$, $A \in C^0 \big([T_0, T_1],  \BBL(X^{n}, X^{n-1})\big)$ and, for any $t \in [T_0, T_1]$, 
\be \label{E:A-bdd} 
|A(t)|_{\BBL(X^{n}, X^{n-1})} \le C_Q^{2(n-1)} |A(t)|_{\BBL(X^1, X^0)} +  C_Q^{2n-3} \sum_{j=1}^{n-1} | [Q, A] (t) |_{\BBL(X^{j}, X^{j-1})}. 
\ee 

The main statement of this subsection is the following 
linear well-posedness and estimates. 

\begin{proposition} \label{P:linear} 
Assume (L.1) -- (L.3) are satisfied with $n_0\ge 1$, then there exists a unique 
\be \label{E:EvoOp-1}
U(t, t_0) \in \BBL(X), \quad (t, t_0) \in \Delta_{T_0, T_1} \triangleq \{(t, t_0) : T_0 \le t_0 \le t \le T_1 \}, 
\ee
such that 
\be \label{E:EvoOp-2}
U(t, t_0) \in \BBL(X^n), \quad U(\cdot, \cdot)v \in C^0(\Delta_{T_0, T_1}, X^n), \quad \forall v\in X^n, \; 0 \le n \le n_0,
\ee
\be \label{E:EvoOp-3}
U(t_0, t_0)=I,   \quad U(t, t_1) U(t_1, t_0) = U(t, t_0),
\ee
\be \label{E:EvoOp-4}
\p_t U(t, t_0)v = A(t) U(t, t_0) v, \;\; \p_{t_0} U(t, t_0) v = - U(t, t_0) A(t_0) v, \quad \forall v \in X^1.
\ee
Moreover $U(t, t_0)$ satisfies the estimate for any $v \in X^n$, $0\le n\le n_0$, 
\be \label{E:EvoOp-5} \begin{split}
|Q(t)^n U(t, & t_0) v |_{L(t)} \\
& \le  l(t, t_0) e^{\omega (t-t_0) + C_L^2 \sum_{j=1}^n C_Q^{2j-1} | Q' + [Q,  A]|_{L^1 ([t_0, t], \BBL(X^j, X^{j-1}))}} |Q(t_0)^n v|_{L(t_0)}.
\end{split} \ee
\end{proposition} 

\begin{proof} 
We first show that $A(t)$ is a stable family of generators of $C^0$ semigroups on $X$ as in Definition 2.1 of Chapter 5 in \cite{Pazy83} (see also \cite{Ka73}). The dissipativity assumption (L.2) of $A(t)$ with respect to the equivalent inner product $\langle L(t) \cdot, \cdot\rangle$ implies the spectrum $\sigma (A(t)) \supset (\omega, + \infty)$ and that the semigroup $e^{\tau A(t)}$ satisfies \eqref{E:semiG-1} uniformly in $t \in [T_0, T_1]$.  
For any $j \in \N$, $T_0 \le t_1 \le t_2 <\ldots \le t_{j}\le T_1$, $\tau_1, \ldots \tau_{j}\ge 0$, and $v \in X$, we obtain from \eqref{E:coercivity-1} and \eqref{E:semiG-1},
\be \label{E:Pazy-1} \begin{split}
& | e^{\tau_{j}A(t_{j})} \ldots e^{\tau_1 A(t_1)} v|_{L(t_{j})} 
\le  e^{\omega \tau_j} | e^{\tau_{j-1} A(t_{j-1})} \ldots e^{\tau_1 A(t_1)} v|_{L(t_j)} \\
\le & l(t_{j-1}, t_j) e^{\omega \tau_j} | e^{\tau_{j-1} A(t_{j-1})} \ldots e^{\tau_1 A(t_1)} v|_{t_{j-1}} \le  \ldots \le  l(t_1, t_j) e^{ \omega (\tau_1+ \ldots + \tau_j)} | v|_{L(t_1)}. 
\end{split} \ee
Therefore 
\[
| e^{\tau_{j}A(t_{j})} \ldots e^{\tau_1 A(t_1)}|_{\BBL(X)}  \le C_L^2 l(t_1, t_j) e^{ \omega (\tau_1+ \ldots + \tau_j)},
\]
and it proves the stability according  to Definition 2.1  and Theorem 2.2 in Chapter 5 of \cite{Pazy83}. 

Since, for any $t \in [T_0, T_1]$,  
\[
|Q(t) A(t) Q(t)^{-1} - A(t)|_{\BBL(X)} = |[Q(t), A(t)] Q(t)^{-1} |_{\BBL(X)} \le C_Q \sup_{t\in [T_0, T_1]} |[Q, A]|_{\BBL(X^1, X)},  
\]
Theorem I in \cite{Ka73} (as well as Theorem 4.6 in Chapter 5 of \cite{Pazy83}) yields a unique $U(t, t_0) \in \BBL(X)$ satisfying  \eqref{E:EvoOp-1}, \eqref{E:EvoOp-2} for $n=0, 1$, \eqref{E:EvoOp-3}, \eqref{E:EvoOp-4}, and the same estimate as in \eqref{E:Pazy-1} 
\be \label{E:temp-2} 
|U(t, t_0)v|_{L(t)} \le l(t_0, t) e^{ \omega (t-t_0)} |v|_{L(t_0)}, \ \forall v\in X. 
\ee 

Next we shall prove \eqref{E:EvoOp-2} 
and \eqref{E:EvoOp-5} for general $n$. 
Consider the integral equation 
\be \label{E:higher-O-1}
U^n(t, t_0) = U^{n -1} (t, t_0) + \int_{t_0}^t U^{n}(t, \tau) \Big(Q^{n-1}  \big(Q' + [Q,  A]\big) Q^{-n} \Big)\Big|_{\tau} U^{n-1} (\tau, t_0) d\tau, 
\ee
where $1\le n \le n_0$ and $U^0 = U$ is understood. Through a standard iteration procedure using assumption (L.3), inductively one may prove that this integral equation has a unique solution $U^n(t, t_0) \in \BBL(X)$ which is 
 strongly $C^0$ in $\Delta_{T_0, T_1}$ 
and $U^n(t_0, t_0) =I$. Moreover, we prove that they satisfy the estimate, for any $v \in X$, 
\be \label{E:temp-3}
|U^n(t, t_0)v|_{L(t)} \le l(t_0, t) e^{\omega (t-t_0) + C_L^2 \sum_{j=1}^n C_Q^{2j-1} | Q' + [Q,  A]|_{L^1 ([t_0, t], \BBL(X^j, X^{j-1}))}} |v|_{L(t_0)}.  
\ee
In fact, let 
\[
f_n(t, t_0) =  \log l(t_0, t) + \omega (t-t_0) +C_L^2 \sum_{j=1}^n C_Q^{2j-1} | Q' + [Q,  A]|_{L^1 ([t_0, t], \BBL(X^j, X^{j-1}))}, 
\]
then \eqref{E:temp-3} is equivalent to 
\be \label{E:temp-4}
e^{-f_{n} (t, t_0)} |U^n(t, t_0)|_{\BBL((X, |\cdot|_{L(t_0)}), (X, |\cdot|_{L(t)}))} \le 1.
\ee
We also notice from \eqref{E:temp-1}  
\[
f_n(t_2, t_1) + f_n (t_1, t_0)= f_n (t_2, t_0), \quad \forall T_0 \le t_0\le t_1\le t_2 \le T_1. 
\]
For $n=0$, the desired estimate \eqref{E:temp-4} is exactly \eqref{E:temp-2}. For $n \ge 1$, we obtain from \eqref{E:higher-O-1}  
\begin{align*}
e^{-f_{n-1} (t, t_0)} & |U^n(t, t_0) |_{\BBL((X, |\cdot|_{L(t_0)}), (X, |\cdot|_{L(t)}))}  \le e^{-f_{n-1} (t, t_0)}|U^{n-1}(t, t_0)|_{\BBL((X, |\cdot|_{L(t_0)}), (X, |\cdot|_{L(t)}))}  \\
& + \int_{t_0}^t C_L^2 C_Q^{2n-1} | Q'(\tau) + [Q,  A](\tau)|_{\BBL(X^n, X^{n-1})} e^{-f_{n-1} (t, \tau)}   \\
& \times |U^{n}(t, \tau)|_{\BBL((X, |\cdot|_{L(\tau)}), (X, |\cdot|_{L(t)}))} e^{-f_{n-1} (\tau, t_0)} |U^{n-1}(\tau, t_0)|_{\BBL((X, |\cdot|_{L(t_0)}), (X, |\cdot|_{L(\tau)}))} d\tau.  
\end{align*}
Therefore \eqref{E:temp-4} follows  inductively from the Gronwall inequality applied to $e^{-f_{n-1} (t, t_0)} |U^n(t, t_0) |_{\BBL(X)}$. 

From (4.14)--(4.15) in Chapter 5 \cite{Pazy83}, or by directly computing $\p_{t_0} (U(t, t_0) Q(t_0)^{-1})$ and using the second equality in \eqref{E:EvoOp-4}, 
we have 
\[
Q(t) U (t, t_0) Q(t_0)^{-1} = U^1(t, t_0) = U (t, t_0) + \int_{t_0}^t U^1 (t, \tau) \big( \big(Q' + [Q,  A]\big) Q^{-1} \big)\big|_{\tau} U (\tau, t_0) d\tau.
\]
Inductively, suppose, for some $1\le n < n_0$ it holds 
\be \label{E:temp-5}
 Q(t)^{n'} U (t, t_0) Q(t_0)^{-n'}=U^{n'} (t, t_0), \quad \forall 0\le n' \le n. 
\ee 
Applying $Q(t)^{-1}$\eqref{E:higher-O-1} for $n+1$ to  $Q(t_0)v$ where $v \in X^1$, we obtain 
\begin{align*}
& Q(t)^{-1} U^{n+1}  (t, t_0) Q(t_0) v= Q(t)^{-1} U^n (t, t_0) Q(t_0)v \\
&\qquad \qquad \qquad + \int_{t_0}^t Q(t)^{-1} U^{n+1}(t, \tau) \Big(Q^{n}  \big(Q' + [Q,  A]\big) Q^{-n-1} \Big)\Big|_{\tau} U^{n} (\tau, t_0) Q(t_0)v d\tau \\
=& U^{n-1} (t, t_0) v + \int_{t_0}^t Q(t)^{-1} U^{n+1}(t, \tau) Q(\tau) \Big(Q^{n-1}  \big(Q' + [Q,  A]\big) Q^{-n} \Big)\Big|_{\tau}  U^{n-1} (\tau, t_0) vd\tau. 
\end{align*}
By the uniqueness of solutions to \eqref{E:higher-O-1}, we have   
\[
Q(t)^{-1} U^{n+1} (t, t_0) Q(t_0) v - U^n (t, t_0) v =0, \quad \forall v \in X^1,
\]
which implies  
\[
Q(t)^{n+1} U(t, t_0) Q(t_0)^{-n-1} = Q(t) U^{n} (t, t_0) Q(t_0)^{-1} = U^{n+1} (t, t_0) \in \BBL(X).
\]
Hence \eqref{E:temp-5} holds for all $0\le n\le n_0$ and thus \eqref{E:EvoOp-2} follows. The estimate \eqref{E:EvoOp-5} is obtained from \eqref{E:temp-3}. 
\end{proof}

\subsection{Local well-posedness of a model quasilinear PDEs} \label{SS:QLPDE-LWP}

Consider  
\be \label{E:QLPDE-0} 
v_t = \BA(v) v  + f(v), \quad v(0) = v_0. 
\ee
Here $A(v)$ is a leading order linear term and $f(v)$ is a lower order nonlinearity. 
We assume that there exist 
\[
k \ge 1, \; \omega \in \R,  \; C_{L}, C_Q \ge 1,   \; C_0, \delta_0, C_{f, 0}, C_{f, 1}>0, \; v_* \in X^k, \; R_0 > C_L^2C_Q^{2k} |v_*|_{X^k}, 
\]
\[
\BL\in C^{1} \big(X^{k-1}(v_*, \delta_0), \BBL(X, X^*)\big), \; \BQ \in C^1 \big(X^{k-1}(v_*, \delta_0), \BBL(X^{1}, X)\big)
\]
such that the following are satisfied.  
\begin{itemize} 

\item [(H.1)] For any $v \in X^{k-1}(v_*, \delta_0)$ and for any $w \in X$
\[
\BL (v) = \BL (v)^*, \quad      
C_{L}^{-1} |w|_{X} \le  |w|_{\BL (v)} \le C_{L} |w|_{X}, \quad |\BD\BL |_{C^0 (X^{k-1}(v_*, \delta_0), \BBL (X^{k-1} \otimes X, X^*))} \le C_0,
\]
where $|w|_{\BL (v)} = \sqrt{\langle \BL (v) w, w \rangle}$ as in \eqref{E:coercivity-1}. 

\item [(H.2)] For any $v \in  X^{k-1}(v_*, \delta_0)$, $Dom (\BA(v))$ of $\BA(v)$ contains $X^1$, $\BA(v): Dom(\BA(v)) \to X$ is the closure of $\BA(v)|_{X^1}$, and $\lambda - \BA(v): Dom (\BA(v)) \to X$ is surjective for some $\lambda > \omega$. Moreover $\BA$ also satisfies, for $1\le r \le k$,    
\be \label{E:BA-1} 
\BA \in C^{0, 1} \big(X^{k-1}(v_*, \delta_0), \BBL(X^{r}, X^{r-1})\big), \quad 
|\BA|_{ C^{0,1}(X^{k-1}(v_*, \delta_0), \BBL(X^{r}, X^{r-1}))} \le C_0,  
\ee
and 
\be  \label{E:dissipativity-2}
\langle \BL (v)w, \BA (v)w \rangle \le \omega \langle \BL (v)w, w \rangle, \quad \forall w \in X^1. 
\ee

\item [(H.3)] For any  $v \in X^{k-1}(v_*, \delta_0)$, $\BQ(v)  \in \BBL(X^r, X^{r-1})$ is an isomorphism for any $1\le r\le k$ and, on the domain $X^{k-1}(v_*, \delta_0)$, $\BQ$ also satisfies 
\be \label{E:S-2} \begin{split}
&|\BQ|_{C^0  \BBL(X^{r}, X^{r-1})}, \ |\BQ^{-1}|_{C^0\BBL(X^{r-1}, X^{r})} \le C_Q, \\
& |\BD \BQ|_{C^0  \BBL(X^{k-1} \otimes X^{r}, X^{r-1})}, \ |\BD \BQ^{-1}|_{C^0 \BBL(X^{k-1} \otimes X^{r-1}, X^{r})}, \ |[\BQ, \BA]|_{C^0 \BBL(X^{r}, X^{r-1})}  \le C_0.
\end{split}\ee

\item [(H.4)] 
Assume $f(v)$ satisfies 
\[
|f|_{C^0 (X^{k-1} (v_*, \delta_0) \cap X^k (R_0), X^{k})} \le C_{f, 0}, 
\]
\[
|f(v_1) - f(v_2)|_{X^{k-1}} \le C_{f, 1} |v_1-v_2|_{X^{k-1}},  \quad \forall v_1, v_2 \in X^{k-1} (v_*, \delta_0) \cap X^k (R_0). 
\]
\end{itemize}

\begin{remark} \label{R:temp-1}
We notice that $\BA(v)$ and $\BQ(v)$ are defined for $v \in X^{k-1}$, but act on $X^k$. This happens when they are derived from differentiating/quasilinearizing the original nonlinear PDE, see Appendix \ref{SS:NLPDE-LWP}. Such assumption is also consistent with assumption (7.4) in \cite{Ka75}. 

Assumptions (H.1) and (H.3) imply  $|\BQ(v)^r w|_{\BL(v)}$, $\forall v\in X^{k-1} (v_*, \delta_0)$, is an equivalent metric of $w \in X^r$ for $0\le r \le k$ with 
\be \label{E:temp-7}
C_L^{-1} C_Q^{-r}  |w|_{X^r} \le |\BQ(v)^r w|_{\BL(v)} \le C_L C_Q^{r}  |w|_{X^r}.  
\ee
The mixed usage of both $|w|_{X^r}$ and $|\BQ(v)^r w|_{\BL(v)}$ leads to the assumption $R_0> C_L^2 C_Q^{2k} |v_*|_{X^k}$.   
\end{remark}

The main statement of this subsection is the local well-posedness of \eqref{E:QLPDE-0} under the above conditions, in the spirit of Theorem 6 in \cite{Ka75}. It is not in the optimal form, but sufficient to be applied in Appendix \ref{SS:NLPDE-LWP}. Certain smooth dependence on the initial data will be given in Appendix \ref{SSS:smoothness-NLPDE} for a more general class nonlinear PDEs \eqref{E:NLPDE-0}. 

\begin{theorem} \label{T:QLPDE-LWP}
Assume (H.1)--(H.4), then for any $\delta\in (0, \delta_0)$ and $R \in (|v_*|_{X^k}, C_L^{-2} C_Q^{-2k} R_0)$, there exist constants $T>0$ determined by $\delta$, $R$, and the constants in (H.1--H.4) 
such that, for any initial value $v_0 \in X^{k-1}(v_*, \delta) \cap X^k (R)$, \eqref{E:QLPDE-0} has a unique solution $v(t) \in X^{k-1}(v_*, \delta_0) \cap X^k (R_0)$, $t \in [0, T]$, so that $v \in C_t^0 X^k \cap C_t^1 X^{k-1}$.
Moreover, let $v_1(t)$ and $v_2(t)$ be solutions with initial values $v_{01}$ and $v_{02}$, respectively, then 
\[
|v_2 - v_1|_{C_t^0 X^{k-1}} \le 4 C_L^2C_Q^{2(k-1)} |v_{02}- v_{01}|_{X^{k-1}}.
\]
\end{theorem}

The theorem will be proved by a fixed point argument of a transformation $\CT (v_0, v(\cdot))$ of functions $v(\cdot) \in Y$ with initial value $v_0$ 
as a parameter where  
\[\begin{split}
Y = \{ v (\cdot)  \in L^\infty ([0, T],  X^{k}) \cap & W^{1, \infty} ([0, T], X^{k-1}) :  |v -v_* |_{L_t^\infty X^{k-1}} \le (\delta_0 +\delta)/2, \\ 
&|v'|_{L_t^\infty X^{k-1}} \le M_1, \ |v(t)|_{L_t^\infty X^k} \le  (C_L^{2} C_Q^{2k} R+R_0)/2
\}. 
\end{split}\]
The constants $T, M_1>0$ will be determined later. 

For $v \in Y$, let  
\be \label{E:temp-7.5}
A (t) = \BA (v(t)), \quad L (t) = \BL (v(t)), \quad Q(t) = \BQ(v(t)). 
\ee
Clearly, $L \in W_t^{1, \infty} \BBL(X, X^*)$, $A \in C_t^{0, 1}\BBL(X^{r}, X^{r-1})$, and $Q \in W_t^{1, \infty}\BBL(X^{r}, X^{r-1})$, for $1\le r \le k$. Moreover, for  a.~e.~$t\in [0, T]$, 
\[
|L'(t) |_{\BBL(X, X^*)}, \ |Q'(t) |_{\BBL(X^{r}, X^{r-1})} \le C_0 |v' (t)|_{X^{k-1}}. 
\]
From \eqref{E:l-1}, assumptions (L.1)--(L.3) given in Appendix \ref{SS:Linear} are satisfied by $L(t)$, $A(t)$, and $Q(t)$ with $n_0=k$ and 
\be \label{E:l-2}
l(t_0, t_1) \le e^{\frac 12 C_0 C_L^2 \int_{t_0}^{t_1} |v'(\tau)|_{X^{k-1}} d\tau}.  
\ee
Let  $U (t, t_0) \in \BBL(X^r)$, $ 0 \le t_0 \le t \le T$ and $0\le r\le k$, be the strongly $C^0$ evolution operator on $X^r$ generated by $A (t)$ due to Proposition \ref{P:linear}. Using \eqref{E:S-2} and the definition of $Y$, its estimate becomes  
\be \label{E:U-bdd-1} \begin{split}
& |Q(t)^r U (t, t_0) v |_{L(t)} \\
\le & e^{\omega (t-t_0) + C_L^2 \int_{t_0}^t  \frac 12 C_0 |v'|_{X^{k-1}} + \sum_{j=1}^r C_Q^{2j-1} (C_0|v'|_{X^{k-1}} + |[Q, A]|_{\BBL(X^{j}, X^{j-1})}) d\tau} |Q(t_0)^r v|_{L(t_0)} \\
\le & e^{(\omega  + C_0 C_L^2 (\frac 12 M_1 + r C_Q^{2r-1}(1+M_1)))  (t-t_0)} |Q(t_0)^r v|_{L(t_0)} \triangleq e^{C_1 (t-t_0)} |Q(t_0)^r v|_{L(t_0)},
\end{split} \ee
for any $v \in X^r$. Hence we also obtain from \eqref{E:temp-7} 
\be \label{E:U-bdd-2} \begin{split}
& |U (t, t_0) |_{\BBL(X^r)} \le C_L^2 C_Q^{2r} e^{C_1 (t-t_0)}. 
\end{split} \ee

Clearly $v \in Y$ is a solution to \eqref{E:QLPDE-0} with $v(0) = v_0$ iff $\CT(v_{0}, v) = v$ where  
\be \label{E:CT-1}
\CT(v_0, v)(t) =\wt v (t) = U (t, 0) v_{0} + \int_{0}^t U (t, \tau) f (v(\tau)) d\tau, \;\;  t\in [0, T],
\ee
which is the solution to 
\be \label{E:QLPDE-iter-1}
\wt v_t = A(v) \wt v + f(v), \quad \wt v(0)=v_0. 
\ee





\begin{lemma} \label{L:Contraction-QLPDE-1}
Let 
\[
M_1= C_0 C_L^2 C_Q^{2k} (2R +1)+C_{f, 0} + C_{f, 1} \delta_0. 
\]
There exists $T>0$ determined by $\delta$, $R$, and the constants in (H.1--H.4) 
such that $\CT: \big(X^{k-1}(v_*, \delta) \cap X^k (R) ) \big) \times Y \to Y$ and satisfies 
\[
|\CT(v_{01}, v_2) - \CT (v_{02}, v_1)|_{L_t^\infty X^{k-1}} \le 2 C_L^2C_Q^{2(k-1)} |v_{02}- v_{01}|_{X^{k-1}} + (1/2) |v_2 - v_1|_{L_t^\infty X^{k-1}}. 
\]
\end{lemma}

\begin{proof} 
For $v_0 \in X^{k-1}(v_*, \delta) \cap X^k (R)$ and $v\in Y$, let $\wt v =\CT(v_0, v)$, then one may estimate using \eqref{E:CT-1}, \eqref{E:U-bdd-2},  and assumption (H.4), for $t \in [0, T]$, 
\begin{align*} 
| \wt v (t) |_{X^k} \le & C_L^2 C_Q^{2k} e^{C_1 t}|v_0 |_{X^k} + C_L^2 C_Q^{2k}  C_{f, 0} \int_0^t e^{C_1 (t-\tau)}
d\tau 
\le C_L^2 C_Q^{2k} \big( e^{C_1 t} R + 
C_{f, 0} t (1+ e^{C_1 t}) \big),
\end{align*}
where $C_1$ was defined in \eqref{E:U-bdd-1}.
Moreover, from \eqref{E:QLPDE-iter-1}, 
\begin{align*}
|\p_t \wt v (t) |_{X^{k-1}} = & |A (t) \wt v (t) + f (v(t))|_{X^{k-1}} \le C_0 |\wt v (t) |_{X^k}
+ |f(v_*)|_{X^{k-1}} + C_{f, 1} \delta_0.
\end{align*}
Since $R < C_L^{-2} C_Q^{-2k} R_0$ and $\delta < \delta_0$, by the choice of $M_1$, there exists $T>0$ such that $\wt v \in Y$. 


To obtain the Lipschitz estimate of  $\CT$ in $v$, let $v_j \in Y$, $j=1,2$, and $U_j(t, t_0)$ be the evolution operators generated by $v_j$, respectively. Let  
\[
\wt w = \wt v_2 - \wt v_1 \triangleq \CT(v_0, v_2) - \CT(v_0, v_1), 
\]
which, due to \eqref{E:QLPDE-iter-1}, satisfies 
\[
\p_t \wt w = \BA ( v_1) \wt w + \big(\BA(v_2) - \BA(v_1) \big) \wt v_2 + f (v_2) - f (v_1), \quad \wt w(0)=0, 
\]
and thus 
\be \label{E:temp-8}
\wt w (t) =  \int_0^t U_{1} (t, \tau) \Big( \big(\BA(v_2(\tau)) - \BA(v_1(\tau)) \big)  \wt v_2(\tau)  + f (v_2(\tau) ) - f (v_1(\tau) ) \Big) d\tau.
\ee
Therefore \eqref{E:U-bdd-2}, \eqref{E:temp-7}, the definitions of $M_1$ and $C_1$, and the fact $\wt v_2 \in Y$ imply 
\begin{align*} 
|\wt w (t)|_{X^{k-1}} \le & 
\int_{0}^t |U_1 (t, \tau)|_{\BBL(X^{k-1})} \Big( C_0 |\wt v_{2} (\tau)|_{X^{k}}  + C_{f, 1} \Big) |v_2 (\tau) - v_1(\tau)|_{X^{k-1}}d\tau  \\
\le & C_L^2 C_Q^{2(k-1)}   
\int_0^t  e^{C_1 
(t-\tau)}   \Big(  M_1 + C_{f,1}\Big) |v_2 (\tau) - v_1(\tau)|_{X^{k-1}}  d\tau \\
\le & C_L^2 C_Q^{2(k-1)} (1+e^{C_1 T} ) (M_1 + C_{f, 1}) T |v_2  - v_1|_{L_t^\infty X^{k-1}}. 
\end{align*}
The  Lipschitz estimate of $\CT$ follows by taking $T>0$ small and also using \eqref{E:U-bdd-2}. 
\end{proof} 

We are ready to prove Theorem \ref{T:QLPDE-LWP} of the local well-posedness of  \eqref{E:QLPDE-0}.

\begin{proof} [Proof of Theorem \ref{T:QLPDE-LWP}] 
For any $\delta\in (0, \delta_0)$ and $R \in (0, C_L^{-2} C_Q^{-2k} R_0)$, choose 
$M_1, T>0$ according to Lemma \ref{L:Contraction-QLPDE-1}.
For any $v_0 \in X^{k-1}(v_*, \delta) \cap X^k (R)$,  let $v^{(n)} = \big(\CT(v_0, \cdot)\big)^{(n)} v^{(0)} \in Y$ be the iteration sequence starting with any initial $v^{(0)} \in Y$. Lemma \ref{L:Contraction-QLPDE-1} implies that $(v^{(n)})$ is a Cauchy sequence in $C^0 ([0, T], X^{k-1})$. Hence there exists $v \in C^0 ([0, T], X^{k-1})$ such that $v^{(n)} \to v$ in $C^0 ([0, T], X^{k-1})$. Moreover, since $v^{(n)}$ is bounded in $L^\infty ([0, T], X^{k}) = \big( L^1 ([0, T], (X^{k})^*) \big)^*$ and $v_t^{(n)}$ bounded in $L^\infty ([0, T], X^{k-1}) = \big( L^1 ([0, T], (X^{k-1})^*) \big)^*$, there exist subsequences of $( v^{(n)})$ weak-* convergent in $L^\infty ([0, T], X^{k})$ with limits in $Y$. Hence all of these subsequences have to  converge to the above strong limit $v$ and thus we obtain $v \in Y$ and $\CT(v_0, v) =v$. The fact $v$ being a fixed point of $\CT(v_0, \cdot)$ and the definition \eqref{E:CT-1} along with \eqref{E:QLPDE-iter-1} imply that $v \in X^k$ and $v_t \in X^{k-1}$ are continuous in $t$. The uniqueness of the fixed point and its Lipschitz dependence on $v$ are obvious from the contraction estimate of $\CT$ in $v$. 
\end{proof}

\subsection{Local well-posedness of a general class nonlinear PDEs} \label{SS:NLPDE-LWP}

Consider 
\be \label{E:NLPDE-0} 
u_t = F(u), \quad u(0)=u_0.
\ee
Conceptually we shall differentiate this equation in $t$ and study the resulted quasilinear equation in the form of \eqref{E:QLPDE-0} of the new unknown $v= u_t = F(u)$. However, $F(u)$ may not be invertible. Instead the quasilinearization will by realized by the local diffeomorphism 
\be \label{E:CB-1}
\CB(u) = F(u) - \omega_* u,
\ee
for some $\omega_* \in \R$. 

We assume that there exist 
\[
\omega_*, a \in \R, \;  C_\CL, C_* \ge 1, \; 2 \le n \in \N,  \text{ an open } \; \CO \subset X^{n-1},  
\text{ and } \; \CL \in C^1 \big(\CO, \BBL(X, X^*)\big), 
\]
such that the following hold, where 
\[
\CO_n \triangleq \CO \cap X^n \; \text{ and } \;  \CO_n (R) \triangleq \{ u\in \CO_n \mid |u|_{X^n} \le R\}
\]
are equipped with the $|\cdot|_{X^n}$ topology.  
\begin{itemize} 
\item [(B.1)]$F \in C^2 (\CO, X^{n-2}) \cap C^2 (\CO_n, X^{n-1})$ satisfies, for any $1\le r \le n-1$ and $R> 0$,
\[
\CA\triangleq \BD F \in  C^{1} (\CO, \BBL (X^r, X^{r-1})), \quad |\CA|_{C^{1} (\CO, \BBL (X^r, X^{r-1}))}
\le C_*, 
\]  
\[
|\CA|_{C^{0} (\CO_n(R), \BBL (X^n, X^{n-1}))} < \infty, \; \forall R>0.
\]

\item [(B.2)] For any $u \in  \CO$ 
and $1\le r\le n-1$, 
it holds that $\omega_* - \CA(u) \in \BBL(X^r, X^{r-1})$ is isomorphic and for any $1\le r\le n-1$ and $R >0$,  
\[
|\omega_* - \CA|_{C^1 (\CO, \BBL(X^{r}, X^{r-1}))}, \, |(\omega_* - \CA)^{-1}|_{C^1 (\CO, \BBL(X^{r-1}, X^{r}))} \le C_*. 
\]
Moreover there exists $\lambda >  C_\CL^2 \big(\omega_* + |a| C_* ( 1 + \omega_* C_*)   \big)$ such that $(\lambda - \CA(u))( X_1) = X$. 

\item [(B.3)] 
For any $u \in \CO$,  $\CL(u) = \CL (u)^*$ and 
\be  \label{E:coercivity-2} 
C_\CL^{-2} |w|_{X}^2 - a |(\omega_*   - \CA(u))^{-1} w|_X^2 \le  \langle \CL (u) w, w \rangle \le C_\CL^2 |w|_{X}^2, \quad \forall w \in X, 
\ee
\[ 
|\BD\CL |_{C^0 (\CO, \BBL (X^{n-1} \otimes X, X^*))}  \le C_*,
\]
\be  \label{E:dissipativity-2.5} 
\langle \CL (u) w , \CA (u)w \rangle \le  \omega_* |w|_X^2, \quad \forall w\in X^1. 
\ee

\item [(B.4)] For any 
$R >  |v_*|_{X^{n-1}}$ where $v_*= \CB(u_*)$, there exists $\delta >0$ such that $\CB^{-1} \in  C^{0} \big(X^{n-2} (v_*, \delta) \cap X^{n-1} (R), \CO_n \big)$ and 
\[
|\CB^{-1}|_{C^{0} (X^{n-2} (v_*, \delta) \cap X^{n-1} (R), X^{n} )} < \infty, 
\] 
where $X^{n-2} (v_*, \delta) \cap X^{n-1} (R)$ is equipped with the $X^{n-1}$ topology. 
\end{itemize}

\begin{remark} \label{R:LWP-R-1}
If \eqref{E:coercivity-2} is satisfied at some $u\in \CO$, then the continuity of $\CA$ and $\CL$ implies that they also hold in a neighborhood of $u$ in $X^{n-1}$ with slightly relaxed constants $C_\CL$ and $a$. This does not apply to the crucial dissipativity  \eqref{E:dissipativity-2.5}.

In the above set-up mainly for fully nonlinear evolutionary PDEs, we did not include a differential operator $\BQ(u)$ as in Appendix \ref{SS:QLPDE-LWP}, but assumed that $\CA(u) - \omega_*$ plays such a role. It may not always be so straight forward to derive the invertibility of $\CA(u) - \omega_*$, see Subsection \ref{SS:examples-Ham} for a general treatment for Hamiltonian PDEs. In certain problems where $Dom(\CA(u)^n)$ do not provide convenient function spaces to work with, e.~g.~the incompressible Euler equation, one may still try to use a separate differential operator $\BQ(u)$. 

From the regularity (B.1) and non-degeneracy (B.2) of $F$, 
the mapping $\CB\in C^{2}  (\CO, X^{n-2}) 
$ is a local diffeomorphism. However, the stronger assumption (B.4), which is not completely local due to the arbitrary $R>0$, is concerned with the regularity of $F$ and is needed to verify condition (H.4), after $\CB$ is applied to quasilinearize  \eqref{E:NLPDE-0}. 
If we assume $F(u_*) \in X^n$, 
then instead of (B.4), 
$(\omega_* - \CA(u_*))^{-1} \in \BBL(X^{n-1}, X^n)$ would be sufficient, see Remark \ref{R:LWP-R-2}. 
Often (B.4) can be verified directly through regularity estimates. 
Lemma \ref{L:CB-1} also provides a sufficient condition 
for (B.4). 
\end{remark} 

The main statement of this subsection is the local well-posdness of \eqref{E:NLPDE-0}. 

\begin{theorem} \label{T:NLPDE-LWP}
Assume (B.1)--(B.4), then, for any $u_* \in \CO_n$ and $R_1 >|u_*|_{X^n}$, there exist $\ep, T, M_0, M_2>0$
such that, for any $u_0 \in X^{n-1}(u_*, \ep) \cap X^n (R_1)$, \eqref{E:NLPDE-0} has a unique solution $u(t) \in X^{n-1}(u_*, 2\ep) \cap X^n (M_0)$, $t \in [0, T]$, so that $u \in C_t^0 X^n \cap C_t^1 X^{n-1}$.
Moreover, let $u_1(t)$ and $u_2(t)$ be solutions with initial values $u_{01}$ and $u_{02}$, respectively, then 
\[
|u_2 - u_1|_{C_t^0 X^{n-1}} \le M_2 |u_{02}- u_{01}|_{X^{n-1}}.
\]
Finally, for any $t \in [0, T]$, the solution map $u(t, \cdot) \in C^{1,1} (X^{n-1}(u_*, \ep) \cap X^n (R_1), X^{n-2})$. 
\end{theorem}


Further analysis of the smooth dependence of the solutions on the initial data is given in Theorem \ref{T:LWP-smoothness} in Appendix \ref{SSS:smoothness-NLPDE}. 


\begin{proof} 
Let $v_* = \CB(u_*) \in X^{n-1}$. From assumptions (B.1)--(B.2) and the Inverse Function Theorem, there exists $\delta_1>0$ 
such that $\CB^{-1} \in C^{2} \big(X^{n-2} (v_*, \delta_1), \CO\big)$. 
For $v \in X^{n-2} (v_*, \delta_1)$, let 
\be \label{E:temp-10.5} \begin{split} 
& u = \CB^{-1} (v), \quad \BA (v)= \CA(u), \quad \BQ (v) = \bar A(u) \triangleq  \BD \CB(u) = \CA(u) -\omega_*, \\ 
&\BL (v) = \CL(u) + a \big(\bar A(u)^{-1} \big)^{adj}\bar A(u)^{-1}, 
\end{split} \ee 
where  $\big(\bar A(u)^{-1} \big)^{adj} \in \BBL(X)$ is the adjoint operator of $\bar A(u)^{-1} \in \BBL(X)$. 

From assumptions (B.1)--(B.2), for any $1 \le r \le n-1$ and in the domain $X^{n-2} (v_*, \delta_1)$, it is straight forward to verify 
\[
|\BA|_{C^1 \BBL(X^{r}, X^{r-1})} \le 2C_*^2, \quad |\BQ|_{C^0 \BBL(X^{r}, X^{r-1})}, \,
|\BQ^{-1}|_{C^0\BBL(X^{r-1}, X^{r})} \le C_*, 
\]
\[
|\BD \BQ|_{C^0 \BBL(X^{n-2} \otimes X^{r}, X^{r-1})} \le C_*^2, \ |\BD \BQ^{-1}|_{C^0\BBL(X^{n-2} \otimes X^{r-1}, X^{r})}  \le C_*^4.
\]
Since $[\BQ, \BA] \equiv 0$, we obtain that assumptions \eqref{E:BA-1} and (H.3) in Appendix \ref{SS:QLPDE-LWP} are satisfied on $X^{n-2} (v_*, \delta_1)$ with 
\[
k=n-1, \; Dom(\CA(v))= X^1 \subset X, \; C_Q = C_*, \; \text{ and } \; C_0 = 2C_*^4.
\] 

From the definition of $\BL$ and (B.3), it is clear $\BL (v) = \BL (v)^*$ and 
\be \label{E:coercivity-3}
C_\CL^{-2} |w|_{X}^2 \le  \langle \BL (v) w, w \rangle \le C_\CL^2 |w|_{X}^2 + a |\bar A(\CB^{-1} (v))^{-1} w|_X^2, \;\; \forall v \in X^{n-2} (v_*, \delta_1), \; w \in X. 
\ee
Moreover, for any 
$v_1\in X^{n-2}$, one may compute
\begin{align*}
\BD \BL (v) v_1 = & \BD \CL (u) (\BD \CB^{-1}(v) v_1) - a (\bar A(u)^{-1} )^{adj} \bar A(u)^{-1} \big( \BD \CA(u) (\BD \CB^{-1}(v) v_1) \big) \bar A(u)^{-1} \\
&- a\big( \bar A(u)^{-1} \big( \BD \CA(u) (\BD \CB^{-1}(v) v_1)\big) \bar A(u)^{-1} \big)^{adj} \bar A(u)^{-1}, 
\end{align*}
where  $u = \CB^{-1} (v)$, and thus (B.1)--(B.3) yield 
\[
|\BD \BL (v) v_1|_{\BBL(X, X^*)} \le C_*^2 ( 1 + 2 |a| C_*^3 ) |v_1|_{X^{n-2}}.  
\]
The definition of $\BL$ and assumption (B.3) also imply, for any $v \in X^{n-2} (v_*, \delta_1)$, $w \in X^1$,  
\begin{align*}
\langle \BL (v)w, \BA (v)w \rangle = & \langle \CL (u)w, \CA (u)w \rangle + a \big( \bar A(u)^{-1}w, w + \omega_* \bar A(u)^{-1}w \big)_X \\
\le & \big(\omega_* + |a| C_* ( 1 + \omega_* C_*)   \big) |w|_X^2 \le C_\CL^2 \big(\omega_* + |a| C_* ( 1 + \omega_* C_*)   \big) \langle \BL (v)w, w \rangle.  
\end{align*}
Summarizing the above estimates, we obtain that assumptions (H.1)--(H.3)  in Appendix \ref{SS:QLPDE-LWP} are satisfied with 
\be \label{E:temp-11} \begin{split}
& k=n-1, \quad Dom(\CA(v))= X^1 \subset X, \quad C_Q = C_*, \quad C_L = (C_\CL^2 + |a| C_*^2 )^{\frac 12}, \\ 
&\omega = C_\CL^2 \big(\omega_* + |a| C_* ( 1 + \omega_* C_*)   \big),  \quad C_0 =2 C_*^4 (1 +  |a| C_*). 
\end{split} \ee

Through the transformation $v = \CB(u)$, one may compute that \eqref{E:NLPDE-0} becomes 
\[
v_t = \big( \CA(u) - \omega_* \big)  F(u) = \big( \CA(u) - \omega_* \big) (v + \omega_* u) 
\triangleq   \BA(v) v + f(v), 
\]
where 
\be \label{E:temp-12}
f(v) = \omega_* \big( (\BA(v) - \omega_*) \CB^{-1} (v) - v\big). 
\ee
Namely $v = \CB(u)$ satisfies an evolution equation in the form of \eqref{E:QLPDE-0} where assumptions (H.1)--(H.3)  in Appendix \ref{SS:QLPDE-LWP} have been satisfied. One may calculate, for any $v \in X^{n-2} (v_*, \delta_1)$ and $v_1\in X^{n-2}$, 
\begin{align*}
\BD f (v) v_1  = & \omega_* \big( \bar A(u) \BD \CB^{-1} (v) v_1 + \BD \CA(u) (\BD \CB^{-1} (v) v_1) \CB^{-1} (v)  - v_1\big) \\
=& \omega_*  \BD \CA(u) \big(\BD \CB^{-1} (v) v_1 \big) \CB^{-1} (v) \implies 
\end{align*}
\be \label{E:temp-12.5}
|\BD f|_{C^0 \BBL(X^{n-2})} \le |\omega_*| C_*^2 \big( |u_*|_{X^{n-1}} + C_* \delta_1) \triangleq C_{f, 1}.  
\ee

To finish the proof we shall identify some $\ep, \delta_0>0$ and $R_0 > 3C_L^2 C_Q^{2(n-1)} |v_*|_{X^{n-1}}$ such that 
\[
\CB \big( X^{n-1}(u_*, \ep) \cap X^n (R_1) \big) \subset X^{n-2} (v_*, \delta_0/2) \cap X^{n-1} \big(C_L^{-2} C_Q^{-2(n-1)}R_0/2 \big),
\]
and obtain the boundedness of $|f(v)|_{X^{n-1}}$ on the latter set, where $C_L, C_Q$ are defined in \eqref{E:temp-11}. 
Let $\ep_0 = \min\{ dist_{X^{n-1}} (u_*, \CO^c), \delta_1/C_*\}$, 
then we have $\CB(X^{n-1} (u_*, \ep_0)) \subset  X^{n-2} (v_*, \delta_1)$.  
From 
\begin{align*}
& |\CB |_{C^0 (X^{n-1} (u_*, \ep_0) \cap X^n (2 R_1), X^{n-1})} \le |v_*|_{X^{n-1}} + 2 R_1 |\bar A|_{C^0 (\CO_n(2 R_1), \BBL(X^n, X^{n-1}))}\triangleq R_2,
\end{align*}
we obtain 
\[
\CB\big( X^{n-1} (u_*, \ep_0) \cap X^n (2 R_1) \big) 
\subset X^{n-2} (v_*, \delta_1) \cap X^{n-1}(R_2 ). 
\]
Let $R_0 = 3 C_L^2 C_Q^{2(n-1)} R_2$. From assumption (B.4) as well as (B.1), there exists $\delta_0  \in (0, \delta_1/2]$ such that 
\[
|\CB^{-1}|_{C^0 (X^{n-2} (v_*, \delta_0) \cap X^{n-1} ( R_0 ), X^n )}, \ 
|f|_{C^0 (X^{n-2} (v_*, \delta_0) \cap X^{n-1} (R_0), X^{n-1})} < \infty. 
\]
Hence assumptions (H.1)--(H.4) are verified. 

Finally, let $\ep \in (0, \ep_0/2)$ such that $\CB(X^{n-1} (u_*, \ep )) \subset X^{n-2} (v_*, \delta_0/2)$, then any initial value $u_0 \in X^{n-1}(u_*, \ep) \cap X^n (R_1)$ satisfies $v_0 = \CB (u_0) \in X^{n-2} (v_*, \delta_0/2) \cap X^{n-1} (R_2)$. The local well-posedness and Lipschitz dependence on the initial values in Theorem \ref{T:NLPDE-LWP} follow immediately from Theorem \ref{T:QLPDE-LWP}. The $C^{1,1}$ dependence on the initial values will be given in Theorem \ref{T:LWP-smoothness} along with Remark \ref{R:smoothness-1}.
\end{proof}

\begin{remark} \label{R:LWP-R-2}
If $F(u_*) \in X^n$, we could use the transformation 
\[
\CB(u) = F(u) - F(u_*) - \omega_*(u-u_*) 
\]
instead, such that $v_*= \CB(u_*)=0$ and correspondingly 
\[
f(v) =\BA(v) F(u_*) + \omega_* \big( (\BA(v) - \omega_*) (\CB^{-1} (v) - u_*) - v - F(u_*)\big).  
\]
If we further assume $(\omega_* - \CA(u_*))^{-1} \in \BBL(X^{n-1}, X^n)$ in (B.2), then instead of assuming (B.4), by the Inverse Function Theorem, there exists $\delta_0, R_0>0$ such that assumption (H.1)--(H.4) in Appendix \ref{SS:QLPDE-LWP} are verified. Hence the same local well-posedness results in Theorem \ref{T:NLPDE-LWP} hold on $X^n(u_*, R_1)$ for small $R_1>0$. In fact $F(u_*) \in X^n$ is comparable to (7.5) in \cite{Ka75}. 
\end{remark}

Besides that one may verify   assumption (B.4) directly for many PDEs, to end this part of discussions we give the following lemma which provides a sufficient condition for (B.4). 

\begin{lemma} \label{L:CB-1}
In addition to (B.1)--(B.2), assume $u_* \in \CO_n$, $v_* = \CB(u_*)$, $(\omega_* - \CA(u_*))^{-1} \in \BBL(X^{n-1}, X^{n})$, and for any $R>0$, 
\be \label{E:tameE-1} \begin{split}
C_{\CA,1} (R) \triangleq \sup \Big \{ \frac {|\BD^2 F (u) (w_1, w_2)|_{X^{n-1}}}{ |w_1|_{X^n} |w_2|_{X^{n-1}} +  |w_1|_{X^{n-1}} |w_2|_{X^n}} \, :& \ u \in \CO_n(R), \\
& \ w_1, w_2 \in X^n \setminus \{0\} \Big\} < \infty, 
\end{split} \ee
then for any 
$R_0 \ge  |v_*|_{X^{n-1}}$, there exists $\delta >0$ (given in \eqref{E:temp-10} below) such that $\CB^{-1} \in  C^{2} \big(X^{n-2} (v_*, \delta) \cap X^{n-1} (R_0), \CO_n \big)$ 
and the norms  
\[
|\BD \CB^{-1}|_{C^{0} (X^{n-2} (v_*, \delta) \cap X^{n-1} (R_0), \BBL(X^{n-1}, X^{n} ))} < \infty 
\]
\begin{align*}  
\sup \Big \{ \frac {|\BD^2 \CB^{-1} (v) (v_1, v_2)|_{X^{n}}}{ |v_1|_{X^{n-2}} |v_2|_{X^{n-1}} +  |v_1|_{X^{n-2}} |v_2|_{X^{n-1}}} : v \in X^{n-2} & (v_*, \delta) \cap X^{n-1} (R_0), \\
& v_1, v_2 \in X^{n-1} \setminus \{0\} \Big\}<\infty.
\end{align*}
\end{lemma} 

Here we notice that the above norm of $\BD^2 \CB^{-1}$ 
is consistent with \eqref{E:tameE-1}. 
Though assumption \eqref{E:tameE-1} is stronger than $F \in C^2 (\CO_n, X^{n-1})$, it is often easily verified by using inequalities like  
\be \label{E:tameE-2}
|u_1u_2|_{W^{n, p} (\R^d)} \lesssim |u_1|_{W^{n, p} (\R^d)} |u_2|_{L^\infty (\R^d)} +  |u_1|_{L^\infty (\R^d)}|u_2|_{W^{n, p} (\R^d)}
\ee
if only products of the unknowns and their  derivatives are involved in $F(u)$. 

\begin{proof} 
Though it is clear from the assumptions that $\CB^{-1}$ is also $C^2$ from a neighborhood of $v_*$ in $X^{n-1}$ to $X^n$, it still remains to show the much stronger statement  that the domain of $\CB^{-1}$ can be extended to a arbitrary ball in $X^{n-1}$. 
For this purpose, we carefully go through the procedure of the proof of the Inverse Function Theorem. 
Let 
\[
\CS (u, v) = u+ \bar A (u_*)^{-1} (v -\CB(u) ), \quad \; (u, v) \in \CO \times X^{n-2},
\]
which satisfies 
\[
\CB(u) = v \Longleftrightarrow \CS(u, v) =u; \quad 
\CS \in C^{2} (\CO \times X^{n-2}, X^{n-1}) \cap C^{2}( \CO_n \times X^{n-1}, X^n). 
\]
Let $\ep_0 = dist(u_*, \CO^c)>0$. 
We shall focus on the subsets
 \[
\CU = \overline{X^{n-1} (u_*, \ep)} \cap \overline{X^{n} (R_1)} \subset \CO_n, \quad \CK = X^{n-2} (v_*, \delta) \cap X^{n-1} (R_0), 
\]
where $R_1 > |u_*|_{X^n}$, $\ep\in (0, \ep_0)$, and $\delta>0$ are constants to be fixed later based on $\ep_0$ and  $R_0 > |v_*|_{X^{n-1}}$.

At any $(u, v) \in \CU \times \CK$ and $u_1 \in X^n$,  
\begin{align*}
\BD_u \CS(u, v) u_1 = & u_1 - \bar A(u_*)^{-1} \bar A(u) u_1 = \bar A(u_*)^{-1} \big(\CA(u_*) - \CA(u) \big) u_1 \\
= & \bar A(u_*)^{-1} \int_0^1 \BD^2 F \big(\tau u_* + (1-\tau) u\big) (u_* -u, u_1) d\tau.  
\end{align*}
Let us denote $b =| (\omega_* - \CA(u_*))^{-1}|_{\BBL(X^{n-1}, X^{n})}$. 
From (B.1), (B.2), and the above assumptions, one can estimate 
\[
|\BD_u \CS(u, v) u_1|_{X^{n-1}} \le C_*^2 |u-u_*|_{X^{n-1}} |u_1|_{X^{n-1}} \le C_*^2 \ep |u_1|_{X^{n-1}}, 
\]
\begin{align*}
|\BD_u \CS(u, v) u_1|_{X^n} \le & b C_{\CA, 1} (R_1) \big(|u-u_*|_{X^{n-1}} |u_1|_{X^n} + |u-u_*|_{X^{n}} |u_1|_{X^{n-1}} \big) \\
\le &  b C_{\CA, 1} (R_1) \big(\ep |u_1|_{X^n} +2R_1 |u_1|_{X^{n-1}} \big).
\end{align*}
Since $\BD_v \CS(u, v) = \bar A(u_*)^{-1}$, for any $(u_1, v_1), (u_2, v_2) \in \CU \times \CK$, we obtain
\begin{align*}
|\CS(u_1, v_1) - \CS(u_2, v_2)|_{X^{n-1}} \le & C_* |v_1-v_2|_{X^{n-2}} + C_*^2 \ep |u_2-u_1|_{X^{n-1}}, 
\\
|\CS(u_1, v_1) - \CS(u_2, v_2)|_{X^{n}} \le & b |v_1-v_2|_{X^{n-1}} \\
&+ b C_{\CA, 1} (R_1) (\ep |u_2-u_1|_{X^{n}} + 2R_1|u_2-u_1|_{X^{n-1}} ).  
\end{align*}
These Lipschitz estimates also imply, for $(u, v) \in \CU \times \CK$, 
\begin{align*}
|\CS(u, v) - u_*|_{X^{n-1}} =& |\CS(u, v) -\CS( u_*, v_*)|_{X^{n-1}} \\
\le & C_* |v-v_*|_{X^{n-2}} + C_*^2 \ep |u-u_*|_{X^{n-1}} \le C_* \delta + C_*^2 \ep^2, \\
|\CS(u, v) - u_*|_{X^{n}} \le & b |v-v_*|_{X^{n-1}}  + b C_{\CA, 1} (R_1) (\ep |u-u_*|_{X^{n}} + 2R_1|u-u_*|_{X^{n-1}} ) \\
\le & b ( 2  R_0 + 4 C_{\CA, 1} (R_1) R_1 \ep). 
\end{align*}

To complete the proof, 
for any $R_0>0$, let 
\be \label{E:temp-10} \begin{split} 
R_1 = \max\{ 9|u_*|_{X^n} +1, \, 9( b +1) R_0\}, \; \ep =  \min\Big\{\frac {\ep_0}2, \, \frac 1{2C_*^2}, \, \frac {1} {9 b C_{\CA, 1} (R_1)} 
\Big\},  \; \delta = \frac {\ep}{3C_*}.  
\end{split} \ee
One may verify $\CS \in C^2 ( \CU \times \CK, \CU)$. Moreover, for any $(u_0, v) \in \CU \times \CK$, the iteration sequence $u_l = \CS(u_{l-1}, v)$, $l\in \N$, satisfies    
\begin{align*}
|u_{l+1} - u_l|_{X^{n-1}} \le & 2^{-l} |u_1-u_0|_{X^{n-1}}, \\
|u_{l+1} - u_l|_{X^{n}} \le & 9^{-l} |u_1-u_0|_{X^{n}} + 2 b C_{\CA, 1} (R_1) R_1 \sum_{j=1}^l 9^{1-j} |u_{l-j+1} - u_{l-j}|_{X^{n-1}} \\
\le & 9^{-l} |u_1-u_0|_{X^{n}} + 2^{2-l} l b C_{\CA, 1} (R_1) R_1 |u_1-u_0|_{X^{n-1}}. 
\end{align*}
Therefore $u_l$ converges to a limit in $X^n$ which is a fixed point in $\CU$ of $\CS(\cdot, v)$. The uniqueness of the fixed point of $\CS(\cdot, v)$ follows from the contraction estimate of $\CS(\cdot, v)$ in the $X^{n-1}$ norm. Hence $\CB^{-1}(v) \in \CU \subset X^n$ for any $v \in \CK$ is obtained. Much as the iteration estimate, it also holds, for any $(u, v) \in \CU \times \CK$ and $u_1 \in X^n$, 
\[
|\big( \BD_u \CS (u, v)\big)^l u_1|_{X_n} \le  9^{-l} |u_1|_{X^{n}} + 2^{2-l} l b C_{\CA, 1} (R_1) R_1 |u_1|_{X^{n-1}}. 
\]
Therefore  $\big(I- \BD_u \CS (u, v) \big)^{-1} \in \BBL(X^n) \cap \BBL(X^{n-1})$ and is uniformly bounded. It is standard to obtain 
\begin{align*}
&\BD \CB^{-1} (v)= (I- \BD_u \CS)^{-1} \bar A(u_*)^{-1}, \\
& \BD^2 \CB^{-1} (v) (v_1, v_2) = -(I- \BD_u \CS)^{-1} \bar A(u_*)^{-1} \BD^2 F (u) \big(\BD \CB^{-1} (v)v_1, \BD \CB^{-1} (v)v_2 \big), 
\end{align*}
where $u = \CB^{-1} (v)$, $\BD_u \CS$ is evaluated at $(u, v)$, and the specific form of $\CS$ is also used. The desired estimates on $\CB^{-1}$ follow immediately. 
\end{proof}

\subsubsection{Smooth dependence on the initial data} \label{SSS:smoothness-NLPDE}

For any $m \in \N$, formally linearizing the solution map $u(t, u_0)$ of \eqref{E:NLPDE-0} with respect to the initial value $u_0$ yields that the symmetric $m$-linear operators $U^m (t, u_0) = \BD_{u_0}^m u(t, u_0)$ satisfies 
\be \label{E:smoothness-1} \begin{cases} 
&\p_t U^m = \CA(u) U^m  + \CF_m (t, u_0) \\
&U^1 (0) = I; \quad U^m (0) =0, \ \text{ if } \ m>1, 
\end{cases} \ee
where  the symmetric $m$-linear operator $\CF_m (t, u_0)
$ is 
in the form of  
\be \label{E:temp-10.1} \begin{split}
& \CF_m (t, u_0) (w, \ldots, w) \\ 
=&\sum_{\substack{1\le m_1, \ldots, m_l < m \\ m_1 + \ldots + m_l =m}} \alpha_{m_1, \ldots, m_l } \BD^{l} F (u(t, u_0)) \big(U^{m_1} (t, u_0) (w, \ldots, w), \ldots, U^{m_l} (t, u_0) (w, \ldots, w) \big)
\end{split} \ee
for some coefficients $\alpha_{m_1, \ldots, m_l }$. Clearly $\CF_1 =0$ and  the following recursive relation holds 
\be \label{E:recursion-1} \begin{split} 
\CF_{m+1} (t, u_0) (w, \ldots) = \BD^2 F(u(t, u_0)) \big( U^1 w, U^m(\ldots) ) + \BD_{u_0} \CF_m(t, u_0) w.
\end{split} \ee
In particular, the above last term includes terms like $\BD_{u_0} (\BD^l F(u(t, u_0)))w$ and $\BD_{u_0} U^{m_j} (t, u_0)w$  
which should be replaced by $\BD^{l+1} F (u(t, u_0)) (U^1(t, u_0), \ldots)$ and $U^{m_j+1} (t, u_0) (w, \ldots)$, respectively. 
According to the Lipschitz dependence of $u(t, u_0)$ on $u_0$ given in Theorem \ref{T:NLPDE-LWP},  one might wish $\BD_{u_0} u \in \BBL(X^{r})$, $0 \le r\le n-1$. However, for $m \ge 2$, the terms like $\BD^2 F(u) (\BD_{u_0}^{m-1} u, \BD_{u_0} u)$ limits the regularity of $\BD_{u_0}^m u$, which thus can be expected to be valued in  $X^{n-m}$ at the best. To prove such a result,  in addition to (B.1--4), we assume 

\begin{enumerate}
\item [(B.5)] There exist $m_0, r_0 \in \N$ such that $1 \le m_0 \le r_0 \le n -1$ and it holds, for any  $r_0 +1 \le r \le n$ and $2\le l \le m \le m_0$,   
\[
\BD^l F \in C^0 \big(\CO, \BBL ( \otimes_{j=1}^l X^{r-m_j},  X^{r-m}) \big), 
\quad | \BD^{l} F|_{C^0 (\CO, \BBL ( \otimes_{j=1}^l X^{r-m_j},  X^{r-m}) )} \le C_*;
\]  
and also for any $1\le l \le m \le m_0$,  
\[
| \BD^{l} F|_{Lip (\CO, \BBL ( \otimes_{j=1}^l X^{r-m_j},  X^{r-m-1}))} \le C_*, \quad \forall \, r_0+1 \le r \le n. 
\]  
Here in both cases, $m_1, \ldots, m_l$ satisfy 
\[
1\le m_j \le m, \; \forall 1\le j \le l, \quad m_1 + \ldots + m_l =m.
\]
\end{enumerate}

\begin{remark} \label{R:smoothness-1} 
The above assumption for $r_0=1$ is already included in (B.1) (as well as in (D.3) in Subsection \ref{SS:NLPDE-LInMa}) if $m_0=1$. 
\end{remark} 

\begin{remark} \label{R:smoothness-2} 
In most of the applications a simpler condition is that there exist $R>0$ and an open $\hat \CO \subset X^{n-m_0}$ such that $F$ is bounded in $C^{m_0, 1} (\hat\CO \cap X^{n'} (R), X^{n'-1})$ for all $n' \in \N \cap [n-m_0, n-1]$. This would yield (B.5) for $r_0=n-1$ and $\CO= \hat \CO \cap X^{n-1}(R)$. 
\end{remark} 

\begin{theorem} \label{T:LWP-smoothness}
Assume (B.1--5), $u_* \in \CO_n$ and $R_1 > |u_*|_{X^n}$, let $\ep, T, C>0$ 
and the solution map $u(t, u_0)$, $t \in [0, T]$, be given by Theorem \ref{T:NLPDE-LWP}, then $u(t, \cdot) \in C^{m, 1}( X^{n-1}(u_*, \ep) \cap X^n (R_1), X^{n-1-m})$ for $1\le m \le m_0$. Moreover, for any $t \in [0, T]$, $1\le m \le m_0$, $r_0 \le r \le n-1$, and $u_0 \in X^{n-1}(u_*, \ep) \cap X^n (R_1)$, the $m$-linear operator $\BD_{u_0}^m u(t, u_0) \in \BBL( \otimes_{j=1}^m X^r, X^{r-m+1})$ is strongly $C^0$ in $t$, satisfies 
\be \label{E:temp-10.2}
|\BD_{u_0}^m u(t, u_0)|_{\BBL( \otimes_{j=1}^m X^r, X^{r-m+1})} \le 
\begin{cases} C, & \text{ if } \ m=1,\\
Ct,  & \text{ if } \ m>1, 
\end{cases} \ee
and for any $u_{01}, u_{02} \in X^{n-1}(u_*, \ep) \cap X^n (R_1)$, 
\[
|\BD_{u_0}^m u(t, u_{01}) - \BD_{u_0}^m u(t, u_{02})|_{\BBL( \otimes_{j=1}^m X^r, X^{r-m})} \le Ct |u_{01} - u_{02}|_{X^{n-1}}. 
\]
\end{theorem} 

\begin{proof}
To prove the theorem, we shall first obtain the existence of solutions to \eqref{E:smoothness-1} for $1\le m \le m_0$ in $\BBL( \otimes_{j=1}^m X^r, X^{r-m+1})$ strongly continuous in $t\in [0, T]$, their Lipschitz property in $u_0$, and then prove the differentiability of $u(t, u_0)$ in $u_0$. 

Let $u(t)$, $t \in [0, T]$, be a solution to \eqref{E:NLPDE-0}. 
As in Appendix \ref{SS:QLPDE-LWP}, $A(t) =\CA(u(t))$ defines an evolution operator $U(t, t_0) \in \BBL(X^r)$, $0\le r \le n-1$, strongly continuous in $t$ and $t_0$ and satisfying \eqref{E:U-bdd-1} and \eqref{E:U-bdd-2} where $Q(t)$ and $L(t)$ are given in \eqref{E:temp-7.5} and \eqref{E:temp-10.5}, $C_Q$, $C_L$, and $C_0$ given in \eqref{E:temp-11}, and $M_1 = \big| \big(\CA (u(t)) - \omega_* \big) u_t \big|_{C_t^0 X^{n-2}}$. 

For $m=1$, clearly $U^1 (t) = U(t, 0)$ solves \eqref{E:smoothness-1} and satisfies the desired upper bound \eqref{E:temp-10.2}. For $m \ge 2$, we obtain from the variation of constant formula 
\[
U^m (t) = \int_0^t U(t, \tau) \CF_m (\tau, u_0)
d\tau, 
\]
which is symmetric $m$-linear for each $t \in [0, T]$ and solves \eqref{E:smoothness-1}. The form \eqref{E:temp-10.1} and assumption (B.5) imply, for $r_0 \le r \le n-1$, 
\be \label{E:temp-10.22}
|\CF_m (t, u_0)|_{\BBL(\otimes_{j=1}^m X^r, X^{r-m+1})} \le C  
\sum_{\substack{1\le m_1, \ldots, m_l < m \\ m_1 + \ldots + m_l =m}} \prod_{j=1}^l |U^{m_j} (t, u_0)|_{\BBL(\otimes_{j=1}^{m_j} X^r, X^{r-m_j+1})}. 
\ee
and thus the desired estimate \eqref{E:temp-10.2} on $|U^m|$ follows inductively. 


Next we study the dependence of $U^m$ in the initial value $u_0$. 
Consider solutions $u_j = u(t, u_{0j})$, $j=1,2$ and $t \in [0, T]$, to \eqref{E:NLPDE-0} with initial values $u_{01}, u_{02} \in X^{n-1}(u_*, \ep) \cap X^n (R_1)$. Let the $m$-linear operators $U_j^m = U^m (t, u_{0j})$ be the solutions to \eqref{E:smoothness-1} along $u_j(t)$.
From \eqref{E:smoothness-1}, one may compute 
\begin{align*}
(U_2^m -  U_1^m )_t   = & \CA (u_1) (U_2^m -  U_1^m ) + ( \BD F(u_2) - \BD F (u_1) ) U_2^m + \CF_m(t, u_{02}) - \CF_m (t, u_{01}). 
\end{align*}
Again \eqref{E:temp-10.1}, assumption (B.5), and the estimate on $|U^m|$ obtained above imply for $m \ge 2$
\be \label{E:temp-10.24} \begin{split}
& |\CF_m(t, u_{02}) - \CF_m (t, u_{01}) |_{\BBL(\otimes_{j=1}^m X^r, X^{r-m})}\\
\le & C \Big( |u_2(t) - u_1(t)|_{X^{n-1}} + \sum_{j=1}^{m-1} |U_2^{j} (t) - U_1^{j} (t)|_{\BBL(\otimes_{j=1}^{j} X^r, X^{r-j})} \Big),
\end{split} \ee
for $r_0 \le r \le n-1$. 
Hence the desired Lipschitz estimates follow from the induction using the variation of constant formula and assumption (B.5). 

Finally, for any initial value 
\[
u_0 + \tilde u_0 \in X^{n-1}(u_*, \ep) \cap X^n (R_1),  \quad 0\le |\tilde u_0|_{X^{n-1}} \ll 1, 
\]
recall $U^m (t, u_0+ \tilde u_0)$ denote the solution to \eqref{E:smoothness-1} along the nonlinear solution $\tilde u(t) = u(t, u_0 + \tilde u_0)$, 
with initial value $u_0 + \tilde u_0$, 
while we often skip the initial $u_0$ in $u$ and $U^m$ if $\tilde u_0=0$. 
Consider 
\[
w(t, \tilde u_0) = u(t, u_0 + \tilde u_0) - u(t) - U^1 (t) \tilde u_0 = \tilde u(t) - u(t) - U^1 (t) \tilde u_0. 
\]
Clearly $w(0, \tilde u_0) =0$ and one may compute 
\begin{align*}
w_t  =&  F(\tilde u(t)) - F(u(t)) - \CA(u(t)) U^1 (t) \tilde u_0 \\
= & \CA(u(t)) w + F(\tilde u(t)) - F(u(t)) - \CA(u(t)) \big(\tilde u(t) - u(t) \big)  \\
=& \CA(u(t)) w + \int_0^1 \Big( \BD F \big( u(t) + \tau (  \tilde u(t) - u(t))\big) - \BD F(u(t))\Big) \big(\tilde u(t) - u(t) \big) d\tau,
\end{align*}
and thus 
\[
w(t, \tilde u_0) = \int_0^t U(t, s) \int_0^1 \Big( \CA \big( u(s) + \tau (  \tilde u(s) - u(s))\big) - \CA(u(s))\Big) \big( \tilde u(s) - u(s) \big) d\tau ds. 
\]
From  assumption (B.1) and the Lipschitz dependence of $u \in X^{n-1}$ in $u_0 \in X^{n-1}$ obtained in Theorem \ref{T:NLPDE-LWP}, we have
\be \label{E:temp-10.3}
|\tilde u(t) - u(t) - U^1 (t) \tilde u_0|_{X^{n-2}} \le C \int_0^t \big|  \tilde u(s) - u(s) \big|_{X^{n-1}}^2 
ds \le C t |\tilde u_0|_{X^{n-1}}^2. 
\ee
Therefore $u(t, \cdot) : X^{n-1}(u_*, \ep) \cap X^n (R_1) \to X^{n-2}$  is Fr\'echet differentiable and $\BD_{u_0} u(t, u_0) = U^1 (t)$ for any $u_0$. Along with the above estimates on $U^1$ it completes the proof of the theorem for the case of $m_0=1$. 

The higher order differentiability of $u(t, u_0)$ in $u_0$ is proved inductively with the above case of $m=1$ as the base case. For any $2\le m \le m_0$, assume that it has been proved 
$\BD_{u_0}^{j} u(t, u_0+ \tilde u_0) = U^{j} (t, u_0 + \tilde u_0)$, $1\le j\le m-1$, and 
\be \label{E:temp-10.4}
|\BD_{u_0}^{j-1} u(t, u_0 + \tilde u_0) - \BD_{u_0}^{j-1} u(t, u_0) - U^{j} (t) ( \tilde u_0, \ldots)|_{\BBL(\otimes_{j=1}^{j} X^r, X^{r-j})} \le Ct |\tilde u_0|_{X^{n-1}}^2.  
\ee
Consider the $m$-linear operator 
\begin{align*}
W^{m}(t, \tilde u_0) = & \BD_{u_0}^{m-1} u(t, u_0 + \tilde u_0) - \BD_{u_0}^{m-1} u(t, u_0) - U^{m} (t) ( \tilde u_0, \ldots) \\
= & U^{m-1} (t, u_0+ \tilde u_0) - U^{m-1} (t) - U^{m} (t) ( \tilde u_0, \ldots).  
\end{align*}
Clearly $W^{m} (0, \tilde u_0) =0$ and one may compute 
\begin{align*}
&W^{m}_t =  \CA(\tilde u(t)) U^{m-1} (t, u_0+ \tilde u_0) - \CA(u(t)) U^{m-1} (t) -  \CA(u(t)) U^{m} (t) ( \tilde u_0, \ldots) \\
& \qquad \quad + \CF_{m-1}(t, u_0 + \tilde u_0) 
- \CF_{m-1} (t, u_0) 
- \CF_{m} (t, u_0) (\tilde u_0 , \ldots) 
\\
=& \CA(u(t)) W^{m}  + \big( \BD F(\tilde u(t)) - \BD F(u(t)) \big) \big(U^{m-1} (t, u_0 + \tilde u_0) -  U^{m-1} (t) \big)  \\
& + \big( \BD F(\tilde u(t)) - \BD F(u(t)) - \BD^2 F(u(t)) U^1(t) \tilde u_0 \big) U^{m-1} (t)+ \CF_{m-1} (t, u_0 + \tilde u_0) \\
& - \CF_{m-1} (t, u_0) +  \BD^2 F(u(t)) \big(U^1(t) \tilde u_0,U^{m-1} (t)  \big)  - \CF_{m}(t, u_0) (\tilde u_0 , \ldots).
\end{align*}
Along with \eqref{E:temp-10.1}, \eqref{E:recursion-1}, \eqref{E:temp-10.3}, \eqref{E:temp-10.4}, assumption (B.5), and the above estimates and induction assumptions on $U^j$, it implies 
\[
|W_t^m - \CA(u) W^m|_{\BBL(\otimes_{j=1}^{m-1} X^r, X^{r-m})} \le C  |\tilde u_0|_{X^{n-1}}^2. 
\]
The variation of constant formula implies that the same estimate \eqref{E:temp-10.4} holds for $j=m$ and thus $\BD_{u_0}^m u(t, u_0) = U^m(t, u_0)$. 
\end{proof}

\bibliographystyle{abbrv}
\bibliography{bibliography}

\begin{thebibliography}{10}

\bibitem{ABZ11}
T.~Alazard, N.~Burq, and C.~Zuily.
\newblock On the water-wave equations with surface tension.
\newblock {\em Duke Math. J.}, 158(3):413--499, 2011.

\bibitem{AM03}
D.~M. Ambrose.
\newblock Well-posedness of vortex sheets with surface tension.
\newblock {\em SIAM J. Math. Anal.}, 35(1):211--244, 2003.

\bibitem{AM05}
D.~M. Ambrose and N.~Masmoudi.
\newblock The zero surface tension limit of two-dimensional water waves.
\newblock {\em Comm. Pure Appl. Math.}, 58(10):1287--1315, 2005.

\bibitem{AM09}
D.~M. Ambrose and N.~Masmoudi.
\newblock The zero surface tension limit of three-dimensional water waves.
\newblock {\em Indiana Univ. Math. J.}, 58(2):479--521, 2009.

\bibitem{AK89}
C.~J. Amick and K.~Kirchg\"assner.
\newblock A theory of solitary water-waves in the presence of surface tension.
\newblock {\em Arch. Rational Mech. Anal.}, 105(1):1--49, 1989.

\bibitem{BJ89}
P.~W. Bates and C.~K. R.~T. Jones.
\newblock Invariant manifolds for semilinear partial differential equations.
\newblock In {\em Dynamics reported, {V}ol.\ 2}, volume~2 of {\em Dynam.
  Report. Ser. Dynam. Systems Appl.}, pages 1--38. Wiley, Chichester, 1989.

\bibitem{BHL93}
J.~T. Beale, T.~Y. Hou, and J.~S. Lowengrub.
\newblock On the well-posedness of two fluid interfacial flows with surface
  tension.
\newblock In {\em Singularities in fluids, plasmas and optics ({H}eraklion,
  1992)}, volume 404 of {\em NATO Adv. Sci. Inst. Ser. C: Math. Phys. Sci.},
  pages 11--38. Kluwer Acad. Publ., Dordrecht, 1993.

\bibitem{BB97}
T.~B. Benjamin and T.~J. Bridges.
\newblock Reappraisal of the {K}elvin-{H}elmholtz problem. {I}. {H}amiltonian
  structure.
\newblock {\em J. Fluid Mech.}, 333:301--325, 1997.

\bibitem{BMV22}
M.~Berti, A.~Maspero, and P.~Ventura.
\newblock Full description of {B}enjamin-{F}eir instability of {S}tokes waves
  in deep water.
\newblock {\em Invent. Math.}, 230(2):651--711, 2022.

\bibitem{BG98}
K.~Beyer and M.~G\"unther.
\newblock On the {C}auchy problem for a capillary drop. {I}. {I}rrotational
  motion.
\newblock {\em Math. Methods Appl. Sci.}, 21(12):1149--1183, 1998.

\bibitem{Br01}
T.~J. Bridges.
\newblock Transverse instability of solitary-wave states of the water-wave
  problem.
\newblock {\em J. Fluid Mech.}, 439:255--278, 2001.

\bibitem{BM95}
T.~J. Bridges and A.~Mielke.
\newblock A proof of the {B}enjamin-{F}eir instability.
\newblock {\em Arch. Rational Mech. Anal.}, 133(2):145--198, 1995.

\bibitem{BSWZ16}
O.~B\"uhler, J.~Shatah, S.~Walsh, and C.~Zeng.
\newblock On the wind generation of water waves.
\newblock {\em Arch. Ration. Mech. Anal.}, 222(2):827--878, 2016.

\bibitem{Ca81}
J.~Carr.
\newblock {\em Applications of centre manifold theory}, volume~35 of {\em
  Applied Mathematical Sciences}.
\newblock Springer-Verlag, New York-Berlin, 1981.

\bibitem{CW22}
R.~M. Chen and S.~Walsh.
\newblock Orbital stability of internal waves.
\newblock {\em Comm. Math. Phys.}, 391(3):1091--1141, 2022.

\bibitem{CSS08}
C.-H.~A. Cheng, D.~Coutand, and S.~Shkoller.
\newblock On the motion of vortex sheets with surface tension in
  three-dimensional {E}uler equations with vorticity.
\newblock {\em Comm. Pure Appl. Math.}, 61(12):1715--1752, 2008.

\bibitem{CH82}
S.~N. Chow and J.~K. Hale.
\newblock {\em Methods of bifurcation theory}, volume 251 of {\em Grundlehren
  der Mathematischen Wissenschaften}.
\newblock Springer-Verlag, New York-Berlin, 1982.

\bibitem{CLL91}
S.-N. Chow, X.-B. Lin, and K.~Lu.
\newblock Smooth invariant foliations in infinite-dimensional spaces.
\newblock {\em J. Differential Equations}, 94(2):266--291, 1991.

\bibitem{CL88}
S.-N. Chow and K.~Lu.
\newblock Invariant manifolds for flows in {B}anach spaces.
\newblock {\em J. Differential Equations}, 74(2):285--317, 1988.

\bibitem{CM19}
T.~H. Colding and W.~P. Minicozzi, II.
\newblock Dynamics of closed singularities.
\newblock {\em Ann. Inst. Fourier (Grenoble)}, 69(7):2973--3016, 2019.

\bibitem{CJ04}
M.~Colin and L.~Jeanjean.
\newblock Solutions for a quasilinear {S}chr\"odinger equation: a dual
  approach.
\newblock {\em Nonlinear Anal.}, 56(2):213--226, 2004.

\bibitem{CFNT89}
P.~Constantin, C.~Foias, B.~Nicolaenko, and R.~Temam.
\newblock {\em Integral manifolds and inertial manifolds for dissipative
  partial differential equations}, volume~70 of {\em Applied Mathematical
  Sciences}.
\newblock Springer-Verlag, New York, 1989.

\bibitem{CS07}
D.~Coutand and S.~Shkoller.
\newblock Well-posedness of the free-surface incompressible {E}uler equations
  with or without surface tension.
\newblock {\em J. Amer. Math. Soc.}, 20(3):829--930, 2007.

\bibitem{Cr85}
W.~Craig.
\newblock An existence theory for water waves and the {B}oussinesq and
  {K}orteweg-de {V}ries scaling limits.
\newblock {\em Comm. Partial Differential Equations}, 10(8):787--1003, 1985.

\bibitem{CG00}
W.~Craig and M.~D. Groves.
\newblock Normal forms for wave motion in fluid interfaces.
\newblock {\em Wave Motion}, 31(1):21--41, 2000.

\bibitem{CN00}
W.~Craig and D.~P. Nicholls.
\newblock Travelling two and three dimensional capillary gravity water waves.
\newblock {\em SIAM J. Math. Anal.}, 32(2):323--359, 2000.

\bibitem{DL88}
G.~Da~Prato and A.~Lunardi.
\newblock Stability, instability and center manifold theorem for fully
  nonlinear autonomous parabolic equations in {B}anach space.
\newblock {\em Arch. Rational Mech. Anal.}, 101(2):115--141, 1988.

\bibitem{AN13}
B.~Deconinck and O.~Trichtchenko.
\newblock Stability of periodic gravity waves in the presence of surface
  tension.
\newblock {\em Eur. J. Mech. B Fluids}, 46:97--108, 2014.

\bibitem{DT14}
B.~Deconinck and O.~Trichtchenko.
\newblock Stability of periodic gravity waves in the presence of surface
  tension.
\newblock {\em Eur. J. Mech. B Fluids}, 46:97--108, 2014.

\bibitem{DR04}
P.~G. Drazin and W.~H. Reid.
\newblock {\em Hydrodynamic stability}.
\newblock Cambridge Mathematical Library. Cambridge University Press,
  Cambridge, second edition, 2004.
\newblock With a foreword by John Miles.

\bibitem{DM09}
T.~Duyckaerts and F.~Merle.
\newblock Dynamic of threshold solutions for energy-critical {NLS}.
\newblock {\em Geom. Funct. Anal.}, 18(6):1787--1840, 2009.

\bibitem{DM10}
T.~Duyckaerts and F.~Merle.
\newblock Dynamic of threshold solutions for energy-critical {NLS}.
\newblock {\em Geom. Funct. Anal.}, 18(6):1787--1840, 2009.

\bibitem{Eck04}
K.~Ecker.
\newblock {\em Regularity theory for mean curvature flow}, volume~57 of {\em
  Progress in Nonlinear Differential Equations and their Applications}.
\newblock Birkh\"auser Boston, Inc., Boston, MA, 2004.

\bibitem{EW87}
C.~L. Epstein and M.~I. Weinstein.
\newblock A stable manifold theorem for the curve shortening equation.
\newblock {\em Comm. Pure Appl. Math.}, 40(1):119--139, 1987.

\bibitem{Ger06}
C.~Gerhardt.
\newblock {\em Curvature problems}, volume~39 of {\em Series in Geometry and
  Topology}.
\newblock International Press, Somerville, MA, 2006.

\bibitem{GT01}
D.~Gilbarg and N.~S. Trudinger.
\newblock {\em Elliptic partial differential equations of second order}.
\newblock Classics in Mathematics. Springer-Verlag, Berlin, 2001.
\newblock Reprint of the 1998 edition.

\bibitem{GGSZ25}
O.~M.~L. Gomide, M.~Guardia, T.~M. Seara, and C.~Zeng.
\newblock On small breathers of nonlinear {K}lein-{G}ordon equations via
  exponentially small homoclinic splitting.
\newblock {\em Invent. Math.}, 240(2):661--777, 2025.

\bibitem{GSS87}
M.~Grillakis, J.~Shatah, and W.~Strauss.
\newblock Stability theory of solitary waves in the presence of symmetry. {I}.
\newblock {\em J. Funct. Anal.}, 74(1):160--197, 1987.

\bibitem{GSS90}
M.~Grillakis, J.~Shatah, and W.~Strauss.
\newblock Stability theory of solitary waves in the presence of symmetry. {II}.
\newblock {\em J. Funct. Anal.}, 94(2):308--348, 1990.

\bibitem{GHS01}
M.~D. Groves, M.~Haragus, and S.-M. Sun.
\newblock Transverse instability of gravity-capillary line solitary water
  waves.
\newblock {\em C. R. Acad. Sci. Paris S\'er. I Math.}, 333(5):421--426, 2001.

\bibitem{Hal61}
J.~K. Hale.
\newblock Integral manifolds of perturbed differential systems.
\newblock {\em Ann. of Math. (2)}, 73:496--531, 1961.

\bibitem{HHSTWW22}
S.~V. Haziot, V.~M. Hur, W.~A. Strauss, J.~F. Toland, E.~Wahl\'en, S.~Walsh,
  and M.~H. Wheeler.
\newblock Traveling water waves---the ebb and flow of two centuries.
\newblock {\em Quart. Appl. Math.}, 80(2):317--401, 2022.

\bibitem{He81}
D.~Henry.
\newblock {\em Geometric theory of semilinear parabolic equations}, volume 840
  of {\em Lecture Notes in Mathematics}.
\newblock Springer-Verlag, Berlin-New York, 1981.

\bibitem{HY23}
V.~M. Hur and Z.~Yang.
\newblock Unstable {S}tokes waves.
\newblock {\em Arch. Ration. Mech. Anal.}, 247(4):Paper No. 62, 59, 2023.

\bibitem{IK92}
G.~Iooss and K.~Kirchg\"assner.
\newblock Water waves for small surface tension: an approach via normal form.
\newblock {\em Proc. Roy. Soc. Edinburgh Sect. A}, 122(3-4):267--299, 1992.

\bibitem{Ka73}
T.~Kato.
\newblock Linear evolution equations of ``hyperbolic'' type. {II}.
\newblock {\em J. Math. Soc. Japan}, 25:648--666, 1973.

\bibitem{Ka75}
T.~Kato.
\newblock Quasi-linear equations of evolution, with applications to partial
  differential equations.
\newblock In {\em Spectral theory and differential equations ({P}roc.
  {S}ympos., {D}undee, 1974; dedicated to {K}onrad {J}\"orgens)}, volume Vol.
  448 of {\em Lecture Notes in Math.}, pages 25--70. Springer, Berlin-New York,
  1975.

\bibitem{Ke67}
A.~Kelley.
\newblock The stable, center-stable, center, center-unstable, unstable
  manifolds.
\newblock {\em J. Differential Equations}, 3:546--570, 1967.

\bibitem{Ki88}
K.~Kirchg\"assner.
\newblock Nonlinearly resonant surface waves and homoclinic bifurcation.
\newblock In {\em Advances in applied mechanics, {V}ol.\ 26}, volume~26 of {\em
  Adv. Appl. Mech.}, pages 135--181. Academic Press, Boston, MA, 1988.

\bibitem{La05}
D.~Lannes.
\newblock Well-posedness of the water-waves equations.
\newblock {\em J. Amer. Math. Soc.}, 18(3):605--654, 2005.

\bibitem{La13}
D.~Lannes.
\newblock {\em The water waves problem}, volume 188 of {\em Mathematical
  Surveys and Monographs}.
\newblock American Mathematical Society, Providence, RI, 2013.
\newblock Mathematical analysis and asymptotics.

\bibitem{LMSW96}
Y.~Li, D.~W. McLaughlin, J.~Shatah, and S.~Wiggins.
\newblock Persistent homoclinic orbits for a perturbed nonlinear
  {S}chr\"odinger equation.
\newblock {\em Comm. Pure Appl. Math.}, 49(11):1175--1255, 1996.

\bibitem{LZ13}
Z.~Lin and C.~Zeng.
\newblock Unstable manifolds of {E}uler equations.
\newblock {\em Comm. Pure Appl. Math.}, 66(11):1803--1836, 2013.

\bibitem{LZ22}
Z.~Lin and C.~Zeng.
\newblock Instability, index theorem, and exponential trichotomy for linear
  {H}amiltonian {PDE}s.
\newblock {\em Mem. Amer. Math. Soc.}, 275(1347):v+136, 2022.

\bibitem{LZ22CPAM}
Z.~Lin and C.~Zeng.
\newblock Separable {H}amiltonian {PDE}s and turning point principle for
  stability of gaseous stars.
\newblock {\em Comm. Pure Appl. Math.}, 75(11):2511--2572, 2022.

\bibitem{Mant11}
C.~Mantegazza.
\newblock {\em Lecture notes on mean curvature flow}, volume 290 of {\em
  Progress in Mathematics}.
\newblock Birkh\"auser/Springer Basel AG, Basel, 2011.

\bibitem{Mi88}
A.~Mielke.
\newblock Reduction of quasilinear elliptic equations in cylindrical domains
  with applications.
\newblock {\em Math. Methods Appl. Sci.}, 10(1):51--66, 1988.

\bibitem{Mi02}
A.~Mielke.
\newblock On the energetic stability of solitary water waves.
\newblock {\em R. Soc. Lond. Philos. Trans. Ser. A Math. Phys. Eng. Sci.},
  360(1799):2337--2358, 2002.
\newblock Recent developments in the mathematical theory of water waves
  (Oberwolfach, 2001).

\bibitem{MZ09}
M.~Ming and Z.~Zhang.
\newblock Well-posedness of the water-wave problem with surface tension.
\newblock {\em J. Math. Pures Appl. (9)}, 92(5):429--455, 2009.

\bibitem{Na74}
V.~I. Nalimov.
\newblock The {C}auchy-{P}oisson problem.
\newblock {\em Dinamika Splo\v sn. Sredy}, (18):104--210, 254, 1974.

\bibitem{NS23}
H.~Q. Nguyen and W.~A. Strauss.
\newblock Proof of modulational instability of {S}tokes waves in deep water.
\newblock {\em Comm. Pure Appl. Math.}, 76(5):1035--1084, 2023.

\bibitem{Pazy83}
A.~Pazy.
\newblock {\em Semigroups of linear operators and applications to partial
  differential equations}, volume~44 of {\em Applied Mathematical Sciences}.
\newblock Springer-Verlag, New York, 1983.

\bibitem{PS04}
R.~L. Pego and S.~M. Sun.
\newblock On the transverse linear instability of solitary water waves with
  large surface tension.
\newblock {\em Proc. Roy. Soc. Edinburgh Sect. A}, 134(4):733--752, 2004.

\bibitem{Pl64}
V.~A. Pliss.
\newblock A reduction principle in the theory of stability of motion.
\newblock {\em Izv. Akad. Nauk SSSR Ser. Mat.}, 28:1297--1324, 1964.

\bibitem{PS16}
J.~Pr\"uss and G.~Simonett.
\newblock {\em Moving interfaces and quasilinear parabolic evolution
  equations}, volume 105 of {\em Monographs in Mathematics}.
\newblock Birkh\"auser/Springer, [Cham], 2016.

\bibitem{RT11}
F.~Rousset and N.~Tzvetkov.
\newblock Transverse instability of the line solitary water-waves.
\newblock {\em Invent. Math.}, 184(2):257--388, 2011.

\bibitem{Sc05}
B.~Schweizer.
\newblock On the three-dimensional {E}uler equations with a free boundary
  subject to surface tension.
\newblock {\em Ann. Inst. H. Poincar\'e{} C Anal. Non Lin\'eaire},
  22(6):753--781, 2005.

\bibitem{SZ03}
J.~Shatah and C.~Zeng.
\newblock Orbits homoclinic to centre manifolds of conservative {PDE}s.
\newblock {\em Nonlinearity}, 16(2):591--614, 2003.

\bibitem{SZ08a}
J.~Shatah and C.~Zeng.
\newblock Geometry and a priori estimates for free boundary problems of the
  {E}uler equation.
\newblock {\em Comm. Pure Appl. Math.}, 61(5):698--744, 2008.

\bibitem{SZ08b}
J.~Shatah and C.~Zeng.
\newblock A priori estimates for fluid interface problems.
\newblock {\em Comm. Pure Appl. Math.}, 61(6):848--876, 2008.

\bibitem{SZ11}
J.~Shatah and C.~Zeng.
\newblock Local well-posedness for fluid interface problems.
\newblock {\em Arch. Rational Mech. Anal.}, 199:653–705, 2011.

\bibitem{SX21}
A.~Sun and J.~Xue.
\newblock Initial perturbation of the mean curvature flow for closed limit
  shrinker.
\newblock {\em arXiv:2104.03101}, 2021.

\bibitem{IV92}
A.~Vanderbauwhede and G.~Iooss.
\newblock Center manifold theory in infinite dimensions.
\newblock In {\em Dynamics reported: expositions in dynamical systems},
  volume~1 of {\em Dynam. Report. Expositions Dynam. Systems (N.S.)}, pages
  125--163. Springer, Berlin, 1992.

\bibitem{Wu97}
S.~Wu.
\newblock Well-posedness in sobolev spaces of the full water wave problem in
  2-d.
\newblock {\em Invent. Math.}, 130:39--72, 1997.

\bibitem{Wu99}
S.~Wu.
\newblock Well-posedness in sobolev spaces of the full water wave problem in
  3-d.
\newblock {\em J. AMS.}, 12(2):445--495, 1999.

\bibitem{Yo83}
H.~Yosihara.
\newblock Capillary-gravity waves for an incompressible ideal fluid.
\newblock {\em J. Math. Kyoto Univ.}, 23(4):649--694, 1983.

\bibitem{Ze00}
C.~Zeng.
\newblock Homoclinic orbits for a perturbed nonlinear {S}chr\"odinger equation.
\newblock {\em Comm. Pure Appl. Math.}, 53(10):1222--1283, 2000.

\bibitem{ZZ08}
P.~Zhang and Z.~Zhang.
\newblock On the free boundary problem of three-dimensional incompressible
  {E}uler equations.
\newblock {\em Comm. Pure Appl. Math.}, 61(7):877--940, 2008.

\end{thebibliography}

\end{document}